%% file: Morse-Bott.tex
\documentclass[reqno,11pt]{amsart}
 
\usepackage{amsmath}
\usepackage{amsthm}  
\usepackage{verbatim}
\usepackage{enumerate} 
\usepackage{mathtools}  
\usepackage{amssymb}
\usepackage{mathrsfs}
\usepackage[top=1.5in, bottom=1.5in, left=0.7in, right=0.7in]{geometry}  
\usepackage[colorlinks]{hyperref}
\hypersetup{
	colorlinks,
	citecolor=blue,
	linkcolor=red
}   
\numberwithin{equation}{section}
\usepackage{xypic}
\usepackage{bm}
\usepackage{tikz}
\usepackage{xcolor}
\usepackage{float}
\usetikzlibrary{shapes,snakes}
\usetikzlibrary{arrows,positioning} 
\tikzset{
	>=stealth',
	punkt/.style={
		rectangle,
		rounded corners,
		draw=black, very thick,
		text width=12em,
		minimum height=2em,
		text centered}, 
	pil/.style={
		->,
		thick,
		shorten <=2pt,
		shorten >=2pt,}
}
\usepackage{graphicx}
\usepackage{lmodern}
\usepackage{array}
\theoremstyle{plain}

\newtheorem{theorem}{\textbf{Theorem}}[section]
\newtheorem{definition}[theorem]{\textbf{Definition}}
\newtheorem*{definition*}{\textbf{Definition}}
\newtheorem{proposition}[theorem]{\textbf{Proposition}}
\newtheorem{lemma}[theorem]{\textbf{Lemma}}
\newtheorem{definition/lemma}[theorem]{\textbf{Definition/Lemma}}
\newtheorem{definition/proposition}[theorem]{\textbf{Definition/Proposition}}
\newtheorem*{theorem*}{\textbf{Theorem}}
\newtheorem{claim}[theorem]{\textbf{Claim}}
\newtheorem*{claim*}{\textbf{Claim}}

\newtheorem{corollary}[theorem]{\textbf{Corollary}}
\newtheorem{remark}[theorem]{\textbf{Remark}}
\newtheorem{example}[theorem]{\textbf{Example}}

\def\N{{\mathbb N}}
\def\R{\mathbb{R}}
\def\Z{{\mathbb Z}}

\def\C{{\mathbb C}}

\newcommand{\CP}{\mathbb{C}\mathbb{P}}

\def\a{\alpha}
\def\b{\beta}

\def\i{\iota}

\def\cB{{\mathcal B}}
\def\cC{{\mathcal C}}
\def\cD{{\mathcal D}}
\def\cE{{\mathcal E}}
\def\cF{{\mathcal F}}

\def\cH{{\mathcal H}}
\def\cI{{\mathcal I}}

\def\cK{{\mathcal K}}

\def\cM{{\mathcal M}}

\def\cP{{\mathcal P}}

\def\cT{{\mathcal T}}

\def\cY{{\mathcal Y}}
\def\cZ{{\mathcal Z}}

\def\fF{{\mathfrak F}}
\def\fG{{\mathfrak G}}
\def\fH{{\mathfrak H}}
\def\fI{{\mathfrak I}}

\def\fP{{\mathfrak P}}

\def\fT{{\mathfrak T}}

\def\sY{{\mathscr Y}}

\def\rD{{\rm D}}

\def\rd{{\rm d}}

\def\sc{{\rm sc}}

\def\la{\langle\,}
\def\ra{\,\rangle}

\def\AB{{\rm AB}}
\def\Fuk{{\rm F}}
\def\BC{{\rm BC}}

\DeclareMathOperator{\Ima}{im}
\DeclareMathOperator{\supp}{supp}
\DeclareMathOperator{\ind}{ind}

\DeclareMathOperator{\id}{id}

\DeclareMathOperator{\Mor}{Mor}
\DeclareMathOperator{\Obj}{Obj}

\DeclareMathOperator{\graph}{graph}

\DeclareMathOperator{\sign}{sign}
\DeclareMathOperator{\rank}{rank}
\DeclareMathOperator{\vol}{vol}
\DeclareMathOperator{\tr}{Tr}
\DeclareMathOperator{\Crit}{Crit}

\def\holim{\mathrm {holim} }
\allowdisplaybreaks
\makeatletter
\def\@tocline#1#2#3#4#5#6#7{\relax
	\ifnum #1>\c@tocdepth 
	\else
	\par \addpenalty\@secpenalty\addvspace{#2}%
	\begingroup \hyphenpenalty\@M
	\@ifempty{#4}{%
		\@tempdima\csname r@tocindent\number#1\endcsname\relax
	}{%
	\@tempdima#4\relax
}%
\parindent\z@ \leftskip#3\relax \advance\leftskip\@tempdima\relax
\rightskip\@pnumwidth plus4em \parfillskip-\@pnumwidth
#5\leavevmode\hskip-\@tempdima
\ifcase #1
\or\or \hskip 1em \or \hskip 2em \else \hskip 3em \fi%
#6\nobreak\relax
\hfill\hbox to\@pnumwidth{\@tocpagenum{#7}}\par
\nobreak
\endgroup
\fi}
\makeatother
\title{Morse-Bott cohomology from homological perturbation theory}
\author{Zhengyi Zhou}

\begin{document}
\maketitle
\begin{abstract}
	In this paper, we construct cochain complexes generated by the cohomology of critical manifolds in the abstract setup of flow categories for Morse-Bott theories under minimum transversality assumptions. We discuss the relations between different constructions of Morse-Bott theories. In particular, we explain how homological perturbation theory is used in Morse-Bott theories, and both our construction and the cascades construction can be interpreted as applications of homological perturbations. In the presence of group actions,  we construct cochain complexes for the equivariant theory. Expected properties like the independence of approximations of classifying spaces and the existence of the action spectral sequence are proven. We carry out our construction for Morse-Bott functions on closed manifolds and prove it recovers the regular cohomology. We outline the project of combining our construction with the polyfold theory.
\end{abstract}
\setcounter{tocdepth}{1}
\tableofcontents

\input{s1}

\input{s2}

\input{s3}
\input{s4}

\input{s5}
\input{s6}
\input{s7}

\input{s8}
\input{s9}

\input{app}

\bibliographystyle{plain} 
\bibliography{ref}

\end{document}

%% file: s1.tex
\section{Introduction}
Morse theory \cite{morse1934calculus} enables one to analyze the topology of a manifold by studying Morse functions on that manifold, or more explicitly by studying critical points and gradient flow lines. Although Morse functions are generic among all differentiable functions, sometimes it is more convenient to work with more special functions. Morse-Bott functions were introduced by Bott in \cite{bott1954nondegenerate} as generalizations of Morse functions and have proven to be extremely useful for studying spaces in the presence of symmetries \cite{bott1956application,bott1958applications}. Inspired by ideas of Witten \cite{witten1982supersymmetry} and Gromov \cite{gromov1985pseudo}, Floer generalized Morse theory to various infinite dimensional settings \cite{floer1988instanton,floer1988morse,floer1989symplectic,floer1989witten}. 
Now there are many invariants in symplectic and contact geometry and low dimensional topology based on Floer's construction. Many of them have a ``Morse theoretical" background, e.g.\cite{dostoglou1994self, kronheimer2007monopoles, ozsvath2004holomorphic, seidel2008biased}. Many other invariants \cite{eliashberg2000introduction,fukaya2010lagrangian,seidel2008fukaya} are closely related to Morse theory. Usually, invariants are defined in the ``Morse" case, i.e.\ critical points are isolated, and invariants or structural maps are defined by counting zero dimensional moduli spaces. However, in many cases, it is more convenient to study the Morse-Bott case, where we need to ``count" higher dimensional moduli spaces, since there are several benefits of working with the Morse-Bott case: (1) Morse-Bott functions usually reflect some extra symmetries of the problem, and computations in Morse-Bott theory are usually simpler because of the extra symmetries \cite{bourgeois2002morse,diogo2018morse}; (2) Morse-Bott theories appear in equivariant theories \cite{austin1995morse, bourgeois2009exact,lin2014morse}.

There are two aspects of Morse-Bott theories in applications. First, we need to construct compactified moduli spaces of gradient flow lines/Floer trajectories from one critical manifold to another critical manifold. Moreover, we need the moduli spaces to be equipped with smooth structures so that the moduli spaces are manifolds or orbifolds. To achieve that, there are mainly three methods. (1) Geometric perturbations \cite{mcduff2012j}, where one perturbs geometric data like almost compact structures or metrics. Such methods were used in many classical treatments of Floer theories. (2) Kurunishi method \cite{fukaya2010lagrangian,joyce2014new,mcduff2017smooth}. (3) Polyfold method \cite{hofer2017polyfold}. There are many other methods for specific geometric settings \cite{cieliebak2007symplectic,ionel2013natural,li1996virtual,ruan1995mathematical} and algebraic treatments \cite{pardon2016algebraic}. Second, from critical manifolds and compactified moduli spaces of gradient flow lines/Floer trajectories we need to construct cochain complexes. This paper focuses on the second part, in particular, we explain how to count when the dimension of moduli spaces is positive assuming the moduli spaces are reasonably nice. However, we will discuss the transversality problem for the finite dimensional Morse-Bott theory in \S \ref{s7} using geometric perturbations and outline the polyfold method for general case in \S \ref{poly}.

\subsection{Cohomology of flow categories}
It turns out that all critical manifolds and \emph{compactified} moduli spaces from a Morse-Bott setting determine a category, namely a flow category, which was first introduced by Cohen, Jones and Segal in \cite{cohen1995floer} to organize all the moduli spaces of flow lines in Morse/Floer theories. Roughly speaking, the objects of a flow category come from critical points, and the morphisms are (broken) flow lines.

In the Morse case, the cochain complex is constructed by counting points in the zero-dimensional moduli spaces (the morphism space). However in a general Morse-Bott case, higher dimensional moduli spaces should contribute non-trivially to the construction. Given a general abstract Morse-Bott flow category, there are several methods to get a chain or cochain complex, namely
\begin{enumerate}
	\item Austin-Braam's model \cite{austin1995morse}. The cochain complex is generated by differential forms of the critical manifolds, and the differential is defined by pullback and pushforward of differential forms through the compactified moduli spaces. 
	\item Fukaya's model \cite{fukaya1996floer}. The chain complex is generated by certain subcomplex of the singular chain complex of the critical manifolds, and the differential is defined by pushforward and pullback of singular chains through the compactified moduli spaces. 
	\item Cascades model by Bourgeois \cite{bourgeois2002morse} and Frauenfelder \cite{frauenfelder2004arnold}. The cochain complex is generated by Morse cochain complexes of critical manifolds after we assign suitable Morse functions to each critical manifold. The differential is defined by counting ``cascades".\footnote{Strictly speaking, the original cascades model \cite{bourgeois2002morse,frauenfelder2004arnold} was phrased using homological conventions, the above mentioned cochain complex is the linear dual of the homological cascades model.}
\end{enumerate}
All of the methods above have to make some assumptions on the compactified moduli spaces of Morse/Floer trajectories. In the Morse-Bott setting, Morse/Floer trajectories can break into pieces with ends matched. Hence the boundary of a compactified moduli spaces consists of fiber products over critical manifolds. The minimal transversality requirement is that these fiber products are cut out transversely. Such a requirement is natural using any reasonable virtual technique. In this paper, we work in the context of flow categories, under such fiber products transversality assumptions. 

The first goal of this paper is to unify the three methods and provide a simple and clean construction, called the \emph{minimal Morse-Bott construction}, to every Morse-Bott flow category. Moreover, we will explain the following guiding principle in Morse-Bott constructions. 
\begin{claim*}\label{claim:phi}
	Formal applications of the homological perturbation lemma tend to give well-defined constructions. 
\end{claim*}
It turns out that both cascades and the minimal construction fit into this principle, and the relations are described in the following diagram.
\begin{center}
	\begin{tikzpicture}[node distance=1cm, auto,]  
	\node(mark2){};
	\node[punkt,left=of mark2](c){Cascades construction};
	\node[punkt, right=of mark2](m){Minimal construction};
	\node[ above=of mark2](mark){Homological pertubation lemma}
	edge[pil] (c.north)
	edge[pil] (m.north){};
	\node[punkt, above=of mark](infty){Austin-Braam's model/Fukaya's model}
	edge[pil] (mark.north);
	\end{tikzpicture}
\end{center}

In applications of the homological perturbation theory, one needs to choose some perturbation data (projections and  homotopies). For the cascades model, the projections and homotopies are provided by Harvey and Lawson's work \cite{harvey2001finite} on Morse theory. The minimal construction is based on a more direct construction of the projections and homotopies, e.g.\ one can choose the projection to harmonic forms and the associated Green operator (as the homotopy) as the perturbation data. The principle above also works for structures more general than a ``linear structure" like flow categories, as long as the all the relevant moduli spaces satisfy the fiber products transversality assumption, e.g.\ \cite{cieliebak}. However, this has gone beyond the scope of the current paper.  

Our main theorem is that, with suitable orientations, one can associate a well-defined cochain complex generated by the cohomology of the object space (i.e.\ critical manifolds) to a flow category.
\begin{theorem*}
	To every oriented flow category, we can assign a minimal Morse-Bott cochain complex $(\BC,d_{\BC})$ over $\R$ generated by the cohomology of the object space (with a suitable completion) in a functorial way. 
\end{theorem*}
Of course, the theorem here bears no meaning yet. We point out here that (1) when the flow category arises from a Morse-Bott function on a closed manifold, the cohomology of the minimal Morse-Bott cochain complex is the cohomology of the manifold. (2) When the flow category arises from a Morse case, i.e.\ critical points are non-degenerate, hence isolated, the cochain complex is the usual cochain complex with differential defined by counting rigid points in the morphism space. (3) There are analogous constructions for continuation maps and homotopies, which, in applications, will yield invariance w.r.t.\ various auxiliary geometric data (Hamiltonians, almost complex structures, metrics etc.).

The construction provides explicit formulae for how higher dimensional moduli spaces contribute in the construction, in particular, there are error correcting terms from moduli spaces related to the boundaries and corners. Like the cascades construction, to write down an explicit cochain complex, we need to make some choices on each critical manifold. One of the advantages of the minimal construction is that the choices do not require any compatibility condition with the morphism space (moduli spaces). The cohomology theory on the level of flow categories in this paper simplifies many geometric constructions including products (\S\ref{ss:product}), quotients (\S \ref{ss:quotient}) and fibrations (\S \ref{ss:Gysin}), as such constructions are natural on the level of flow categories.

The theorem above is the simplest version. We also discuss several generalizations in this paper: (1) The critical manifold $C_i$ can be non-compact; (2)  The critical manifold $C_i$ can be equipped with local systems and does not have to be orientable; (3) It is not necessary that the cochain complex is generated by the cohomology, any finite dimensional subspace of differential forms satisfying a cohomological relation is sufficient. Such flexibility allows us to prove a Gysin exact sequence for sphere bundles over flow categories. In \cite{ring}, we use the Gysin exact sequence to show that any exact filling of a simply connected flexibly fillable contact manifold has the same cohomology ring structure on even degrees. 

\subsection{Equivariant theories}
The second goal of this paper is developing an equivariant theory on the level of flow categories, which would serve as a model for defining equivariant Floer theory. When there is a group $G$ symmetry on the Morse-Bott theory, the cohomology theory should be enriched to a $G$-equivariant theory. One typical method is approximating the homotopy quotient. Bourgeois and Oancea  \cite{bourgeois2013gysin} used a construction inspired by the cascades method to define the $S^1$-equivariant symplectic homology in this spirit. In our case, the homotopy quotient construction is very natural on the level of flow categories.  Hence we can combine the Borel's construction and our minimal construction and realize the equivariant cochain complex as a homotopy limit.
\begin{theorem*}
	Assume a compact Lie group $G$ acts on an oriented flow category $\cC$ and preserves the orientations. Then there is a cochain complex $(\BC^G,d^G_{\BC})$, whose homotopy type is unique, i.e.\ independent of all the choices in the construction, in particular, the choice of finite dimensional approximations of the classifying space $EG\to BG$.  
\end{theorem*}

\subsection{Constructions of flow categories}
The remaining problem of applying the minimal construction in applications is  constructing a flow category. In \S \ref{s7}, we construct flow categories for the finite dimensional Morse-Bott theory using geometric methods. In general, geometric perturbations, i.e.\ perturbing metrics in Morse theory and perturbing almost complex structures in Floer theory,  may not be enough to guarantee the transversality assumption, hence one needs to apply some abstract perturbations. In fact, our minimal construction is applicable to the polyfold theory. We can enrich a flow category, i.e.\ a system of manifolds to a system of polyfolds with  sc-Fredholm sections, and the boundaries/corners of the polyfolds come from transverse fiber products of polyfolds. We will refer this system as a polyflow category. Then we can find a coherent perturbation scheme and apply the abstract perturbation theorem for polyfolds from \cite{hofer2017polyfold} to get a flow category. In the presence of a group action, the theorem  above on equivariant cohomology requires $G$-equivariant transversality. But we know that $G$-equivariant transversality is typically obstructed. In general, we need to apply the Borel construction using quotient theorems from \cite{quotient} to the whole polyflow category instead of the flow category.

\subsection{Organization of the paper}
This paper is organized as follows. \S \ref{perb} discusses the motivation of the minimal construction from homological perturbation theory and interprets the cascades construction as an example of an application of the homological perturbation theory. \S \ref{morsebott} defines the minimal cochain complex as well as continuation maps and homotopies explicitly, and proves that they satisfy the desired properties.  \S \ref{specseq} discusses the action spectral sequence. \S \ref{s5} explains how the orientations used in \S \ref{morsebott} arise in Morse/Floer theories. \S \ref{s5} also generalizes the construction to the case with local systems and non-orientable manifolds. \S \ref{s6} generalizes the construction to flow categories with non-compact critical manifolds and also provides a more general setup which allows us to prove statements like the Gysin exact sequence. \S \ref{equi} discusses the equivariant theory. \S \ref{s7} is devoted to the Morse-Bott theory on finite dimensional manifolds (both open and closed) and proves the minimal construction recovers the cohomology of the underlying manifold. \S \ref{poly} outlines the project of combining our construction with the polyfold theory. 

\subsection*{Acknowledgements}
The results presented here are part of my Ph.D. thesis; I would like to express my deep gratitude to my thesis advisor Katrin Wehrheim for guidance, encouragement and enlightening discussions. I would like to thank Kai Cieliebak and Michael Hutchings for helpful conversations. The author is in debt to the anonymous referee for a thorough checking of this very long and technical paper and providing many helpful suggestions which improve the manuscript. Part of the writing was completed  during  my  stay  at  the  Institute  for  Advanced  Study  supported  by  the National Science Foundation under Grant No. DMS-1638352.  It is a great pleasure to acknowledge the institute for its warm hospitality. This paper is dedicated to the memory of Chenxue.

%% file: s2.tex
\section{Motivation From Homological Perturbation Theory}\label{perb}
\subsection{Differential topology notation}
We first set up some notation for manifolds with boundaries and corners, transversality theory of them and orientation conventions.
\subsubsection{Manifolds and submanifolds with boundaries and corners}
Unless stated otherwise, all manifolds considered in this paper are manifolds possibly with boundaries and corners \cite[Definition 1.6.1]{analysis}, i.e.\ for every point in the manifold, there is an open neighborhood diffeomorphic to an open subset of $\R^n_+$, where $\R_+:=[0,\infty)$. A \emph{closed} manifold is a \emph{compact} manifold \emph{without boundary}.
\begin{definition}\label{def:degenercy}
	Let $M$ be a manifold and $x\in M$ a point, by choosing a chart $\phi: \R_+^n\supset U\to M$ near $x\in M$,  the degeneracy index $d(x)$ of the point $x$ is defined to be $\#\{v_i| v_i=0 \}$, where $(v_1,\ldots,v_n)\in \R_+^n$ and $\phi(v_1,\ldots, v_n)=x\in M$.    
\end{definition}
The degeneracy index $d$ does not depend on the local chart $\phi$ \cite[Corollary 1.5.1]{analysis}.  For $i\ge 0$, we define the \textbf{depth-$\bm{i}$ boundary $\bm{\partial_i M}$} to be 
\begin{equation}\label{eqn:depthi}
\partial_i M:=\left\{x\in M\left|d(x)=i \right.\right\}.
\end{equation}
Then $\partial_0M$ is the set of interior points of $M$. Note that all $\partial_iM$ are manifolds without boundary, and in most cases they are noncompact. Submanifolds of manifolds should be compatible with structures defined in \eqref{eqn:depthi}, i.e. we have the following.
\begin{definition}\label{def:submanifold}
	A closed subset $N\subset M$ is a submanifold of $M$ iff $N$ is a manifold, such that the inclusion $N\to M$ is a smooth embedding and for all $i\ge 0$, $\partial_iN=N\cap \partial_iM$. In other words, for every $x\in N$, $(M,N)$ near $x$ is locally modelled on $(\R_+^{k}\times \R^{n-k},\R_+^k\times \R^{n-m}\times \{0\}^{m-k})$ near $0$.
\end{definition}
An instant corollary is that if $N$ is submanifold of $M$ and $M$ is submanifold of $K$, then $N$ is also a submanifold of $K$. In this paper, unless stated otherwise, we will only consider submanifolds defined as above. In particular, when $M$ has no boundary, a submanifold does not have boundary either. Note that $\partial_iM$ is not a submanifold of $M$ in the sense of Definition \ref{def:submanifold} unless $\dim M = 0$. 

\begin{remark}\label{rmk:submanifold}
A few remarks regarding the notion of manifolds with boundaries and corners and their submanifolds are in order.
\begin{enumerate}
    \item Some authors require, in the definition of manifolds with boundaries and corners, the additional property that faces (the closure of connected components of  $\partial_1M$) are submainfiolds (not in the sense of Definition \ref{def:submanifold} but a weaker sense, e.g.\ $t$-submanifolds in \cite[Definition 1.7.3]{analysis}), e.g.\ \cite[Definition 1.8.5]{analysis}. Such a definition will rule out the ``tear drop" shape. Although we do not use this definition, we note here that in Floer/Morse cohomology theories,  which are the main applications of the abstract construction in this paper, the compactified moduli spaces of Floer/Morse trajectories are manifolds with boundaries and corners in this stronger sense. However, if we were to consider more general algebraic structures (i.e.\ more complicated than a cochain complex) arising from the compactified moduli spaces of pseudo-holomorphic curves, a ``tear drop” moduli space may appear, e.g.\ \cite[Figure 8]{VFC}.
    \item There are different notions of submanifolds in a manifold with boundaries and corners depending on the purpose. For example, there are notions of t-, d-, and p-submanifolds \cite[\S 1.7]{analysis} depending on the compatibility of tangent spaces at the boundary.  However, our notion of submanifolds is stronger than any of that, as we require that $l=k$ in the definition of p-submanifolds \cite[Definition 1.7.4]{analysis} (which is equivalent to requiring that $(M,N)$ near $x$ is locally modelled on $(\R_+^{k}\times \R^{n-k},\R_+^k\times \R^{n-m}\times \{0\}^{m-k})$ near $0$ for $x\in N$.).
    \item Submanifolds in the sense of Definition \ref{def:submanifold} arise naturally as zero sets of sections $s:M\to E$ of a vector bundle $E$ over a manifold $M$ with boundaries and corners, if $s|_{\partial_iM}$ is transverse to $0$ for all $i$. This can be viewed as a prototype of how compactified moduli spaces of Floer cylinders/holomorphic curves can be equipped with a structure of manifold with boundaries and corners in the polyfold perspective. The transversality requirements above are equivalent to that $s$ is in a general position \cite[Definition 5.3.9]{hofer2017polyfold}.
\end{enumerate}
\end{remark}

\begin{definition}\label{def:trans}
	Transversality is defined as follows to accommodate the boundary and corner structures.
	\begin{enumerate}	
		\item\label{trans:1} Let $C$ be a manifold \emph{without boundary}, $B$ a submanifold of $C$ and $M$ a manifold possibly with boundaries and corners. A smooth map $f:M\to C$ is \textbf{transverse to $\bm{B}$}, iff $f|_{\partial_iM}\pitchfork B$ for all $i$ in the classical sense, i.e.\ $\rD f_x(T\partial_iM)+T_{f(x)}B=T_{f(x)} C$ for all $x\in \partial_i M$ such that $f(x)\in B$. 
		\item\label{trans:2} Let $M$ be a manifold and $N_1,N_2$ two submanifolds, then we say \textbf{$\bm{N_1}$ is transverse to $\bm{N_2}$} iff for all $i\ge 0$ and every $x\in \partial_iN_1\cap \partial_iN_2$, we have $\partial_iN_1$ is transverse to $\partial_iN_2$ in $\partial_iM$ in the classical sense, i.e. $T_x\partial_iN_1+T_x\partial_iN_2=T_x\partial_iM$.
	\end{enumerate}	
\end{definition}

\begin{proposition}\label{prop:trans}
    We have the following implicit function theorems.
    \begin{enumerate}
        \item  Let $C$ be a manifold \emph{without boundary} and $B$ be a submanifold. Given a manifold  $M$ along with a smooth map $f$, assume that $f:M\to C$ is transverse to $B$ in the sense of Definition \ref{def:trans} \eqref{trans:1}, then $f^{-1}(B)$ is submanifold of $M$ (in the sense of Definition \ref{def:submanifold}).
        \item Let $N_1,N_2$ be two submanifolds of a manifold $M$ such that $N_1$ is transverse to $N_2$ in the sense of Definition \ref{def:trans} \eqref{trans:2}, then $N_1\cap N_2$ is a submanifold of $M$. The codimension of $N_1\cap N_2$ is the sum of codimensions of $N_1$ and $N_2$.
    \end{enumerate}
\end{proposition}
\begin{proof}
The first claim is standard. We give a sketch of the second claim using the first claim (but not the ``obvious" one, as we can not assume $C=M,B=N_2$ in the first claim since $M,N_2$ have nonempty boundaries). Let $x\in N_2$ with $d(x)=k$, we may assume the pair $(M,N_2,x)\cap U$ for an open set $U\subset M$ is modelled on $(\R_+^k\times \R^{n-k}, \R_+^k\times \R^{m-k}\times \{0\}^{n-m},0)$ following Remark \ref{rmk:submanifold}. We consider $f:N_1\cap U\to \R^{n-m}$ the projection to the last $n-m$ coordinates. It is straightforward to check that transversality in Definition \ref{def:trans} \eqref{trans:2} implies (is actually equivalent to) that $0$ is a regular value of $f$. Since $f^{-1}(0)=N_1\cap N_2\cap U$, we endow $N_1\cap N_2$ a structure of submanifold with boundaries and corners in $N_1$ by the first claim, hence a structure of submanifold with boundaries and corners in $M$.
\end{proof}

Since measure-zero sets on differentiable manifolds are well-defined and our construction is based on integration, errors over a measure-zero set can be tolerated. In particular, we have the following useful notion.
\begin{definition}\label{def:measure}
	Let $M,N$ be two manifolds. A smooth map $f:M\to N$ is a \textbf{diffeomorphism up to zero-measure} iff there exist measure zero closed sets $M_1\subset M, N_1\subset N$, such that $f|_{M\backslash M_1}: M\backslash M_1\rightarrow N\backslash N_1$ is a diffeomorphism. 
\end{definition}

\subsubsection{Orientations}\label{notori}
This paragraph fixes our orientation conventions. Given an \textit{oriented} vector bundle $E$ over a manifold $M$,  the determinant bundle $\det E$ is a trivial line bundle.  $\det E$ can be reduced further to a  trivial $\Z/2$ bundle $\sign E$. Moreover, we can assign  $\sign E $ with a $\Z/2$ grading $|\sign E|=\rank E$. The fiber of $\sign E$ over $x\in M$ is the set of equivalence classes of ordered bases $[(e_1,\ldots,e_n)]$ of the fiber $E_x$, where $(e_1,\ldots,e_n)$ is equivalent to $(e'_1,\ldots,e'_n)$ iff the transformation matrix between them has positive determinant. Then the orientation of $E$ induces a continuous section  of $\sign E$, and we use $[E] \in \Gamma(\sign E)$ to denote the section induced by the orientation.

Given two vector bundles $E,F$ over $M$, we fix a bundle isomorphism: 
$$m_{E,F}: \begin{array}{rcl} \sign(E)\otimes_{\Z/2}\sign(F) & \to & \sign(E\oplus F), \\
 {}[(e_1,\ldots, e_n)] \otimes [(f_1,\ldots, f_m)] & \mapsto & [(e_1,\ldots, e_n, f_1,\ldots, f_m)].        \end{array} $$
Therefore orientations  $[E]$ and $[F]$ determine an orientation of $E\oplus F$ through $m_{E,F}$, hence we denote the induced orientation by
\begin{equation}\label{eqn:oriproduct}
[E][F]:=m_{E,F}([E],[F]).
\end{equation}
Since $[(e_1,\ldots, e_n, f_1,\ldots, f_m)]=(-1)^{nm}[(f_1,\ldots, f_m, e_1,\ldots, e_n)]$, we have:
$$[E][F]=(-1)^{|F||E|}[F][E].$$ 

\begin{definition}\label{def:orinot}
	For simplicity of notation, we introduce the following.
	\begin{itemize}
		\item A manifold $M$ is oriented iff the tangent bundle $TM$ is oriented, and we use $[M]$ to denote the orientation.
		\item $\partial[M]$ denotes the induced orientation\footnote{In the usual sense such that Stokes' theorem holds without extra sign.} on the depth-$1$ boundary $\partial_1 M$ for an oriented manifold $M$. 
		\item Let $E\to M$ and $F\to N$ be two oriented vector bundles, we use $[E]+[F]$ to denote the induced orientation on $E\cup F\to M\cup N$.  And we use $-[E]$ to denote the opposite orientation.
		\item Unless stated otherwise, the product $M\times N$ is oriented by the product orientation of $M$ and $N$ and we use $[M\times N]$ to denote the product orientation. Then we have
		\begin{equation}
			\partial[M\times N]=\partial[M]\times [N]+(-1)^{\dim M}[M]\times\partial [N].
		\end{equation}
		\item If $f:M\to N$ is a diffeomorphism, we use $f_*[M]$ as the orientation on $N$ induced by $\rD f:TM\to TN$ and $[M]$. 
		\item Let $E\to N$ be an oriented vector bundle and $f:M\to N$ a smooth map, then the bundle map $f^*E\to E$ induces a bundle map $\sign(f^*E)\to \sign(E)$. Through this map, the orientation $[E]$ induces an orientation on $f^*E$ over $M$, the induced orientation is denoted by $f^*[E]$.
	\end{itemize} 
\end{definition}

\begin{example}\label{ex:ori}
	Let $C$ be a closed oriented manifold. We now explain our orientation convention for the normal bundle $N$ of the diagonal $\Delta \subset C\times C=C_1\times C_2$ using the notation introduced in Definition \ref{def:orinot}. 
	$\Delta$ is oriented by the condition ${\pi_1}_*[\Delta]=[C_1]$\footnote{It is  equivalent to ${\pi_2}_*[\Delta]=[C_2]$.}, where $\pi_1:C_1\times C_2 \to C_1$ is the projection. Then there exists a unique orientation of $N$, such that when restricted to $\Delta$ we have
	$$[\Delta][N]=[TC_1][TC_2]|_{\Delta}.$$ 
	For simplicity, we suppress the restrictions and the subscripts\footnote{We will never switch the order of the two copies of $C$ throughout this paper.}, and the equation becomes
	\begin{equation}\label{oricri}
	[\Delta][N]=[C][C] \text{ or equivalently } [N][\Delta]=(-1)^{(\dim C)^2}[C][C].
	\end{equation}
	\eqref{oricri} determines our orientation convention for the normal bundle $N$ in this paper.
\end{example}

\subsection{Flow categories.}
Flow categories was introduced by Cohen-Jones-Segal \cite{cohen1995floer} to organize the moduli spaces in Floer (co)homology and were used to construct a stable homotopy type for Floer theories. Our construction will be based on the concept of flow categories, hence we recall the definition first.
\begin{definition}\label{def:flow}
	A \textbf{flow category} is a small category $\cC$ with the following properties.
	\begin{enumerate}
		\item\label{F1} The objects space $\Obj(\cC)=\sqcup_{i\in \Z} C_i$ is a disjoint union of \emph{closed}  manifolds $C_i$, i.e.\ $C_i$ is a compact manifold without boundary. The morphism space $\Mor(\cC)=\cM$ is a manifold. The source and target maps $s,t:\cM\to C$ are smooth. 
		\item\label{F2} Let $\cM_{i,j}$ denote $(s\times t)^{-1}(C_i\times C_j)$. Then $\cM_{i,i}=C_i$, corresponding to the identity morphisms and $s,t$ restricted to $\cM_{i,i}$ are identities. $\cM_{i,j}=\emptyset$ for $j<i$, and $\cM_{i,j}$ is a \emph{compact} manifold for $j>i$.
		\item\label{F3}  Let $s_{i,j},t_{i,j}$ denote $s|_{\cM_{i,j}}, t|_{\cM_{i,j}}$. For every strictly increasing sequence $i_0<i_1<\ldots<i_k$, $t_{i_0,i_1}\times s_{i_1,i_{2}}\times t_{i_1,i_2}\times \ldots \times s_{i_{k-1},i_k}:{\cM_{i_0,i_1}}\times {\cM_{i_1,i_2}}\times\ldots \times\cM_{i_{k-1},i_k}\to C_{i_1}\times C_{i_1}\times C_{i_2}\times C_{i_2}\times \ldots \times C_{i_{k-1}}\times C_{i_{k-1}}$ is transverse to the submanifold $\Delta_{i_1}\times \ldots \times \Delta_{i_{k-1}}$ in the sense of Definition \ref{def:trans}.  Therefore the fiber product ${\cM_{i_0,i_1}}\times_{i_1} {\cM_{i_1,i_2}}\times_{i_2}\ldots \times_{i_{k-1}}\cM_{i_{k-1},i_k}:=(t_{i_0,i_1}\times s_{i_1,i_{2}}\times t_{i_1,i_2}\times \ldots \times s_{i_{k-1},i_k})^{-1}(\Delta_{i_1}\times \Delta_{i_2}\times \ldots \times \Delta_{i_{k-1}})\subset {\cM_{i_0,i_1}}\times {\cM_{i_1,i_2}}\times\ldots \times\cM_{i_{k-1},i_k} $ is a submanifold by Proposition \ref{prop:trans}.
		\item\label{F4} The composition $m:{\cM_{i,j}}\times_j {\cM_{j,k}}\to \cM_{i,k}$ is a smooth  map, such that $$m:  \bigsqcup_{i < j < k} \cM_{i,j}\times_j \cM_{j,k}\to \partial \cM_{i,k} \text{ is a diffeomorphism up to zero-measure.} $$
	\end{enumerate}
\end{definition}

\begin{example}\label{ex:ex}
Fixing a Morse-Bott function $f$ on a closed manifold $M$, then there are finitely many critical values $v_1<\ldots<v_n$. Let $C_i$ denote the critical manifold corresponding to the critical value $v_i$ and $\cM_{i,j}$ the \emph{compactified} moduli spaces of \emph{unparametrized gradient flow lines} from $C_i$ to $C_j$. Since the function value increases along a gradient flow line, we have $\cM_{i,j}=\emptyset$ when $i>j$. The source map $s:\cM_{i,j}\to C_i$ and target map $t:\cM_{i,j}\to C_j$ are  defined to be the evaluation maps at the negative/positive end of the flow line in $\cM_{i,j}$. The composition map $m$ is the concatenation of flow lines. It's a folklore theorem that $\cM_{i,j}$ are smooth manifolds with boundaries and corners if one chooses a suitable metric, c.f.\ \cite{austin1995morse,fukaya1996floer} and \S \ref{s7}. Therefore $\{C_i,\cM_{i,j}\}$ form a flow category.  We emphasize here the subscript $i$ in $C_i$ has nothing to do with Morse-Bott indices. Similar constructions also exist in Floer theories, as long as there is a background ``Morse-Bott" functional and all the transversality conditions are met. For example, \cite{cohen1995floer} gave an explicit construction of the flow category for the Hamiltonian Floer cohomology theory on $\CP^n$, where the background Morse-Bott functional is the symplectic action functional with the Hamiltonian $H=0$.\footnote{\cite{cohen1995floer} used homological convention, which gave the opposite category of a flow category in the sense of Definition \ref{def:flow}.}  There are also flow categories without obvious background Morse-Bott functionals, for example, the flow category for Khovanov homology  \cite{lipshitz2014khovanov}. 
\end{example}

This paper associates a natural cochain complex to each (oriented) flow category in a functorial way. The main application would be defining Hamiltonian-Floer cohomology or Morse cohomology under Morse-Bott non-degenerate conditions. Although we will be discussing the \emph{abstract notion of flow categories}, it would be helpful to keep Example \ref{ex:ex} in mind. In view of this, with a bit abuse of notation, we will refer to elements of $\cM_{i,j}$ as Morse (or Floer) trajectories from $C_i$ to $C_j$. Inspired from Example \ref{ex:ex}, Condition \eqref{F2} in Definition \ref{def:flow} is usually the consequence of the existence of some background functional and the morphism space $\cM_{i,j}$ is the \emph{compactified} moduli space of ``gradient flow lines"\footnote{It could be Floer flow lines, which, strictly speaking, are not gradient flow lines.}, i.e.\ the space of possibly broken ``gradient flow lines". Condition \eqref{F3} is necessary for the smoothness of the composition map $m$. Roughly speaking, Condition \eqref{F4} is that the boundary of the morphism space is the space of nontrivial compositions of morphisms, although Condition \eqref{F4} is only about an essential portion of the correspondence. In applications, we can stratify $\cM_{i,j}$ in a cell-like manner by a poset similar to the construction in \cite{VFC} such that $m$ respects the structure, but we will not need that level of precision in this paper.


\begin{remark}\label{rmk:flow}
A few remarks on Definition \ref{def:flow} in applications are in order.
\begin{enumerate}
	\item A flow category is called \textbf{Morse}, if $C$ is a discrete set. Then the fiber product transversality becomes tautological, and it recovers the definition of a flow category in \cite{cohen1995floer} up to taking the opposite category.
	\item 	In the context of Floer theories, the moduli spaces may not be manifolds in general, but some weighted objects with local symmetries, e.g. weighted branched orbifolds in \cite{HWZ3}. Every argument in this paper holds for weighted branched orbifolds, since there is a well-behaved integration theory with Stokes' theorem \cite{hofer2010integration}.
	\item  	When the flow category comes from a Morse-Bott functional $f$, but $f$ is not single valued\footnote{For example, Hamiltonian-Floer cohomology on $(M,\omega)$ with  $\omega|_{\pi_2(M)}\ne 0$ has such property.}, we need to lift $f$ to $\tilde{f}$ over the cyclic cover \cite{cohen1995floer} to guarantee Condition \eqref{F2} in Definition \ref{def:flow}. Such modification was already reflected in the usual construction by introducing the Novikov coefficient.
	\item \label{compact_remark} In Definition \ref{def:flow}, we require $C_i$ to be compact and without boundary. However the compactness assumption can be dropped, i.e.\ $C_i$ could be a disjoint union of infinitely many closed manifolds or $C_i$ could have noncompact components\footnote{But those noncompact manifolds should have finite topology, see \S \ref{ss:prop} for details.}. In such generalizations,  compactness of $\mathcal{M}_{i,j}$ can be weakened to that the target maps $t:M_{i,j} \to C_j$ are proper\footnote{One can replace it by asking $s:M_{i,j}\to C_i$ to be proper, but it will result in a theory analogous to the compactly supported cohomology.}, see \S \ref{ss:prop} for details.
	\item For a background Morse-Bott function $f$, sometimes it is impossible to partition the critical manifolds by $\Z$ and in the order of increasing critical values, i.e.\ critical values may accumulate.  For example, Hamiltonian-Floer cohomology with Novikov coefficient will have this problem, if the symplectic form is irrational. However, Gromov compactness for the Hamiltonian-Floer equation implies that there is an action gap $\hbar$, such that there are no non-constant flow lines when the action difference (energy) is smaller than $\hbar$. Therefore we can still divide all the critical manifolds into groups indexed by $\Z$, such that there are no non-constant flow lines inside each group. Then the flow category can still be defined using the generalization in \eqref{compact_remark}.
	\item 	In most of this paper, we will work with oriented $C_i$, see Definition \ref{oridef}. This assumption can be dropped with the price of working with local systems. We discuss this generalization in \S \ref{s5}. 
	\item We point out here that the requirement of the partition of $\Obj(\cC)$ by $\Z$ is not necessary. We can certainly work with $\Obj(\cC)$ indexed by any set $I$, as long as we require that $\cM_{i,j}$ has only finitely many degenerations for any $i,j\in I$ and the finite set of degeneration configurations is equipped with a partial order, whose minimum elements are built from $\cM_{i,j}$ without boundary. This is precisely the setup in \cite[\S 7]{pardon2016algebraic} and satisfied by more general constructions in \cite{VFC}. When  $\Obj(\cC)$ is indexed by $\Z$ with the properties in Definition \ref{def:flow}, the set of degeneration configurations of $\cM_{i,j}$ is precisely the set of strictly increasing sequences $S:=\{i<\ldots<j\}$,  where the partial order is given by $S_1\le S_2$ iff $S_2\subset S_1$. Then the minimum element is $\{i<i+1<\ldots<j-1<j\}$, which corresponds to the fiber product of manifolds $\cM_{*,*+1}$ without boundary. However, this level of generalization does not add much to applications  we have in mind, hence we choose to work with the more down-to-earth version (Definition \ref{def:flow}) to avoid more complication in notation.
\end{enumerate}
\end{remark}	

Flow categories can be equipped with extra structures. For our construction, the most relevant structures are gradings and orientations. Given a flow category $\cC = \{C_i,\cM_{i,j}\}$, for the simplicity of notation, \emph{we assume through out this paper that $\dim \cM_{i,j}$  and $\dim C_i$ are well-defined}. This requirement usually holds when each $C_i$ has one component. 

\begin{remark}
	When $\dim C_i$ and $\dim \cM_{i,j}$ are not well-defined, then we need to work in a component-wise way. For example, if a function $f$ in Example \ref{ex:ex} is Morse and $C_i$ contains critical points of different Morse indices, then $\cM_{i,i+1}$ has multiple connected components of different dimension. This generalization only results in complexity of notation, it is straightforward to see that all proofs in this paper hold for such generalizations. The proofs presented in this paper can be viewed as formulae on one component.
\end{remark}
Let $m_{i,j}:= \dim \cM_{i,j}$ for $i<j$ and $c_i:=\dim C_i$. We formally define $m_{i,i} := c_i -1$. By  \eqref{F3}, \eqref{F4} of Definition \ref{def:flow} and Proposition \ref{prop:trans}, $t_{i,j}\times s_{j,k}:\cM_{i,j}\times \cM_{j,k}\to C_j\times C_j$ is transverse to $\Delta_j$ and an open dense part of $\cM_{i,j}\times_j\cM_{j,k}$ can be identified with part of the boundary of $\cM_{i,k}$. Then we have:
\begin{equation}
\label{dim} m_{i,j}+m_{j,k}-c_j+1=m_{i,k}, \quad  \forall i\le j\le k.
\end{equation}
\begin{definition}\label{def:grade}
 A flow category is \textbf{graded} if for each $i\in \Z$, there is an integer $d_i$, such that $d_i=d_j+c_j-m_{i,j}-1$ for all $i<j$. \footnote{When $\dim M_{i,j}$ or $\dim C_i$ are not well-defined, a grading is an assignment of integers to each component of $C_i$ satisfying similar relations.} We will refer to $\{d_i\}$ as the grading structure. Similarly we define a $\Z/k$ grading structure if $d_i\in \Z/k$ and the relation holds in $\Z/k$.
\end{definition}

\begin{remark}\label{rmk:grading}
The $\Z/k$ grading structure on a flow category is used to equip the Morse-Bott cochain complex with a $\Z/k$ grading. In the finite dimensional Morse-Bott theory, a $\Z$ grading structure exists, i.e.\ $d_i$ can be the dimension of the negative eigenspace of $\mathrm{Hess}(f)$ on $C_i$. For Hamiltonian-Floer cohomology, a $\Z/2$ grading structure always exists and a $\Z$ grading structure exists if the first Chern class of the symplectic manifold vanishes, then $d_i$ is related to the generalized Conley-Zehnder index \cite{robbin1993maslov}.
\end{remark}

Next, we define orientations on a flow category. Since $t_{i,j}\times s_{j,k}:\cM_{i,j}\times \cM_{j,k}\to C_j\times C_j$ is transverse to the diagonal $\Delta_j$,  the pullback $(t_{i,j}\times s_{j,k})^*N_j$ of the normal bundle $N_j$ of $\Delta_j$  by  $t_{i,j}\times s_{j,k}$ is the normal bundle of $\cM_{i,j}\times_j\cM_{j,k}:=(t_{i,j}\times s_{j,k})^{-1}(\Delta_j)$ in $\cM_{i,j}\times \cM_{j,k}$. If $N_j$ is oriented, then we can pull back this orientation to orient the normal bundle of $\cM_{i,j}\times_j\cM_{j,k}$. We define a coherent orientation on a flow category as follows.
\begin{definition}\label{oridef} 
	A \textbf{coherent orientation} on a flow category is an assignment of orientations for each $C_i$, $\cM_{i,j}$ and $\cM_{i,j}\times_j \cM_{j,k}$, such that the following holds.
	\begin{enumerate}
		\item The normal bundle $N_i$ of $\Delta_i \subset C_i\times C_i$ is oriented by $[N_i][\Delta_i]=(-1)^{c_i^2}[C_i][C_i]$ as in Example \ref{ex:ori}. 
		\item\label{cond:ori2} $(t_{i,j}\times s_{j,k})^*[N_j][\cM_{i,j}\times_j \cM_{j,k}]=(-1)^{c_jm_{i,j}}[\cM_{i,j}][\cM_{j,k}].$
		\item \label{cond:ori3}	$\partial[\cM_{i,k}]=\sum_j (-1)^{m_{i,j}}m\left([\cM_{i,j}\times_j \mathcal{M}_{j,k}]\right).$\footnote{More precisely, the relation holds on where $m$ is a diffeomorphism. }
	\end{enumerate}
	Or one can combine \eqref{cond:ori2} and \eqref{cond:ori3} as 
		$$(t_{i,j}\times s_{j,k})^*[N_j] m^{-1}\left(\partial[\cM_{i,k}]|_{m(\cM_{i,j}\times_j \cM_{j,k})}\right)=(-1)^{(c_j+1)m_{i,j}}[\cM_{i,j}][\cM_{j,k}].$$
\end{definition}
\begin{remark}\label{rmk:orientation}
 Orientation conventions are by no means unique, however they typically differ by a global change. For example, in the context of Morse theory, Definition \ref{def:def}  differs from \cite{MR2748215} by an opposite sign on the orientation of every $\cM_{i,j}$. We point out here, although our orientation conventions for fiber products are different from \cite{joyce2009manifolds}, our conventions also enjoy the associativity \cite[Proposition 7.5(a)]{joyce2009manifolds} hence the uniqueness property in \cite[Remark 7.6 (iii)]{joyce2009manifolds} holds.
\end{remark}
We will discuss how coherent orientations arise in applications in \S \ref{orientation}.  When the flow category is oriented as in Definition \ref{oridef}, we have  the following form of Stokes' theorem:
$$\int_{\cM_{i,k}} d\alpha=\sum_{i<j<k} (-1)^{m_{i,j}}\int_{{\cM_{i,j}}\times_j \cM_{j,k}}m^*\alpha$$
Let $\alpha \in \Omega^*(C_i),\beta \in \Omega^*(C_k)$ and $i<j<k$. Because $s_{i,k}\circ m|_{\cM_{i,j}\times_j\cM_{j,k}}=s_{i,j}\circ \pi_1$ and $t_{i,k}\circ m|_{\cM_{i,j}\times_j\cM_{j,k}}=t_{j,k}\circ \pi_2$, where $\pi_1,\pi_2$ are natural projections, we have \begin{equation}\label{identification}
\int_{m({\cM_{i,j}}\times_j \cM_{j,k})} s_{i,k}^* \alpha \wedge t_{i,k}^* \beta=\int_{{\cM_{i,j}}\times_j \cM_{j,k}} m^*s_{i,k}^* \alpha \wedge m^*t_{i,k}^*\beta=\int_{{\cM_{i,j}}\times_j \cM_{j,k}}\pi_1^*s_{i,j}^*\alpha \wedge \pi_2^*t_{j,k}^*\beta\end{equation} 
Since we will only consider pullbacks of forms by source and target maps, it is convenient to think that $\cM_{i,j}\times_j \cM_{j,k}$ is contained in $\partial \cM_{i,k}$, and suppress the composition map $m$.

\subsubsection{Conventions for cochain complexes} In a typical homological algebra textbook e.g.\ \cite{weibel1995introduction}, a cochain complex is $\Z$ graded or $\Z/k$ graded for $k\ge 2$. As mentioned in Remark \ref{rmk:grading}, the grading of the Morse-Bott cochain complex is a consequence of the grading structure in Definition \ref{def:grade}, which is an extra piece of data on flow categories. Although the applications in our mind always have at least a $\Z/2$ grading structure, we will not assume this and only work with Definition \ref{def:flow}. As a result, our cochain complex is simply a vector space $C$ with an operator $d:C\to C$ such that  $d^2=0$. Then the cohomology $H(C,d)$ is defined as $\ker d/\Ima d$. The definitions of cochain maps and homotopies are similar and have the usual properties. It is clear that by forgetting the grading on a $\Z/k$ graded cochain complex, we get a cochain complex in the above sense.  Many basic properties in homological algebra survive for ungraded cochain complexes, e.g.\ the spectral sequence from a filtration, the exact triangle\footnote{When we have a $\Z$ grading, the exact triangle is a long exact sequence.} from a short exact sequence, the mapping cone and mapping cylinder constructions.

\subsection{Review of existing constructions}
Throughout this subsection, we fix a flow category $\cC:=\{C_i,\cM_{i,j}\}$, such that there are finitely many nonempty $C_i$ for simplicity (for example one can take the flow category from Example \ref{ex:ex}). Before giving our construction of the minimal Morse-Bott cochain complex in \S \ref{mbcochain}, we review the three constructions in the existing literature: Austin-Braam's pull-push construction, Fukaya's push-pull construction and the cascades construction. For simplicity, we completely neglect the issue of signs\footnote{For curious readers who would like to verify those constructions, we point out here that Austin-Braam \cite{austin1995morse} have got incorrect orientations and signs. Although our construction is motivated from theirs, we will not appeal to any of their specific formulae in our proofs.} and orientations. 

\subsubsection{Austin-Braam's Morse-Bott cochain complex  $(BC^{\AB}, d^{\AB})$}\label{AB}
Austin-Braam \cite{austin1995morse} defined the Morse-Bott cochain complex of a flow category to be 
$$(BC^{\AB}:=\oplus_i \Omega^*(C_i), d^{\AB}),$$
where $\Omega^*(C_i)=\oplus_{j=0}^{\dim C_i} \Omega^j(C_i)$ is the space of differential forms on $C_i$. The differential $d^{\AB}$ is defined as $\sum_{k\ge 0 }d_k$, where $d_k$ is defined by:
\begin{align}
d_0=\rd &:  \Omega^*(C_i)\to \Omega^*(C_i)\text{ is the usual exterior differential on differential forms.} \nonumber \\
d_k&:  \Omega^*(C_i)\to \cD^*(C_{i+k}),  \alpha \mapsto {t_{i,i+k}}_*\circ s^*_{i,i+k}(\alpha), \text{ for } k\ge 1 \label{eqn:AB}.
\end{align}
 Here $\cD^*(C)$ is the space of currents on $C$. The operator $d_k$ taking value in $\cD^*(C)$ instead of $\Omega^*(C)$ causes difficulties getting a well-defined \emph{ungraded} cochain complex $(BC^{\AB}, d^{\AB})$. Thus to make it well-defined, the target maps $t_{i,j}$ are assumed to be fibrations in Austin-Braam's model. Under such assumptions, ${t_{i,j}}_*$ is integration along the fiber, hence $d_k$ actually lands in $\Omega^*(C_{i+k})$. However it was noticed in \cite[Remark 2.4]{latschev2000gradient} that the fibration condition is obstructed for some Morse-Bott functions, i.e.\ there exists a Morse Bott function $f$, such that the fibration property fails for all metrics. 
 
 \begin{remark}\label{rmk:MBS}
 An equivalent form of the fibration condition was studied by Banyaga-Hurtubise under the name \textbf {Morse-Bott-Smale condition} \cite[Definition 3.4]{banyaga2010morse}. More precisely, let $\phi_t$ be the gradient flow of $f$,  the Morse-Bott-Smale condition holds iff the unstable manifold $U(C_i)=\{x|x\in M, \lim_{t\to -\infty}\phi_t(x)\in C_i\}$ and the stable manifold $S(p)=\{x|x\in M, \lim_{t\to \infty} \phi_t(x)=p\}$ for $p\in C_j$ intersect transversely for all $C_i,C_j$ and $p\in C_j$.\footnote{Note that we use (un)stable manifolds of the \emph{positive} gradient flow, this explains the discrepancy with \cite[Definition 3.4]{banyaga2010morse}.} Note that $(U(C_i)\cap S(p))/\R$ is the intersection of the preimage $t_{i,j}^{-1}(p)$ with the open stratum of $\cM_{i,j}$ (the space of unbroken flow lines), it is easy to check that $U(C_i)$ is transverse to $S(p)$ iff $p$ is a regular value of $t_{i,j}$ restricted to the open stratum. In particular, the fibration condition implies the Morse-Bott-Smale condition. On the other hand,  the Morse-Bott-Smale condition implies the fibration condition by \cite[Corollary 5.20]{banyaga2010morse} and Ehresmann’s Theorem. Latschev introduced another even stronger condition \cite[Definition 2.3]{latschev2000gradient} to make sure the generalization of Harvey-Lawson's method \cite{harvey2001finite} can work in the context of Morse-Bott functions. On the other hand, the existence of a flow category only requires that $U(C_i)$ and $S(C_j)$ (the stable manifold of $C_j$) intersect transversely, and the iterated source and target maps
from these transverse intersections are transverse for all $i, j$, see \S \ref{s7} (this holds automatically when the  Morse-Bott-Smale condition holds). We refer to such a pair $(f,g)$ of a function and a metric as a \textbf{Morse-Bott-Smale pair} in \S \ref{s7}. It is important to note that the  \textbf{Morse-Bott-Smale pair} condition is much weaker than the \textbf {Morse-Bott-Smale condition} (namely transversality v.s.\ point-wise transversality in a family). Moreover, Morse-Bott-Smale pairs always exist, in particular, there is a metric for Latschev's example to form a  Morse-Bott-Smale pair.
 \end{remark}
 
 \begin{remark}
 	One way to get the fibration property is to fatten up all moduli spaces systematically, a construction in this spirit was carried out in \cite{fukaya2015kuranishi} using CF-perturbations.
 \end{remark}
 
 \begin{remark}
 The Austin-Braam cochain complex  $(BC^{\AB}, d^{\AB})$ explained here is ungraded. However, we can grade $\alpha\in \Omega^j(C_i)$ by $j+d_i$, where $d_i$ the dimension of the negative eigenspace of $\mathrm{Hess}(f)$ on $C_i$, i.e.\ the grading structure in Remark \ref{rmk:grading}, then  $(BC^{\AB}, d^{\AB})$ is graded by $\Z$ and the degree of $d^{\AB}$ is $1$. It is clear that $BC^{\AB}$ is equipped with a (action) filtration $F_i:=\oplus_{j=i}^\infty \Omega^*(C_j)\subset F_{i-1}$ compatible with the differential, which induces a spectral sequence. This structure does not depend on the grading and always exists for all flow categories, we will discuss the induced spectral sequence in \S \ref{specseq}. On the other hand, if there is a $\Z$ grading structure, then the cochain complex has the structure of a multicomplex studied in \cite{MR2718643}, which can decompose the spectral sequence further by the grading.
 \end{remark}

\subsubsection{Fukaya's Morse-Bott chain complex}
Fukaya \cite{fukaya1996floer} used ``singular" chains of critical manifolds to model the \emph{homology} of the manifold for the flow category in Example \ref{ex:ex}, and the Austin-Braam's model can be viewed as the dual of Fukaya's model. The chain complex is defined to be 
$$(BC^\Fuk:=\oplus_i C_*(C_i),\partial^\Fuk),$$
where $C_*(C_i)$ is the space of singular chains on $C_i$ and $\partial^\Fuk:=\sum_{k\ge 0} \partial_k$ with $\partial_k$ is defined by: 
$$\begin{aligned}
\partial_0=\partial & :C_*(C_i)\to C_{*}(C_i)\text{ is the usual boundary operator on singular chains.}\\
\partial_k& : C_*(C_{i+k})\to C_*(C_{i}), P \mapsto {s_{i,i+k}}_*\circ t^*_{i,i+k}(P), \text{ for } k\ge 1.
\end{aligned}$$
Now pushforward is well-defined. Pullback is defined as follows. Let $P:\Delta\to C_{i+k}$ be a singular chain, assume the fiber product $\Delta\times_{C_{i+k}} \mathcal{M}_{i,i+k}$ is cut out transversely in the sense of Definition \ref{def:trans}, hence a manifold with boundaries and corners. Then the projection to the second factor\footnote{To be more precise, we need to choose a triangulation of $\Delta\times_{C_{i+k}} \cM_{i,i+k}$},
$$\pi_{\cM_{i,i+k}}: \Delta\times_{C_{i+k}} \cM_{i,i+k}\to \cM_{i,i+k}$$ 
is defined to be  the pullback $t^*_{i,i+k}(P)$.

To guarantee this pullback is well-defined for all singular chains in $C_{i+k}$, one also needs to assume the target map $t_{i,i+k}$ is a fibration. To drop this constraint, Fukaya constructed a quasi-isomorphic subset $C_{\mathrm{geo}}(C_i)\subset C_*(C_i)$, such that the fiber products in the definition of pullbacks are defined over $C_{\mathrm{geo}}(C_i)$ and the operators $\partial_k$ are closed on $C_{\mathrm{geo}}(C_i)$.  Then $(\oplus_i C_{\mathrm{geo}}(C_i), \sum_{k\ge 0} \partial_k)$ defines a chain complex. It is important to note that the construction of $C_{\mathrm{geo}}(C_i)$ depends on $\cM_{i,j}$ and $s_{i,j}, t_{i,j}$.

\subsubsection{The cascades model}\label{cas}
The cascades construction was first introduced by Bourgeois \cite{bourgeois2002morse} and Frauenfelder \cite{frauenfelder2004arnold}.  In the following, we review their constructions but in the cohomology context to align with Austin-Braam's construction. For each $C_i$, we choose a Morse-Smale pair $(f_i,g_i)$.\footnote{That is stable manifolds and unstable manifolds of $\nabla_{g_i} f_i$ intersect transversely.} Then the cascade cochain complex is defined to be
$$(BC^\mathrm{C}:=\oplus_i  MC(f_i,g_i), d^\mathrm{C}),$$ 
where $MC(f_i,g_i)$ is the Morse \emph{cochain} complex of $C_i$ using the Morse-Smale pair $(f_i,g_i)$. The differential $d^\mathrm{C}$ is defined to be $\sum_{k\ge 0} d^\mathrm{C}_k$, where $d^\mathrm{C}_k$ is defined by:
$$
\begin{aligned}
d^\mathrm{C}_0=d_\mathrm{M} & :   MC(f_i,g_i)\to MC(f_i,g_i) \text{ is the usual Morse differential for $(f_i,g_i)$.}   \\
d^{\mathrm{C}}_k &: MC(f_i,g_i)\to MC(f_{i+k},g_{i+k}) \text{ is defined by the number of rigid cascades from $C_i$ to $C_{i+k}$, } \forall k \ge 1. 
\end{aligned}
$$
A $0$-cascade is an unparameterized gradient flow line for $(f_i,g_i)$. For $k\ge 1$, a $k$-cascade from $a\in \Crit(f_i)$ to $b\in \Crit(f_j)$ for $i<j$ is a tuple for $i<r_1<\ldots<r_{k}<j$,
$$(\gamma_i, m_{i,r_1}, \gamma_{r_1},t_{r_1}, \ldots, m_{r_{k-1},r_k}, \gamma_{r_k}, ,t_{r_k}, m_{r_k,j}, \gamma_j),$$ 
where $\gamma_*$ is a gradient flow line in $C_*$, and $m_{*,*}$ is a point in $\mathcal{M}_{*,*}$, $t_*$ are positive real numbers, such that  $\gamma_i(-\infty)=a ,\gamma_i(0)=s(m_{i,r_1}), \gamma_j(+\infty)=b$, $\gamma_j(0)=t(m_{r_k,j})$ and $\gamma_{r_s} (t_{r_s})=s(m_{r_s,r_{s+1}})$, $\gamma_{r_s} (0)=t(m_{r_{s-1},r_s})$.

\begin{figure}[H]
	\begin{center}
	  \scalebox{.5}{
		\begin{tikzpicture}
		\draw (0,0) to [out=90, in=180] (2,1) to [out=0,in=90] (4,0) to [out=270,in=0] (2,-1) to [out=180,in=270] (0,0);
		\draw (3,3) to [out=90, in=180] (5,4) to [out=0,in=90] (7,3) to [out=270,in=0] (5,2) to [out=180,in=270] (3,3);
		\draw (6,6) to [out=90, in=180] (8,7) to [out=0,in=90] (10,6) to [out=270,in=0] (8,5) to [out=180,in=270] (6,6);
		\draw[->] (0.5,0) to (2,0);
		\draw[->] (2,0) to (3.5,0) to (3.5,1.5);
		\draw[->] (3.5,1.5) to (3.5,3) to (5,3);
		\draw[->] (5,3) to (6.5,3) to (6.5,4.5);
		\draw[->] (6.5,4.5) to (6.5,6) to (8,6);
		\draw (8,6) to (9.5,6);
		
		\node at (0.5,0) [circle,fill,inner sep=1.5pt] {};
		\node at (9.5,6) [circle,fill,inner sep=1.5pt] {};
		
		\node at (2,-0.7) {$C_1$};
		\node at (5,2.3) {$C_2$};
		\node at (8,5.3) {$C_3$};
		
		\node at (0.5,0.3) {$a$};
		\node at (9.5,6.3) {$b$};
		
		\node at (2,0.3) {$\small{\nabla_{g_1}f_1}$};
		\node at (5,3.3) {$\small \nabla_{g_2}f_2$};
		\node at (8,6.3) {$\small \nabla_{g_3}f_3$};
		
		\node at (4.2,1.5) {$\in \cM_{1,2}$};
		\node at (7.2,4.5) {$\in \cM_{2,3}$};
		\end{tikzpicture}
		}
	\end{center}
	\caption{A $2$-cascade}
\end{figure}
When appropriate transversality assumptions are met, the moduli space of all cascades from $a$ to $b$ form a manifold.  Moreover, there is a natural compactification of the moduli space by including the ``broken" cascades. Then the differential $d^\mathrm{C}$ for the cascades cochain complex comes from counting the zero dimensional compactified moduli spaces of cascades.

\begin{remark}\label{rmk:choice}
The transversality for all compactified moduli spaces of cascades will become tautological if we assume $t_{i,j}$ is a fibration. In principle, we can obtain transversality for the cascades moduli spaces with generic choices of $(f_i,g_i)$. However the choice depends on  $\cM_{i,j}$ and $s_{i,j}, t_{i,j}$ just like Fukaya's model.
\end{remark}

\begin{remark}
The cascades construction is very popular and was deployed in many applications, see \cite{biran2008lagrangian,bourgeois2002morse,diogo2018symplectic,frauenfelder2004arnold,schmaschke2016floer}. One advantage of the cascades model, besides being locally finite dimensional, is the clear relation between the cascades model and the Morse model. More precisely, the additional Morse function $f_i$ can be used to perturb the Morse-Bott function into a Morse function whose gradient flow lines can be identified with cascades. This identification was carried out by Banyaga-Hurtubise \cite{banyaga2013cascades} in the context of finite dimensional Morse-Bott theory and Bourgeois-Oancea \cite{autonomous} in the context of symplectic homology with autonomous Hamiltonians.
\end{remark}

\subsection{Homological perturbation theory}
The fibration  condition in Austin-Braam's construction plays an important role in resolving the problem of the differential $d_k$ taking value in the space of currents.  Since fibration conditions are usually stronger than what one can get in any virtual techniques, we want to replace the fibration condition with a weaker transversality requirement, i.e.\ the fiber product transversality condition in Definition \ref{def:flow}, which is generic in every reasonable virtual technique. Note that the operator $d_k$ is defined using pushforward of differential forms. Since pushforward is defined as the dual operator of pullback, the problem is rooted in the fact that the dual space of differential forms $\Omega^*(C_i)$ is the space of currents $\cD^*(C_i)$ instead of itself. However, this problem never appears for finite dimensional vector spaces, i.e.\ whenever a finite dimensional space is equipped with a non-degenerate bilinear form, the dual space is identified with itself. To make use of this fact, we use the homological perturbation lemma, which is a method of constructing small cochain complexes from  larger ones. The strategy is to formally apply the homological perturbation lemma to the almost existing Austin-Braam's cochain complex, and then verify that the formula suggested by the perturbation lemma is well-defined directly and gives the desired algebraic relations.  The theme of this paper can be summarized as the following slogan: 

\begin{center}
	\begin{tikzpicture}[punkt/.style={
           rectangle,
           rounded corners,
           draw=black, thin,
           text width=43em,
           minimum height=1em}]
	\node[punkt] {Formal applications of the homological perturbation lemma can resolve the technical difficulty of infinite dimensional cochain models.};
	\end{tikzpicture}
\end{center}

\subsubsection{A homological perturbation theorem}
Roughly speaking, the homological perturbation lemma is a procedure that takes in a cochain complex and perturbation data (in most cases, projections and homotopies) and produces another cochain complex, which is quasi-isomorphic to the input cochain complex. For simplicity, we consider a cochain complex $A=\oplus_{i=1}^n A_i$, where $A_i$ are $\Z/2$ linear spaces (ungraded as usual, $i$ is \emph{not} the grading!). Assume the differential $d$ is in the form of $\sum_{k\ge 0}d_k$ with $d_k:A_i\to A_{i+k}$ for $k\ge 0$. Then $d^2=0$ implies that $(A_i,d_0)$ is also a cochain complex for all $i$. The perturbation data consists of for each $1\le i \le n$ projections $p_i:A_i\to A_i$ and homotopies $H_i:A_i\to A_i$ between the identity and $p_i$, i.e.
\begin{equation} \label{homotopyrel}\id-p_i=d_0\circ H_i+H_i\circ d_0.\end{equation}
With this perturbation data, we have the following homological perturbation lemma.
\begin{lemma}\label{hplthm}  
There is a differential on $\bigoplus_i p_i(A_i)$, such that  $\bigoplus_i p_i(A_i)$ is quasi-isomorphic to $A$.  
\end{lemma}
The lemma holds for general coefficient rings and graded complexes, once appropriate signs are assigned. Since we only use Lemma \ref{hplthm} to explain the motivation behind the formulae we give in \S \ref{morsebott}, we will not go into the details of the signs nor the proof. What is more relevant to our purpose is the pattern of the formula for the differential on  $\bigoplus p_i(A_i)$, which can be viewed as an analog of the perturbation theorem for $A_\infty$ structures proved in \cite{kontsevich2000homological}.  For a strictly increasing sequence of integers $T=\{i_0=0, i_1,\ldots,i_{r+1}=k\}$ for $r\ge 0$, we define the an operator $D_{k,T}: p_i(A_i)\to p_{i+k}(A_{i+k})$ for all integers $i$:
\begin{equation}\label{hpl}D_{k,T}=p_{i+k}\circ  d_{i_{r+1}-i_{r}} \circ H_{i+i_{r}}\circ\ldots\circ H_{i+i_2}\circ d_{i_2-i_1} \circ H_{i+i_1} \circ d_{i_1-i_0}\circ \iota_i \end{equation}
where  $\iota_i: p_i(A_i)\to A_i$ denotes the inclusion. $D_{k,T}$ can be schematically explained as follows:
\begin{center}
    \begin{tikzpicture}
    \draw[->] (0,-1) to (0,0) to (1,0);
    \draw[->] (1,0) to (3,0);
    \draw[->] (3,0) to (5,0);
    \draw (5,0) to (6,0);
    \draw [dotted] (6,0) to (7,0);
    \draw[->] (7,0) to (9,0);
    \draw[->] (9,0) to (10,0) to (10,-1);
    \draw (10,-2) to (10,-1);
    \draw[->] (0,-2) to (0,-1);
    
    \node at (0,0) [circle,fill,inner sep=1.5pt] {};
    \node at (2,0) [circle,fill,inner sep=1.5pt] {};
    \node at (4,0) [circle,fill,inner sep=1.5pt] {};
    \node at (8,0) [circle,fill,inner sep=1.5pt] {};
    \node at (10,0) [circle,fill,inner sep=1.5pt] {};
    \node at (10,-2) [circle,fill,inner sep=1.5pt] {};
    \node at (0,-2) [circle,fill,inner sep=1.5pt] {};
    
    \node at (-0.3,0) {$A_i$};
    \node at (2,-0.3) {$A_{i+i_1}$};
    \node at (4,-0.3) {$A_{i+i_2}$};
    \node at (8,-0.3) {$A_{i+i_{r}}$};
    \node at (10.5,0) {$A_{i+k}$};
    
    \node at (-0.3,-2.3) {$p_i(A_i)$};
    \node at (10.3,-2.3) {$p_{i+k}(A_{i+k})$};
    
    \node at (-0.3,-1) {$\iota_i$};
    \node at (10.5,-1) {$p_{i+k}$};
    
    \node at (1,0.3) {$d_{i_1}$};
    \node at (3,0.3) {$d_{i_2-i_1}$};
    \node at (9,0.3) {$d_{i_{r+1}-i_{r}}$};
    
    \draw[->] (2,0) to [out=180, in=180]  (2,0.8);
    \draw (2,0.8) to [out=0, in =0] (2,0);
    \node at (2,1.1) {$H_{i+i_1}$};
    
    \draw[->] (4,0) to [out=180, in=180]  (4,0.8);
    \draw (4,0.8) to [out=0, in =0] (4,0);
    \node at (4,1.1) {$H_{i+i_2}$};
    
    \draw[->] (8,0) to [out=180, in=180]  (8,0.8);
    \draw (8,0.8) to [out=0, in =0] (8,0);
    \node at (8,1.1) {$H_{i+i_{r}}$};
    \end{tikzpicture}
\end{center}
The new differential $D$  on $\bigoplus_i p_i(A_i)$ is defined as $$D=\sum_{k=0}^\infty D_k,$$
where $D_k=\sum_T D_{k,T}$ is the summation over all strictly increasing sequences $T$ from $0$ to $k$. 
  
\subsubsection{Cascades from homological perturbation}
In this part, we explain how to heuristically interpret the cascades \emph{cochain} complex as a homological perturbation on the Austin-Braam cochain complex. The feature that the cascades construction does not require the fibration condition also reflects the theme of the paper. 
 
We first explain the perturbation data used to get the cascades cochain complex, i.e.\ a pair of projection and homotopy $(p_i, H_i)$ on $\Omega^*(C_i)$ for every $i$.  We require that the image $\Ima p_i$ is a finite dimensional subspace of $\Omega^*(C_i)$. Given such perturbation data, we can formally write down operators $D_{k,T}$ from \eqref{hpl}.  Note that in the cascades construction, we choose a Morse-Smale pair $(f_i,g_i)$ on each critical manifold $C_i$. The perturbation data is then given by such a Morse-Smale pair using the construction in \cite{harvey2001finite}.  Before giving the construction, we first set up some notation. We will not be precise about signs and orientations.
\begin{definition}\label{def:current}
Let $C$ be an oriented closed manifold.
	\begin{enumerate}
	\item\label{sign_1} $\cD^*(C)$ denotes the space of currents\footnote{For basics of currents, we refer readers to \cite{griffiths2014principles}.} on $C$. There is a natural inclusion $\iota: \Omega^*(C)\to \cD^*(C)$ given by
	$$\iota(\alpha)(\beta)=\int_C \alpha \wedge \beta, \quad \forall \alpha \in \Omega^*(C).$$
	\item\label{sign_2} Let $\kappa \in \cD^*(C\times C)$ be a current, then the induced integral operator $I_\kappa:\Omega^*(C)\to \cD^*(C)$ is defined as:
	\begin{equation}\label{intop} I_\kappa(\alpha)(\beta):=(-1)^{\dim C} \kappa(\pi_1^*\alpha\wedge \pi_2^*\beta)  \qquad \forall \alpha, \beta \in \Omega^*(C)\footnote{We make the signs in \eqref{sign_1}, \eqref{sign_2}  precise for the sake of \S \ref{morsebott}.},
	\end{equation}
	where $\pi_1, \pi_2$ are  projections of $C\times C$ to the first and the second factor respectively.  
	\item\label{singular} Let $B$ be an oriented compact manifold  and  $i:B\to C$ a smooth inclusion. Then we can define a current $[B] \in \cD^*(C)$ by:
	$$[B](\alpha):=\pm \int_{B} i^*\alpha, \quad \forall \alpha \in \Omega^*(C).$$ 
	In general, one can define a current $[B]$ for any oriented singular chain $B$.
\end{enumerate}   
\end{definition} 
Let $\Crit(f_i)$ be the set of critical points of the Morse function $f_i$ on $C_i$.  We use $\phi^i_t:C_i\to C_i$ to denote the time $t$ flow of the gradient vector field $\nabla_{g_i} f_i$ on $C_i$. Then the pullback operator  ${\phi^i_{-t}}^*:\Omega^*(C_i)\to \Omega^*(C_i)$ can be understood as the integral operator $I_{[\graph \phi^i_t]}$ of the current of $\graph \phi^i_t:=\{(x,\phi^i_t(x))\}\subset C_i\times C_i$\footnote{$\phi_t^*$ is represented by $\{(\phi^i_t(x),x)\}\subset C_i\times C_i$, which was used in \cite{harvey2001finite}. }. The manifold $\cup_{0<t'< t}\graph \phi^i_{t'}\subset C_i\times C_i$  defines an integral operator $H^i_t:=I_{[\cup_{0< t'< t}\graph \phi^i_{t'}]}=I_{[\cup_{0\le  t'\le t}\graph \phi^i_{t'}]}$.  Since $\partial (\cup_{0\le t'\le t}\graph \phi^i_{t'}) = \Delta_i \cup \graph \phi^i_t$, Stokes' theorem implies that
\begin{equation}\label{hpleqn}
	\id-{\phi^i_{-t}}^*=\rd \circ H^i_t+H^i_t\circ \rd.
\end{equation}	
It was proven in \cite{harvey2001finite} that when $t\rightarrow \infty$, \eqref{hpleqn} converges to a projection-homotopy relation. To be more specific, let $U_x$, $S_x$ denote the unstable and stable manifolds of the critical point $x\in \Crit(f_i)$, i.e. 
\begin{eqnarray}
U_x &:= &\{y\in C_i|\lim_{t\to -\infty}\phi^i_t(y)=x\}; \nonumber 
\\S_x & := & \{y\in C_i|\lim_{t\to \infty}\phi^i_t(y)=x\}. \nonumber 
\end{eqnarray}
In the sense of currents, we have the following,
\begin{eqnarray}
\lim_{t\to \infty}\left[\graph \phi^i_t\right] &= &\sum_{x\in \Crit(f_i)}\left[S_x\times U_x\right]\label{limit1},\\
\lim_{t\to \infty}\left[\bigcup_{0<t'<t}\graph \phi^i_{t'}\right] &= &\left[\bigcup_{0<t'<\infty}\graph \phi^i_{t'}\right]\label{limit2},
\end{eqnarray}
see \cite[Theorem 2.3, 3.3]{harvey2001finite} for details.
\begin{remark}
    It is important to note that \cite{harvey2001finite} studied $\lim_{t\to \infty }\phi_t^*$ and \cite[Theorem 3.3]{harvey2001finite} stated that $\lim_{t\to \infty }\phi_t^*$ can be represented by $\sum_{x\in \Crit(f_i)}\left[U_x]\times [S_x\right]$. Then \eqref{hpleqn} projects $\Omega^*(C_i)$ to the Morse \emph{chain} complex \cite[Proposition 4.5]{harvey2001finite}, or equivalently the Morse \emph{cochain} complex of $-f_i$. Since we need a projection to the Morse \emph{cochain} complex of $f_i$ to explain the cascades model, we need to work with $\lim_{t\to \infty }\phi_{-t}^*$ instead. This explains the discrepancy with \cite{harvey2001finite}.
\end{remark}

Hence \eqref{limit1} and \eqref{limit2} define two integral operators ${\phi^i_{-\infty}}^*, H^i_\infty:\Omega^*(C_i)\to \cD^*(C_i)$, such that \begin{equation} \label{projhomo}\iota-{\phi^i_{-\infty}}^*=\rd \circ H^i_\infty+H^i_\infty \circ \rd, \text{\quad see \cite[Theorem  2.3,3.3]{harvey2001finite}}\end{equation}
where $\iota$ is the natural embedding $\Omega^*(C_i)\hookrightarrow \cD^*(C_i)$. 
Note that
\begin{equation}\label{eqn:projection}
{\phi^i_{-\infty}}^*(\alpha)=\sum_{x\in \Crit(f_i)} \left(\int_{C_i} \alpha\wedge[S_x]\right)\cdot[U_x]=\sum_{x\in \Crit(f_i)}  \left(\int_{S_x} \alpha|_{S_x}\right)\cdot [U_x], \text{\quad see \cite[Theorem  4.1]{harvey2001finite}}
\end{equation}
can be viewed as the projection from $\Omega^*(C_i)$ to the Morse cochain complex. By \eqref{projhomo}, $H^i_{\infty}$ defines a homotopy between $\iota$ and the projection ${\phi^i_{-\infty}}^*$.  
\begin{remark}\label{rmk:strict}
	Strictly speaking, \eqref{projhomo} is not a genuine  projection-homotopy relation, since ${\phi^i_{-\infty}}^*$ lands in space of currents instead of differential forms. To get an honest projection-homotopy relation, we need to enlarge $\Omega^*(C)$ by adding some currents of singular chains. Roughly speaking, the enlargement is the minimal extension which contains $[U_x],[S_x]$ for $x\in \mathrm{Crit}(f_i)$, such that it is closed under ${\phi^i_{-\infty}}^*$, $H_\infty^i$ and $\rd$. Such enlargement depends on $\cM_{i,j}$ and $s_{i,j}, t_{i,j}$, which leads to the choices in Remark \ref{rmk:choice}.
\end{remark}

From now on, we will neglect the issue in Remark \ref{rmk:strict} and show formally that the cascades construction can be understood as applying the construction in \eqref{hpl} to the Austin-Braam cochain complex using the perturbation data $({\phi^i_{-\infty}}^*, H^i_\infty)$. Before ``proving" the claim, we first ``define" the integration of pullbacks of currents from singular chains.
\begin{definition}\label{def:intcur}
	Let $\cM$ be a compact manifold with two smooth maps $s,t:\cM \to C_1,C_2$. Assume $B_1\subset C_1$ and $B_2\subset C_2$ are two submanifolds without boundary\footnote{The inclusion $B_*\subset C_*$ is not required to be proper, hence $B_*$ may not be closed. We only require $B_*$ is the interior of a compact manifold with boundaries and corners $\overline{B}_*$ so that the inclusion $B_*\hookrightarrow C_*$ is the restriction of a smooth map $\overline{B}_*\to C_*$. Therefore \eqref{singular} of Definition \ref{def:current} makes sense for $B_1$. In particular, the (un)stable manifolds satisfy the condition.}. If $s$ is transverse to $B_1$ and $t$ is transverse to $B_2$ and $s^{-1}(B_1)$ is transverse to $t^{-1}(B_2)$ with finite intersections, then we define
	$$\int_{\cM} s^*([B_1])\wedge t^*([B_2]) :=  \sum_{p \in s^{-1}(B_1)\cap t^{-1}(B_2)  }\pm 1.$$
\end{definition}
Definition \ref{def:intcur} is natural in the sense that if we approximate the current $[B_1]$ by differential forms supported in a tubular neighborhood \cite[Chapter 3 \S 1]{griffiths2014principles}, then the limit of the integration of the pullbacks of  the approximations is indeed the number of intersection points counted with sign.\footnote{The sign is determined by the orientations of $B_1,B_2,C_1,C_2$ and $\cM$.}

Now we apply \eqref{hpl}. For $x\in \mathrm{Crit}(f_i)$, the first term $D_0$ in $D=\sum_{k\ge 0} D_k$ is defined by 
\begin{eqnarray*}
D_0([U_x]) &:= & {\phi^i_{-\infty}}^*(d_0([U_x]))={\phi^i_{-\infty}}^*(\rd([U_x]))\\
& = &\sum_{y\in \mathrm{Crit}(f_i) }\left(\int_{C_i} \rd([U_x]) \wedge [S_y]\right)\cdot[U_y].
\end{eqnarray*}
It was proven in \cite[Proposition 4.5]{harvey2001finite} that when the Morse-Smale condition holds  we have $\int_{C_i} \rd([U_x]) \wedge [S_y]$ equals the signed counts of rigid gradient flow lines from $x$ to $y$.  Therefore $D_0$ recovers the Morse differential on $C_i$.  Next, we study the higher operators in $D$. Let $x\in \mathrm{Crit}(f_i)$, we have
$$\begin{array}{>{\displaystyle}r>{\displaystyle}c>{\displaystyle}l}D_1([U_x])={\phi^{i+1}_{-\infty}}^* d_1[U_x] &  \stackrel{\eqref{eqn:projection}}{=} &\displaystyle \sum_{y\in \Crit(f_{i+1})}\left(\int_{C_{i+1}} d_1[U_x]\wedge [S_y] \right) \cdot[U_y] \\
 & \stackrel{\eqref{eqn:AB}}{=} & \sum_{y\in \Crit(f_{i+1})}\left(\int_{\cM_{i,i+1}} s_{i,i+1}^*[U_x]\wedge t_{i,i+1}^*[S_y] \right) \cdot[U_y]  \\
 & \stackrel{\text{Def \ref{def:intcur}}}{=} & \sum_{y\in \Crit(f_{i+1})} \#(s_{i,i+1}^{-1} (U_x) \cap t_{i,i+1}^{-1}(S_y))\cdot[U_y].
 \end{array}$$
 Here the last equality requires that $s_{i,i+1}^{-1}(U_x) \pitchfork t_{i,i+1}^{-1}(S_y)$. Therefore $D_1$ counts points in $s_{i,i+1}^{-1} (U_x) \cap t_{i,i+1}^{-1}(S_y)$, which is exactly the 1-cascades in \cite{bourgeois2002morse, frauenfelder2004arnold}. By the same argument, $D_{2,\{0,2\}}$ counts rigid 1-cascades from $C_i$ to $C_{i+2}$. Next we consider the operator $D_{2,\{0,1,2\}}$,
\begin{eqnarray}
 D_{2,\{0,1,2\}}([U_x]) & =& {\phi^{i+2}_{-\infty}}^*\circ d_1\circ H^{i+1}_\infty\circ d_1([U_x]) \nonumber\\
& \stackrel{\eqref{eqn:projection}}{=} & \sum_{y\in \Crit(f_{i+2})}\left(\int_{C_{i+2}}(d_1\circ H^{i+1}_{\infty}\circ d_1 [U_x])\wedge[S_y]\right)\cdot [U_y] \nonumber \\
& \stackrel{\eqref{eqn:AB}}{=} & \sum_{y\in \Crit(f_{i+2})}\left(\int_{\cM_{i+1,i+2}} s_{i+1,i+2}^*(H^{i+1}_\infty\circ d_1[U_x]) \wedge t_{i+1,i+2}^*[S_y]\right)\cdot[U_y] \label{eqn:D_2} 
\end{eqnarray}
Let us treat currents just like differential forms for simplicity. By definition we have
\begin{eqnarray*}
\int_{C_{i+1}} H^{i+1}_{\infty}\circ d_1([U_x])\wedge \alpha & = & \int_{C_{i+1}\times C_{i+1}} \pi_1^*(d_1([U_x])) \wedge \left[\bigcup_{0<t'<\infty}\graph \phi^{i+1}_{t'}\right]\wedge \pi_2^*\alpha \\
& = & \int_{\cM_{i,i+1}\times C_{i+1}} s_{i,i+1}^*[U_x]\wedge (t_{i,i+1}\times \id_{C_{i+1}})^* \left[\bigcup_{0<t'<\infty}\graph \phi^{i+1}_{t'}\right]\wedge \pi_2^*\alpha.
\end{eqnarray*}
Then we have
$$H^{i+1}_{\infty}\circ d_1([U_x])= \int_{\cM_{i,i+1}} s_{i,i+1}^*[U_x]\wedge (t_{i,i+1}\times \id_{C_{i+1}})^* \left[\bigcup_{0<t'<\infty}\graph \phi^{i+1}_{t'}\right],$$
the right hand side is the integration along the fiber $\cM_{i,i+1}$ in the trivial fibration $\cM_{i,i+1}\times C_{i+1}$. Therefore we have
$$\eqref{eqn:D_2}= \sum_{y\in \Crit(f_{i+2})}\left(\int_{\cM_{i,i+1}\times \cM_{i+1,i+2}} s_{i,i+1}^*[U_x] \wedge (t_{i,i+1}\times s_{i+1,i+2})^*\left[\bigcup_{0<t'<\infty}\graph \phi^{i+1}_{t'}\right] \wedge t_{i+1,i+2}^*[S_y]\right)\cdot [U_y].$$
When transversality holds,  by Definition \ref{def:intcur}, the last line above equals to 
$$\sum_{y\in \Crit(f_{i+2})}\#\left( \left(s_{i,i+1}^{-1}(U_x) \times t_{i+1,i+2}^{-1}(S_y) \right)\pitchfork \left((t_{i,i+1}\times s_{i+1,i+2})^{-1}(\bigcup_{0<t'<\infty}\graph \phi^{i+1}_{t'})\right)\right)\cdot [U_y].$$ 
It can be interpreted as the counting of 2-cascades from $C_i$ to $C_{i+2}$ staying on $C_{i+1}$ for finite time. Therefore $D_2 = D_{2,\{0,2\}}+D_{2,\{0,1,2\}}$ counts all rigid cascades from $C_i$ to $C_{i+2}$.  In general, assuming the transversality for the cascade moduli spaces, we shall recover the whole cascades construction from \eqref{hpl}. Hence the cascades construction fits into the homological perturbation philosophy.

%% file: s3.tex
\section{The Minimal Morse-Bott Cochain Complexes}\label{morsebott}
In this section, we carry out the construction of the minimal Morse-Bott cochain complex for an abstract oriented flow category, which is applicable to both finite dimensional Morse-Bott theory and Floer theories.  The motivation of the construction is from Lemma \ref{hplthm} and formula \eqref{hpl} with different perturbation data. We still need to make some choices (Definition \ref{def:def}) in the construction of the perturbation data, however unlike the cascades construction, the choices in the minimal construction only depend on $C_i$, i.e.\ there is no compatibility requirement with the morphism spaces $\cM_{i,j}$. 

This section is organized as follows: \S \ref{pertdata} constructs the perturbation data for the minimal Morse-Bott cochain complex; \S \ref{mbcochain} constructs the Morse-Bott cochain complexes for every oriented flow category; \S \ref{flowmor} defines flow morphisms which can be viewed as the geometric analogue of the continuation maps and shows that flow morphisms induce morphisms between Morse-Bott cochain complexes; \S \ref{flowcomp} explains the compositions of flow morphisms; \S \ref{flowhomotopy} defines flow homotopies and proves that flow homotopies induce homotopies between morphisms; \S \ref{canonical} establishes that our construction is canonical on the cochain complex level, i.e. it is independent of all choices; \S \ref{sub:sub} introduces flow subcategories and quotient categories, which are the geometric analogues of subcomplexes and quotient complexes respectively. From now on, we will be very specific about the orientations and signs and provide rigorous arguments. Proofs in this section involve a lot of sign computations, we provide a detailed proof of $d^2_{\BC}=0$ for the coboundary map $d_{\BC}$ in \S \ref{mbcochain}. Proofs of other results in \S \ref{flowmor},  \ref{flowcomp}, and \ref{flowhomotopy} will only be sketched.
   
\subsection{Perturbation data for the minimal Morse-Bott cochain complex}\label{pertdata}
In this subsection, we construct the perturbation data $\{(p_i,H_i)\}$ for the minimal Morse-Bott cochain complex of an oriented flow category $\mathcal{C}:=\{C_i,\cM_{i,j} \}$. Then \eqref{hpl} will motivate the definition of $D_{k,T}$ for the differential. We will show in the next subsection that they indeed define a cochain complex.

\subsubsection{Projection $p_i$} We start by defining a projection on $p_i$ on $\Omega^*(C_i)=\oplus_{j=1}^{\dim C_i} \Omega^j(C_i)$. First note that we have bilinear form on $\Omega^*(C_i)$ given by
\begin{equation}\label{pair}
\langle  \alpha, \beta \rangle_{i}:=(-1)^{\dim C_i\cdot |\beta| }\int_{C_i} \alpha \wedge \beta, \quad \forall \alpha, \beta \in \Omega^*(C_i).\end{equation}
We can pick representatives $\{\theta_{i,a}\}_{1\le a \le \dim H^*(C_i)}\subset \Omega^*(C_i)$ of a basis of $H^*(C_i)$, i.e.\ $\theta_{i,a}$ are closed forms such that the corresponding cohomology classes form a basis of $H^*(C_i)$. Such choice gives us a quasi-isomorphic embedding $H^*(C_i)\to \Omega^*(C_i)$. Let $h(i)$ denote the image of the embedding above, i.e.\ $h(i) := \la \theta_{i,1},\ldots, \theta_{i,\dim H^*(C_i)}\ra\subset \Omega^*(C_i)$. Note that \eqref{pair} is non-degenerate on cohomology, let $\{\theta^*_{i,a} \}_{1\le a \le \dim H^*(C_i)}\subset h(i)$ be the dual basis to the basis $\{ \theta_{i,a}\}$ in the sense that
\begin{equation}\label{eqn:dual}
    \langle \theta_{i,a}^*, \theta_{i,b}   \rangle_i=\delta_{ab}.
\end{equation}
Then we can define a projection $p_i:\Omega^*(C_i)\rightarrow h(i)\subset \Omega^*(C_i)$ by
\begin{equation}
p_i(\alpha):=\sum_{a=1}^{\dim H^*(C_i)} \langle \alpha, \theta_{i,a}\rangle_i \cdot \theta_{i,a}^*. \label{eqn:pi}
\end{equation}
If we identify $H^*(C_i)$ with $h(i)$, then $p_i$ can be thought of as a projection from $\Omega^*(C_i)$ to $H^*(C_i)$. 
\subsubsection{Homotopy $H_i$} 
We now explain the related homotopy $H_i$. First note that the Poincar\'e dual of the diagonal $\Delta_i\subset C_i\times C_i$ can be represented by Thom classes. We can identify a tubular neighborhood of the diagonal $\Delta_i$ with the unit disk bundle of the normal bundle $N_i$ of $\Delta_i$. Then one way of writing Thom classes of the diagonal $\Delta_i$ is
\begin{equation}\label{eqn:delta}
\delta_i^n:=\rd(\rho_n \psi_i),
\end{equation}
where $\psi_i$ is the angular form of the sphere bundle $S(N_i)$ \cite[\S 6]{bott2013differential} using the orientation in Example \ref{ex:ori} and $\rho_n:\R^+\to \R$ are smooth functions, such that $\rho_n$ is increasing, supported in $[0,\frac{1}{n}]$ and is $-1$ near $0$. For details of this construction, we refer readers to \cite[\S 6]{bott2013differential}. We also include a brief discussion of this construction and properties of it in Appendix \ref{conv}. The most important property of $\delta_i^n$ is that it converges to the Dirac current of $\Delta_i$. 
\begin{figure}[H]
	\begin{center}
		\begin{tikzpicture}
		\draw[->] (-1,0) -- (7,0) node[right] {$r$};
		\draw[->] (0,-2.5) -- (0,1) node[right] {$\rho_n(r)$};
		\draw (0,-2) .. controls (2,-2) and (1,0) .. (3,0);
		\node[left] at (0,-2) {$-1$};
		\node[above] at (3,0) {$\frac{1}{n}$};
		\node[above] at (6,0) {$1$};
		\end{tikzpicture}
		\caption{Graph of $\rho_n$}
	\end{center}
\end{figure}
\begin{lemma}\label{conv4}
	The Thom classes $\delta^n_i$ converge to the Dirac current $\delta_i$ of the diagonal $\Delta_i$ in the sense of currents, i.e.\ $\forall \alpha \in \Omega^*(C_i\times C_i)$ we have
	$$\lim_{n\to \infty}\int_{C_i\times C_i} \alpha \wedge \delta^n_i= \int_{C_i\times C_i} \alpha \wedge \delta_i :=\int_{\Delta_i} \alpha|_{\Delta_i}.$$
\end{lemma}
We will prove Lemma \ref{conv4} in Appendix \ref{conv}. By \eqref{intop},  for $\alpha,\beta \in \Omega^*(C_i)$, we have $\int_{C_i\times C_i} \pi_1^*\alpha \wedge \pi_2^*\beta \wedge \delta^n_i=(-1)^{(\dim C_i)^2} \int_{C_i\times C_i} \delta^n_i\wedge \pi_1^*\alpha\wedge \pi_2^*\beta=(-1)^{\dim C_i}\int_{C_i\times C_i}\delta^n_i\wedge \pi_1^*\alpha \wedge \pi_2^*\beta=I_{\delta^n_i}(\alpha)(\beta)$, then Lemma \ref{conv4} can be rewritten as 
$$\lim_{n\rightarrow \infty} I_{\delta^n_i}=I_{\delta_i}=\id:\Omega^*(C_i)\to \Omega^*(C_i)$$
in the weak topology. On the other hand, under the orientation convention \eqref{oricri}, we have another representative of the Poincar\'e dual of the diagonal by $\sum_a\pi_1^*\theta_{i,a}\wedge \pi_2^* \theta_{i,a}^*$, where $\pi_1,\pi_2$ are the projections to the first and second factor of $C_i\times C_i$ respectively.  
\begin{proposition}\label{prop:cohomologus}
$\sum_a\pi_1^*\theta_{i,a}\wedge \pi_2^* \theta_{i,a}^*$ is cohomologous to $\delta^n_i$ for all $n$. 	
\end{proposition}
\begin{proof} 
	Since the pairing \eqref{pair} is non-degenerate on $H^*(C_i\times C_i)$, it suffices to prove that
	$$\int_{C_i\times C_i}\alpha\wedge \delta_i^n=\int_{C_i\times C_i}\alpha \wedge \sum_a\pi_1^*\theta_{i,a}\wedge \pi_2^* \theta_{i,a}^*$$
	for any closed form $\alpha$. Since all $\delta^n_i$ are cohomologous to each other for different $n$, Lemma \ref{conv4} implies that if $\alpha \in \Omega^*(C_i\times C_i)$ is closed, then for all $n$, 
	$$\int _{C_i\times C_i} \alpha \wedge \delta^n_i=\int_{\Delta_i} \alpha|_{\Delta_i}.$$
	Therefore it suffices to show that for all closed form $\alpha\in \Omega^*(C_i\times C_i)$ we have 
	\begin{equation}\label{coh}
	\int_{C_i\times C_i} \alpha \wedge 	\left(\sum_a\pi_1^*\theta_{i,a}\wedge \pi_2^* \theta_{i,a}^*\right)=\int_{\Delta_i} \alpha|_{\Delta_i}
	\end{equation}
	Since the cohomology of $C_i\times C_i$ is spanned by $\{\pi_1^*\theta^*_{i,c}\wedge \pi_2^*\theta_{i,d} \}_{1\le c, d \le \dim H^*(C_i)}$, it is enough to verify \eqref{coh} for $\alpha=\pi_1^*\theta^*_{i,c}\wedge \pi_2^*\theta_{i,d}$. Note that by definition, we have $\left\langle\theta^*_{i,a},\theta_{i,b} \right\rangle_i=\delta_{ab}$. Then if $c\ne d$, we have 
	\begin{equation*}
	\int_{C_i\times C_i} \pi_1^*\theta^*_{i,c}\wedge \pi_2^*\theta_{i,d} \wedge \left(\sum_a\pi_1^*\theta_{i,a}\wedge \pi_2^* \theta_{i,a}^*\right)=\sum_{a} \pm \int_{C_i\times C_i} \pi_1^*\theta^*_{i,c}\wedge \pi_1^*\theta_{i,a}\wedge \pi_2^*\theta_{i,d}\wedge \pi_2^*\theta_{i,a}^*=\sum_{a} \pm \delta_{ca}\delta_{da}=0. 
	\end{equation*}
	Similarly, when $c=d$, we have 
	\begin{eqnarray*}
	\int_{C_i\times C_i} \pi_1^*\theta^*_{i,c}\wedge \pi_2^*\theta_{i,c} \wedge \left(\sum_a\pi_1^*\theta_{i,a}\wedge \pi_2^* \theta_{i,a}^*\right) &= & \int_{C_i\times C_i}  \pi_1^*\theta^*_{i,c}\wedge \pi_2^*\theta_{i,c} \wedge \pi_1^*\theta_{i,c}\wedge \pi_2^* \theta_{i,c}^*+\sum_{a\ne c}\pm \delta_{ca}\delta_{ca}\\
	&= & (-1)^{|\theta_{i,c}|^2+|\theta_{i,c}|\cdot |\theta^*_{i,c}|} \int_{C_i\times C_i}\pi_1^*\theta^*_{i,c}\wedge \pi_1^* \theta_{i,c} \wedge \pi_2^*\theta^*_{i,c} \wedge \pi_2^* \theta_{i,c}\\
	&= & (-1)^{|\theta_{i,c}|^2+|\theta_{i,c}|\cdot |\theta^*_{i,c}|+\dim C_i |\theta_{i,c}|} \left( \int_{C_i} \theta^*_{i,c}\wedge  \theta_{i,c} \right) \left \langle \theta_{i,c}^*, \theta_{i,c} \right\rangle_i\\
	&= & \int_{C_i}\theta^*_{i,c}\wedge  \theta_{i,c} = \int_{\Delta_i}  (\pi_1^*\theta^*_{i,c}\wedge \pi_2^*\theta_{i,c})|_{\Delta_i}.
	\end{eqnarray*}
	Thus \eqref{coh} is proven.
\end{proof}

As a consequence of Proposition \ref{prop:cohomologus}, there exist primitives $f_i^n\in \Omega^*(C_i\times C_i)$ such that
\begin{eqnarray}
\rd f_i^n & = & \delta_i^n-\sum_a\pi_1^*\theta_{i,a}\wedge \pi_2^* \theta_{i,a}^*;\label{minihomo}\\
f_i^n-f_i^m & = & (\rho_n-\rho_m)\psi_i \label{eqn:diff}. 
\end{eqnarray}
Note that the integral operator $I_{\delta_i}$ of the Dirac current $\delta_i$ is the identity map from $\Omega^*(C_i)$ to itself.  The integral operator $I_{\sum_a \pi_1^*\theta_{i,a}\wedge \pi_2^* \theta_{i,a}^*}$ is the projection $p_i$ in \eqref{eqn:pi}.  Therefore by  \eqref{minihomo}, the integral operator $I_{f_i^n}$ of the primitive $f_i^n$ satisfies the relation:
\begin{equation}\label{eqn:approhomo}
I_{\delta_i^n}- I_{\sum_a \pi_1^*\theta_{i,a}\wedge \pi_2^* \theta_{i,a}^*}=I_{\rd f_i^n}=\rd \circ I_{f_i^n}+I_{f_i^n}\circ \rd.\end{equation}
It is proven in the Appendix \ref{conv} that $f_i^n$ converges to a current $f_i\in \mathcal{D}^*(C_i\times C_i)$, and the corresponding integral operator $I_{f_i}$ satisfies the following relation:
\begin{equation}\label{eqn:ph}
\id -p_i=\rd \circ I_{f_i}+I_{f_i}\circ \rd,
\end{equation}
which is the limit of \eqref{eqn:approhomo}.  Therefore the integral operator $I_{f_i}=\lim I_{f_i^n}$ gives us the homotopy $H_i$ for the projection $p_i$. This explains the perturbation data, which shall motivate the differential on the minimal Morse-Bott cochain complex. However, we will not use \eqref{eqn:ph} to avoid working with currents ($f_i$ is only a current), and always work with the approximation \eqref{eqn:approhomo} and then take limits. More precisely, we will only use the ``classical relation" \eqref{minihomo}.

From the discussions above, we have the following definition.
\begin{definition}\label{def:def}
	\textbf{Defining data $\bm{\Theta}$} for an oriented flow category $\cC$ consists of the following:
\begin{itemize}
	\item Quasi-isomorphic embeddings $H^*(C_i)\to \Omega^*(C_i)$, the image is denoted by $h(\cC,i)$ and we fix a basis $\{\theta_{i,a}\}$ of $h(\cC,i)$ and a dual basis $\{\theta_{i,a}^*\}$ in the sense that $\langle \theta_{i,a}^*, \theta_{i,b}\rangle_{i}=\delta_{ab}$;
	\item A sequence of Thom classes in the form of $\delta_i^n=\rd (\rho_n \psi_i)$ of the diagonal $\Delta_i \subset C_i\times C_i$ for all $i$;
	\item Primitives $f_i^n$, such that $\rd f^n_i=\delta^n_i-\sum_a \pi_1^*\theta_{i,a}\wedge \pi_2^*\theta_{i,a}^*$ and $f^n_i-f^m_i=(\rho_n-\rho_m)\psi_i$ for all $i$. 
\end{itemize}  
\end{definition}	
\begin{remark}
	The form $\sum_a \pi_1^*\theta_{i,a}\wedge \pi_2^*\theta_{i,a}^*$ in Definition \ref{def:def} does not depend on the basis $\{\theta_{i,a}\}$ for a fixed quasi-isomorphic embedding $H^*(C_i)\to \Omega^*(C_i)$. 
\end{remark}

\subsubsection{The perturbed operator $D_{k,T,\Theta}$}
Given defining data $\Theta$, we are able to write down the operator $D_{k,T,\Theta}$ from \eqref{hpl} using the perturbation data introduced above. Those  $D_{k,T,\Theta}$ will then be assembled to the differential on the minimal Morse-Bott cochain complex. To simplify the presentation, we first introduce the following notation.

\begin{enumerate}
	\item We use $[\alpha]$ to denote the cohomology class of a closed form $\alpha \in h(\cC,i)$ and $|\alpha|$ to denote the degree of the differential form.
	\item We write $\cM^{v,k}_{i_1,\ldots, i_r}:=\cM_{v,v+i_1}\times \ldots \times \cM_{v+i_r,v+k}$ for $0=i_0<i_1<i_2<\ldots<i_r<i_{r+1}=k$ for $r\ge 0$ with the product orientation.
	\item For $\alpha \in \Omega^* (C_v), \gamma \in \Omega^*(C_{v+k})$ and $f_{v+i_j}\in \Omega^*(C_{v+i_j}\times C_{v+i_j}), 1\le j \le r$, we define the pairing $\cM^{v,k}_{i_1,\ldots,i_r}[\alpha, f_{v+i_1},\ldots, f_{v+i_r}, \gamma]$ to be
	\begin{equation}\label{pair1}
   \int_{\cM^{v,k}_{i_1,\ldots,i_r}}  s^*_{v,v+i_1}\alpha\wedge (t_{v,v+i_1}\times s_{v+i_1,v+i_2})^* f_{v+i_1}\wedge \ldots  \wedge (t_{v+i_{r-1},v+i_r}\times s_{v+i_r,v+k})^*f_{v+i_r}\wedge t_{v+i_r, v+k}^*\gamma.
	\end{equation}
	Strictly speaking, before taking the wedge product, we need to pullback $s_{v,v+i_1}^*\alpha$, $(t_{v+i_{j-1},v+i_j}\times s_{v+i_j,v+i_{j+1}})^*f_{v+i_j}$,  $t_{v+i_r, v+k}^*\gamma$ to $\cM^{v,k}_{i_1,\ldots,i_r}$ through the natural projections. This also applies to all other similar formulae in this paper.
	\item For $\alpha\in h(\cC,v)$ and $k\ge 1$, we define
	\begin{eqnarray} 
	\dagger(\cC,\alpha,k)&:= &(|\alpha|+m_{v,v+k})(c_{v+k}+1); \\ \ddagger(\cC,\alpha,k)&:= & (|\alpha|+m_{v,v+k}+1)(c_{v+k}+1),
	\end{eqnarray}
	where $c_i:=\dim C_i$, $m_{i,j}:=\dim\cM_{i,j}$ when $i<j$ and $m_{i,i}:=c_i-1$.
\end{enumerate}	

Then the perturbation data in \S \ref{pertdata} and \eqref{hpl} motivate the following definition.
\begin{definition}
	Given defining data $\Theta$ and an increasing sequence $T:=\{0=i_0<i_1<\ldots<i_r<i_{r+1}=k\}$, we define a linear map $D_{k,T,\Theta}:H^*(C_v)\simeq h(\cC,v)\to h(\cC,v+k)\simeq H^*(C_{v+k})$ such that the following holds for any $\gamma \in h(\cC,v+k)$
	\begin{equation}\label{eqn:dkt}
	\langle D_{k,T,\Theta}[\alpha],[\gamma] \rangle_{v+k}:=(-1)^\star \lim_{n \to \infty} \mathcal{M}^{v,k}_{i_1,\ldots, i_r}[\alpha, f^n_{v+i_1},\ldots, f^n_{v+i_r},\gamma],
	\end{equation} 
	where $\star:=\sum_{j=0}^r \ddagger(\cC,\alpha,i_j)$. In other words, by \eqref{eqn:dual}, we can write 
	\begin{equation}\label{eqn:dkt'}
	D_{k,T,\Theta}([\alpha])=\sum_{a} (-1)^\star \lim_{n \to \infty} \mathcal{M}^{v,k}_{i_1,\ldots, i_r}[\alpha, f^n_{v+i_1},\ldots, f^n_{v+i_r},\theta_{v+k,a}]\cdot  [\theta^*_{v+k,a}].
	\end{equation}
\end{definition}

\begin{remark}\label{rmk:sign}
	One way to understand the signs in \eqref{eqn:dkt} is to treat $D_{k,T,\Theta}$ as a composition of certain operators. Let $\alpha \in \Omega^*(C_i)$ and $f \in \Omega^*(C_j\times C_j)$, then $\cM_{i,j}$ defines an operator:
	$$\cM_{i,j}(\alpha, f):=(-1)^{\ddagger(\cC,\alpha,0)}\int_{\cM_{i,j}} s_{i,j}^*\alpha \wedge  (t_{i,j}\times \id_j)^* f\in \Omega^*(C_j),$$
	where $t_{i,j}\times \id_j:\cM_{i,j}\times C_j\to C_j\times C_j$. Here, by omitting the pullback of projections for simplicity, $s_{i,j}^*\alpha \wedge  (t_{i,j}\times \id_j)^* f$ is a differential form on $\cM_{i,j}\times C_j$, we then integration along the $\cM_{i,j}$ fiber in the trivial fibration $\cM_{i,j}\times C_j$ to obtain a form on $C_j$. If $|f|=c_j-1$, then $|\cM_{i,j}(\alpha, f)|=|\alpha|+c_j-1-m_{i,j}$ and hence
	\begin{eqnarray*}
	\ddagger(\cC,\cM_{i,j}(\alpha, f),0) & = & (|\alpha|+c_j-1-m_{i,j}+m_{j,j}+1)(c_j+1) \\
	& = &(|\alpha|+c_j-1-m_{i,j}+c_j)(c_j+1) \\
	&\equiv & \ddagger(\cC,\alpha,j) \mod 2.
	\end{eqnarray*}
	Then for $g\in \Omega^*(C_k\times C_k)$, we have
	$$
	\cM_{j,k}(\cM_{i,j}(\alpha, f),g) = (-1)^{\ddagger(\cC,\alpha,0)+\ddagger(\cC,\alpha,j)} \int_{\cM_{i,j}\times \cM_{j,k}} s_{i,j}^*\alpha \wedge (t_{i,j}\times s_{j,k})^* f\wedge (t_{j,k}\times \id_k )^* g.
	$$
	In general, $(-1)^{\star} \cM^{s,k}_{i_1,\ldots,i_r}[\alpha, f_{s+i_1}^n,\ldots ,f_{s+i_r}^n, \gamma]$ is the integral of the wedge product of compositions of such operators with $t_{s+i_r,s+k}^*\gamma$ on $\cM_{s+i_r,s+k}$. 	When $f$ is $f^n_j$ for $n\gg 0$, $\cM_{i,j}(\alpha,f)$ should be viewed as an approximation of $H_j\circ d_{j-i}\circ \iota_i(\alpha)$ in \eqref{hpl}. In general, \eqref{eqn:dkt'} can be viewed as \eqref{hpl} applied to the Austin-Braam complex using the perturbation data in this subsection.
\end{remark}

The following lemma asserts that \eqref{eqn:dkt} is well-defined and will be used in the proof of the main theorem, we prove it in Appendix \ref{conv}. 
\begin{lemma}\label{conv1}
	For every $\alpha \in \Omega^*(C_s),\gamma \in \Omega^*(C_{s+k})$ and any defining data, we have
	$\displaystyle \lim_{n\to \infty} \cM^{s,k}_{i_1,\ldots,i_r}[\alpha, f^{n}_{s+i_1},\ldots,\allowbreak f^{n}_{s+i_r}, \gamma]\in \R$ exists.
\end{lemma}

\subsection{The minimal Morse-Bott cochain complex}\label{mbcochain}
The main theorem of this subsection is that we can get a well-defined cochain complex out of an oriented flow category with any defining data. The cochain complex is generated by the cohomology $H^*(C_i)$ of the flow category, hence it is called the minimal Morse-Bott cochain complex.
\begin{definition}\label{def:MMBCC}
	Given defining data $\Theta$, the minimal Morse-Bott complex of an oriented flow category $\cC:=\{C_i,\cM_{i,j}\}$ is defined by 
	$$\BC(\cC,\Theta):=\BC:=\varinjlim_{q\to -\infty}\prod_{j=q}^\infty H^*(C_j),$$ i.e.\ the direct sum near the negative end and direct product near the positive end.  To be more precise, every element in $\BC$ is a function $A:\Z \to \prod_{i=-\infty}^{\infty}H^*(C_i)$, such that $A(i)\in H^*(C_i)$, and there exists $N_A\in \Z$, such that $A(i)=0$ $\forall i<N_A$\footnote{Assume $\cC$ arises from a Morse-Bott function $f$ on a non-compact manifold (but $\cM_{i,j}$ is still compact, so it can not be any Morse-Bott function on any non-compact manifold), the differential in the cochain complex should increase the value of $f$, which forces the cochain complex to take direct limit in the positive direction. }. The differential $d_{\BC,\Theta}:\BC\to \BC$ is defined as $\prod_{k\ge 1} d_{k,\Theta}$,  where $d_{k,\Theta}: H^*(C_v)\to H^*(C_{v+k})$ is defined as
	$$d_{k,\Theta}:=\sum_{T} D_{k,T,\Theta},$$
	for all increasing sequence $T=\{0=i_0<i_1<\ldots,<i_r<i_{r+1}=k \}$ with $r\ge 0$. In other words, we have
	\begin{equation}\label{def}\langle d_{k,\Theta}[\alpha], [\gamma] \rangle_{v+k}  =  \lim_{n\to \infty}\sum_{T} (-1)^{\star}\cM^{v,k}_{i_1,\ldots,i_r}[\alpha, f_{v+i_1}^n,\ldots ,f_{v+i_r}^n, \gamma] \end{equation}
	for $\alpha\in h(\cC,v)$, $\gamma \in h(\cC,v+k)$ and $\star=\sum_{j=0}^r \ddagger(\cC,\alpha, i_j)$. If we define $d_{i,\Theta}=0$ for $i\le 0$, then for $A\in \BC$ we have
	$$(d_{\BC,\Theta}A)(i):=\sum_{j\in \Z} d_{{i-j},\Theta} A(j).$$
	Note that it is a finite sum. If moreover the flow category has a grading structure $\{d_i\}$, then $\BC$ is also graded. The grading of an element $\alpha\in H^*(C_i)$ is $|\alpha|+d_i$, which shall be viewed as in $\Z/k$ if $\{d_i\}$ is only a grading structure in $\Z/k$.
\end{definition}

\begin{remark}
    The degree of $d_{k,\Theta}[\alpha]$ in $H^*(C_{v+k})$ is $|\alpha|+c_{v+k}-m_{v,v+k}$ under the simplifying assumption after Remark \ref{rmk:flow} that $c_i,m_{i,j}$ are well-defined. If the assumption is not satisfied, then $d_{k,T,\Theta}$ can be decomposed w.r.t.\ connected components of $\cM^{v,k}_{i_1,\ldots,i_r}$, such that each component has a well-defined degree in $H^*(C_{v+k})$. Then we need to keep track of the connected component in the proofs, which only results in complication of notation.
\end{remark}
The main result of this section in the following. 

\begin{theorem}\label{thm:main}
	Given an oriented flow category $\cC$ and defining data $\Theta$, $(\BC,d_{\BC,\Theta})$ is a cochain complex. The cohomology $H(\BC,d_{\BC,\Theta})$ is independent of the defining data $\Theta$. If in addition the flow category is graded, then $\BC$ is also graded and the degree of $d_{\BC,\Theta}$ is $1$. 
\end{theorem}  
\begin{remark} 
	A few remarks of the non-triviality of Theorem \ref{thm:main} are in order.
	\begin{enumerate}
		\item  We prove in \S \ref{s7} that when the flow category comes from a Morse-Bott function $f$ on a closed manifold $M$, the cohomology of the minimal Morse-Bott cochain complex is the regular cohomology  $H^*(M,\R)$. Note that this follows from the definition if $f$ is constant: since the flow category is $\{C_0=M\}$ with only identities in the morphism space, as a result, $\BC=H^*(C_0,\R)=H^*(M;\R)$ with $d_{\BC}=0$. Therefore it suffices to show that the cohomology of the minimal Morse-Bott cochain complex is independent of the Morse-Bott function $f$.
		
		\item If all the critical manifolds $C_i$ are discrete, then the defining data $\Theta$ is unique. Assume for simplicity, each $C_i$ consists of one point, the minimal Morse-Bott cochain complex $\BC$ is generated by the critical points and equals to the usual Morse cochain complex, 
		$$\BC=\varinjlim_{q\to -\infty}\prod_{j=q}^\infty H^*(C_j)=\varinjlim_{q\to -\infty}\prod_{j=q}^\infty \R.$$
		Since $|f^n_i|=-1$, $d_{k,\Theta}:H^*(C_v)\to H^*(C_{v+k})$ only has the leading term
		$$\langle d_{k,\Theta}[1],[1]\rangle_{v+k}=\cM^{v,k}[1,1]=\int_{\cM_{v,v+k}}1.
		$$
		Therefore the differential $d_{\BC,\Theta}:=\sum_{k\ge 1} d_{k,\Theta_0}$ is just the signed counting of all zero-dimensional moduli spaces $\cM_{v,v+k}$, which is the usual cochain differential in a non-degenerate Morse/Floer theory. 
	\end{enumerate} 
\end{remark}
\begin{remark}
	Theorem \ref{thm:main} is the most simple version. We generalize Theorem \ref{thm:main} in \S \ref{s5} and \S \ref{s6} to following cases: (1) $C_i$ is not oriented; (2) $C_i$ is not compact; (3) the defining data is not minimal, i.e.\ the rank of the projection in the perturbation data is larger than $\dim H^*(C_i)$. 
\end{remark}
\begin{corollary}\label{cor:dim}
	If the oriented flow category $\cC$ has the property that $\dim C_i \le k$ for all $i$, then the minimal Morse-Bott cochain complex $\BC(\cC)$ only depends on $\cM_{i,j}$ with $\dim \cM_{i,j}\le 2k$.
\end{corollary}
\begin{proof}
	Since $|f^n_i| = \dim C_i - 1 \le k -1$ and $|\alpha|,|\gamma|\le k$, if $\cM_{i,j}$ appears in an integration in the definition of the differential with $\dim \cM_{i,j} > 2k$, there is no way the pullbacks of those forms can contain a nontrivial component in $\wedge^{\dim \cM_{i,j}}\cM_{i,j}$. Therefore the integration must be zero.  Note that when $k=0$, this amounts to say the fact that the cochain complex only depends on zero dimensional moduli spaces (although the existence of $1$ dimensional moduli spaces is needed to show that $d^2=0$).
\end{proof}

We first show that $(\BC,d_{\BC,\Theta})$ is a cochain complex, the invariance is deferred to the next subsection. For simpicity, we first introduce the following notation.
\begin{enumerate}
	 \item For $ 0<i_1<i_2\ldots<i_r<k$, we define  
	 \begin{equation}\label{fiber}
	 \cM^{v,k}_{i_1,\ldots,\overline{i_p},\ldots,i_r}:=\cM_{v,v+i_1}\times \ldots \times(\cM_{v+i_{p-1},v+i_p}\times_{v+i_p}\cM_{v+i_p,v+i_{p+1}})\times \ldots \times \cM_{v+i_r,v+k}
	 \end{equation}
	 with the product orientation.
	 \item For $\alpha \in \Omega^* (C_v), \gamma\in \Omega^*(C_{v+k})$ and $f_{v+i_j}\in \Omega^*(C_{v+i_j}\times C_{v+i_j})$, we define $\cM^{v,k}_{i_1,\ldots,i_r}[\rd (\alpha, f_{v+i_1},\ldots, f_{v+i_r}, \gamma)]$ to be
	 \begin{equation}\label{paird}
	 \int_{\cM^{v,k}_{i_1,\ldots,i_r}} \rd \left(s^*_{v,v+i_1}\alpha\wedge (t_{v,v+i_1}\times s_{v+i_1,v+i_2})^*f_{v+i_1}\wedge \ldots \wedge (t_{v+i_{r-1},v+i_r}\times s_{v+i_r,v+k})^*f_{v+i_r}\wedge t_{v+i_r, v+k}^*\gamma\right).
	 \end{equation}
	 \item We define the pairing $\cM^{v,k}_{i_1,\ldots,\overline{i_p},\ldots, i_r}[\alpha, f_{v+i_1},\ldots, f_{v+i_{p-1}},  f_{v+i_{p+1}},\ldots, f_{v+i_r}, \gamma]$ over  $\cM^{v,k}_{i_1,\ldots,\overline{i_p},\ldots,i_r}$ to be
      \begin{equation}\label{pair2}
      \begin{aligned}
       \int_{\cM^{v,k}_{i_1,\ldots,\overline{i_p},\ldots,i_r}}  s^*_{v,v+i_1}\alpha\wedge (t_{v,v+i_1}\times s_{v+i_1,v+i_2})^* f_{v+i_1}\wedge \ldots \wedge (t_{v+i_{p-2},v+i_{p-1}}\times s_{v+i_{p-1},v+i_{p+1}})^*f_{v+i_{p-1}} \\
        \qquad \wedge(t_{v+i_{p-1},v+i_{p+1}}\times s_{v+i_{p+1},v+i_{p+2}})^*f_{v+i_{p+1}}\wedge \ldots \wedge (t_{v+i_{r-1},v+i_r}\times s_{v+i_r,v+k})^*f_{v+i_r}\wedge t_{v+i_r, v+k}^*\gamma.
     \end{aligned}
     \end{equation}
     \item When we compose two operators,  a trace term  will appear. Therefore we introduce $\tr^{v+i_p}\cM^{v,k}_{i_1,\ldots,i_r}[\alpha, \allowbreak f_{v+i_1}, \ldots, f_{v+i_{p-1}}, \theta\theta^*_{v+i_p}, f_{v+i_{p+1}}, \ldots, f_{v+i_r},\gamma]$ to denote the following,
      \begin{equation}\label{pair3}
     \begin{aligned} 
      \int_{\cM^{v,k}_{i_1,\ldots,i_r}} & s^*_{v,v+i_1} \alpha\wedge (t_{v,v+i_1}\times s_{v+i_1,v+i_2})^*f_{v+i_1}\wedge \ldots \\
      & \wedge (t_{v+i_{p-1},v+i_p}\times s_{v+i_p,v+i_{p+1}})^*(\sum_a\pi_1^*\theta_{v+i_p,a}\wedge \pi_2^*\theta_{v+i_p,a}^*) \wedge \ldots \wedge (t_{v+i_r}\times s_{v+i_r})^*f_{v+i_r}\wedge t_{v+k}^*\gamma,
     \end{aligned}
     \end{equation}
     where $\pi_1,\pi_2$ are the projections of $C_{v+i_p}\times C_{v+i_p}$ to the first and second factor respectively.
\end{enumerate}

Heuristically speaking, the ``Thom class" of $\cM^{v,k}_{i_1,\ldots,i_{p-1},\overline{i_p}, i_{p+1},\ldots, i_r}\subset \cM^{v,k}_{i_1,\ldots,i_r}$ is given by the pullback of $(t_{v+i_{p-1},v+i_p}\times s_{v+i_p,v+i_{p+1}})^*\delta^n_{v+i_p}\in \Omega^*(\cM_{v+i_{p-1},v+i_p}\times \cM_{v+i_p,v+i_{p+1}})$ to $ \cM^{v,k}_{i_1,\ldots,i_r}$ by the natural projection. Hence we will have the following lemma, which is crucial to the proof of $d_{\BC,\Theta}^2=0$ and will be proven in Appendix \ref{conv}. 
\begin{lemma}\label{conv2}
	For an oriented flow category $\cC$ and any defining data, we have
	\begin{equation*}
	\lim_{n\to \infty}  \cM^{v,k}_{i_1,\ldots,i_r}[\alpha, f_{v+i_1}^{n},\ldots ,\delta^{n}_{v+i_p},\ldots, f_{v+i_r}^{n}, \gamma]
	= (-1)^* \lim_{n\to \infty}\cM^{v,k}_{i_1,\ldots,i_{p-1},\overline{i_p}, i_{p+1},\ldots, i_r}[\alpha, f_{v+i_1}^{n},\ldots, f_{v+i_r}^{n},\gamma],\end{equation*}
	where $*=(|\alpha|+m_{v,v+i_p})c_{v+i_p}$.
\end{lemma}
 
\begin{proposition}{\label{dsquare}}
	$(\BC,d_{\BC,\Theta})$ is a cochain complex, i.e.\ $d_{\BC,\Theta}^2=0$.
\end{proposition} 
\begin{proof}
For simplicity, we will suppress the subscript $\Theta$ in the proof. It suffices to show that for all $\alpha \in h(\cC,v)$ and $\gamma\in h(\cC,v+k)$, we have
\begin{equation}\label{goal1}\left\langle \sum_{i=1}^{k-1} d_{k-i}\circ d_i[\alpha],[\gamma]\right\rangle_{v+k}=0.\end{equation}
We first prove the following lemma.

\begin{lemma}\label{lemma:dsquare}
For $r\ge 1$, we have 
\begin{equation}
\begin{aligned}
0=(-1)^{|\alpha|c_v} & \int_{\partial \cM_{v,v+k}}  s_{v,v+k}^*\alpha \wedge t_{v,v+k}^* \gamma  = \lim_{n\to \infty}\sum_{0<i_1<\ldots<i_r<k}  (-1)^{\star_1} \cM^{v,k}_{i_1,\ldots,i_r}[\rd (\alpha, f^n_{v+i_1},\ldots,f^n_{v+i_r},\gamma)] +  \label{exact} \\
& \lim_{n\to \infty}\sum_{\substack{1\le p \le q\le r\\ 0<i_1<\ldots<i_q<k}} (-1)^{\star_2} \tr^{v+i_p}\cM^{v,k}_{i_1,\ldots,i_q} [\alpha, f^n_{v+i_1},\ldots, f^n_{v+i_{p-1}},\theta\theta_{v+i_p}^*, f^n_{v+i_{p+1}}, \ldots, f^n_{v+i_q}, \gamma],
\end{aligned}    
\end{equation}
where $$
\begin{aligned}\star_1 & =|\alpha|c_v+\sum_{j=1}^r \dagger(\cC,\alpha,i_j),\\
\star_2 & =|\alpha|(c_v+1)+\sum_{j=1}^{p-1} \ddagger(\cC,\alpha,i_j)+\sum_{j=p}^q \dagger(\cC,\alpha,i_j).
\end{aligned}
$$
\end{lemma}
\begin{proof}
\textbf{Step 1: the $r=1$ case.} In this case, since $p=q=r=1$ for the second term, we write $i=i_1$. Then we have $\star_2=|\alpha|(c_v+1)+\dagger(\cC,\alpha,i)$. Using the equation $\delta_*^n-\sum_a \pi_1^*\theta_{*,a} \wedge \pi_2^*\theta_{*,a}^*=\rd f^n_{*}$ for any $n\in \N$, we have
\begin{eqnarray}
(-1)^{\star_2} \tr^{v+i}\cM^{v,k}_{i}[\alpha,\theta\theta^*_{v+i},\gamma] & = &\sum_i (-1)^{\star_2}  \cM^{v,k}_{i}[\alpha, \delta^n_{v+i}-\rd f^n_{v+i},\gamma] \nonumber\\
& = &\lim_{n\to \infty}\sum_i (-1)^{\star_2}  \cM^{v,k}_{i}[\alpha, \delta^n_{v+i}-\rd f^n_{v+i},\gamma] \nonumber\\
& = & \lim_{n\to \infty}\sum_i (-1)^{\star_2}  \cM^{v,k}_{i}[\alpha, \delta^n_{v+i},\gamma] \label{relation3}\\
& & +\lim_{n\to \infty} \sum_i (-1)^{\star_2+1}  \cM^{v,k}_{i}[\alpha,\rd f^n_{v+i},\gamma]. \label{formula2}
\end{eqnarray}
By Lemma \ref{conv2}, we have
\begin{equation}\label{relation4}
	\lim_{n\to \infty}\sum_i (-1)^{\star_2}  \cM^{v,k}_{i}[\alpha, \delta^n_{v+i},\gamma] =\sum_i (-1)^{\star_2+(|\alpha|+m_{v,v+i})c_{v+i}}  \mathcal{M}^{v,k}_{\overline{i}}[\alpha,\gamma].
\end{equation}	
Since $(-1)^{\star_2+(|\alpha|+m_{v,v+i})c_{v+i}}=(-1)^{|\alpha|c_v+m_{v,v+i}}$, and  $\partial[\cM_{ik}]=\sum(-1)^{m_{i,j}}[\cM_{ij}]\times_j[\cM_{jk}]$, by Stokes' theorem  
$$\begin{array}{>{\displaystyle}r>{\displaystyle}c>{\displaystyle}l} \eqref{relation4}  &= &\sum_i(-1)^{|\alpha|c_v+m_{v,v+i}} \int_{\cM_{v,v+i}\times_{v+i}\cM_{v+i,v+k}} s_{v,v+i}^*\alpha \wedge t_{v+i,v+k}^* \gamma\\
&= &(-1)^{|\alpha|c_v} \int_{\partial \cM_{v,v+k}} s_{v,v+k}^*\alpha \wedge t_{v,v+k}^* \gamma\\ 
&= &(-1)^{|\alpha|c_v} \int_{ \cM_{v,v+k}} \rd\left(s_{v,v+k}^*\alpha \wedge t_{v,v+k}^* \gamma\right)\\
&=& 0.\end{array}$$    
Now, we have 
$$\eqref{formula2} = \displaystyle\lim_{n\to \infty}\sum_i(-1)^{\star_2+1+|\alpha|}\cM^{v,k}_i[\rd(\alpha,f_{v+i}^n,\gamma)].$$
Note that the difference between $\star_1$ and $\star_2$ in the $r=1$ case is indeed $|\alpha|$. This proves the $r=1$ case.

\textbf{Step 2: independence of $r$.}
We need to prove the value of the RHS does not change from $r$ to $r+1$.  We apply Stokes' theorem to the exact term in\eqref{exact} in the $r$ case. The boundary $\partial (\cM_{v,v+i_1}\times \ldots \times \cM_{v+i_r,v+k})$ comes from fiber product at $v+w$ for all $t,w$ such that $0<i_1<\ldots<i_t<w<i_{t+1}<\ldots i_r<k$. Consider the boundary coming from the fiber product at $v+w$, after applying Stokes' theorem to the exact term in \eqref{exact}, we have the contribution from  integration over the  $\cM^{v,k}_{i_1,\ldots,i_t,\overline{w},\ldots, i_r}\subset \cM^{v,k}_{i_1,\ldots, i_r}$ is
\begin{equation}\label{dsquare3}
(-1)^{\star_3}\lim_{n\to \infty}\cM^{v,k}_{i_1,\ldots,i_t,\overline{w},\ldots, i_r} [\alpha, f^n_{v+i_1}, \ldots, f^n_{v+i_r},\gamma],
\end{equation}
where $\star_3=|\alpha|c_v+\sum_{j=1}^r \dagger(\cC,\alpha,i_j) +m_{v,v+i_1}+\ldots+m_{v+i_t,v+w}$. By replacing the fiber product in  $\cM^{v,k}_{i_1,\ldots,i_t,\overline{w},\ldots, i_r}$ with the Cartesian product $\cM^{v,k}_{i_1,\ldots,i_t,w,\ldots, i_r}$, we have by Lemma \ref{conv2},
\begin{equation}\label{induction3}
\eqref{dsquare3}=(-1)^{\star_3+(|\alpha|+m_{v,v+w})c_{v+w}}\lim_{n\to \infty} \cM^{v,k}_{i_1,\ldots,i_t,w,\ldots, i_r}[\alpha, f_{v+i_1}^n,\ldots, \delta^n_{v+w},\ldots, f_{v+i_r}^n,\gamma].
\end{equation}
We replace the Thom class $\delta_*^n$ by $\sum_a\pi_1^*\theta_{*,a} \wedge \pi_2^*\theta_{*,a}^*+\rd f^n_*$ to get
\begin{eqnarray}
\eqref{induction3} & = &(-1)^{\star_3+(|\alpha|+m_{v,v+w})c_{v+w}}\lim_{n\to\infty}\tr^{v+w}  \cM^{v,k}_{i_1,\ldots,i_t,w, \ldots, i_r}[\alpha, f_{v+i_1}^n,\ldots, \theta\theta^*_{v+w},\ldots, f_{v+i_r}^n,\gamma] \\ & & +(-1)^{\star_3+(|\alpha|+m_{v,v+w})c_{v+w}}\lim_{n\to \infty}\cM^{v,k}_{i_1,\ldots,i_t,w,\ldots, i_r}[\alpha, f_{v+i_1}^n,\ldots, \rd f^n_{v+w},\ldots, f_{v+i_r}^n,\gamma]. \label{end1}
\end{eqnarray}
Let $\star_4$ denote $\star_3+(|\alpha|+m_{v,v+w})c_{v+w}$.
By \eqref{dim} we have
$$\star_4=|\alpha|(c_v+1)+\sum_{j=1}^t \ddagger(\cC,\alpha,i_j)+\dagger(\cC,\alpha,w)+\sum_{j=t+1}^r\dagger(\cC,\alpha,i_j) \mod 2.$$
Because $\star_5:=\star_4+|\alpha|+\sum_{j=1}^t(c_{v+i_j}+1)\equiv |\alpha|c_v+\sum_{j=1}^r\dagger(\mathcal{C},\alpha,i_j)+\dagger(\mathcal{C},\alpha,w) \mod 2$ and $|f_{v+i_j}^n|\equiv c_{v+i_j}+1\mod 2$, we have
\begin{equation}
\eqref{end1}= \lim_{n\to \infty}\sum_{0<i_1<\ldots<i_t<w<i_{t+1}<i_r<k}(-1)^{\star_5}\cM^{v,k}_{i_1,\ldots,i_t,w,i_{t+1},\ldots,i_r}[\rd(\alpha, f^n_{v+i_1},\ldots, f^n_{v+w},\ldots, f_{v+i_r}^n, \gamma)].
\end{equation}
Therefore we arrive at
\begin{eqnarray}
RHS & = & \lim_{n\to \infty}\sum_{\substack{1\le p \le q\le r\\ 0<i_1<\ldots<i_q<k}} (-1)^{\star_2} \tr^{v+i_p}\cM^{s,k}_{i_1,\ldots,i_q} [\alpha, f^n_{v+i_1},\ldots, f^n_{v+i_{p-1}},\theta\theta_{v+i_p}^*, f^n_{v+i_{p+1}}, \ldots, f^n_{v+i_q}, \gamma] \nonumber \\
& &+ \lim_{n\to \infty}\sum_{0<i_1<\ldots<i_t<w<i_{t+1}<i_r<k} (-1)^{\star_4} \tr^{v+w}\cM^{s,k}_{i_1,\ldots,i_t,w,i_{t+1},\ldots,i_r}[\alpha, f^n_{v+i_1},\ldots, \theta\theta_{v+w}^*,\ldots, f_{v+i_r}^n, \gamma] \nonumber\\
& &+\lim_{n\to \infty} \sum_{0<i_1<\ldots<i_t<w<i_{t+1}<i_r<k} (-1)^{\star_5} \cM^{v,k}_{i_1,\ldots,i_t,w,i_{t+1},\ldots,i_r}[\rd(\alpha, f^n_{v+i_1},\ldots, f^n_{v+w},\ldots, f_{v+i_r}^n, \gamma)]. \nonumber
\end{eqnarray}
This is the $r+1$ case, so we have proved the claim.
\end{proof}

Going back to the proof of Proposition \ref{dsquare}, in the case of $r=k-1$ in Lemma \ref{lemma:dsquare},  the following two terms sum to zero
\begin{equation} \label{closed}	
	 \lim_{n\to \infty}  (-1)^{\star_1} \cM^{v,k}_{1,\ldots,k-1}[\rd (\alpha, f^n_{v+1},\ldots,f^n_{v+k-1},\gamma)],
\end{equation}	 
\begin{equation}\label{trace}\lim_{n\to \infty}\sum_{\substack{1\le p \le q\le k-1\\ 0<i_1<\ldots<i_q<k}} (-1)^{\star_2} \tr^{v+i_p}\cM^{v,k}_{i_1,\ldots,i_q} [\alpha, f^n_{v+i_1},\ldots, f^n_{v+i_{p-1}},\theta\theta_{v+i_p}^*, f^n_{v+i_{p+1}}, \ldots, f^n_{v+i_q}, \gamma],
\end{equation}
where 
$$\begin{aligned}\star_1 & =|\alpha|c_v+\sum_{j=1}^{k-1} \dagger(\cC,\alpha,j),\\
 \star_2 & =|\alpha|(c_v+1)+\sum_{j=1}^{p-1} \ddagger(\cC,\alpha,i_j)+\sum_{j=p}^q \dagger(\cC,\alpha,i_j).
\end{aligned}
$$
Since $\cM^{v,k}_{1,\ldots,k-1}$ is a closed manifold, \eqref{closed} is $0$ by Stokes' theorem. For the remaining term, we claim that 
\begin{equation}\label{eqn:claimtrace}
\eqref{trace}=\left\langle \sum_{i=1}^{k-1} d_{k-i}\circ d_i[\alpha],[\gamma]\right\rangle_{v+k}.\end{equation}
Since $|d_i\alpha|=|\alpha|+m_{v,v+i}+c_{v+i}\mod 2$,  we have 
$$\ddagger(\cC,d_i\alpha,j)=\dagger(\cC,\alpha,i+j), \mod 2.$$
Then the claim simply follows from the definition of $d_i$.
\end{proof}
                         
\begin{remark}
	From the proof of Proposition \ref{dsquare}, we see that there is no harm in suppressing the index $n$ and $\displaystyle \lim_{n\to\infty}$ by Lemma \ref{conv1} and \ref{conv2}. If we write $f_i$ as the limit of $f_i^n$ in the space of currents, such that 
	\begin{equation}\label{rel}\delta_i=\pi_1^*\theta_{i,a}\wedge \pi_2^*\theta_{i,a}^*+\rd f_{i}\end{equation} 
	where $\delta_i$ is the  Dirac current, then we can use \eqref{rel} to do formal computations.  
\end{remark}

\subsection{Flow morphisms induce cochain morphisms}\label{flowmor}
\S \ref{mbcochain} shows that a flow category carries enough geometric structure to define a cochain complex. In the following subsections, we study the analogous geometric data for cochain complex morphisms and homotopies. In this subsection, we introduce flow morphisms between flow categories, which is the underlying geometric data to define continuation maps \cite[Chapter 11]{audin2014morse}. We show that every flow category has an identity flow morphism from the flow category to itself.  Using the identity flow morphism, we show that $H(\BC,d_{\BC,\Theta})$ is independent of the defining-data $\Theta$, thus we finish the proof of Theorem \ref{thm:main}.    

\subsubsection{Flow morphisms}
\begin{definition}\label{morphism}
	An \textbf{oriented flow morphism} $\fH$ from an oriented flow category $\cC:=\{C_i,\cM^{C}_{i,j}\}$ to another oriented flow category $\cD:=\{D_i,\cM^{D}_{i,j}\}$ is a family of compact oriented manifolds  $\{\cH_{i,j}\}_{i,j\in \Z}$, such that the following holds. 
	\begin{enumerate}
		\item\label{m0} There are two smooth maps $s: \cH_{i,j}\to C_i, t:\cH_{i,j}\to D_j$.
		\item\label{m00}  $\exists N\in \Z$,  such that when $i-j>N$, $\cH_{i,j}=\emptyset$.
		\item\label{m2} For every $i_0<i_1<\ldots<i_k$, $ j_0<\ldots<j_{m-1}<j_m$, the fiber product $$\cM^C_{i_0,i_1}\times_{i_1}\ldots \times_{i_k} \cH_{i_k,j_0}\times_{j_0}\ldots\times_{j_{m-1}}\cM^D_{j_{m-1},j_m}$$  is cut out transversely.
		\item\label{m1} There are smooth maps $m_L:\cM^C_{i,j}\times_j \cH_{j,k}\to \cH_{i,k}$ and $m_R: \cH_{i,j}\times_j \cM^D_{j,k}\to \cH_{i,k}$, such that 
		$$\begin{aligned}s\circ m_L(a,b) & =s^C(a), &
		 t\circ m_L(a,b) & =t(b),\\
		s\circ m_R(a,b) & =s(a), &
		 t\circ m_R(a,b) & =t^D(b),
		\end{aligned}$$
		where the map $s^C$ is the source map for the flow category $\cC$ and the map $t^D$ is the target map for the flow category $\cD$.
		\item\label{m4}The map $m_L\cup m_R: \left(\cup_j\cM^C_{i,j}\times_j \cH_{j,k}\right)\cup \left( \cup_j \cH_{i,j}\times_j \cM^\cD_{j,k}\right)\to \partial \cH_{i,k}$ is a diffeomorphism up to zero-measure (Definition \ref{def:measure}).
		\item\label{m3} The orientation $[\cH_{i,j}]$ has the following properties,
		$$\partial [\cH_{i,j}]=\sum_{p>0} (-1)^{m^C_{i,i+p}} m_L\left([\cM^C_{i,i+p}\times_{i+p} \cH_{i+p,j}]\right)+\sum_{p>0} (-1)^{h_{i,j}}m_R\left([\cH_{i,j-p}\times_{j-p} \cM^D_{j-p,j}]\right),$$
		$$(t^C\times s)^*[N_j][\cM^C_{i,j}\times_j \cH_{j,k}]=(-1)^{c_j m^C_{i,j}}[\cM^C_{i,j}][\cH_{j,k}],$$ 
		$$(t\times s^{D})^*[N_j][\cH_{i,j}\times_j \cM^D_{j,k}]=(-1)^{d_jh_{i,j}}[\cH_{i,j}][\cM^D_{j,k}].$$
		Here $c_i:=\dim C_i, m_{i,j}^C:=\dim \cM^C_{i,j}, d_j:=\dim D_j$  and $h_{i,j}=\dim \cH_{i,j}$. 
	\end{enumerate}
\end{definition}
By condition \eqref{m1}, we have a formula similar to \eqref{identification}. Thus it is convenient to use $m_L,m_R$ to identify $\cM^C_{i,j}\times_j \cH_{j,k}, \cH_{i,j}\times_j \cM^D_{j,k}$ with the corresponding parts of $\partial \cH_{i,k}$. Hence in the following, we will suppress $m_L,m_R$, and treat $\cM^C_{i,j}\times_j \cH_{j,k}, \cH_{i,j}\times_j \cM^D_{j,k}$ as though they are contained in $\partial \cH_{i,k}$.


\begin{remark}
 Condition \eqref{m00} is important for us  to obtain a finite sum in the definition of the induced cochain morphism. In the context of Morse/Floer theories, the existence of  $N$  usually comes from some energy estimates. More precisely, $\cH_{i,j}$ is typically the compactification of the space of solutions to parameterized Floer equations/gradient flow equations interpolating the  geometric data for $\cC$ and $\cD$. Then there is usually some notion of energy $E(u)$ for a Floer cylinder/gradient flow $u$ in the moduli space $\cH_{i,j}$, such that $E(u)\ge 0$. Now we assume that the energy $E(u)$ satisfies inequality $E(u)\le g(D_j)-f(C_i)+C$, where $f$ and $g$ are the background Morse-Bott functionals for $\cC$ and $\cD$, and $C$ is a universal constant depending on the interpolating data we use to define the moduli space $\cH_{i,j}$. Assuming the critical values do not accumulate for simplicity\footnote{When critical values accumulate, see Remark \ref{rmk:flow}}, then if $j\ll i$, we have $E(u)<0$, i.e.\ there are no curves in $\cH_{i,j}$.
\end{remark}

\begin{remark}
    Similar to Definition \ref{def:grade}, we say $\fH$ is compatible with the grading structures on $\cC,\cD$ iff $d(C_i)=d(D_j)+d_j-h_{i,j}$, where $\{d(C_i)\},\{d(D_j)\}$ are grading structures on $\cC,\cD$ respectively. When this holds, the cochain morphism $\phi^H$ below will have degree $0$.
\end{remark}
   
The main result of this subsection is that oriented flow morphisms induce cochain morphisms between the minimal Morse-Bott cochain complexes. Let $\cC:=\{C_i,\cM^{\cC}_{i,j}\}$ and $\cD:=\{D_i,\cM^{\cD}_{i,j}\}$ be two oriented flow categories. Assume $\fH=\{\cH_{i,j}\}$ is an oriented flow morphism from $\cC$ to $\cD$, then we introduce the following.
	\begin{enumerate}
		\item We write $c_i:=\dim C_i, d_i:=\dim D_i, m^{C}_{i,j}:=\dim \cM^C_{i,j}$, $m^D_{i,j}:=\dim \cM^D_{i,j}$ and $h_{i,j}:=\dim \cH_{i,j}$. We formally define $m^C_{i,i}=c_i-1$ and $m^D_{i,i}=d_i-1$ as before. We assume as before that those numbers are well-defined. Then we have that 
		$$h_{i,j}+m^D_{j,k}-d_{j}+1=h_{i,k},\quad j\le k$$
		and 
		$$m^C_{i,j}+h_{j,k}-c_i+1=h_{i,k}, \quad i\le j$$
		by Definition \ref{morphism}.
		\item For $v, k\in \Z$ and $0<i_1<\ldots<i_p$ and $j_1<\ldots<j_q<k$, We define
		$$\cH^{v,k}_{i_1,\ldots,i_p|j_1\ldots,j_q}:=\cM^C_{v,v+i_1}\times \ldots \times\cM^C_{v+i_{p-1},v+i_p}\times \cH_{v+i_p,v+j_1}\times \cM^D_{v+j_1,v+j_2}\times \ldots \times \cM^D_{v+j_q,v+k}$$ with the product orientation.
		\item $\cH^{*,*}_{\ldots|\ldots }[\alpha,  f_*, \ldots, f_*, \ldots, \gamma]$ is defined similarly to $\cM^{*,*}_{\ldots}[\alpha, f_*,\ldots, \gamma]$ in \eqref{pair1}.
		\item  For $\alpha \in \Omega^*(C_v)$, we define $\dagger(\fH,\alpha,k)=(|\alpha|+h_{v,v+k})(d_{v+k}+1)$ and $\ddagger(\fH,\alpha, k):=(|\alpha|+h_{v,v+k}+1)(d_{v+k}+1)$
	\end{enumerate}
   
Let $\Theta_1:=\{h(\cC,i), f_i^{C,n}\}$ and $\Theta_2:=\{h(\cD,i), f_i^{D,n}\}$ be defining data for flow categories $\cC$ and $\cD$ respectively. Let $\fH := \{\cH_{i,j}\}$ be an oriented flow morphism from $\cC$ to $\cD$. The counterparts of Lemma \ref{conv1} and \ref{conv2} hold for $\cH$ by the same argument. Then we define an linear operator $\phi^{H}_{k,\Theta_1,\Theta_2}: H^*(C_v) \to H^*(D_{v+k})$ for every $v,k\in \Z$ as follows. 
\begin{equation}\label{mordef}
\begin{aligned}
\left\langle\phi^H_{k,\Theta_1,\Theta_2}[\alpha],[\gamma]\right\rangle_{v+k}&:= \sum_{\substack{p,q\ge 0\\0=i_0<i_1<\ldots<i_p\\j_1<\ldots<j_q<j_{q+1}=k}} (-1)^{\ast} \cH^{v,k}_{i_1,\ldots,i_p|j_1,\ldots,j_q}[\alpha, f^C_{v+i_1},\ldots, f^C_{v+i_p},f^D_{v+j_1},\ldots, f^D_{v+j_q},\gamma]\\
&:= \lim_{n\to \infty}\sum_{\substack{p,q\ge 0\\0=i_0<i_1<\ldots<i_p\\j_1<\ldots<j_q<j_{q+1}=k}} (-1)^{\ast} \cH^{v,k}_{i_1,\ldots,i_p|j_1,\ldots,j_q}[\alpha, f^{C,n}_{v+i_1},\ldots, f^{C,n}_{v+i_p},f^{D,n}_{v+j_1},\ldots, f^{D,n}_{v+j_q},\gamma],\\
\end{aligned}
\end{equation}
where $$\ast:=|\alpha|c_{v}+h_{v,v+j_1}+\sum_{w=1}^p\ddagger(\cC,\alpha,i_w)+\sum_{w=1}^q \ddagger(\fH,\alpha,j_w).$$ 
The existence of $N$ in the condition \eqref{m00} of Definition \ref{morphism} implies that \eqref{mordef} is a finite sum and $\phi^H_{k,\Theta_1,\Theta_2}=0$ for $k<-N$.

\begin{theorem}\label{Mor}
	Let $\fH:\cC \Rightarrow \cD$ be an oriented flow morphism. If we fix defining data $\Theta_1:=\{h(\cC,i), f_i^{C,n}\}$ and $\Theta_2:=\{h(\cD,i), f_i^{D,n}\}$  for $\cC$ and $\cD$ respectively, then there is a linear map $\phi_{\Theta_1,\Theta_2}^H = \prod_{k\in \Z} \phi^H_{k,\Theta_1,\Theta_2}: \BC(\cC,\Theta_1) \to \BC(\cD,\Theta_2)$ given by \eqref{mordef}, such that 
     $$\phi^H_{\Theta_1,\Theta_2}\circ d^C_{\BC,\Theta_1}- d^D_{\BC,\Theta_2}\circ\phi^H_{\Theta_1,\Theta_2}=0.$$
    In particular, $\phi^H_{\Theta_1,\Theta_2}$ induces a map $H(\BC(\cC), d^C_{\BC,\Theta_1})\to H(\BC(\cD),d^D_{\BC,\Theta_2})$ on cohomology.
\end{theorem}    
\begin{proof}
Similar to the proof of Proposition \ref{dsquare}, this theorem follows from the claim below:

For $\alpha\in h(\cC,v),\gamma\in h(\cC,v+k)$ with $k\in \Z$, and any $r\ge 1$, we have
\begin{eqnarray*}
	0&= &(-1)^{1+|\alpha|c_v+h_{v,v+k}}\int_{\partial H_{v,v+k}}s^*\alpha \wedge t^*\gamma  =\\
	&  & \qquad \sum_{\substack{0\le p\le r\\ 0<i_1<\ldots<i_p\\j_1<\ldots<j_{r-p}<k}} (-1)^{\ast_1} \cH^{v,k}_{i_1,\ldots,i_p|j_1,\ldots,j_{r-p}}[ \rd (\alpha, f^C_{v+i_1},\ldots, f^C_{v+i_p},f^D_{v+j_1},\ldots, f^D_{v+j_{r-p}}, \gamma)] \\
	&  & \qquad +\sum_{\substack{0\le p\le q\le r, 1\le t\le p\\ 0<i_1\ldots<i_p\\j_1<\ldots<j_{q-p}<k}} (-1)^{\ast_2} \tr^{v+i_t} \cH^{v,k}_{i_1,\ldots,i_p|j_1,\ldots,j_{q-p}} [\alpha, f^C_{v+i_1},  \ldots, \theta^C{\theta^C_{v+i_t}}^*,\ldots, f^D_{v+i_{q-p}},\gamma] \\
	& &\qquad  +\sum_{\substack{0\le p\le q\le r, 1\le t\le q-p\\ 0<i_1\ldots<i_p\\j_1<\ldots<j_{q-p}<k}} (-1)^{\ast_3} \tr^{v+j_t}\cH^{v,k}_{i_1,\ldots,i_p|j_1,\ldots,j_{q-p}}[\alpha, f^C_{v+i_1}, \ldots, \theta^D {\theta^D_{v+j_{t}}}^*, \ldots, f^D_{v+i_{q-p}}, \gamma] .
\end{eqnarray*}
	Here 
\begin{eqnarray*}
		\ast_1 &= & 1+|\alpha|(c_v+1)+h_{v,v+j_1}+\sum_{w=1}^p\dagger(\cC,\alpha,i_w)+\sum_{w=1}^{r-p}\dagger(\fH,\alpha,j_w),\\
		\ast_2 &= & 1+|\alpha|c_v+h_{v,v+j_1}+ \sum_{w=1}^{t-1}\ddagger(\cC,\alpha,i_w)+\sum_{w=t}^{p}\dagger(\cC,\alpha,i_w) + \sum_{w=1}^{q-p} \dagger(\fH,\alpha,j_w),\\
		\ast_3 &= & 1+|\alpha|c_v+h_{v,v+j_1}+ \sum_{w=1}^{p}\ddagger(\cC,\alpha,i_w)+\sum_{w=1}^{t-1}\ddagger(\fH,\alpha,j_w)+\sum_{w=t}^{q-p} \dagger(\fH,\alpha,j_w).
\end{eqnarray*}
The proof is again by induction, which we omit. Then for $r>k+N$, the first exact term is zero, as $\cH^{v,k}_{i_1,\ldots,i_p|j_1,\ldots,j_{r-p}}$ is necessarily empty by \eqref{m00} of Definition \ref{morphism}. The remaining terms are exactly $\la (\phi^H\circ d^C-d^D\circ \phi^H)\alpha,\gamma \ra_{v+k}$ by a direct check, hence the theorem holds.
\end{proof}

Similar to Corollary \ref{cor:dim}, we have the following.
\begin{corollary}\label{cor:mordim}
	Assume that oriented flow categories $\cC,\cD$ have the property that $\dim C_i,\dim D_i \le k$ for all $i$. If $\fH:\cC \Rightarrow \cD$ is an oriented flow morphism, then $\phi^H:\BC(\cC,\Theta_1)\to \BC(\cD,\Theta_2)$ only depends on those $\cM^C_{i,j},\cH_{i,j},\cM^D_{i,j}$ of dimension $\le 2k$.
\end{corollary}

\subsubsection{The identity flow morphism}
Next we show that for every oriented flow category $\cC$, there is an oriented flow morphism $\fI:\cC \Rightarrow \cC$, which is referred to as the identity flow morphism. Roughly speaking, when the flow category has a background Morse-Bott function,  the identity flow morphism comes from the compactified moduli space of parameterized gradient flow lines, i.e.\ flow lines without modulo by the $\R$ translation action. Using the identity flow
morphism, we show the Morse-Bott cohomology is independent of the defining data.
\begin{definition/lemma}\label{identitymor}
		For an oriented flow category $\cC$, there is a canonical oriented flow morphism $\fI:\cC\Rightarrow \cC$. Given by $\cI_{i,j}=\cM_{i,j}\times[0,j-i]$ with the product orientation, for $i\le j$, and $\cI_{i,j}=\emptyset$ for $i>j$. 	The source and target maps $s,t:\cI_{i,j}\to C_i,C_j$ are defined as
		$$s=s^C\circ \pi_1,\quad t=t^C\circ \pi_1,$$
		where $\pi_1$ is the projection to the $\cM$ component. The compositions $m_L,m_R$ are defined as follows,
		$$
		\begin{aligned}
		m_L&: \cM_{i,k}\times_k \cI_{k,j}\to \cI_{i,j}, \quad (a,b,t)\mapsto (m(a,b),t+k-i),\\
		m_R&: \cI_{i,k}\times_k \cM_{k,j}\to \cI_{i,j}, \quad (a,t,b)\mapsto(m(a,b),t),
		\end{aligned}
		$$
		where $m$ is the composition in $\cC$.
\end{definition/lemma}

Before giving the proof, we will first use Lemma \ref{identitymor} to finish the proof of Theorem \ref{thm:main}. 
\begin{proof}[Proof of Theorem \ref{thm:main}]
	Let $\Theta_1,\Theta_2$ be defining data for the oriented flow category $\cC$. We have shown in Proposition \ref{dsquare} that $(\BC,d_{\BC,\Theta_1})$ and $(\BC,d_{\BC,\Theta_2})$ are cochain complexes. 	By \eqref{mordef}, the cochain morphism $\phi^I_{\Theta_1,\Theta_2}:(\BC,d_{\BC,\Theta_1}) \to (\BC,d_{\BC,\Theta_2})$ induced by the identity flow morphism $\fI$ can be written as $\id+N$,  where $N$ is strictly upper triangular, i.e.\ $N$ sends $H^*(C_s)$ to $\prod_{t={s+1}}^\infty H^*(C_t)$. Note that $\sum_{n=0}^\infty (-N)^n$ is well-defined on the cochain complex $\BC$, and $\sum_{n=0}^\infty (-N)^n$ is the inverse to $\id+N$. Thus $\phi^I_{\Theta_1,\Theta_2}$ is an isomorphism and hence induces an isomorphism on cohomology. 
\end{proof}
\begin{remark}
	When $\Theta_1 = \Theta_2$, we show in \S \ref{canonical} that $\phi^I_{\Theta_1,\Theta_2}$ is homotopic to the identity map. In particular, we will show that the construction up to homotopy is functorial w.r.t.\ the choice of defining data. 
\end{remark}

\begin{proof}[Proof of Lemma \ref{identitymor}]
	Condition \eqref{m00} of Definition \ref{morphism} follows from $\cI_{i,j}=\emptyset$, for $i>j$. Condition \eqref{m2} holds for $\fI$ due to the transversality property of the flow category $\cC$. Since $m_L(\cM_{i,k}\times_k \cI_{k,j})=\cM_{i,k}\times_k \cM_{k,j}\times [k-i,j-i]$ and $m_R(\cI_{i,k}\times_k \cM_{k.j})=\cM_{i,k}\times_k \cM_{k,j}\times [0,k-i]$, therefore condition \eqref{m1}, \eqref{m4} of flow morphism are satisfied by $\fI$. Therefore only orientation condition \eqref{m3} remains to check. 
	
	Unless stated otherwise, products of manifolds are always equipped with the product orientation. For $i<j$, we have
	\begin{eqnarray}
	\partial [\cI_{i,j}] & = &\partial [\cM_{i,j}\times [0,j-i]] \nonumber \\
	&= &(-1)^{m_{i,j}+1}[\cM_{i,j}\times \{0\}]+(-1)^{m_{i,j}}[\cM_{i,j}\times \{j-i\}]\ \nonumber \\
	&&+ \sum_{i<k<j}(-1)^{m_{i,k}}[\cM_{i,k}\times_k \cM_{k,j}\times [0,j-i]] \nonumber \\
	&= & (-1)^{m_{i,j}+1}[\cM_{i,j}\times \{0\}]+(-1)^{m_{i,j}}[\cM_{i,j}\times \{j-i\}] \label{idends}\\
	&& + \sum_{i<k<j}(-1)^{m_{i,k}}[\cM_{i,k}\times_k \cM_{k,j}\times [0,k-i]] \label{idbound1}\\
	&&+ \sum_{i<k<j}(-1)^{m_{i,k}}[\cM_{i,k}\times_k \cM_{k,j}\times [k-i,j-i]]. \label{idbound2}
	\end{eqnarray}
	Since the flow category $\cC$ is oriented, for $i<k<j$  we have
	\begin{equation}\label{idmor2}
	(t^C\times s^C)^*[N_k][\cM_{i,k}\times_k\cM_{k,j}]=(-1)^{c_km_{i,k}}[\cM_{i,k}][\cM_{k,j}].
	\end{equation}
	Let $\pi$ be the projection $\cI_{i,j}\rightarrow \cM_{i,j}$ for $i<j$, then we have 
	\begin{eqnarray*}(t\times s^C)^* N_k &= &\pi^*(t^C\times s^C)^*N_k|_{\cM_{i,k}\times_k \cM_{k,j}\times [0,k-i]};\\
	(t^C\times s)^* N_k & = & \pi^*(t^C\times s^C)^*N_k|_{\cM_{i,k}\times_k \cM_{k,j}\times [k-i,j-i]}.
	\end{eqnarray*}
	Therefore \eqref{idmor2} implies the following
	\begin{eqnarray}
	(t\times s^C)^*[N_k][\cM_{i,k}\times_k\cM_{k,j}\times[0,k-i]] &=& (-1)^{c_{i,k}m_{i,k}+m_{k,j}}[\cM_{i,k}\times[0,k-i]][\cM_{k,j}] \nonumber \\
	 &=&(-1)^{c_{i,k}m_{i,k}+m_{k,j}}[\cI_{i,k}][\cM_{j,k}]\label{idfiber1};\\
	(t^C\times s)^*[N_k][\cM_{i,k}\times_k\cM_{k,j}\times[k-i,j-i]] &=&(-1)^{c_km_{i,j}}[\cM_{i,k}][\cM_{k,j}\times [k-i,j-i]] \nonumber \\
	&=&(-1)^{c_km_{i,j}}[\cM_{i,k}][\cI_{k,j}]\label{idfiber2}.
	\end{eqnarray}
	If we orient $\cI_{i,k}\times_k \cM_{k,j}$ by $(-1)^{m_{k,j}+c_k}[\cM_{i,k}\times_k\cM_{k,j}][[0,k-i]]$ and  $[\cM_{i,k}\times_k \cI_{k,j}]$ by $[\cM_{i,k}\times_k\cM_{k,j}][[k-i,j-i]]$, then \eqref{idfiber1} implies that
	\begin{eqnarray}\label{idmor4}
	\eqref{idbound1}=(-1)^{m_{i,k}}[\cM_{i,k}\times_k \cM_{k,j}\times [0,k-i]]& = &(-1)^{m_{i,j}+1}[\cI_{i,k}\times_k \cM_{k,j}] \label{idrel1}\\
	(t\times s^C)^*[N_k][\cM_{i,k}\times_j \cI_{k,j}] & = & (-1)^{c_k(m_{i,k}+1)}[\cI_{i,k}][\cM_{k,j}].\label{idrel2}
	\end{eqnarray}
	And \eqref{idfiber2} implies that
	\begin{eqnarray}
	\eqref{idbound2}=(-1)^{m_{i,k}}[\cM_{i,k}\times_k \cM_{k,j}\times [k-i,j-i]]& = &(-1)^{m_{i,k}}[\cM_{i,k}\times_k \cI_{k,j}]\label{idrel3}\\
	(t^C\times s)^*[N_k][\cM_{i,k}\times_k \cI_{k,j}] & = & (-1)^{c_km_{i,k}}[\cM_{i,k}][\cI_{k,j}].\label{idrel4}
	\end{eqnarray}
	We still have to consider the first two copies of $\cM_{i,j}$ in \eqref{idends}. Since $m_L:  \cI_{i,i}\times_i \cM_{i,j}\rightarrow \cM_{i,j}$ and $m_R:\cM_{i,j}\times_j\cI_{j,j}\rightarrow \cM_{i,j}$ are diffeomorphisms. Therefore we can orient  $\cI_{i,i}\times_i \cM_{i,j}=C_i\times_i \cM_{i,j}$ and $\cM_{i,j}\times_j \cI_{j,j}=\cM_{i,j}\times_jC_j$ by $m_L^{-1}([\cM_{i,j}])$ and $m_R^{-1}([\cM_{i,j}])$. Then by Lemma \ref{linearori} below and the discussion after, we have
	\begin{eqnarray}
	(t\times s^C)^*[N_i][C_i\times \cM_{i,j}] &= &(-1)^{c_i^2}[C_i][\cM_{i,j}]; \label{a}\\
	(t^C\times s)^*[N_j][\cM_{i,j}\times_j C_j] &= &(-1)^{c_jm_{i,j}}[\cM_{i,j}][C_j] \label{b}.
	\end{eqnarray}
	Therefore we have 
	\begin{equation}\label{idmor5}
	\begin{array}{>{\displaystyle}rc>{\displaystyle}l}
	(-1)^{m_{i,j}+1}[\cM_{i,j}\times \{0\}] & = & (-1)^{m_{i,j}+1} m_R([\cI_{i,i}\times_i\cM_{i,j}]);\\
	
	[(t\times s^C)^*N_j][\cI_{i,i}\times_i\cM_{i,j}] & = &(-1)^{c_i^2}[\cI_{i,i}][\cM_{i,j}]; \\
	(-1)^{m_{i,j}}[\cM_{i,j}\times \{j-i\}] & = & (-1)^{m_{i,j}} m_L([\cM_{i,j}\times_j \cI_{j,j}]);\\
	
	[(t^C\times s)^*N_i][\cM_{i,j}\times_j\cI_{j,j}]&= &(-1)^{c_jm_{i,j}}[\cM_{i,j}][\cI_{j,j}].
	 \end{array}	 
	 \end{equation}	 
To sum up, \eqref{idrel1}, \eqref{idrel2}, \eqref{idrel3}, \eqref{idrel4} and \eqref{idmor5} prove the orientation condition \eqref{m3} of Definition \ref{morphism}. 
\end{proof}	

To state Lemma \ref{linearori}, we need to set up some notation. Let $E,F$ be two oriented finite dimensional vector spaces and $l:E\to F$ be a linear map. $\Delta_F$ denotes the diagonal subspace of $F\times F$. Suppose the ordered basis $(f_1,\ldots, f_n)$ represents the orientation $[F]$ of $F$ and the ordered basis $(e_1,\ldots, e_m)$ represents the orientation of $E$.  Then $((f_1,f_1),\ldots, (f_n,f_n))$ determines an orientation $[\Delta_F]$ of $\Delta_F$. Like \eqref{oricri}, we orient the quotient bundle (i.e. the normal bundle) $(F\times F)/\Delta_F$, such that $[\Delta_F][(F\times F)/\Delta_F]=[F][F]$. The fiber product $E\times_l F$ is the graph of $l$ in $E\times F$, then $\left((e_1,l(e_1)),\ldots, (e_m,l(e_m)\right)$ determines an orientation $[E\times_l F]$ on $E\times_l F=\graph l$. The projection $\pi: E\times_l F\to E$ is an isomorphism and the orientation we put on $E\times_l F$ has the property that $\pi([E\times_l F])=[E]$. Since $(l,\id): (E\times F)/(E\times_l F)\to (F\times F)/\Delta_F$ is an isomorphism, thus we can orient $(E\times F)/(E\times_lF)$ by $(l,\id)([(E\times F)/(E\times_l F)])=[(F\times  F)/\Delta_F]$.  What we describe here is the tangent picture of $\cM_{i,j}\times_j C_j$: let $(m,c)\in \cM_{i,j}\times_j C_j$, then the correspondences are $E=T_m \cM_{i,j}$, $F=T_c C_j$, $l=Ds|_m$, and the orientations match up.
\begin{lemma}\label{linearori} Following the notation above, we have
$$[(E\times F)/(E\times_l F)][E\times_l F ]=(-1)^{\dim E \dim F} [E][F]$$
\end{lemma}
\begin{proof}
	The ordered basis $\left((0_F,f_1),\ldots, (0_F, f_n)\right)$ represents a basis for $(F\times F)/\Delta_{F}$ as well as the orientation $[(F\times F)/\Delta_F]$. Note that $\left((0_E,f_1),\ldots, (0_E,f_n)\right)$ represents a basis for $(E\times F)/(E\times_l F)$, and is mapped to $\left((0_F,f_1),\ldots, (0_F, f_n)\right)$ through the map $(l, \id)$, thus  $\left((0_E,f_1),\ldots, (0_E,f_n)\right)$ represents the orientation on $(E\times F)/E\times_l F$. Since $\left((e_1,l(e_1)),\ldots, (e_m,l(e_m)),(0_E,f_1),\ldots, (0_E,f_n)\right)$ represents the orientation $[E][F]$, so we have 
	$$[E\times_l F ][(E\times F)/(E\times_l F)]=[E][F] \text{ or } [(E\times F)/E\times_l F][(E\times_l F)]=(-1)^{\dim E\dim F} [E][F],$$
	which would yield \eqref{b}.
\end{proof}
Similarly, if we consider $F\times_l E$ and it is oriented by $\left((l(e_1),e_1),\ldots, (l(e_m),e_m)\right)$. If we orient $(F\times E)/(F\times_lE)$ by 
$(\id,l)([(F\times E)/(F\times_lE)])=[(F\times F)/\Delta_F]$, then we have 
$$[(F\times E)/(F\times_l E)][F\times_l E]=(-1)^{(\dim F)^2} [F][E],$$
which yields \eqref{a}.

\subsection{Compositions of flow morphisms}\label{flowcomp}
In this subsection,  we study the compositions of flow morphisms. Roughly speaking, the composition of flow morphisms is taking fiber products. Hence in the Morse-Bott case, not every flow morphism can be composed and we introduce the following concept.
\begin{definition}\label{CompMor1}
	Two flow morphisms $\fH:\cC \to  \cD$, $\fF:\cD \to \cE$ are \textbf{composable} iff the fiber products $\cM^C_{i_1,i_2}\times_{i_2}\ldots \times_{i_{p-1}}\cM^C_{i_{p-1},i_p}\times_{i_p}\cH_{i_p,j_1}\times_{j_1}\cM^D_{j_1,j_2}\times_{j_2}\ldots \times_{j_{q-1}}\cM^D_{j_{q-1},j_q}\times_{j_q} \cF_{j_q,k_1}\times_{k_1}\cM^E_{k_1,k_2}\times_{k_2} \ldots \times_{k_{r-1}}\cM^E_{k_{r-1},k_r}$ are cut out transversely.
\end{definition}
Heuristically, one can define the composition $\fF\circ \fH $ of two composable morphisms $\fF$ and $\fH$ to be $(\cF\circ \cH)_{i,k}=\cup_j \cH_{i,j}\times_j \cF_{j,k}$, where the orientation is determined by 
\begin{equation}\label{ori}(t^H\times s^F)^*[N_j][\cH_{i,j}\times_j \cF_{j,k}]=(-1)^{d_jh_{i,j}}[\cH_{i,j}][\cF_{j,k}].\end{equation}
By \eqref{m00} of Definition \ref{morphism}, $(\cF\circ \cH)_{i,k}$ is a compact manifold. However, this is no longer a flow morphism, since the boundary can also come from  fiber products in the middle in addition to fiber products at the two ends\footnote{Although, in this case, the breaking from fiber products in the middle should pair up and ``cancel" with each other, this is morally why we have Theorem \ref{CompMor2}.}, violating \eqref{m4} of Definition \ref{morphism}. Hence we introduce the following definition.

\begin{definition}
	An \textbf{oriented flow premorphism} $\fF:\cC\Rightarrow\cD$ is a family of compact oriented manifolds $\cF_{i,j}$ with smooth  maps $s:\cF_{i,j}\to C_i, t:\cF_{i,j}\to D_j$. Moreover, there exists $N$, such for $i-j>N$, $\cF_{i,j}=\emptyset$ and the fiber products $\cM^C_{i_0,i_1}\times_{i_1}\ldots \times_{i_{k}} \cF_{i_k,j_0}\times_{j_0} \ldots \times_{j_{l-1}}\cM^D_{j_{l-1},j_{l}}$ are cut out transversely for all $i_0<\ldots<i_k, j_0<\ldots<j_l$. 
\end{definition}
Given a flow premorphism $\fF$, one can still define $\phi^F$ by \eqref{mordef}, which may not be a cochain morphism. Let $\fH$ and $\fF$ be two composable flow morphisms, then $\fF\circ \fH$ is a flow premorphism by definition. We need to understand the relation between $\phi^{F\circ H}$ and $\phi^F\circ \phi^H$. The main result of this subsection is that they differ by a homotopy. Before stating the theorem, we first introducing some notation.
\begin{enumerate}
	\item $\cE:=\{E_i,\cM^E_{i,j} \}$ is an oriented flow category, $e_i:=\dim E_i$, $m^E_{i,j}:=\dim \cM^E_{i,j}$ and $f_{i,j}:=\dim \cF_{i,j}$, which are again assumed to be well-defined for simplicity. 
	\item For $k\in \Z$, $0<i_1<\ldots<i_p$, $j_1<\ldots<j_q$  and $k_1<\ldots<k_r<k$, we define $\cF\times \cH^{v,k}  _{i_1,\ldots,i_p|j_1\ldots,j_q|k_1,\ldots,k_r}$ to be:
	$$	\cM^{C}_{v,v+i_1}\times \ldots \times \cH_{v+i_p,v+j_1}\times \cM^D_{v+j_1,v+j_2}\times \ldots \times \cF_{v+j_q,v+k_1}\times \ldots \times \cM^E_{v+k_r,v+k}.$$
	Note that we must have $q\ge 1$ for this to be defined. 
	\item $(\cF\times \cH)^{v,k}  _{i_1,\ldots,i_p|j_1\ldots,j_q|k_1,\ldots,k_r}[\alpha, f^C_{v+i_1},\ldots, f^C_{v+i_p}, f^D_{v+j_1},\ldots, f^D_{v+j_q}, f^E_{v+k_1},\ldots, f^E_{v+k_r},\gamma]$ is defined similarly to \eqref{pair1}.
\end{enumerate}
To define the homotopy operator $P_{\Theta_1,\Theta_2,\Theta_3}$ or $P$ for simplicity, for $k\in \Z$, $\alpha \in h(\cC,v)$ and $ \gamma \in h(\cE,v+k)$,  we define $P$ by 
\begin{equation}\label{pdef}
\begin{aligned}\left\langle  P[\alpha],[\gamma]\right\rangle_{v+k} =\sum_{\substack{p,r\ge 0,q\ge 1\\ 0=i_0<i_1<\ldots<i_p,\\ j_1<\ldots<j_q \\ k_1\le \ldots<k_{r+1}=k}} (-1)^\bigstar F\times H^{v,k}_{i_1, \ldots,i_p|j_1,\ldots,j_q|k_1,\ldots,k_r}[\alpha,& f^C_{v+i_1},\ldots,f^C_{v+i_p},f^D_{v+j_1},\ldots\\
&\ldots ,f^D_{v+j_{q}},f^E_{v+k_1}, \ldots, f^E_{v+k_r},\gamma], \end{aligned}\end{equation}
where $$\begin{aligned}\bigstar:=1+|\alpha|(c_v+1)+\dim (\cF\circ \cH)_{v,v+k_1}+&\sum_{w=1}^p \ddagger(\cC,\alpha,i_w)+h_{v,v+j_1}+ \sum_{w=1}^q\ddagger(\fH,\alpha,j_w)\\& +\sum_{w=1}^r\dagger(\fF\circ \fH, \alpha, k_w)\end{aligned}.$$

\begin{theorem}\label{CompMor2} Let $\fH,\fF$ be composable oriented flow morphisms from $\cC$ to $\cD$ and from $\cD$ to $\cE$ respectively. If we fix defining data $\Theta_1,\Theta_2$ and $\Theta_3$ for $\cC,\cD$ and $\cE$, then there exists an operator $P_{\Theta_1,\Theta_2,\Theta_3}:\BC(\cC)\to \BC(\cE)$ defined by \eqref{pdef}, such that 
	$$\phi^{F\circ H}_{\Theta_1,\Theta_3}-\phi^F_{\Theta_2,\Theta_3}\circ \phi^H_{\Theta_1,\Theta_2}+P_{\Theta_1,\Theta_2,\Theta_3}\circ d^C_{\BC,\Theta_1}+d^E_{\BC,\Theta_3}\circ P_{\Theta_1,\Theta_2,\Theta_3}=0.$$ 
\end{theorem}
\begin{proof}
The theorem follows from the claim below.

For $\alpha \in h(\cC,v),\gamma\in h(\cE,v+k)$ with $k\in \Z$ and any $l\ge 1$, we have
\begin{eqnarray*}
0 & = &\sum_{r\le l} \sum_{p+q=r-1}(-1)^{\bigstar_1}  (\cF\circ \cH)^{v,k}_{i_1,\ldots,i_p|k_1,\ldots k_{q}}[\alpha,\ldots f^C_{v+i_*},\ldots,f^E_{v+k_{*}},\ldots,\gamma] \\
&  & + \sum_{p+r+w=l}(-1)^{\bigstar_2}(\cF\times \cH)^{v,k}_{i_1,\ldots,i_p|j_1,\ldots,j_r|k_1,\ldots,k_{w}}[\rd (\alpha,\ldots f^C_{v+i_*},\ldots, f^D_{v+j_*}, \ldots, f^E_{v+k_*},\ldots,\gamma)] \\
& & + \sum_{\substack{p+q+w\le l,\\ u\ge 1}}(-1)^{\bigstar_3} \tr^{v+i_u} (\cF\times \cH)^{v,k}_{i_1,\ldots,i_p|j_1,\ldots,j_q|k_1,\ldots,k_{w}}[\alpha,\ldots, f^C_{v+i_*},\ldots,\theta{\theta^C_{v+i_u}}^*, \ldots,f^D_{v+j_*},\ldots,f^E_{v+k_*},\ldots\gamma]\\
& & + \sum_{\substack{p+q+w\le l,\\ u\ge 1}}(-1)^{\bigstar_4}\tr^{v+j_u} \cF\times \cH^{v,k}_{i_1,\ldots,i_p|j_1,\ldots,j_q|k_1,\ldots,k_{w}}[\alpha, \ldots f^C_{v+i_*},\ldots, f^D_{v+j_*},\ldots,  \theta{\theta^D_{v+j_u}}^*,\ldots f^E_{v+k_*},\ldots,\gamma]\\
& & + \sum_{\substack{p+q+w\le l,\\ u\ge 1}}(-1)^{\bigstar_5} \tr^{v+k_u} \cF\times \cH^{v,k}_{i_1,\ldots,i_p|j_1,\ldots,j_q|k_1,\ldots,k_{w}}[\alpha,\ldots,  f^C_{v+i_*}, \ldots, f^D_{v+j_*}\ldots f^E_{v+k_*}	\ldots, \theta{\theta^E_{v+k_u}}^*,\ldots, \gamma]
\end{eqnarray*}
where we omit the obvious constraints $0<i_1<\ldots<i_p,j_1<\ldots<j_q,k_1<\ldots<k_w<k$. The indices for signs are
	\begin{eqnarray*}
		\bigstar_1 & = & 1+|\alpha|c_v+\dim (\cF\circ \cH)_{v,v+k_1}+\sum_{s=1}^p\ddagger(\cC,\alpha, i_s) +\sum_{s=1}^{q}\ddagger(\fF\circ \fH, \alpha, k_s),\\
		\bigstar_2 &= &|\alpha|c_v+\dim (\cF\circ \cH)_{v,v+k_1}+\sum_{s=1}^p\dagger(\cC,\alpha,i_s)+h_{v,v+j_1}+\sum_{s=1}^r\dagger(\fH,\alpha,j_s)\\ && \qquad +\sum_{s=1}^{w}\ddagger(\fF\circ \fH,\alpha,k_s),\\
		\bigstar_3 & = &|\alpha|(c_v+1)+\dim (\cF\circ \cH)_{v,v+k_1}+\sum_{s=1}^{u-1}\ddagger(\cC,\alpha, i_s)+\sum_{s=u}^p\dagger(\cC,\alpha,i_s)+h_{v,v+j_1}\\ && \qquad +\sum_{s=1}^r \dagger(\fH,\alpha,j_s)+\sum_{s=1}^{w}\ddagger(\fF\circ \fH, \alpha, k_s),\\
		\bigstar_4 & = & |\alpha|(c_v+1)+\dim (\cF\circ \cH)_{v,v+k_1}+\sum_{s=1}^p\ddagger(\cC,\alpha,i_s)+h_{v,v+j_1}+\sum_{s=1}^{u-1} \ddagger(\fH,\alpha,j_s)\\ &&\qquad+\sum_{s=u}^r \dagger(\fH,\alpha,j_s)+\sum_{s=1}^{w}\ddagger(\fF\circ \fH,\alpha,k_s),\\
		\bigstar_5 & = &|\alpha|(c_v+1)+\dim (\cF\circ \cH)_{v,v+k_1}+\sum_{s=1}^p \ddagger(\cC,\alpha,i_s)+h_{v,v+j_1} +\sum_{s=1}^r \ddagger(\fH,\alpha,j_s)\\
		&&\qquad+\sum_{s=1}^{u-1} \dagger(\fF\circ \fH, \alpha,k_s)+\sum_{s=u}^{w}\ddagger(\fF\circ \fH, \alpha, k_s).
	\end{eqnarray*}
The proof is again by induction on $l$, which we omit. Then for $l\gg 0$, the exact term is zero. It is direct to check that the first term is $-\la \phi^{F\circ H}\alpha,\gamma\ra_{v+k}$, the third term is $-\la P\circ d^C \alpha,\gamma \ra_{v+k}$, the fourth term is $\la \phi^F\circ \phi^H \alpha,\gamma\ra_{v+k}$ and the last term is $-\la d^E\circ P\alpha,\gamma \ra_{v+k}$, hence the theorem follows.
\end{proof}
As a corollary,  $\phi^{F\circ H}_{\Theta_1,\Theta_3}$ is a cochain map between $(\BC(\cC), d^C_{\BC,\Theta_1})$ and $(\BC(\cE), d^E_{\BC,\Theta_3})$ and that is homotopic to $\phi^F_{\Theta_2,\Theta_3}\circ\phi^H_{\Theta_1,\Theta_2}$.

\subsection{Flow homotopies induce cochain homotopies}\label{flowhomotopy}
In this subsection, we introduce the flow homotopies between flow premorphisms. Such structures can be viewed as the analogue of the geometric data needed to define homotopies between continuation maps in Floer theories \cite[Chapter 11]{audin2014morse}.

\begin{definition}\label{homot}
	An \textbf{oriented flow homotopy} $\sY$ between two flow \emph{premorphisms} $\fF=\{\cF_{i,j}\}$ and $\fH=\{\cH_{i,j}\}$ from $\cC$ to $\cD$ is a family of oriented compact manifolds $\{\cY_{i,j}\}$ with smooth source and target maps $s:\cY_{i,j}\to C_i$ and $t:\cY_{i,j}\to D_j$, such that the following holds.
	\begin{enumerate}
		\item\label{h1} There are smooth maps $\iota_F,\iota_H: \cF_{i,j},\cH_{i,j}\to \cY_{i,j}$, such that $s\circ\iota_{F}=s^{F}$, $s\circ\iota_{H}=s^{H}$, $t\circ \iota_{F}=t^{F}$ and $t\circ \iota_{H}=t^{H}$ where $s^{F},s^{H},t^{F},t^{H}$ are the source and target maps for $\fF$ and $\fH$ respectively. 
	    \item\label{h2}  $\exists N \in \N$,  such that when $i-j>N$, we have $\cY_{i,j}=\emptyset$.
	    \item\label{h3}  For all $i_0<\ldots<i_k, j_0<\ldots<j_l$, the fiber products $\cM^C_{i_0,i_1}\times_{i_1}\ldots \times_{i_k} \cY_{i_k,j_0}\times_{j_0}\ldots\times_{j_{l-1}}\cM^D_{j_{l-1},j_l}$ is cut out transversely.
	    \item\label{h4} There are smooth maps $m_L:\cM^C_{i,j}\times_j \cY_{j,k}\to \cY_{i,k}$ and $m_R: \cY_{i,j}\times_j \cM^D_{j,k}\to \cY_{i,k}$, such that the following holds,
	    $$
	    \begin{array}{rclrcl}
	     s\circ m_L(a,b) & = & s^C(a), &
	     t\circ m_L(a,b) & =& t(b),\\ 
	     s\circ m_R(a,b) & = & s(a), &
	     t\circ m_R(a,b) & = & t^D(b). 
	    \end{array}
	    $$
        Here $s^C$ is the source map for $\cC$ and $t^D$ is the target map for $\cD$.
	    \item\label{h5} The map $\iota_F\cup \iota_H\cup m_L\cup m_R: \cF_{i,k}\cup\cH_{i,k}\cup \left(\cup_j \cM^\cC_{i,j}\times_j \cY_{j,k}\right) \cup \left(\cup_j \cY_{i,j}\times_j \cM^\cD_{j,k}\right)\to \partial \cY_{i,k}$ is a diffeomorphism up to measure zero sets. 
		\item\label{h6} The orientation $[\cY_{i,j}]$ has the following properties.
		$$\begin{aligned}\partial [\cY_{i,j}]=\iota_F([\cF_{i,j}])-\iota_H([\cH_{i,j}])&+\sum_{p>0} (-1)^{c_{i+p}+1} m_L([\cM^C_{i,i+p} \times_{i+p} \cY_{i+p,j}]) \\ & +\sum_{p>0} (-1)^{y_{i,j}}m_R([\cY_{i,j-p}\times_{j-p} \cM^D_{j-p,j}])\end{aligned},$$
		$$(t^C\times s)^*[N_j][\cM^C_{i,j}\times_j \cY_{j,k}]=(-1)^{c_j m^C_{i,j}}[\cM^C_{i,j}][\cY_{j,k}],$$ 
		$$(t\times s^{D})^*[N_j][\cH_{i,j}\times_j \cM^D_{j,k}]=(-1)^{d_jy_{i,j}}[\cY_{i,j}][\cM^D_{j,k}],$$
		where $y_{i,j}:=\dim \cY_{i,j}$.
	\end{enumerate}	
\end{definition}

The main result of this subsection is that flow homotopies induce homotopies between the maps induced by the boundary flow premorphisms (which are not necessarily cochain morphisms). Before stating the theorem, we introduce the following notation.
\begin{enumerate}
	\item 
	 For $k\in \Z$ and $0<i_1<\ldots<i_p$ and $j_1<\ldots<j_q<k$,
	 $$\cY^{v,k}_{i_1,\ldots,i_p|j_1\ldots,j_q}:=\cM^C_{v,v+i_1}\times \ldots \times\cM^C_{v+i_{p-1},s+i_p}\times \cY_{v+i_p,v+j_1}\times \cM^D_{v+j_1,v+j_2}\times \ldots \times \cM^D_{v+j_q,v+k}.$$
	\item $\cY^{*,*}_{\ldots }[\alpha,  f^C_*, \ldots, f^D_* \ldots, \gamma]$ is defined similarly to \eqref{pair1}.
	\item For $\alpha \in h(\cC,v)$, we define $\dagger(\sY,\alpha, k):=(|\alpha|+y_{v,v+k})(d_{v+k}+1)$ and $\ddagger(\sY,\alpha, k):=(|\alpha|+y_{v,v+k}+1)(d_{v+k}+1)$.
\end{enumerate}
To state the formula for the homotopy operator $\Lambda^Y$, we suppress the subscripts $\Theta_1,\Theta_2$ for simplicity. Let $\alpha\in h(\cC,v)$ and $\gamma \in h(\cD,v+k)$, then $\left\langle\Lambda^Y [\alpha],[\gamma]\right\rangle_{v+k}$ is defined to be:
\begin{equation}\label{hdef}\sum_{\substack{p,q\ge 0\\ 0=i_0<\ldots<i_p\\ j_1<\ldots<j_{q+1}=k}}(-1)^\clubsuit \cY^{v,k}_{i_1,\ldots,i_p|j_1,\ldots,j_q} [\alpha,f^C_{v+i_1},\ldots, f^D_{v+j_q},\gamma],\end{equation}
where 
$$\clubsuit:=|\alpha|(c_v+1)+y_{v+i_p,v+j_1}+\sum_{w=1}^p\ddagger(\cC,\alpha,i_w)+\sum_{w=1}^q \ddagger(\sY,\alpha,j_w).$$
\begin{theorem}\label{thm:homotopy}
	Given an oriented flow homotopy $\sY$ between two oriented flow premorphisms $\fF,\fH:\cC\Rightarrow \cD$, after fixing defining data $\Theta_1,\Theta_2$ for $\cC,\cD$ respectively, there exists an operator $\Lambda^Y_{\Theta_1,\Theta_2}: \BC(\cC)\to \BC(\cD)$ defined by \eqref{hdef}, such that
	$$d^D_{\BC,\Theta_2}\circ\Lambda^Y_{\Theta_1,\Theta_2}+\Lambda_{\Theta_1,\Theta_2}^Y\circ d^C_{\BC,\Theta_1}+\phi^F_{\Theta_1,\Theta_2}-\phi^H_{\Theta_1,\Theta_2}=0.$$
\end{theorem}
\begin{proof}
Similar to the proof of Proposition \ref{dsquare} and Theorem \ref{Mor}, this theorem follows from the following claim, whose proof is again by induction and will be omitted.

For $\alpha\in h(\cC,v),\gamma \in h(\cD,v+k)$ with $k\in \Z$ and any $r\ge 0$, we have
\begin{eqnarray*}
0 & = &  \sum_{0\le p\le r} (-1)^{\clubsuit_1} \cY^{v,k}_{i_1,\ldots,i_p|j_1,\ldots, j_{r-p}} [\rd(\alpha, f^C_{v+i_1},\ldots ,f^D_{v+j_{r-p}},\gamma)]\\
& &	+ \sum_{\substack{0\le p\le q\le r \\1\le u\le p}}(-1)^{\clubsuit_2} \tr^{v+i_u}\cY^{v,k}_{i_1,\ldots,i_p|j_1,\ldots,j_{q-p}}[\alpha, f^C_{v+i_1},\ldots, \theta{\theta^C_{v+i_u}}^*,\ldots, f^D_{v+i_{q-p}},\gamma]\\
& & +  \sum_{\substack{0\le p\le q\le r\\1\le u\le q-p}} (-1)^{\clubsuit_3}\tr^{v+j_u} \cY^{v,k}_{i_1,\ldots,i_p|j_1,\ldots,j_{q-p}} [\alpha, f^C_{v+i_1},\ldots, \theta{\theta^D_{v+j_u}}^*,\ldots, f^D_{v+i_{q-p}},\gamma]\\
& & +  \sum_{\substack{0\le p\le q< r}}(-1)^{\clubsuit_4}\left(\cF^{v,k}|_{i_1,\ldots, i_p|j_1,\ldots, j_{q-p}}-\cH^{v,k}|_{i_1,\ldots, i_p|j_1,\ldots, j_{q-p}}\right)[\alpha, f^C_{v+i_1},\ldots, f^D_{v+j_{q-p}},\gamma].
\end{eqnarray*}
Here
	\begin{eqnarray*}
	\clubsuit_1 & = & |\alpha|c_v+y_{v+i_p,v+j_1}+\sum_{w=1}^p\dagger(\cC,\alpha,i_w)+\sum_{w=1}^{r-p} \dagger(\sY,\alpha,j_w),\\
	\clubsuit_2 & = & |\alpha|(c_v+1)+y_{v+i_p,v+j_1}+\sum_{w=1}^{u-1}\ddagger(\cC,\alpha,i_w)+\sum_{w=u}^p\dagger(\cC,\alpha,i_w)+\sum_{w=1}^{q-p} \dagger(\sY,\alpha,j_w),\\
	\clubsuit_3 & = & |\alpha|(c_v+1)+y_{v+i_p,v+j_1}+ \sum_{w=1}^p\ddagger(\cC,\alpha,i_w)+\sum_{w=1}^{u-1} \ddagger(\sY,\alpha,j_w) +\sum_{w=u}^{q-p}\dagger(\sY,\alpha,j_w),\\
	\clubsuit_4 &= & |\alpha|c_v+y_{v,v+j_1}+\sum_{w=1}^p\ddagger(\cC,\alpha,i_w)+\sum_{w=1}^{q-p}\dagger(\sY,\alpha,j_w).
	\end{eqnarray*}
\end{proof}

\begin{remark}
	Theorem \ref{thm:homotopy} does not require that $\Phi^F_{\Theta_1,\Theta_2}$ or $\Phi^H_{\Theta_1,\Theta_2}$ is a cochain morphism. When they are cochain morphisms (in fact, that one of them is a cochain morphism would imply the other is also by Theorem \ref{thm:homotopy}), Theorem \ref{thm:homotopy} implies that they are homotopic to each other. 
\end{remark}

\subsection{The minimal Morse-Bott cochain complex is canonical}\label{canonical}
Unlike the Morse case, where the defining data is unique, there are a lot of freedom in choosing the defining-data for the minimal Morse-Bott cochain complex, i.e.\ choices of quasi-isomorphic embeddings, choices of Thom classes and choices of $f_i^n$. The cochain morphism $\phi^H_{\Theta,\Theta'}$ induced from the flow morphism $\fH$ by \eqref{mordef} also depends on $\Theta,\Theta'$. Although Theorem \ref{thm:main} asserts that the cohomology is independent of the defining data. It is important to have the isomorphism canonical in a functorial way w.r.t.\ the choice of defining data. In this section,  we prove the construction of the minimal Morse-Bott cochain complex $(\BC,d_{\BC,\Theta})$ is natural w.r.t.\ the defining data $\Theta$.  Moreover, we  will show in this subsection that the cochain morphism $\phi^H_{\Theta,\Theta'}$ from \eqref{mordef} is also canonical in a suitable sense. To explain the claim above in more detail, we introduce the following category of defining date of an oriented flow category. 

\begin{definition}
	Given an oriented flow category $\cC$,  $\cD ata(\cC)$ is defined to be the category whose objects are defining data of $\cC$ and there is exactly one morphism between any two objects. 
\end{definition}

For every object $\Theta$ in $\cD ata(\cC)$, we can associate it with a cochain complex $(\BC,d_{\BC,\Theta})$. The following theorem says that such an assignment can be completed to a functor from $\cD ata(\cC)\to\cK(\cC h)$, where $\cK(\cC h)$ is the homotopy category of cochain complexes.
\begin{theorem}\label{canonicalthm1}
	$$
	\begin{aligned}
	\Theta & \mapsto (\BC,d_{\BC,\Theta}),\\ 
	(\Theta_1\to \Theta_2) & \mapsto \left(\phi^I_{\Theta_1,\Theta_2}:(\BC,d_{\BC,\Theta_1})\to (\BC,d_{\BC,\Theta_2})\right)
	\end{aligned}$$ 
	defines a functor $\mathcal{BC}(\cC): \cD ata(\cC)\to \cK (\cC h)$, where $\fI$ is the identity flow morphism used to define $\phi^I_{\Theta_1,\Theta_2}$ by \eqref{mordef}. 
\end{theorem}
\begin{proof}
\textit{Step 1, $\phi^I_{\Theta,\Theta}$ is homotopic to identity.} 
It is not hard to check that $\phi^{I\circ I}_{\Theta,\Theta}$ can be written as $\id+M$ with $M$  strictly upper triangular. Note that for $i<j$, $I_{i,j}=\cM_{i,j}\times [0, j-i]$ and $(I\circ I)_{i,j}=\cup_{k, i\le k\le j} I_{i,k}\times_k I_{k,j}$ have an interval direction.  Since the pullback of differential forms by source and target maps can not cover that interval direction, we have
$$I^{v,k}_{\ldots,p|q,\ldots} [\ldots, f_{v+p},f_{v+q},\ldots]=(I\circ I)^{v,k}_{\ldots,p|q,\ldots} [\ldots, f_{v+p},f_{v+q},\ldots]=0, \text{ if }p\ne q,$$
$$I^{v,k}_{\ldots,p|}=(I\circ I)^{v,k}_{\ldots,p|}=0, \text{ if } p\ne k,$$
$$I^{v,k}_{|q,\ldots}=(I\circ I)^{v,k}_{|q,\ldots}=0, \text{ if } q\ne 0.$$
Therefore, for $k\in \mathbb{N}^+$, $\alpha \in h(\cC,v)$ and $\gamma \in h(\cC,v+k)$, we have
$$
\begin{array}{>{\displaystyle}rc>{\displaystyle}l}
\left\langle M[\alpha],[\gamma]\right\rangle_{v+k} & = & \sum_{\substack{1\le p\le q \le k\\ 0<i_1<\ldots<i_q<k}} (-1)^{\spadesuit_1} I\circ I^{v,k}_{i_1,\ldots,i_p|i_p,\ldots,i_q}[\alpha, f_{v+i_1},\ldots, f_{v+i_p},f_{v+i_p}, \ldots, f_{v+i_q},\gamma]\\
&  & + \sum_{\substack{1\le p \\ 0<i_1<\ldots<i_p=k}} (-1)^{\spadesuit_2} I\circ I^{v,k}_{i_1,\ldots,i_{p}|}[\alpha, f_{v+i_1},\ldots, f_{v+i_p}, f_{v+i_p},\gamma] \\
&  &+\sum_{\substack{1\le p \\ 0=i_1<\ldots<i_p<k}} (-1)^{\spadesuit_3} I\circ I^{v,k}_{|i_1,\ldots,i_{p}}[\alpha, f_{v+i_1}, f_{v+i_1}, \ldots, f_{v+i_p},\gamma],
\end{array}
$$
where $\spadesuit_1, \spadesuit_2,\spadesuit_3$ are determined according to \eqref{mordef}.

Similarly, we have a decomposition $\phi^I_{\Theta,\Theta}=\id+N$ with and $N$ strictly upper triangular. Note that $(I\circ I)_{v+i_p,v+i_p}=I_{v+i_p,v+i_p}=C_{v+i_p}$, hence we have 
$$\begin{aligned}& (I\circ I)^{v,k}  _{i_1,\ldots,i_p|i_p,\ldots,i_q}[\alpha, f_{v+i_1},\ldots, f_{v+i_p},f_{v+i_p}, \ldots, f_{v+i_q},\gamma]\\ = &I^{v,k}_{i_1,\ldots,i_p|i_p,\ldots,i_q}[\alpha, f_{v+i_1},\ldots, f_{v+i_p},f_{v+i_p}, \ldots, f_{v+i_q},\gamma].\end{aligned}$$
Similarly for the remaining two terms of $M,N$, thus we have $N=M$.  Then by Theorem \ref{CompMor2},
$$(\id+M)-(\id+M)^2=P\circ d_{\BC,\Theta}+d_{\BC,\Theta}\circ P.$$
Since $\id+M$ is a cochain isomorphism, we have:
$$\id-(\id+M)=(\id+M)^{-1}\circ P\circ d_{\BC,\Theta}+d_{\BC,\Theta}\circ (\id+M)^{-1}\circ P.$$
Thus $\id+M=\id+N=\phi^I_{\Theta,\Theta}$ is homotopic to identity.

\noindent\textit{Step 2, functoriality.} Given three defining-data $\Theta_1,\Theta_2,\Theta_3$, by the same argument as above we have up to homotopy that
$$\phi^I_{\Theta_1,\Theta_3}=\phi^{I\circ I}_{\Theta_1,\Theta_3}.$$
By Theorem \ref{CompMor2}, 
$$\phi^{I\circ I}_{\Theta_1,\Theta_3}-\phi^I_{\Theta_2,\Theta_3}\circ \phi^I_{\Theta_1,\Theta_2}+P\circ d_{\BC,\Theta_1}+d_{\BC,\Theta_3}\circ P=0.$$
Thus $\phi^I_{\Theta_1,\Theta_3}$ is homotopic to $\phi^I_{\Theta_2,\Theta_3}\circ \phi^I_{\Theta_1,\Theta_2}$.
\end{proof}

\begin{remark}
    A similar mechanism of proof appeared in \cite[Proposition 7.7.4]{pardon2016algebraic}, where the situation is Morse and the auxiliary data (which can be viewed as the analogue of the defining data) are choices in the construction of virtual fundamental cycles.
\end{remark}

To explain the functoriality for flow morphisms, we introduce the following category.
\begin{definition}
	Let $\cC,\cD$ be two oriented flow categories, $\cD ata(\cC\to \cD)$ is defined to be the category whose objects are defining data of $\cC$ and $\cD$. There is exactly one morphism from $\Theta_1$ to $\Theta_2$ if $\Theta_1,\Theta_2$ are defining data for the same flow category or $\Theta_1,\Theta_2$ are defining data for $\cC$ and $\cD$ respectively.
\end{definition}
Then $\cD ata(\cC)$ and $\cD ata(\cD)$ are full subcategories of $\cD ata(\cC\to \cD)$. If there is an oriented flow morphism $\fH:\cC\rightarrow \cD$, then for any defining data $\Theta$ and $\Theta'$ of $\cC$ and $\cD$ respectively, we can assign a cochain morphism  $\phi^H_{\Theta,\Theta'}:(\BC(\cC), d^C_{\BC,\Theta}) \to (\BC(\cD), d^D_{\BC,\Theta'})$. The next theorem states that such assignment along with $\cB\cC(\cC)$ and $\cB\cC(\cD)$ is a functor.

\begin{theorem}\label{canonicalthm2}
	For an oriented flow morphism $\fH$, there is a functor $$\Phi^H:  \mathcal{D}ata(\cC\to \cD)\to \mathcal{K}(\mathcal{C}h)$$ which extends functors $\mathcal{BC}(\cC)$ and $\mathcal{BC}(\cD)$ by sending the morphism $\Theta^C\to \Theta^D$ to $\phi^H_{\Theta^C,\Theta^D}$. Here $\Theta^C, \Theta^D$ are defining data for $\cC,\cD$ respectively.
\end{theorem}	
\begin{proof}
 We only need to prove the functoriality. We use $\Theta^C,\Theta^D$ to denote defining data for $\cC,\cD$ respectively.  By Theorem \ref{CompMor2}, we have $\phi^{H\circ I}_{\Theta^C_1,\Theta^D}$ is homotopic to  $\phi^H_{\Theta_2^C,\Theta^D}\circ \phi^I_{\Theta_1^C,\Theta_2^C}$
and $\phi^{H\circ I}_{\Theta^C_1,\Theta^D}$ is homotopic to $ \phi^H_{\Theta_1^C,\Theta^D}\circ \phi^I_{\Theta_1^C,\Theta_1^C}$. Since $\phi^I_{\Theta_1^C,\Theta_1^C}$ is homotopic to identity by Theorem \ref{canonicalthm1}, thus $\phi^H_{\Theta_2^C,\Theta^D}\circ \phi^I_{\Theta_1^C,\Theta_2^C}$ is homotopic to $\phi^H_{\Theta_1^C,\Theta^D}$. Similarly, we have $\phi^I_{\Theta_1^D,\Theta_2^D}\circ \phi^H_{\Theta^C,\Theta_1^D}$ is homotopic to $\phi^H_{\Theta^C,\Theta_2^D}$.
\end{proof}

\subsection{Flow subcategories and flow quotient categories}\label{sub:sub}
In this part, we introduce subcategories and quotient categories in the setting of flow categories, which on the cochain complex level correspond to subcomplexes and quotient complexes. 
\begin{definition}\label{def:sub}
	Let $\cC = \{C_i,\cM_{i,j}\}$ be an oriented flow category.  A subset $A$ of $\Z$ is called a $\cC$-subset if for all $i \in A$ we have $j \notin A$ implies $\cM_{i,j} = \emptyset$. 
\end{definition}
A basic example of $\cC$-subset is the set of integers bigger than a fixed number.
\begin{proposition}\label{prop:subcat}
	Let $\cC = \{C_i,\cM_{i,j}\}$ be an oriented flow category and $A$ be a $\cC$-subset.  Then we have the following two flow categories $\cC_A = \{C_i, \cM_{i,j}, i, j \in A\}$ and $\cC_{/A} = \{C_i,\cM_{i,j},i,j,\notin A\}$.	
\end{proposition}
\begin{proof}
    It is clear that both $\cC_A$ and $\cC_{/A}$ are subcategories. Then it is sufficient to prove that the boundary of morphism spaces comes from fiber products of the morphisms spaces for both $\cC_A$ and $\cC_{/A}$. Since the boundary $\partial \cM_{i,k}$ comes from $\cM_{i,j}\times_j\cM_{i,k}$, if both $i,k \in A$, then $j \in A$, otherwise one of $\cM_{i,j}$ and $\cM_{j,k}$ is empty. Similarly for $\cC_{/A}$.
\end{proof}
We will call $\cC_A$ a \textbf{flow subcategory} and $\cC_{/A}$ the associated \textbf{flow quotient category}. 
\begin{remark}
	A finer definition of subcategory is using a subset of components of $\Obj(\cC)$, such that a similar condition to  Definition \ref{def:sub} holds.
\end{remark}	

From Definition \ref{def:MMBCC}, when the  defining data of $\cC_A$ and $\cC_{/A}$ are the restriction of a defining data on $\cC$, we have the following tautological short exact sequence,
\begin{equation}\label{eqn:naturalexact}
0 \to \BC(\cC_A) \to \BC(\cC) \to \BC(\cC_{/A}) \to 0,
\end{equation}
by the obvious inclusion and projection. To make the structure more compatible with concepts introduced here and our future applications \cite{ring}, we lift the short exact sequence to the flow morphism level. We first introduce the following.

\begin{lemma}\label{lemma:commutative}
Assume $(V_0\oplus V_1,d)$ is a cochain complex with the property that $d(V_0)\subset V_0$, i.e.\ $d$ has a decomposition into $d_{00}+d_{10}+d_{11}$, where $d_{ab}:V_a\to V_b$. Suppose we have another cochain complex  $(V'_0\oplus V'_1,d')$ with the same property. Assume the following squares are commutative up to homotopies $H_1,H_2$ with the property that $\Ima H_1\subset V'_0$ and $V_0\subset \ker H_2$ and the middle morphism $\phi$ has the same decomposition $\phi_{00}+\phi_{10}+\phi_{11}$, i.e.\ $\phi(V_0)\subset V'_0$
$$\xymatrix{
0 \ar[r] & V_0 \ar[r]\ar[d]^{\psi} & V_0\oplus V_1 \ar[r]\ar[d]^{\phi} & V_1\ar[d]^{\eta} \ar[r] & 0\\
0 \ar[r] & V'_0 \ar[r] & V'_0\oplus V'_1 \ar[r] & V'_1 \ar[r] & 0
}$$
Then they induce a morphism between the long exact sequences of cohomology.
\end{lemma}
\begin{proof}
    We only need to prove the following square is commutative:
    $$
    \xymatrix{
    H(V_1) \ar[r]^{d_{10}}\ar[d]^{\eta} & H(V_0)\ar[d]^{\psi} \\
    H(V'_1) \ar[r]^{d'_{10}} & H(V'_0)}
    $$
    By $\Ima H_1\subset V'_0$ and $V_0\subset \ker H_2$, we have $\psi=\phi_{00}$ and $\eta=\phi_{11}$ on cohomology. Then the claim follows because the square below is commutative up to the homotopy $\phi_{10}$,\footnote{See Remark \ref{rmk:triangle} for the explanation of the sign, although it does not affect the map on cohomology.}
    $$
    \xymatrix{
    (V_1,d_{11}) \ar[r]^{d_{10}}\ar[d]^{\phi_{11}} & (V_0,-d_{00})\ar[d]^{\phi_{00}} \\
    (V'_1,d'_{11}) \ar[r]^{d'_{10}} & (V'_0,-d'_{00})}
    $$
\end{proof}

\begin{proposition}\label{prop:exact}
Let $\cC = \{C_i,\cM_{i,j}\}$ be an oriented flow category and $A$ a $\cC$-subset. Then we have two flow morphisms $\fI_A:\cC_A \Rightarrow \cC$ and $\fP_A:\cC \Rightarrow \cC_{/A}$, which induces a short exact sequence  $0 \to \BC(\cC_A) \to \BC(\cC) \to \BC(\cC_{/A}) \to 0$. The induced long exact sequence is isomorphic to  that of \eqref{eqn:naturalexact}, if the defining data for $\cC_A,\cC_{/A}$ are the restriction of defining data on $\cC$.
\end{proposition}
\begin{proof}
	$\fI_A$ is the identity flow morphism of $\cC_A$ when the target lands in $A$, and empty set otherwise. $\fP_A$ is the identity flow morphism of $\cC_{/A}$ when the source lands outside $A$, and the empty set otherwise. Similar to the proof of Proposition \ref{prop:subcat}, both $\fI_A$ and $\fP_A$ are oriented flow morphisms. Since the induced cochain morphism of  $\fI_A$ maps $\BC(\cC_A)$ isomorphically to the subspace of $\BC(\cC)$  generated by $H^*(C_i)$ for $i\in A$ and the induced cochain morphism of $\fP_A$ vanishes on the subspace of $\BC(\cC)$  generated by $H^*(C_i)$ for $i\in A$ and map the subspace generated by $H^*(C_i)$ for $i\notin A$ isomorphically to $\BC(\cC_{/A})$, then we have a short exact sequence as below. Moreover, we claim that we have the following diagram of short exact sequences which is commutative up to homotopy.
	$$
	\xymatrix{
	0  \ar[r] & \BC(\cC_A) \ar[r]^{\phi^{I_A}} \ar[d]^{\id} & \BC(\cC) \ar[r]^{\phi^{P_A}} \ar[d]^{\id} & \BC(\cC_{/A}) \ar[r] \ar[d]^{\id} & 0\\
	0  \ar[r] & \BC(\cC_A) \ar[r]^{i} & \BC(\cC) \ar[r]^{\pi}  & \BC(\cC_{/A}) \ar[r]  & 0,}
	$$
	where the second row is the tautological short sequence \eqref{eqn:naturalexact}. To prove the claim it is equivalent to prove $\phi^{I_A}$ is homotopic to inclusion $i$ and $\phi^{P_A}$ is homotopic to the projection $\pi$.  Note that $\phi^{I_A} = i +N$ with $N$ a strict upper triangular matrix and $N = \phi^{I_A} - i =i\circ (\phi^{I_{\cC_A}}-\id)$. Similar to the proof of Theorem \ref{canonicalthm1}, we have $\fI_A\circ \fI_{\cC_A}$ and $\fI_A$ induces the same map. Hence we have $(i+N)\circ (\id +N)$ is homotopic to $i+N$ by Theorem \ref{CompMor2}. Hence $i+N$ is homotopic to $i$ if we multiply $(\id+N)^{-1}$ to the right of the homotopy relation. Similarly, we have $\phi^{P_A}$ is homotopic to the projection $\pi$. It is clear from  Theorem \ref{CompMor2} that those homotopies satisfy the conditions of Lemma \ref{lemma:commutative}, hence the claim follows.
\end{proof}

\begin{remark}\label{rmk:triangle}
    The conclusion of Lemma \ref{lemma:commutative} can be rephrased as $V_0\to V_0\oplus V_1\to V_1\to V_0[1]$ and  $V'_0\to V'_0\oplus V'_1\to V'_1\to V'_0[1]$ are equivalent distinguished triangles in $\cK(\cC h)$\footnote{When $(V,d)$ is ungraded, $V[1]$ simply means $(V,-d)$.}. In view of \S \ref{canonical}, the minimal Morse-Bott cochain complex is only well-defined in $\cK(\cC h)$. It is natural to expect that we only get well-defined distinguished triangles in  $\cK(\cC h)$.
\end{remark}

\begin{definition}
Let $\cC,\cD$ be two oriented flow categories, $A$ a $\cC$-subset and $B$ a $\cD$-subset.  We say an oriented flow morphism $\fH$ maps $A$ to $B$, iff $\cH_{i,j} = \emptyset$ whenever $i\in A, j \notin B$. 
\end{definition}
\begin{proposition}\label{prop:func_exact}
	Let $\cC,\cD$ be two oriented flow categories, $A$ a $\cC$-subset and $B$ a $\cD$-subset.  Assume an oriented flow morphism $\fH$ maps $A$ to $B$. Then we have oriented flow morphisms $\fH_A:\cC_A \Rightarrow \cD_B$ and $\fH_{/A}:\cC_{/A} \Rightarrow \cD_{/B}$ and on the cochain level, they induce a morphism between the long exact sequences.
\end{proposition}
\begin{proof}
	$\fH_A$ is the restriction of $\fH$ when both source and target land in $A$ and $B$ respectively.  	$\fH_{/A}$ is the restriction of $\fH$ when both source and target land in complements of $A$ and $B$ respectively.  Then $\fH_A$ and $\fH_{/A}$ are flow morphisms by a direct check similar to Proposition \ref{prop:subcat}. We define $\fF$ to be the flow morphism from $\cC_A$ to $\cD$, which is the restriction of $\fH$ to $\cC_A$. Since $\cH_{i,j} = \emptyset$ whenever $i\in A, j \notin B$, we have $\fH$ must land in $\cD_{B}$. Then by the same argument in Theorem \ref{canonicalthm1}, $\fH\circ \fI_A$, $\fI_B\circ \fH_A$ and $\fF$ induce the same cochain morphism. Then Theorem \ref{CompMor2} implies that both $\phi^{H}\circ \phi^{I_A}$ and $\phi^{I_B}\circ \phi^{H_A}$ are homotopic to $\phi^{\fF}$. Similarly, we have $\phi^{H_{/A}}\circ \phi^{P_A}$ and $\phi^{P_B}\circ \phi^H$ are homotopic. It is clear that the homotopies and $\phi^H$ satisfy the conditions in Lemma \ref{lemma:commutative}, hence the claim follows. 
\end{proof}
\begin{remark}
    It is clear that the identity flow morphism maps $A$ to $A$. Hence Proposition \ref{prop:func_exact} implies that the long exact sequence from Proposition \ref{prop:exact} is independent of the defining data and is isomorphic to the long exact sequence induced from \eqref{eqn:naturalexact}.
\end{remark}

%% file: s4.tex
\section{Action Spectral Sequence}\label{specseq}
Given a Morse-Bott function on a closed manifold $M$, there is a spectral sequence converging to $H^*(M)$, with the first page generated by the cohomology of critical manifolds (sometimes twisted by a local system). Such a spectral sequence is sometimes referred to as the Morse-Bott spectral sequence. For flow categories, Austin-Braam's construction \cite{austin1995morse} comes with a spectral sequence, which is induced by the an action filtration.  Moreover, it was shown under the fibration condition that the spectral sequence from Austin-Braam's construction (from the first page) is isomorphic to the Morse-Bott spectral sequence.  Similar spectral sequences from action filtration in Floer theory can be found in many places, e.g.\ \cite{seidel2008biased}. In many cases, the spectral sequence is an invariant of the Morse-Bott function, i.e.\ independent of other auxiliary structures. For example, in the finite dimensional Morse-Bott theory, any reasonable construction should recover the Morse-Bott spectral sequence, which can be constructed using only the Morse-Bott function in a purely topological manner.

The goal of this section is to prove those results for the minimal Morse-Bott cochain complex. The existences of an ``action" filtration is encoded in the definition of a flow category by requiring $\cM_{i,j} = \emptyset$ for $i > j$, since we secretly order $C_i$ by their critical values of the hypothetical Morse-Bott functional. For basics of spectral sequences arising from filtrations, we refer readers to \cite{mccleary2001user,weibel1995introduction}.

Let $\mathcal{C}:=\{C_i,\mathcal{M}_{i,j}\}$ be an oriented flow category, we have the following ``action" filtration on the minimal Morse-Bott cochain complex $\BC$,
$$F_p\BC:=\prod_{i\ge p} H^*(C_i)\subset F_{p-1}\BC\subset \BC.$$
It is clear from definition that the differential $d_{\BC,\Theta}$ is compatible with this filtration for any defining data $\Theta$. The associated spectral sequence can be described explicitly as follows. We define $Z^p_{k+1}$ to be space of $\alpha_0 \in H^*(C_p)$, such that there exist $\alpha_1,\alpha_2 \ldots \alpha_{k-1}\in H^*(C_*)$ and (we suppress the subscript $\Theta$ in $d_{i,\Theta}$ for simplicity.)
\begin{equation}\label{Z} \begin{array}{rcl}
d_1\alpha_0&=&0,\\
d_2\alpha_0+d_1\alpha_1&= & 0,\\
d_3\alpha_0+d_2\alpha_1+d_1\alpha_2 &=& 0,\\
\ldots & &\\
d_{k}\alpha_0+d_{k-1}\alpha_1+\ldots d_1\alpha_{k-1} &=& 0.
\end{array}\end{equation}
We define $B^p_{k+1}$ to be space of $\alpha \in H^*(C_p)$, such
that there exist $\alpha_0,\alpha_1,\ldots, \alpha_{k-1}\in H^*(C_*)$  and
\begin{equation}\label{B}
\begin{array}{rcl}
\alpha&=& d_{k}\alpha_0+d_{k-1}\alpha_1+\ldots d_1\alpha_{k-1}, \\
0 &=& d_{k-1}\alpha_0+d_{k-2}\alpha_1+\ldots d_1\alpha_{k-2}, \\
 & & \ldots\\
0&= & d_1\alpha_0. \end{array}\end{equation}
On $Z^p_{k+1}/B^p_{k+1}$, there is a map $\partial_{k+1}:Z^p_{k+1}/B^p_{k+1} \to Z^{p+k+1}_{k+1}/B^{p+k+1}_{k+1}$ defined by $\partial_{k+1}\alpha_0:=d_{k+1}\alpha_0+d_{k}\alpha_1+\ldots d_2\alpha_{k-1}$. Since the differential on the minimal Morse-Bott cochain complex has the special form of $\prod d_i$, unwrapping Larey's theorem on the spectral sequence associated to a filtered complex, we have the following. 

\begin{proposition}[{\cite{mccleary2001user}}]{\label{ss}} Following the notation above, we have
	 $$B^p_1\subset B^p_2\subset \ldots B^p_k\subset \cup_{k} B^p_k=B^p_\infty\subset Z^p_\infty=\cap_k Z^p_k\ldots\subset Z^p_k\subset \ldots \subset Z^p_2\subset Z^p_1.$$
	In addition, $\partial_k$ is a well-defined map from $Z^p_k/B^p_k$ to $Z^{p+k+1}_k/B^{p+k+1}_k$, such that $\partial_k^2=0$ and $Z^p_{k+1}/B^p_{k+1}\simeq H^p(Z_k/B_k,\partial_k)$ (here we view the supscript $p$ as grading and then $\partial_k$ has grading $k+1$ on $Z_k/B_k$.). Hence we have a spectral sequence $(E^p_{k}:=Z^p_k/B^p_k,\partial_k)$, 
	with $$E^p_\infty:=Z^p_\infty/B^p_\infty\simeq F_pH(\BC,d_{\BC})/F_{p+1}H(\BC, d_{\BC}),$$
	where $F_pH(\BC,d_{\BC})$ is the associated filtration on the cohomology of $(\BC,d_{\BC})$. In other words, the spectral sequence $(E^p_{k},\partial_k)$ is the spectral sequence induced from the filtration $F_p\BC$. 
\end{proposition}

\begin{remark}
    Since we do not assume $\cC$ carries a grading structure, we do not have a grading on $\BC$ (as well as its relation to the natural degree on $H^*(C_*)$) in general. In particular, we will not get a multi-complex in \cite{banyaga2010morse}. The cost is that we can not further refine the spectral sequence in $E^p_k$ using their degrees on $H^*(C_p)$.
\end{remark}

The second page of the spectral sequence is computed by taking the cohomology with respect to $\partial_1=d_1$ in \eqref{def}. Since $d_1$ is computed using $\mathcal{M}_{*,*+1}$, which are manifolds without boundary, $d_1$ is simply pullback and pushforward of cohomology. It is more accessible in good cases, works in this direction using cascades constructions can be found in \cite{diogo2018morse,diogo2018symplectic}. In general, even though $d_i$ depends on defining data in general for $i\ge 2$, $\partial_i$ does not for any $i$.

\begin{proposition}\label{prop:specinv}
Every page of the spectral sequence is independent of the defining data. 
\end{proposition}
\begin{proof}
The identity flow morphism $\fI$ induces a cochain map  $ \phi^I_{\Theta_1,\Theta_2}:(\BC,d_{\BC,\Theta_1})\to (\BC,d_{\BC,\Theta_2})$. The cochain map $ \phi^I$ preserves the filtrations, thus it induces a morphism between spectral sequences. Since the induced map on the zeroth page is the identity, thus it induces isomorphisms on every page.
\end{proof}	
\begin{remark}
	Proposition \ref{prop:specinv} only asserts the invariance of the spectral sequence w.r.t. defining data for a fixed flow category. However, the spectral sequence is expected to be an invariant of the hypothetical Morse-Bott functional, i.e.\ independent of other choices (metrics, almost complex structures, abstract perturbations) in the construction of the flow category. To prove such claim, one needs to study the underlying moduli problem and deploy some virtual techniques. We will touch this aspect of the theory briefly in \S \ref{poly}.  The spectral sequence is also expected to be independent of the specific construction method. It is an interesting question to find applications of those invariants, particularly in the quantitative aspects of symplectic geometry like the symplectic embedding problems.
\end{remark}

The final page of the spectral sequence only recovers the associated graded of the cohomology w.r.t.\ the induced filtration. If we denote
$$E_\infty := \varprojlim_p\varinjlim_{q}\oplus_{i=q}^p E^i_{\infty},$$
i.e.\ direct sum at the negative end and direct limit at the positive end of $E^i_\infty$. Following \cite[Proof of Lemma 3.10]{mccleary2001user}, we have the following exact sequence (note that we are using field coefficients),
$$0\to \varprojlim_p F_p H(\BC,d_{BC})\to H(\BC,d_{\BC})\to E_\infty\to \varprojlim_p{}^1 F_pH(\BC,d_{\BC})\to 0.$$
In some good case like $F_p\BC=0$ for $p\gg 0$, $E_\infty$ is (non-canonically) isomorphic to the Morse-Bott cohomology. For example, the symplectic cohomology considered in \cite{seidel2008biased} satisfies such condition, as the symplectic action is bounded from above.

%% file: s5.tex
\section{Orientations and Local Systems}\label{s5}
The aim of this section is explaining how orientation conventions in Definition \ref{oridef}, Definition \ref{morphism} and Definition \ref{homot} arise in applications. In applications like Morse theories or Floer theories, coherent orientations usually use extra structures from the moduli problem, namely the gluing theorem for the determinant line bundles of  Fredholm sections, see \cite{floer1993coherent}. Similar properties and constructions exist in Floer theories of different flavors beyond cohomology theory, e.g.\ \cite{bourgeois2004coherent,fukaya2010lagrangian,seidel2008fukaya}. In this section, we explain the structure which is necessary for the existence of coherent orientations on flow categories and how they arise in applications. Then we generalize the construction of the minimal Morse-Bott cochain complex to flow categories with local systems, where critical manifolds $C_i$ can be non-orientable.

\subsection{Orientations for flow categories}\label{orientation}
\subsubsection{Orientations in the Morse case} 
We first review how coherent orientations arise in the construction of Hamiltonian-Floer cohomology in the non-degenerate (Morse) case following \cite{abouzaid2013symplectic}. We will not just orient $0$ and $1$ dimensional moduli spaces but all of them and show that they satisfy Definition \ref{oridef}. Assume a symplectic manifold $(M,\omega)$ is symplectically aspherical, i.e.\ $\omega|_{\pi_2(M)}=0$. Let $H_t:S^1\times M \to \R$ be a Hamiltonian, such that all contractible 1-periodic orbits of the Hamiltonian vector field $X_{H_t}$ are non-degenerate. For simplicity, we assume that every moduli space of Floer cylinders is cut out transversely. We note here that the orientation problem is independent from many other aspects of the theory, in particular the transversality problem\footnote{In the non-transverse case, the discussion of the determinant line bundle below can be lifted to the underlying Banach manifolds/polyfolds. However, when transversality holds, there is a canonical isomorphism depending on the section/perturbation from the determinant bundle of the moduli space to $o_{i,j}$, such that it is compatible with gluing, i.e.\ \eqref{o4} and \eqref{o5} below.}. In other words, we have a flow category $\{x_i,\cM_{i,j} \}$, where $x_i$ is a non-degenerate contractible periodic orbit and $\cM_{i,j}$ is the \emph{compactified} moduli space of Floer cylinders from $x_i$ to $x_j$, where the symplectic action of $x_i$ is smaller than that of $x_j$ iff $i<j$.

To orient $\cM_{i,j}$ in a coherent way such that Definition \ref{oridef} holds, we recall the following extra structures that can be associated to the moduli spaces $\cM_{i,j}$ in the Hamiltonian-Floer cohomology. 
\begin{enumerate}
	\item\label{o1} For every periodic orbit $x_i$, we can assign an orientation line $o_i$ with a $\Z/2$ grading.  Such a line is constructed from the determinant line of a perturbed $\overline{\partial}$ operator over $\C$ with one positive end at infinity \cite[(1.4.8)]{abouzaid2013symplectic} and the grading is the index of the operator ($\mod 2$).
	\item\label{o2} For every point in $\cM_{i,j}$, there is an orientation line with a $\Z/2$ grading coming from the determinant line bundle of the linearized Floer equation at that point. All these lines form a line bundle $o_{i,j}$ over $\cM_{i,j}$. We refer readers to  \cite{zinger2013determinant} for the topology on the determinant bundle.
	\item\label{o3} By the gluing theorem for linear Fredholm operators \cite[Lemma 1.4.5]{abouzaid2013symplectic}, we have a grading preserving isomorphism over $\cM_{x,y}$,
	\begin{equation}\label{ori1}
	\rho_{i,j}:s^*o_i\otimes o_{i,j}\to t^*o_j.
	\end{equation} 
	Over $\cM_{i,j}\times \cM_{j,k}\subset \partial \cM_{i,k}$, there is a grading preserving isomorphism,
	$$\rho_{i,j,k}: \pi_1^*o_{i,j}\otimes \pi_2^*o_{j,k}\to o_{i,k},$$
	where $\pi_1,\pi_2$ are the two projections.  $\rho_{i,j}$ and $\rho_{i,j,k}$ are compatible in the sense that there is commutative diagram  over $\cM_{i,j}\times \mathcal{M}_{j,k}$ up to multiplying by a positive number,
	$$\resizebox{14cm}{!}{\xymatrix{s^*o_i \otimes \pi_1^* o_{i,j}\otimes \pi_2^*o_{j,k}\ar[rr]^{\quad\rho_{i,j}\otimes \id }\ar[d]^{\id \otimes \rho_{i,j,k}} & & \pi_1^*t^*o_j\otimes \pi_2^*o_{j,k}\ar[r]^{=} & \pi_2^*s^*o_j\otimes \pi_2^*o_{j,k}\ar[rr]^{\qquad \pi_2^*\rho_{j,k}} & & \pi_2^*t^*o_k\ar[r]^{=} & t^*o_k\ar[d]\\
	 	s^*o_i\otimes o_{i,k} \ar[rrrrrr]^{\rho_{i,k}}& & & & & & t^*o_k.
	 }}$$
	\item\label{o4} Let $\overline{\partial}_{i,j}$ be the Floer operator cutting out $\cM_{i,j}$. When transversality holds for every moduli space, $\ker \rD \overline{\partial}_{i,j}$ is a vector bundle over $\cM_{i,j}$. $\ker \rD \overline{\partial}_{i,j}$ contains an oriented trivial line subbundle $\underline{\R}$ induced by the $\R$ translation action and 
	\begin{equation}\label{eqn:trivial}
	\ker \rD \overline{\partial}_{i,j} = T\cM_{i,j}\oplus \underline{\R}. 
	\end{equation}
	Moreover, we have a grading preserving isomorphism $\phi_{i,j}:o_{i,j} \to \det \ker \rD \overline{\partial}_{i,j}$. 
	\item\label{o5} On $\cM_{i,j}\times \cM_{j,k}$, we have an isomorphism $\ker \rD \overline{\partial}_{i,j}\oplus \ker \rD \overline{\partial}_{j,k} \stackrel{\phi}{\to} \ker \rD \overline{\partial}_{i,k}$ and the following holds (we suppress the pullbacks),
	$$\xymatrix{
		o_{i,j}\otimes o_{j,k}\ar[d]^{\phi_{i,j}\otimes \phi_{j,k}} \ar[r]^{\rho_{i,j,k}} & o_{i,k}\ar[d]^{\phi_{i,k}} \\
		\det \ker \rD \overline{\partial}_{i,j} \otimes \det \ker \rD \overline{\partial}_{j,k} \ar[r]^{\qquad \quad \det \phi} & 	\det \ker \rD \overline{\partial}_{i,k}.
		}
	$$
	\item\label{o6} Let $\underline{\R}_r$, $\underline{\R}_s$ and $\underline{\R}_t$ be the trivial subbundle in $\ker \rD \overline{\partial}_{i,j}$, $\ker \rD \overline{\partial}_{j,k}$ and $\ker\rD \overline{\partial}_{i,k}$ respectively. Then by \cite[Lemma 1.5.7]{abouzaid2013symplectic}, we have
	\begin{equation}\label{boundary}
	\phi(\la r,s\ra) = t, 
	\quad \phi(\la -r,s\ra) \text{ is pointing out along } \cM_{i,j}\times \cM_{j,k}\subset \partial \cM_{i,k} \text{ in  \eqref{eqn:trivial}}.
	\end{equation} 
\end{enumerate}

\begin{proposition}\label{prop:ori}
	If we fix an orientation for every $o_i$, then \eqref{o3} and \eqref{o4} determine an orientation of $\cM_{i,j}$ and  $[\cM_{i,j}][\cM_{j,k}]=(-1)^{m_{i,j}+1}\partial [\cM_{i,k}]|_{\cM_{i,j}\times \cM_{j,k}}$.
\end{proposition}
\begin{proof}
	Given orientations of $o_i$, then the isomorphism $\rho_{i,j}$ determines an orientation of $o_{i,j}$.  Then by \eqref{o4} and $\phi_{i,j}$, there is an induced orientation $[\cM_{i,j}]$. We claim this orientation satisfies the claimed relation. By \eqref{o3}, $\rho_{i,j,k}$ preserves the orientations. Therefore $\phi:\ker \rD \overline{\partial}_{i,j} \oplus \ker \rD \overline{\partial}_{j,k} \to \ker \rD \overline{\partial}_{i,k}$ preserves the orientations. That is $[\cM_{i,j}][\underline{\R}_r][\cM_{j,k}][\underline{\R}_s]=[\cM_{i,k}][\underline{\R}_t]$. Then by \eqref{o6}, we have $[\cM_{i,j}][\cM_{j,k}]=(-1)^{m_{i,j}+1}\partial [\cM_{i,k}]|_{\cM_{i,j}\times \cM_{j,k}}$.
\end{proof}

Orientations from Proposition \ref{prop:ori} can be used to prove $d^2=0$ for Hamiltonian-Floer cohomology in the non-degenerate case. Moreover orientations $-[\cM_{i,j}]$ fit into the orientation convention in Definition \ref{oridef}.

\subsubsection{Orientations in the Morse-Bott case}\label{ss:ori_MB}
We should expect similar structures and properties in Morse-Bott theories. We phrase the structures as a definition and explain how to get an oriented flow category from there. Before stating the definition, we first introduce some notation.
\begin{enumerate}
	\item Let $E \to M$ be a vector bundle, then $\det E:= \wedge^{\max} E$ with $\Z/2$ grading $\rank E$ ($\mod 2$). We write $\det C:=\det TC$.
	\item For $\Z/2$ graded line bundles $o_1,o_2$, unless stated otherwise, the map $o_1\otimes o_2 \to o_2\otimes o_1$ is defined by 
	\begin{equation}\label{eqn:switch}
	v_1\otimes v_2 \to (-1)^{|o_1||o_2|}v_2\otimes v_1
	\end{equation}
	for vectors $v_1,v_2$ in $o_1, o_2$ respectively. 
	\item Let $\Delta$ be the diagonal in $C\times C$ with normal bundle $N$. Unless stated otherwise, the map $\det \Delta\otimes \det N \to \det C \otimes \det C$ on $\Delta$ is the map induced by the isomorphism $T\Delta\oplus N \to TC \oplus TC$. In particular, if we orient $N$ following Example \ref{ex:ori}, such a map preserves orientations.
\end{enumerate}

\begin{definition}\label{os}
	An \textbf{orientation structure} on a flow category $\cC=\{C_i,\cM_{i,j}\}$ consists of the following structures.
	\begin{enumerate}
		\item\label{os1} There are topological line bundles $o_i$ over $C_i$ with $\Z/2$ gradings for every $C_i$ and topological line bundles $o_{i,j}$ over $\cM_{i,j}$ with $\Z/2$ gradings for every $\cM_{i,j}$.
		\item\label{os2} There is a grading preserving bundle isomorphism over $\cM_{i,j}$,
		\begin{equation}\label{ori4}
		\rho_{i,j}: s^* o_i\otimes s^* \det C_i \otimes o_{i,j}\to t^*o_j\end{equation}
		and a grading preserving bundle isomorphism over $\cM_{i,j}\times_j \cM_{j,k}\subset \partial \cM_{i,k}$,
		\begin{equation}\label{ori55}\rho_{i,j,k}: \pi_1^*o_{i,j}\otimes (t\times s)^*\det T\Delta_{j}\otimes \pi_2^* o_{j,k}\to o_{i,k}|_{\cM_{i,j}\times_j \cM_{j,k}}.
		\end{equation}
		The bundle isomorphisms are compatible in the sense that the following diagram over $\cM_{i,j}\times_j \cM_{j,k}$ is commutative up to multiplying by a positive number,	
		\begin{equation}\label{ori5}
		\resizebox{14cm}{!}{
		\xymatrix{s^*o_i \otimes s^*\det C_i\otimes \pi_1^*o_{i,j}\otimes (t\times s)^*\det \Delta_j\otimes \pi_2^* o_{j,k}\ar[rr]^{\qquad \qquad  \rho_{i,j}\otimes \id}\ar[d]^{\id \otimes \rho_{i,j,k}} & & \pi_2^*s^*o_j\otimes \pi_2^* s^* \det C_j\otimes \pi_2^*o_{j,k}\ar[r]^{\qquad \qquad \quad \pi_2^*\rho_{j,k}} & t^*o_k\ar[d] \\
		 	s^*o_i \otimes s^* \det C_i \otimes  o_{i,k}\ar[rrr]^{\rho_{i,k}}& & & t^*o_k.
		 }}\end{equation}
		 The diagram makes sense because over the fiber product $\cM_{i,j}\times_j\cM_{j,k}$, $\pi_1^*t^*o_j=\pi_2^*s^*o_j$ and $(t\times s)^* \det \Delta_j=\pi_2^* s^* \det C_j$.
		\item\label{os3} There are vector bundles $V_{i,j}$ over $\cM_{i,j}$ with smooth bundle maps $S_{i,j}: V_{i,j} \to TC_i, T_{i,j}: V_{i,j} \to TC_j$ covering $s_{i,j}:\cM_{i,j}\to C_i$ and $t_{i,j}:\cM_{i,j}\to C_j$ respectively. Moreover, there is an oriented trivial subbundle $\underline{\R}$  of $V_{i,j}$ such that $S_{i,j}(\underline{\R})=T_{i,j}(\underline{\R})=0$ and
		\begin{equation}
		V_{i,j} = T\cM_{i,j}\oplus \underline{\R},
		\end{equation}
		and $S_{i,j}|_{ T\cM_{i,j}}=\rd s_{i,j}$,$T_{i,j}|_{ T\cM_{i,j}}=\rd t_{i,j}$. There is a grading preserving isomorphism $\phi_{i,j}: s^*\det C_i\otimes o_{i,j}\otimes t^*\det C_j \to \det V_{i,j}$.
		\item\label{os4} On $\cM_{i,j}\times_j \cM_{j,k}$, we have $V_{i,j}\times_{TC_j}V_{j,k} = V_{i,k}$ and the following holds,
		\begin{equation}\label{dia:MB}
		\resizebox{14cm}{!}{
		\xymatrix{(t\times s)^*\det N_j\otimes s^*\det C_i \otimes o_{i,j}\otimes (t\times s)^*\det \Delta_j \otimes o_{j,k}\otimes t^*\det C_k \ar[r]^{\qquad \qquad  \quad\rho_{i,j,k}}\ar[d] & (t\times s)^*\det N_j\otimes s^*\det C_i \otimes o_{i,k} \otimes t^* \det C_k \ar[ddd]^{\phi_{i,k}}\\
		s^*\det C_i \otimes o_{i,j}\otimes (t\times s)^*(\det \Delta_j\otimes\det N_j) \otimes o_{j,k}\otimes t^*\det C_k \ar[d] & \\
		s^*\det C_i \otimes o_{i,j}\otimes t^*\det C_j \otimes s^* \det C_j \otimes o_{j,k}\otimes t^*\det C_k \ar[d]^{\phi_{i,j}\otimes \phi_{j,k}} & \\
		\det V_{i,j}\otimes \det V_{j,k} \ar[r] & (t\times s)^* \det N_j \otimes \det V_{i,k},}
		}
		\end{equation}
		where the last map is induced by the isomorphism $V_{i,j}\oplus V_{j,k} =  (t\times s)^*N_j \oplus V_{i,k}$. 
		\item\label{os5} Let $\underline{\R}_r$, $\underline{\R}_s$ and $\underline{\R}_t$ be the trivial subbundle in $V_{i,j}$, $V_{j,k}$ and $V_{i,j}$ respectively. We have
		\begin{equation}\label{boundaryMB}
		\la r,s\ra = t, \quad \la -r,s\ra \text{ is pointing out along } \cM_{i,j}\times_j \cM_{j,k}\subset \partial \cM_{i,k}.
		\end{equation} 
	\end{enumerate}
\end{definition}

In applications, the topological line bundle $o_i$ is the determinant line bundle of a perturbed Floer equation with exponential decay at the end over a domain with one positive end. For the details on exponential decay, we refer readers to \cite{bourgeois2002morse, frauenfelder2004arnold}. The topological line bundle $o_{i,j}$ usually comes from the determinant bundle of the Floer equation with exponential decays at both ends over a cylinder. The bundle isomorphism and their compatible diagram come from a version of the linear gluing theorem for Fredholm operators \cite{abouzaid2013symplectic, floer1993coherent}. $V_{i,j}$ is the kernel of the linearized Floer operator defining $\cM_{i,j}$ and the trivial subbundle comes from the $\R$ translation. The last condition \eqref{os5} comes from a similar argument in \cite[Lemma 1.5.7]{abouzaid2013symplectic}. $o_{i,j}$ can be defined on the background Banach manifold or polyfolds \cite[Chapter 6]{hofer2017polyfold}, however $V_{i,j}$ is defined only when transversality holds. \eqref{os3} of Definition \ref{os} states the relation between $V_{i,j}$, $o_{i,j}$ and $T\cM_{i,j}$. \eqref{os4} states the compatibility with the gluing map $\rho_{i,j,k}$. 
\begin{remark}
    Similar to Definition \ref{def:grade}, Definition \ref{os} is a simplified version. In general, we should associate each component of $C_i$ with a line bundle and each component of $\cM_{i,j}$ with a bundle isomorphism satisfying similar compatibility conditions.
\end{remark}
\begin{remark}
	Definition \ref{os} is modeled on the classical treatment of the Floer equation \cite{bourgeois2002morse, frauenfelder2004arnold}, that is, we modulo out the $\R$ translation after solving the Floer equation. Hence we expect bundles $V_{i,j}$ over $\cM_{i,j}$ contains a trivial oriented $\R$ direction. If we use the polyfold setup, then the Floer operator is defined on polyfolds of cylinders with the $\R$ translation already quotiented out, see \cite{hofer2017application, wehrheim2012fredholm}. One can adjust Definition \ref{os} to be consistent with such a point of view. 
\end{remark}

\begin{proposition}\label{orientation1}
	Assume the flow category $\cC$ has an orientation structure and all the line bundles $o_i$ are oriented and $C_i$ are oriented.  Then $\cC$ can be coherently oriented.
\end{proposition}
\begin{proof}
By the map $\rho_{i,j}$ in \eqref{ori4}, if $o_i$ are oriented and $C_i$ are oriented, then there are induced orientations $[o_{i,j}]$ on $o_{i,j}$. By \eqref{ori5}, over the fiber product $\cM_{i,j}\times_j\cM_{j,k}$ we have
\begin{equation}\label{eqn:comp}
\rho_{i,j,k}(\pi_1^*[o_{i,j}]\otimes (t\times s)^*[\Delta_j]\otimes \pi_2^*[o_{j,k}]) = [o_{i,k}].
\end{equation}
Using $\phi_{i,j}$ in \eqref{os4} in Definition \ref{os}, we have an orientation $[V_{i,j}]$ on $V_{i,j}$. Then by \eqref{eqn:comp}, the commutative diagram \eqref{dia:MB} in Definition \ref{os} implies that the natural map $V_{i,j}\oplus V_{j,k}\to (t\times s)^*N_j \times V_{i,k}$ induces 
$$[V_{i,j}]\otimes [V_{j,k}] \mapsto (-1)^{c_j(m_{i,j}+1)}(t\times s)^*[N_j]\otimes [V_{i,k}].$$
on the prescribed orientations. By \eqref{os3} of Definition \ref{os}, the orientation $[V_{i,j}]$ induces an orientation $[\cM_{i,j}]$. Hence we have on $\cM_{i,j}\times_{j}\cM_{j,k}\subset \partial \cM_{i,k}$, we have
$$[\cM_{i,j}][\underline{\R}_r][\cM_{j,k}][\underline{\R}_s] = (-1)^{c_j(m_{i,j}+1)}(t\times s)^*[N_j][\cM_{i,k}][\underline{\R}_t].$$
Then \eqref{os5} of Definition \ref{os} implies that
$$[\cM_{i,j}][\cM_{j,k}] = (-1)^{c_jm_{i,j}+m_{i,j}+1}(t\times s)^*[N_j]\partial[\cM_{i,k}]|_{\cM_{i,j}\times_j\cM_{j,k}}.$$
Then orientations $-[\cM_{i,j}]$ satisfy Definition \ref{oridef}.\footnote{One can certainly modify the definition of coherent orientations of a flow category (Definition \ref{oridef}) so that $[\cM_{i,j}]$ gives a coherent orientation. Then the signs in \eqref{def} do not factorize nicely.}
\end{proof}
When the $o_i$ are not oriented or the $C_i$ are not oriented, Definition \ref{os} gives all the structures we need to work with the local system $o_i$. We discussion such generalization in \S \ref{sub:localsystem}.

\subsubsection{Orientations for flow morphisms}
Next, we explain how the orientation convention in Definition \ref{morphism} arise in application. 
\begin{definition}\label{MorOS}
	Assume $\fH=\{\cH_{i,j}\}$ is a flow morphism from flow category $\cC$ to $\cD$, such that $\cC$ and $\cD$ have orientation structures. A compatible orientation structure on $\fH$ is the following.
	\begin{enumerate}
		\item\label{MorOS1} There are $\Z/2$ graded line bundles $o^H_{i,j}$ over $\cH_{i,j}$. Over $\cH_{i,j}$, we have a grading preserving isomorphism
		\begin{equation}\label{ori7}
		\rho^H_{i,j}: s^*o^C_i\otimes s^*\det C_i\otimes o^H_{i,j}\to t^*o^D_j.
		\end{equation}
		\item\label{MorOS2}  Over the fiber product $\cM^C_{i,j}\times_j\cH_{j,k}\subset \partial \cH_{i,k}$, we have a grading preserving isomorphism
		\begin{equation}\label{ori8}
		\rho^{C,H}_{i,j,k}:\pi_1^*o^C_{i,j}\otimes (t\times s)^* \det \Delta^C_j\otimes \pi_2^*o^{H}_{j,k}\to o^H_{i,k}.
		\end{equation} 
		Over the fiber product $\cH_{i,j}\times_j \cM^D_{j,k}\subset \partial \cH_{i,k}$, we have a grading preserving isomorphism
		\begin{equation}\label{ori8'}
		\rho^{H,D}_{i,j,k}:\pi_1^*o^H_{i,j}\otimes (t\times s)^* \det \Delta^D_j\otimes \pi_2^*o^{D}_{j,k}\to o^H_{i,k}.
		\end{equation} 
		\item\label{MorOS3} The bundle isomorphisms in \eqref{MorOS1} and \eqref{MorOS2} are compatible in the sense that over $\cM^C_{i,j}\times_j \cH_{j,k}$ and $\cH_{i,j}\times_j\cM^D_{j,k}$, we have the following commutative diagrams respectively,
		\begin{equation}\label{ori9'}
		\resizebox{14cm}{!}{
		\xymatrix{s^*o^C_i \otimes s^*\det C_i\otimes \pi_1^*o^C_{i,j}\otimes (t\times s)^*\det \Delta^C_j\otimes \pi_2^* o^H_{j,k}\ar[rr]^{\qquad \qquad \quad\rho^C_{i,j}\otimes \id}\ar[d]^{\id \otimes\rho^{C,H}_{i,j,k}}& & \pi_2^*s^*o^D_j\otimes \pi_2^* s^* \det D_j\otimes \pi_2^*o^D_{j,k}\ar[r]^{\qquad \qquad \quad \rho^{H}_{j,k}} & t^*o_k\ar[d]\\
			s^*o^C_i \otimes s^* \det C_i \otimes  o^H_{i,k}\ar[rrr]^{\rho^H_{i,k}}& & & t^*o_k,
		}
		}\end{equation}		
		\begin{equation}\label{ori9}
		\resizebox{14cm}{!}{
		\xymatrix{s^*o^C_i \otimes s^*\det C_i\otimes \pi_1^*o^H_{i,j}\otimes (t\times s)^*\det \Delta^D_j\otimes \pi_2^* o^D_{j,k}\ar[rr]^{\qquad \qquad \quad\rho^H_{i,j}\otimes \id}\ar[d]^{\id \otimes\rho^{H,D}_{i,j,k}}& & \pi_2^*s^*o^D_j\otimes \pi_2^* s^* \det D_j\otimes \pi_2^*o^D_{j,k}\ar[r]^{\qquad \qquad \quad \rho^{D}_{j,k}} & t^*o_k\ar[d]\\
			s^*o^C_i \otimes s^* \det C_i \otimes  o^H_{i,k}\ar[rrr]^{\rho^H_{i,k}}& & & t^*o_k.
		}}\end{equation}	
		\item\label{MorOS4} There is a grading preserving isomorphism $\phi^H_{i,j}:s^*\det C_i \otimes o^H_{i,j}\otimes t^* \det D_j \to \det T\cH_{i,j}$.
		\item On $\cM^C_{i,j}\times_j \cH_{j,k}\subset \partial \cH_{i,k}$, we have $V^C_{i,j}\times_{TC_j} T\cH_{j,k} = T\cH_{i,k}$ and the following holds,
		\begin{equation}\label{ori''}
		\resizebox{14cm}{!}{
		\xymatrix{ (t\times s) \det N^C_j \otimes s^*\det C_i \otimes o^C_{i,j}\otimes (t\times s)^*\det \Delta^C_j \otimes o^{H}_{j,k} \otimes t^*\det D_k \ar[d]\ar[r]^{\qquad \qquad \rho^{C,H}_{i,j,k}} & (t\times s)^* \det N^C_j \otimes s^*\det C_i\otimes o^H_{i,k} \otimes t^*\det D_k \ar[ddd]^{\phi^H_{i,k}} \\
		s^*\det C_i \otimes o^C_{i,j}\otimes (t\times s)^*(\det \Delta^C_j\otimes N^C_j) \otimes o^{H}_{j,k} \otimes t^*\det D_k \ar[d] &  \\
		s^*\det C_i \otimes o^C_{i,j}\otimes t^*\det C_j \otimes s^* \det C_j \otimes o^{H}_{j,k} \otimes t^*\det D_k \ar[d]^{\phi^C_{i,j}\otimes \phi^H_{j,k}}&  \\
		\det V^C_{i,j}\otimes \det T\cH_{j,k} \ar[r] & (t\times s)^* \det N^C_j \otimes \det T\cH_{i,k},}
		}
		\end{equation}
		where the last row is induced by the isomorphism $V^C_{i,j}\oplus T\cH_{j,k}\to (t\times s)^*N^C_j \oplus T\cH_{i,k}$.
		
		On $\cH_{i,j}\times_j \cM^D_{j,k}\subset \partial \cH_{i,k}$, we have $T\cH_{i,j}\times_{TD_j} V^D_{j,k} = T\cH_{i,k}$ and the following holds,
		\begin{equation}\label{ori'}
		\resizebox{14cm}{!}{
		\xymatrix{ (t\times s) \det N^D_j \otimes s^*\det C_i \otimes o^H_{i,j}\otimes (t\times s)^*\det \Delta^D_j \otimes o^{D}_{j,k} \otimes t^*\det D_k \ar[d]\ar[r]^{\qquad \qquad \rho^{H,D}_{i,j,k}} & (t\times s)^* \det N^D_j \otimes s^*\det C_i\otimes o^H_{i,k} \otimes t^*\det D_k \ar[ddd]^{\phi^H_{i,k}} \\
		s^*\det C_i \otimes o^H_{i,j}\otimes (t\times s)^*(\det \Delta^D_j\otimes N^D_j) \otimes o^D_{j,k} \otimes t^*\det D_k \ar[d] &  \\
		s^*\det C_i \otimes o^H_{i,j}\otimes t^*\det D_j \otimes s^* \det D_j \otimes o^D_{j,k} \otimes t^*\det D_k \ar[d]^{\phi^H_{i,j}\otimes \phi^D_{j,k}}&  \\
		\det T\cH_{i,j}\otimes \det V^D_{j,k} \ar[r] & (t\times s)^* \det N^D_j \otimes \det T\cH_{i,k},}
		}
		\end{equation}
		where the last row is induced by the isomorphism $T\cH_{i,j}\oplus V^D_{j,k}\to (t\times s)^*N^D_j \oplus T\cH_{i,k}$.
		\item Let $\underline{\R}_s,\underline{\R}_t$ be the trivial line in $V^C_{i,j}$ and $V^D_{j,k}$ respectively, then $s$ points in along $\cM^C_{i,j}\times_j \cH_{j,k}\subset \partial \cH_{i,k}$ and $t$ points out along $\cH_{i,j}\times_j \cM^D_{j,k}\subset \partial \cH_{i,k}$.
	\end{enumerate}
\end{definition}
In the example of Hamiltonian Floer cohomology for non-degenerate Hamiltonians, the bundle $o^H_{i,j}$ is the determinant line bundle of the time-dependent Floer equation \cite[p.~384]{audin2014morse}. In the Morse-Bott case, $o^H_{i,j}$ is the determinant line bundle of the time-dependent Floer equation with exponential decays at both ends. By the same argument in Proposition \ref{orientation1}, we have the following.
\begin{proposition}\label{orientation2}
Let $\cC$, $\cD$ be two flow categories with orientation structures and $\fH$ be a flow morphism from $\cC$ to $\cD$ with a compatible orientation structure. Assume $o_i^C, o_i^D, C_i, D_i$ are oriented and $\cC$ and $\cD$ are oriented using Proposition \ref{orientation1}. Then \eqref{MorOS1} and \eqref{MorOS4} of Definition \ref{MorOS} determine orientations on $\cH_{i,j}$, such that $\fH$ is an oriented flow morphism from $\cC$ to $\cD$.  
\end{proposition}
\begin{remark}
	A compatible orientation structure on a flow premorphism is \eqref{MorOS1} and \eqref{MorOS4} of Definition \ref{MorOS}, hence we have enough structures to orient the spaces in a flow premorphism when $o^C_i,o^D_j, C_i,D_i$ are oriented.  The composition $\fF\circ \fH$ of two composable flow morphisms $\fF,\fH$ with compatible orientation structures has a natural compatible orientation structure, where $o^{F\circ H}_{i,j}|_{\cH_{i,j}\times_j \cF_{j,k}} = \pi_1^*o^H_{i,j}\otimes (t^H_{i,j}\times s^F_{j,k})^*\det \Delta^D_j \otimes \pi_2^*o^F_{j,k}$.
\end{remark}

\subsubsection{Orientations for flow homotopies}    
In applications, a flow homotopy from $\fH$ to $\fF$ usually comes from considering a time-dependent Floer equation with an extra $[0,1]_z$ parameter \cite[p.~414]{audin2014morse}, such that when $z=0$, the equation defines the flow morphism $\fH$ and when $z=1$, the equation defines the flow morphism $\fF$.  Hence we have the following definition.

\begin{definition}\label{HomOS}
	Let $\fH,\fF$ be two flow premorphisms with orientation structures from $\cC$ to $\cD$ with orientation structures compatible with that of $\cC$ and $\cD$. A flow homotopy $\sY$ between $\fH$ and $\fF$ is said to have a compatible orientation structure if the following holds.
	\begin{enumerate}
		\item\label{HomOS1} There are $\Z/2$ graded line bundles $o^Y_{i,j}$ over $\cY_{i,j}$. Over $\cY_{i,j}$ we have a grading preserving isomorphism
		\begin{equation}\label{ori11}\rho^Y_{i,j}: s^*o^C_i\otimes s^*\det C_i\otimes o^Y_{i,j}\to t^*o^D_j.\end{equation}		
		\item\label{HomOS2} Over the fiber product $\cM^C_{i,j}\times_j\cY_{j,k}\subset \cY_{i,k}$, we have a grading preserving isomorphism
		\begin{equation}\label{ori12} \rho^{C,Y}_{i,j,k}: \pi_1^*o^C_{i,j}\otimes (t\times s)^* \det \Delta^C_j\otimes \pi_2^*o^{Y}_{j,k}\to o^Y_{i,k}.\end{equation}
		 Over the fiber product $\cY_{i,j}\times_j \cM^D_{j,k}\subset \partial \cY_{i,k}$, we have a grading preserving isomorphism 
		\begin{equation}\label{ori12'} \rho^{Y,D}_{i,j,k}: \pi_1^*o^Y_{i,j}\otimes (t\times s)^* \det \Delta^D_j\otimes \pi_2^*o^{D}_{j,k}\to o^Y_{i,k}.\end{equation} 
		\item\label{HomOS3} $\rho^Y_{i,j}$, $\rho^{C,Y}_{i,j,k}$ and $\rho^{Y,D}_{i,j,k}$ are compatible such that the similar commutative diagrams in \eqref{MorOS3} of Definition \ref{MorOS} hold.
		\item\label{HomOS4} On $\cH_{i,j}\subset \partial \cY_{i,j}$, we have $o^Y_{i,j}|_{\cH_{i,j}} = o^H_{i,j}$ and $\rho^Y_{i,j}|_{\cH_{i,j}} = \rho^H_{i,j}$, similarly for $\cF_{i,j}\subset \partial \cY_{i,j}$.
		\item\label{HomOS5} We have $T\cY_{i,j}|_{\cH_{i,j}} = \underline{\R}_z \oplus T\cH_{i,j}$ with $z$ pointing in along the boundary and $T\cY_{i,j}|_{\cF_{i,j}} = \underline{\R}_z \oplus T\cF_{i,j}$ with $z$ pointing out along the boundary. And there is a $\Z/2$ bundle isomorphism $\phi^Y_{i,j}:\underline{\R}_z\otimes s^*\det C_i \otimes o^Y_{i,j}\otimes t^*\det D_j \to \det T\cY_{i,j}$, such that $\phi^Y_{i,j}|_{\cH_{i,j}} = \id_{\underline{\R}_z}\otimes \phi^H_{i,j}$ and $\phi^Y_{i,j}|_{\cF_{i,j}} = \id_{\underline{\R}_z}\otimes \phi^F_{i,j}$. 
		\item\label{HomOS6} On $\cM^C_{i,j}\times_j \cY_{j,k}\subset \partial \cY_{i,k}$, we have $V^C_{i,j}\times_{TC_j} T\cY_{j,k} = T\cY_{i,k}$ and the following holds (we suppress the pullback notation),
		\begin{equation}\label{ori13}
		\resizebox{14cm}{!}{
		\xymatrix{\underline{\R}_z\otimes \det N^C_j \otimes \det C_i \otimes o^C_{i,j}\otimes \det \Delta^C_j \otimes o^{Y}_{j,k} \otimes \det D_k \ar[d]\ar[r]^{\qquad \qquad \rho^{C,Y}_{i,j,k}} & \underline{\R}_z\otimes \det N^C_j \otimes \det C_i\otimes o^Y_{i,k} \otimes \det D_k \ar[ddd]^{\phi^Y_{i,k}} \\
		\underline{\R}_z\otimes \det C_i \otimes o^C_{i,j}\otimes \det \Delta^C_j\otimes N^C_j \otimes o^{Y}_{j,k} \otimes \det D_k \ar[d] &  \\
		\det C_i \otimes o^C_{i,j}\otimes \det C_j \otimes \underline{\R}_z \otimes  \det C_j \otimes o^{Y}_{j,k} \otimes \det D_k \ar[d]^{\phi^C_{i,j}\otimes \phi^Y_{j,k}}&  \\
		\det V^C_{i,j}\otimes \det T\cY_{j,k} \ar[r] & (t\times s)^* \det N^C_j \otimes \det T\cY_{i,k},}
		}
		\end{equation}
		where the last row is induced by the isomorphism $V^C_{i,j}\oplus T\cY_{j,k}\to (t\times s)^*N^C_j \oplus T\cY_{i,k}$. 
		
		On $\cY_{i,j}\times_j \cM^D_{j,k}\subset \partial \cY_{i,k}$, we have $T\cY_{i,j}\times_{TD_j} V^D_{j,k} = T\cY_{i,k}$ and the following holds,
		\begin{equation}\label{ori13'}
		\resizebox{14cm}{!}{
		\xymatrix{\underline{\R}_z\otimes\det N^D_j \otimes \det C_i \otimes o^Y_{i,j}\otimes \det \Delta^D_j \otimes o^{D}_{j,k} \otimes \det D_k \ar[d]\ar[r]^{\qquad \qquad \rho^{Y,D}_{i,j,k}} & \det N^D_j \otimes \det C_i\otimes o^Y_{i,k} \otimes\det D_k \ar[ddd]^{\phi^Y_{i,k}} \\
		\underline{\R}_z\otimes\det C_i \otimes o^Y_{i,j}\otimes \det \Delta^D_j\otimes N^D_j \otimes o^D_{j,k} \otimes \det D_k \ar[d] &  \\
		\det C_i \otimes o^Y_{i,j}\otimes \det D_j \otimes \underline{\R}_z\otimes \det D_j \otimes o^D_{j,k} \otimes \det D_k \ar[d]^{\phi^Y_{i,j}\otimes \phi^D_{j,k}}&  \\
		\det T\cY_{i,j}\otimes \det V^D_{j,k} \ar[r]^{(-1)^{d_j}} & \det N^D_j \otimes \det T\cY_{i,k},}
		}
		\end{equation}
		where the last row is induced by the isomorphism $T\cY_{i,j}\oplus V^D_{j,k}\to (t\times s)^*N^D_j \oplus T\cY_{i,k}$ twisted by $(-1)^{d_j}$ (because of the extra $\R_z$).
		\item\label{HomOS7} Let $\underline{\R}_s,\underline{\R}_t$ be the trivial line in $V^C_{i,j}$ and $V^D_{j,k}$ respectively, then $s$ points in along $\cM^C_{i,j}\times \cY_{j,k}\subset \partial \cY_{i,k}$ and $t$ points out along $\cY_{i,j}\times \cM^D_{j,k}\subset \partial \cY_{i,k}$. 
	\end{enumerate}
\end{definition}
If we can fix orientations of $o^C_i$,$o^D_i$, $C_i$ and $D_i$, then \eqref{HomOS1}, \eqref{HomOS4} and \eqref{HomOS5} imply the induced orientations of $\cY_{i,j}$, $\cH_{i,j}$ and $\cF_{i,j}$ satisfy 
$$
\partial [\cY_{i,j}|_{\cH_{i,j}}]=-[\cH_{i,j}], \quad \partial [\cY_{i,j}|_{\cF_{i,j}}]=[\cF_{i,j}].
$$
In general, we have the analog of Proposition \ref{orientation1} and \ref{orientation2} as follows.
\begin{proposition}
	Let $\sY$ be a flow homotopy between two flow premorphisms $\fH,\fF$ from $\cC$ to $\cD$. Assume everything is equipped with compatible orientation structures and $o^C_i, o^D_i, C_i, D_i$ are oriented. If $\cC,\cD,\fH,\fF$ are oriented by Proposition \ref{orientation1} and \ref{orientation2}, then $\cY_{i,j}$ can be oriented by \eqref{HomOS1} and \eqref{HomOS5} of Definition \ref{HomOS} such that $\sY$ is an oriented flow homotopy between $\fH$ and $\fF$.  
\end{proposition}

\subsection{Local Systems}\label{sub:localsystem}
From the discussion in \S \ref{orientation}, to orient a flow category, a flow morphism or a flow homotopy with orientation structures, we need to orient $o_i$ and $C_i$. However, in the Morse-Bott case, it is possible that $C_i$ is not orientable or $o_i$ is not orientable. Hence we need to upgrade the minimal Morse-Bott cochain complex to a version with local systems. In fact, Definition \ref{os}, Definition \ref{MorOS} and Definition \ref{HomOS} already provides all the structures needed to define a cochain complex without any orientable assumptions, the generator will be the cohomology of $C_i$ twisted by $o_i$. In the case of finite dimensional Morse-Bott theory, let $C$ be a critical manifold with stable bundle $S$. Then in view of the Thom isomorphism, the contribution from a critical manifold $C$ to the  total cohomology should be the cohomology with local system $H^*(C,\det S)$.  In the abstract setting, if a flow category has an orientation structure, then the line bundle $o_i$ plays the role of $\det S$. 

We will introduce a more compact definition just like Definition \ref{oridef}. First we introduce some notation.  Let $\cC = \{C_i,\cM_{i,j}\}$ be a flow category. Over $\cM_{i,j}\times_j\cM_{j,k}\subset \partial \cM_{i,k}$, we have an induced isomorphism $T\cM_{i,j}\oplus T\cM_{j,k} \to (t\times s)^*N_j \oplus T\partial \cM_{i,k}.$ If we use the identification $t^*TC_j \to (t\times s)^* N_j, v \mapsto (-v,v)$, we have an isomorphism $T\cM_{i,j}\oplus T\cM_{j,k} \to t^*C_j \oplus T\partial \cM_{i,k}.$ Therefore we have an isomorphism over $\cM_{i,j}\times_j\cM_{j,k}$
$$\det \cM_{i,j}\otimes \det \cM_{j,k} \to t^* \det C_j \otimes \det \partial \cM_{ik}.$$
Using the isomorphism $\underline{\R}_{out} \oplus T\partial \cM_{i,k} = T\cM_{i,k}$, there is a natural isomorphism $\det \partial \cM_{i,k} \to \det \cM_{i,k}$ preserving compatible orientations. Hence have an isomorphism of degree $1$
$$\det \cM_{i,j}\otimes \det \cM_{j,k} \to t^* \det C_j \otimes \det \cM_{ik},$$
which induces an isomorphism
\begin{equation}\label{eqn:localcon}
f: \det \cM_{i,j}\otimes t^* \textstyle{\det^*} C_j \otimes \det \cM_{j,k} \to \det \cM_{j,k},
\end{equation}
where $\det^*C_j=(\det C_j)^*$. Here $f$ is induced by  the natural isomorphism $t^* \det C_j \otimes t^* \det^* C_j =\R$ and the order switch convention \eqref{eqn:switch}. 
\begin{definition}\label{def:local}
	Let $\cC = \{C_i,\cM_{i,j}\}$ be a flow category. Then \textbf{a local system} on $\cC$ consists of the following.
	\begin{enumerate}
		\item There is a line bundle $o_i$ on each $C_i$. 
		\item Over the $\cM_{i,j}$, there is a bundle isomorphism, 
		$$\rho_{i,j}:s^*o_i\otimes \det \cM_{i,j} \otimes t^* \textstyle{\det^*} C_j \to t^*o_j,$$
		such that the following digram over $\cM_{i,j}\times_j\cM_{j,k}\subset \partial \cM_{i,k}$ is commutative,
		$$
		\resizebox{14cm}{!}{
		\xymatrix{
			s^*o_i\otimes \det \cM_{i,j} \otimes t^* \det^* C_j \otimes \det \cM_{j,k}  \otimes t^* \det^* C_k \ar[r]^{\qquad \qquad \rho_{i,j}} \ar[d]^{f} & s^*o_j\otimes \det \cM_{j,k} \otimes t^*\det ^* C_k \ar[r]^{\qquad \qquad \rho_{j,k}} & t^*o_k\ar[d]\\
			s^*o_i\otimes \det \cM_{i,k}\otimes t^*\det^* C_k \ar[rr]^{(-1)^{m_{i,j}+1}\rho_{i,k}} & & t^*o_k,
			}
			}
		$$
		where $f$ is defined in \eqref{eqn:localcon}.
	\end{enumerate}
\end{definition}

\begin{proposition}\label{prop:localsys}
	If $\cC$ has an orientation structure, then $o_i$ is a local system on $\cC$.
\end{proposition}
\begin{proof}
	Since $\cC$ has an orientation structure, i.e.\ we have isomorphisms $\rho^C_{i,j}:s^*o_i\otimes s^*\det C_i \otimes o_{i,j} \to t^*o_j$, $V_{i,j} = T\cM_{i,j}\oplus \underline{\R}$ and $\phi_{i,j}:s^*o_i \otimes o_{i,j}\otimes t^*o_j \to \det V_{i,j}$. Then using the natural orientation on $\underline{\R}$ and isomorphisms $\phi_{i,j}$, $\rho^C_{i,j}$, we get an isomorphism $\rho_{i,j}:s^*o_i\otimes \det \cM_{i,j}\otimes t^*\det^*C_j \to t^*o_j$. By the same argument in Proposition \ref{orientation1}, \eqref{os4} and \eqref{os5} of Definition \ref{os} imply the commutative diagram in Definition \ref{def:local}. 
\end{proof}
Similarly, we have the following definitions of local systems on flow morphism and flow homotopies. 
\begin{definition}
	Let $\fH = \{\cH_{i,j}\}$ be a flow morphism from the flow category $\cC$ to the flow category  $\cD$. Both $\cC$ and $\cD$ are equipped with local systems. We say $\fH$ has a \textbf{compatible local system} if on each $\cH_{i,j}$ we have an isomorphism 
	$$\rho^H_{i,j}:s^*o^C_i\otimes \det \cH_{i,j} \otimes t^* \textstyle{\det^*} C_j \to t^*o^D_j,$$
	such that we have the following two commutative diagrams over $\cM^C_{i,j}\times_j\cH_{j,k}\subset \partial \cH_{i,k}$ and $\cH_{i,j}\times_j \cM^D_{j,k}\subset \partial \cH_{i,k}$ respectively, 
	$$
	\resizebox{14cm}{!}{
	\xymatrix{
	s^*o^C_i \otimes \det \cM^C_{i,j} \otimes t^*\det^* C_j \otimes \det \cH_{j,k} \otimes t^*\det^* D_k \ar[r]^{\qquad \qquad \quad \rho^C_{i,j}} \ar[d]^f & s^*o^C_j\otimes \det \cH_{j,k} \otimes t^* \det^* D_k \ar[r]^{\qquad \qquad \quad\rho^H_{j,k}} & t^*o^D_k\ar[d] \\
	s^*o^C_i \otimes \cH_{i,k} \otimes t^*\det^* D_k \ar[rr]^{(-1)^{m^C_{i,j}+1}\rho^H_{i,k}} & & t^*o^D_k,
	}
	}
	$$
	$$
	\resizebox{14cm}{!}{
	\xymatrix{
		s^*o^C_i \otimes \det \cH_{i,j} \otimes t^*\det^* D_j \otimes \det \cM^D_{j,k} \otimes t^*\det^* D_k \ar[r]^{\qquad \qquad \quad\rho^H_{i,j}} \ar[d]^f & s^*o^D_j\otimes \det \cM^D_{j,k} \otimes t^* \det^* D_k \ar[r]^{\qquad \qquad \quad\rho^D_{j,k}} & t^*o^D_k\ar[d] \\
		s^*o^C_i \otimes \cH_{i,k} \otimes t^*\det^* D_k \ar[rr]^{(-1)^{h_{i,k}+1}\rho^H_{i,k}} & & t^*o^D_k,
	}}
	$$	
	where the map $f$ in the first columns of both digram is defined in a similar way to \eqref{eqn:localcon}.
\end{definition}

\begin{definition}
	A compatible local system on a flow premorphism $\fH$ from $\cC$ to $\cD$ consists of bundle isomorphisms $\rho^H_{i,j}:s^*o^C_i \otimes \det \cH_{i,j} \otimes t^* \det^* D_j \to t^* o^D_j$ on every $\cH_{i,j}$.
\end{definition}
\begin{definition}
		Let $\sY$ be a flow morphism between flow premorphisms $\fH$ and $\fF$ from the flow category $\cC$ to the flow category  $\cD$. Assume $\cC$, $\cD$, $\fH$ and $\fF$ are equipped with compatible local systems. We say $\sY$ has a \textbf{compatible local system} if on each $\cY_{i,j}$ we have an isomorphism $\rho^Y_{i,j}:s^*o^C_i\otimes \det \cY_{i,j} \otimes t^*\det^* D_j \to t^*o^D_j$ such the the following holds.
		\begin{enumerate}
			\item Under the identification $\det \cY_{i,j}|_{\cF_{i,j}} = \det \cF_{i,j}$ induced by $\underline{\R}_{out} \oplus T\cF_{i,j} = T\cY_{i,j}|_{\cF_{i,j}}$, we have $\rho^Y_{i,j}|_{\cF_{i,j}}= \rho^F_{i,j}$. Under the identification $\det \cY_{i,j}|_{\cH_{i,j}} = \det \cH_{i,j}$ induced by $\underline{\R}_{in} \oplus T\cH_{i,j} = T\cY_{i,j}|_{\cH_{i,j}}$, we have $\rho^Y_{i,j}|_{\cH_{i,j}}= \rho^H_{i,j}$. 
			\item  We have the following two commutative diagrams over $\cM^C_{i,j}\times_j\cY_{j,k}\subset \partial \cY_{i,k}$ and $\cY_{i,j}\times_j \cM^D_{j,k}\subset \partial \cY_{i,k}$ respectively, 
			$$
			\resizebox{14cm}{!}{
			\xymatrix{
				s^*o^C_i \otimes \det \cM^C_{i,j} \otimes t^*\det^* C_j \otimes \det \cY_{j,k} \otimes t^*\det^* D_k \ar[r]^{\qquad \qquad \quad \rho^C_{i,j}} \ar[d]^f & s^*o^C_j\otimes \det \cY_{j,k} \otimes t^* \det^* D_k \ar[r]^{\qquad \qquad \quad\rho^Y_{j,k}} & t^*o^D_k\ar[d] \\
				s^*o^C_i \otimes \cY_{i,k} \otimes t^*\det^* D_k \ar[rr]^{(-1)^{c_{j}}\rho^Y_{i,k}} & & t^*o^D_k,
			}
			}
			$$
			$$
			\resizebox{14cm}{!}{
			\xymatrix{
				s^*o^C_i \otimes \det \cY_{i,j} \otimes t^*\det^* D_j \otimes \det \cM^D_{j,k} \otimes t^*\det^* D_k \ar[r]^{\qquad \qquad \quad\rho^Y_{i,j}} \ar[d]^f & s^*o^D_j\otimes \det \cM^D_{j,k} \otimes t^* \det^* D_k \ar[r]^{\qquad \qquad \quad\rho^D_{j,k}} & t^*o^D_k\ar[d] \\
				s^*o^C_i \otimes \cY_{i,k} \otimes t^*\det^* D_k \ar[rr]^{(-1)^{y_{i,k}+1}\rho^Y_{i,k}} & & t^*o^D_k,
			}}
			$$	
			where the map $f$ in the first columns of both digram is defined in a similar way to \eqref{eqn:localcon}.
		\end{enumerate}
\end{definition}

Propositions below follow from  arguments similar to the proof of Proposition \ref{prop:localsys}.
\begin{proposition}\label{prop:localmor}
	Let $\cC$ and $\cD$ be two flow categories with orientation structures. Assume $\fH$ is a flow morphism with compatible orientation structures. If $\cC$ and $\cD$ are given local systems using Proposition \ref{prop:localsys}, then $\fH$ has a compatible local system. 	If $\fH$ is only a flow premorphism from $\cC$ to $\cD$ with compatible orientation structure, then $\fH$ can be given a compatible local system.
\end{proposition}

\begin{proposition}
	Let $\cC, \cD$ be two flow categories with orientation structures and $\fH, \fF$ two flow premorphism with compatible orientation structures. Assume $\sY$ is a flow morphism with compatible orientation structures. If $\cC, \cD$ are given local systems using Proposition \ref{prop:localsys} and $\fH, \fF$ are given local systems using Proposition \ref{prop:localmor}, then $\sY$ has a compatible local system. 
\end{proposition}

\subsubsection{de Rham theory with local systems} To generalize the construction of the minimal Morse-Bott cochain complex to flow categories with local systems, we first recall the de Rham theory with local systems \cite[\S 7]{bott2013differential}.  Let $M$ be manifold and $o$ a local system over $M$. The de Rham complex $\Omega^*(M,o)$ with local system $o$ is defined as sections of $\wedge T^*M \otimes_{\Z/2} o$. The usual exterior differential lifts to a differential on $\Omega^*(M,o)$ to a differential, which is still denoted by $\rd$. The associated cohomology is denoted by $H^*(M,o)$. The wedge product defines a bilinear map 
$$\Omega^*(M,o)\times \Omega^*(M,o')\to \Omega^*(M,o\otimes o'),$$ 
which induces a map on cohomology. Using local systems, the integration is well-defined for forms in $\Omega^*(M,\det M)$ without assuming $M$ is oriented. Moreover, we have the following form of Stokes's theorem
$$\int_{M} \rd \alpha = \int_{\partial M} i^*\alpha,$$
where $i:\Omega^*(M,\det M) \to \Omega^*(\partial M, \det \partial M)$ is defined by the restriction map and the isomorphism $\det M|_{\partial M} \to \det \partial M$ induced by the isomorphism $\underline{\R}_{out}\oplus T\partial M = T M$. 

Let $C$ be a closed manifold with a local system $o$. Since there is a canonical isomorphism from $o^*\otimes o$ to the trivial line bundle, we have a paring \begin{equation}\label{eqn:pairlocal}
H^*(C,o^*)\times H^*(C,o\otimes \det C)\to \R
\end{equation} 
by integrating over $C$. It is a non-degenerate pairing just like the usual case.

\subsubsection{The minimal Morse-Bott cochain complex for flow categories with local systems} Let $\cC = \{C_i,\cM_{i,j}\}$ be a flow category with a local system. We define $o_i^*\boxtimes(o_i\otimes \det C_i)$ to be $\pi_1^*o_i^*\otimes \pi_2^*(o_i\otimes \det C_i)$. Since $\pi_2^*\det C_i$ is canonically isomorphic to $\det \Delta_i$ and $(o^*_i\boxtimes o_i)|_{\Delta_i} = o^*_i\otimes o_i = \underline{\R}$,  when  $\omega \in \Omega^*(C_i\times C_i, o_i^*\boxtimes(o_i\otimes \det C_i))$ restricted to the diagonal $\Delta_i$, we have $\omega|_{\Delta_i} \in \Omega^*(\Delta_i, \det \Delta_i)$. Therefore $\int_{\Delta_i} \omega$ is well-defined. In particular, $\int_{\Delta_i}$ descends to a well-defined map on $H^*(C_i\times C_i, o_i^* \boxtimes(o_i\otimes \det C_i))$. Since the pairing in \eqref{eqn:pairlocal} is non-degenerate, we have $\int_{\Delta_i}$ is represented by a class in $H^*(C_i\times C_i, (o_i\otimes \det C_i)\boxtimes o_i^*) = H^*(C_i, o_i\otimes \det C_i)\otimes H^*(C_i, o_i^*)$. 

If we choose representatives $\{\theta_{i,a}\}\subset \Omega^*(C_i,o_i\otimes \det C_i)$ of a basis of $H^*(C_i, o_i\otimes \det C_i)$ and the representatives $\{\theta^*_{i,a}\}\subset \Omega^*(C_i,o_i^*)$ of the dual basis in $H^*(C_i,o_i^*)$ in the sense that $\la \theta_{i,a}^*, \theta_{i,b}\ra = (-1)^{\dim C_i \cdot |\theta_{i,b}|}\int_{C} \theta^*_{i,a} \wedge \theta_{i,b} = \delta_{ab}$.  Then $\sum_a \pi_1^* \theta_{i,a} \wedge \pi_2^* \theta_{i,a}^*$ represents $\int_{\Delta_i}$ by the same proof as Proposition \ref{prop:cohomologus}. On the other hand, there is a natural isomorphism $\pi_1^*\det C_i\otimes \pi_2^*\det C_i\simeq \det \Delta_i \otimes \det N_i$ over the diagonal $\Delta_i$, induced by the isomorphism $TC_i\oplus TC_i = T\Delta_i \oplus N_i$. Using the natural identification $\pi_2^* \det C_i\to \det \Delta_i$, there is an induced isomorphism $\pi_1^* \det C_i\to N_i$. A similar isomorphism was already used in the definition of \eqref{eqn:localcon}. Using this isomorphism, if a form in $\Omega^*(C_i\times C_i, (o_i \otimes \det C_i)\boxtimes o^*_i)$ is supported in the tubular neighborhood of $\Delta_i$, then it can be viewed as a form in  $\Omega^*(N_i, \det N_i)$. Using the twisted Thom isomorphism in \cite{sklyarenko2006thom}, we get another representative of $\int_{\Delta_i}$ by the Thom classes $\delta_i^n$. As a consequence we find primitives $f_i^n\in \Omega^*(C_i\times C_i,(o_i\otimes\det C_i)\boxtimes o^*_i)$, such that 
$$\rd f_i^n =\delta_i^n -\sum_a \pi_1^*\theta_{i,a}\wedge \pi_2^*\theta^*_{i,a},$$
with a relation similar to \eqref{eqn:diff} holds. Similarly to Definition \ref{def:def}, such choices are referred to as \textbf{defining data}.

Given defining data on a flow category with a local system, we define the minimal Morse-Bott chain complex to be
\begin{equation}\label{BC}
\BC(\cC):=\varinjlim_{q\to -\infty}\prod_{j=q}^\infty H^*(C_j,o^*_j)=\varinjlim_{q\to -\infty}\prod_{j=q}^\infty H^*(C_j,o_j)\text{ (since } o_i\simeq o_i^*\text{, but not canonically.)}
\end{equation}
Next, we explain how the formula \eqref{def} for $d_k$  in the construction of the minimal Morse-Bott cochain complex still makes sense in the setting of local systems. Let $\alpha \in \Omega^*(C_v,o^*_v), \gamma \in \Omega^*(C_{v+k},o_{v+k}\otimes \det C_{v+k})$, then $s^*\alpha\wedge t^*\gamma\in \Omega^*(\cM_{v,v+k},s^*o^*_v\otimes t^*o_{v+k}\otimes t^*\det C_{v+k})$. By Definition \ref{def:local}, we have an isomorphism
$$\rho_{v,v+k}: s^*o_v\otimes \det \cM_{v,v+k}\otimes t^*\textstyle{\det^*} C_{v+k} \to t^*o_{v+k},$$
which induces an isomorphism
\begin{equation}\label{eqn:identification}
\det \cM_{v,v+k} \to s^*o_v^* \otimes t^*o_{v+k} \otimes t^*\det C_{v+k}.
\end{equation}
Let $\psi_{v,v+k}$ denote the inverse of \eqref{eqn:identification} with an extra negative sign. The extra negative sign is to match up to the negative sign in the proof of Proposition \ref{orientation1}. Using $\psi_{v,v+k}$, we can view $s^*\alpha\wedge t^*\gamma$ as in $\Omega^*(\cM_{v,v+k},\det \cM_{v,v+k})$, hence the integration $\int_{\cM_{v,v+k}}s^*\alpha\wedge t^*\gamma$ is well-defined.

Next, we consider $\cM^{v,k}_{i}[\alpha, f^n_{v+i}, \gamma]$. Then $s^*\alpha \wedge (t\times s)^* f^n_{v+i} \wedge t^*\gamma$ is a form in $\Omega^*(\cM_{v,v+i}\times \cM_{v+i,v+k}, s^*o^*_v\otimes (t\times s)^*((o_{v+i} \otimes \det C_{v+i})\boxtimes o^*_{v+i})\otimes t^*(o_{v+k}\otimes \det C_{v+k}))$. 
Since we have
\begin{align*}
s^*o^*_v\otimes (t\times s)^*((o_{v+i}\otimes & \det C_{v+i})\boxtimes o^*_{v+i})\otimes t^*(o_{v+k} \otimes \det C_{v+k})=\\
& \left(s^*o^*_v \otimes t^*(o_{v+i}\otimes \det C_{v+i}) \right)\boxtimes \left(s^*o^*_{v+i}\otimes t^*(o_{v+k}\otimes \det C_{v+k})\right).
\end{align*}
Then using $\psi_{v,v+i}$ and $\psi_{v+i,v+k}$, we get a bundle isomorphism
\begin{align*}
    s^*o^*_v\otimes (t\times s)^*((o_{v+i} \otimes \det C_{v+i})\boxtimes o^*_{v+i})\otimes t^*(o_{v+k}\otimes \det C_{v+k}) & \to \det \cM_{v,v+i}\boxtimes \det \cM_{v+i,v+k} \\
    & \to \det (\cM_{v,v+i}\times \cM_{v+i,v+k}).
\end{align*}
Thus $\cM^{v,k}_{i}[\alpha, f^n_{v+i}, \gamma]$ is defined. Similarly,  $\cM^{v,k}_{i_1,\ldots,i_r}[\alpha, f_{v+i_1}^n,\ldots, f_{v+i_r}^n, \gamma]$ makes sense in the local system setting.  Thus the differential $d_{\BC}=\prod d_k$ is well-defined and $d^2_{\BC}=0$ by the same proof as Theorem \ref{thm:main}. 

\begin{theorem}
	Let $\cC$ be a flow category with a local system, then $(\BC(\cC),d_{\BC})$ is cochain complex for any defining data and the cohomology is independent of of defining data. 
\end{theorem}
Similarly, we have analogues of Theorem \ref{Mor}, \ref{CompMor2}, \ref{thm:homotopy}, \ref{canonicalthm1} and \ref{canonicalthm2} in the setting of local systems by the same argument.

%% file: s6.tex
\section{Generalizations}\label{s6}
In this section, we give two generalizations of the minimal Morse-Bott cochain complex. The first generalization is dropping the compactness assumption on the $C_i$ in flow categories. The second generalization extracts abstract properties used in the construction of the minimal Morse-Bott cochain complex and provides more flexibility in choosing the ``homological perturbation" data. Such generalization leads to Gysin exact sequences for flow categories.
\subsection{Proper flow categories}\label{ss:prop}
We first generalize to the case that $C_i$ is not compact. However, we can not work with every noncompact manifold. Hence we introduce the following.

\begin{definition}\label{def:finite}
	A manifold $C$ is called of \textbf{finite type} iff $C$ is the interior of a compact manifold $M$ with boundary.  
\end{definition}
In particular, any closed manifold is of finite type. An infinite genus surface is not of finite type. For any manifold $C$ of finite type, $H^*(C)$ is a finite dimensional vector space.

\begin{definition}\label{def:prop}
	A \textbf{proper flow category} is defined similarly to Definition \ref{def:flow} except for the following two differences.
	\begin{enumerate}
	\item $C_i$ is manifold, such that each connected component of $C_i$ is of finite type. 
	\item $\cM_{i,j}$ is not assumed to be compact. However, the target map $t_{i,j}:\cM_{i,j}\to C_j$ is proper.
	\end{enumerate} 
\end{definition}
To explain the generalization of the minimal Morse-Bott cochain complex to proper flow categories, we first explain the counterpart of the perturbation data. Although the following discussion does not require a coherent orientation as explained in \S \ref{s5}. We assume $\{C_i,\cM_{i,j}\}$ is equipped with a coherent orientation for simplicity, in particular $C_i$ is oriented. Let $C$ be an oriented manifold of finite type and $\Omega_c^*(C)$ denote the space of compactly supported differential forms on $C$. Then we have a  bilinear form:
$$\Omega^*(C)\times \Omega^*_c(C)\to \R, \quad (\alpha,\beta) \mapsto \la \alpha,\beta \ra:= (-1)^{\dim C\cdot |\beta|}\int_{C} \alpha \wedge \beta,$$
and Lefschetz duality asserts that the bilinear form is non-degenerate on cohomology. 

\begin{definition}
	Let $C$ be an oriented manifold of finite type. We define $\Omega^*_{c,\cdot}(C\times C)$ to be 
	$$\{\alpha \in \Omega^*_{c,\cdot}(C\times C)| \supp(\alpha) \subset K\times C \text{ for some compact set } K \}.$$
	Similarly, we define $\Omega^*_{\cdot,c}(C\times C)$ to be 
	$$\{\alpha \in \Omega^*_{\cdot,c}(C\times C)| \supp(\alpha) \subset C \times K \text{ for some compact set } K \}.$$
\end{definition}
$\Omega^*_{c,\cdot}(C\times C)$ and $\Omega^*_{\cdot,c}(C\times C)$ are both cochain complexes using the usual exterior differential. Moreover, we have $H^*_{c,\cdot}(C\times C):=H^*(\Omega^*_{c,\cdot}(C\times C),\rd) = H^*_c(C)\otimes H^*(C)$ and  $H^*_{\cdot,c}(C\times C):=H^*(\Omega^*_{\cdot,c}(C\times C),\rd) = H^*(C)\otimes H^*_c(C)$, where $H^*_c(C)$ is the cohomology of compactly supported differential forms. The following proposition is an analogue of the Lefschetz duality with a similar proof to \cite[Theorem 12.15]{bott2013differential}.
\begin{proposition}\label{prop:nondeg}
	The bilinear form $\Omega^*_{c,\cdot}(C\times C)\times \Omega^*_{\cdot,c}(C\times C) \to \R$ defined by $(\a,\b) \mapsto \int_{C\times C} \alpha \wedge \beta$ descends to cohomology. The induced bilinear form on cohomology is non-degenerate. 
\end{proposition}
 
To explain the perturbation data for proper flow categories, we need to interpret the diagonal $\Delta\subset C\times C$ as a cohomology class and represent the cohomology class two different ways, namely one as Thom classes which approximate the Dirac current of the diagonal, the other is a trace term. Let $\alpha \in \Omega^*_{\cdot,c}(C\times C)$, then $\supp(\alpha)\cap \Delta$ is compact, hence $\int_{\Delta} \alpha$ is well-defined. Moreover, for $\alpha \in \Omega^*_{\cdot,c}(C\times C)$, we have $\int_{\Delta} \rd \alpha = 0$ by Stokes' theorem. Therefore $\Delta$ determines a linear function $[\Delta]$ on $H^*_{\cdot,c}(C\times C)$. It particular, $[\Delta]$ can be represented by a cohomology class in $H^*_{c,\cdot}(C\times C)$ by Proposition \ref{prop:nondeg}. Since $C$ is of finite type, both $H^*(C)$ and $H^*_c(C)$ are finite dimensional. Let $\{\theta_a \in\Omega^*_c(C) \}_{1\le a \le \dim H^*_c(C)}$ be representatives of a basis of $H^*_c(C)$ in $\Omega^*_c(C)$ and $\{\theta^*_a \in\Omega^*(C) \}_{1\le a \le \dim H^*(C)}$ be representatives of a basis of $H^*(C)$ in $\Omega^*(C)$, such that $\la\theta^*_a, \theta_b \ra = \delta_{ab}$. The following proposition is proven by the same argument in Proposition \ref{prop:cohomologus}.

\begin{proposition}\label{prop:trace}
	$\sum_a \pi_1^*\theta_a \wedge \pi_2^*\theta^*_a \in \Omega^*_{c,\cdot}(C\times C)$ represents $[\Delta]$, i.e.\  for any closed form $\alpha \in \Omega^*_{\cdot,c}(C \times C)$ we have
	$$\int_{C\times C} \alpha \wedge \left(\sum_a \pi_1^*\theta_a \wedge \pi_2^*\theta^*_a\right) = \int_\Delta \alpha.$$
\end{proposition}
The Dirac current $\delta$ of the diagonal $\Delta$ and any of its approximations given in \eqref{eqn:delta} are not in $\Omega^*_{c,\cdot}(C\times C)$. To overcome the problem, we need to perturb $\Delta$ to $\Delta^n$, such that $\Delta^n \subset K\times C$ for a compact set $K$ and $\Delta^n$ approximates $\Delta$ in a suitable sense. In order to do this, we write $C$ as $M \cup (0,1)\times \partial M$ for a manifold $M$ with boundary $\partial M$. We fix a smooth non-decreasing function $f:\R \to \R_+$ such that $f(x) = 0$ for $x\le 0$ and $f(x) < x$ for $x>0$. Then we define $\Delta^n \subset C\times C$ to be
$$\Delta^n:= \Delta_M \cup \Delta_{(0,1-\frac{1}{n})\times \partial M} \cup \widetilde{\Delta}^n,$$
where 
$$[1-\frac{1}{n},1)\times \partial M \times [1-\frac{1}{n},1)\times \partial M \supset \widetilde{\Delta}^n:= \left\{(1-\frac{1}{n}+f(r),x,1-\frac{1}{n}+r,x)\left|r\in[0,\frac{1}{n}),x\in \partial M \right.\right\}.$$
\begin{figure}[H]
    \centering
    	\begin{tikzpicture}
    	\draw (-1,0) to (2,0) to (2,-3);
    	\draw[dashed] (0,-2) to (2,0);
    	\draw (0,-2) to (0.5,-1.5) to [out=45, in=250] (1.6,0);
    	\node at (1.2,-0.2) {$\Delta^n$};
    	\node at (1.8,-0.6) {$\Delta$};
    	\end{tikzpicture}
    \caption{Graph of $\Delta^n$ near the boundary}
\end{figure}

\begin{proposition}
	$\int_{\Delta^n}$ defines the same map on $H^*_{\cdot,c}(C\times C)$ for all $n\in \N$ and equals to $\int_{\Delta}$.
\end{proposition}
\begin{proof}
	The claim follows from the fact that any class in $H^*_c(C)$ has a representative supported in $M \subset C = M\cup (0,1)\times \partial M$ and $\Delta_n \cap (C\times M) = \Delta \cap (C\times M)$ for all $n$.
\end{proof}

The Thom class of $\Delta^n$ constructed from \eqref{eqn:delta} gives a form $\delta^n \in \Omega^*_{c,\cdot}(C\times C)$ (in a sufficiently small tubular neighborhood of $\Delta^n$) representing the map $\int_{\Delta^n}=\int_{\Delta}$ through the non-degenerate bilinear form in Proposition \ref{prop:nondeg}. As a consequence of Proposition \ref{prop:nondeg} and \ref{prop:trace}, we have $\delta^n$ and $\sum_a \pi_1^*\theta_a \wedge \pi_2^*\theta^*_a$ are cohomologous in $\Omega^*_{c,\cdot}(C\times C)$, i.e.\ we can find primitives $f^n$ such that 
$$\rd f^n = \delta^n - \sum_a \pi_1^*\theta_a \wedge \pi_2^*\theta^*_a.$$
The following proposition shows that we can choose $\delta^n$ and $f^n$ carefully, such that the relation  \eqref{eqn:diff} holds asymptotically. Such result is crucial for setting up the convergence results and follows directly from the construction. 
\begin{proposition}\label{prop:nice}
	Fix a tubular neighborhood of $\Delta\subset C\times C$, then there exist Thom classes $\delta^n$ of $\Delta^n$ and primitives $f^n$ such that 
	$f^n-f^m = (\rho_n-\rho_m)\psi$ on $C\times (M\cup (0,1-\frac{2}{\min(n,m)})\times \partial M)$.
\end{proposition}
 Following the same idea in Definition \ref{def:def}, the bases $\{\theta_{i,a}\},\{\theta_{i,a}^*\}$ along with Thom classes $\delta_i^n$ and primitives $f^n_i$ in Proposition \ref{prop:nice} for each $C_i$ give defining data for a proper flow category. Next, we show the analogue of Lemma \ref{conv1} and \ref{conv2} hold for proper flow categories.
\begin{lemma}\label{lemma:conv5}
	Let $\cC$ be an oriented proper flow category. Given defining date defined above, then for every $\alpha \in \Omega^*(C_v),\gamma \in \Omega_c^*(C_{v+k})$, we have the following.
	\begin{enumerate}
		\item $\displaystyle \lim_{n\to \infty} \cM^{v,k}_{i_1,\ldots,i_r}[\alpha, f^{n}_{v+i_1},\ldots,\allowbreak f^{n}_{v+i_r}, \gamma]\in \R$ exists.
		\item $\displaystyle \lim_{n\to \infty}  \cM^{v,k}_{i_1,\ldots,i_r}[\alpha, f_{v+i_1}^{n},\ldots ,\delta^{n}_{v+i_p},\ldots, f_{v+i_r}^{n}, \gamma]
		= (-1)^* \lim_{n\to \infty}\cM^{v,k}_{i_1,\ldots,i_{p-1},\overline{i_p}, i_{p+1},\ldots, i_r}[\alpha, f_{v+i_1}^{n},\ldots, f_{v+i_r}^{n},\gamma],$
	where $*=(|\alpha|+m_{v,v+i_p})c_{v+i_p}$.
	\end{enumerate}
\end{lemma}
\begin{proof}
	Since the target map $t$ is proper, we have $t^*\gamma \in \Omega^*_c(\cM_{v+i_r,v+k})$ and $(t\times s)^*f^n_{v+i_j} \in \Omega_{c,\cdot}^*(\cM_{v+i_{j-1},v+i_j}\times \cM_{v+i_j,v+i_{j+1}})$. Therefore $s^*\alpha \wedge (t\times s)^* f^n_{v+i_1}\wedge \ldots \wedge (t\times s)^*f^n_{v+i_r}\wedge t^*\gamma \in \Omega^*_c(\cM^{v,k}_{i_1,\ldots,i_r})$. Hence $\cM^{v,k}_{i_1,\ldots,i_r}[\alpha, f^{n}_{v+i_1},\ldots,\allowbreak f^{n}_{v+i_r}, \gamma]$ makes sense. For the convergence, take $\cM^{v,k}_i[\alpha, f^n_{v+i},\gamma]$ as an example. Let $K$ denote the subset $s_{v+i,v+k}(t_{v+i,v+k}^{-1}(\supp(\gamma)))$ of $C_{v+i}$, then we only need $f^n_{v+i}$ for its value on $C_{v+i}\times K$ to determine $\cM^{v,k}_i[\alpha, f^n_{v+i},\gamma]$. By the properness, we have $K$ is compact. We write $C_{v+i} = M\cup (0,1)\times \partial M$. Therefore for $n$ big enough, we have $K\subset M\cup (0,1-\frac{2}{n})\times \partial M$. Hence for $n,m$ big enough, the difference $f^n_{v+i}-f^m_{v+i}$ on $C_{v+i}\times K$ is prescribed in Proposition \ref{prop:nice}. Hence the argument in the proof of Lemma \ref{conv1} can be applied to prove the convergence. Similarly,  we have $ \lim_{n\to \infty} \cM^{v,k}_{i_1,\ldots,i_r}[\alpha, f^{n}_{v+i_1},\ldots,\allowbreak f^{n}_{v+i_r}, \gamma]$ exists. The second claim follows from a similar argument and the proof of Lemma \ref{conv2}.   
\end{proof}

Similarly to Definition \ref{def:prop}, we have proper flow morphisms, proper flow premorphisms and proper flow homotopies if we require the target maps to be proper. With Lemma \ref{lemma:conv5}, all results in \S \ref{morsebott} can be generalized to proper flow categories with the same proof.

\subsection{Flexible defining data}\label{ss:reduction}
The following discussion works for proper flow categories with orientation structures. However, for the simplicity of notation, we will only work with oriented flow categories. Let $\cC$ be an oriented flow category. From the discussion in \S \ref{morsebott}, the essential property we need for the construction is the following relation
\begin{equation}\label{eqn:key}
\delta_i^n=\rd f_i^n+\sum_a \pi_1^*\theta_{i,a}\wedge \pi_2^*\theta_{i,a}^*.
\end{equation}
In fact, it is not necessary to construct our cochain complex from the cohomology of the critical manifolds. We only need to find differential forms $\{\theta_{i,a}\},\{\theta^*_{i,a}\}$ such that \eqref{eqn:key} holds and they are dual bases in the sense that $\la \theta^*_{i,a}, \theta_{i,b} \ra_i=\delta_{ab}$. Such generalization provides some flexibility in applications, for example one can use the generalized construction to prove Gysin exact sequences for sphere bundles over flow categories. 

\begin{definition}\label{def:reduction}
	For an oriented closed manifold $C$, a reduction of $\Omega^*(C)$ is a pair $(A,A^*)$, such that the following holds.
	\begin{enumerate}
		\item $A,A^*$ are finite dimensional subspaces of $\Omega^*(C)$ with $\dim A=\dim A^*$.
		\item  There exists a basis $\{\theta_{a}\}$ of $A$ and a basis  $\{\theta_a^*\} $ of $A^*$, such that $\langle\theta_a^*, \theta_b \rangle=\delta_{ab}$.
		 \item $\sum_a\pi_1^* \theta_a\wedge \pi_2^* \theta_a^*$ is cohomologous to the Thom class $\delta^n$.
	\end{enumerate}
\end{definition}
\begin{example}
	In the construction of the minimal Morse-Bott cochain complex on an oriented flow category, we use the reduction $A=A^*$ equals to the image of chosen quasi-embedding $H^*(C)\to \Omega^*(C)$. 
\end{example}
Using Definition \ref{def:reduction}, we can generalize defining data to the following data: (1) a reduction $(A_i,A_i^*)$ for $C_i$, (2) a family of Thom classes $\delta^n_i$ converging to the Dirac current of the diagonal $\Delta_i$ and (3) primitives $f^i_n$ such that \eqref{eqn:key} and \eqref{eqn:diff} holds. We will call it defining data with reductions. 

Let $\cC$ be an oriented flow category. Given a defining data with reductions $A$ we define

\begin{equation}\label{redcochain}
\BC(\cC,A):=\varinjlim_{j\rightarrow -\infty}\prod_{i=j}^\infty A^*_i.\end{equation}
The differential is defined as $d_A=\prod_{i=0} d_{A,i}$, where
\begin{equation} \label{redd1}
d_{A,0} \alpha :=(-1)^{|\alpha|(c_v+1)+c_v}\sum_a(\int_{C_v}\rd\alpha \wedge\theta_{v,a}) \theta_{v,a}^* = (-1)^{c_v+|\alpha|}\sum_a \la \rd \alpha, \theta_{v,a} \ra \theta_{v,a}^*
\end{equation}
with $\rd$ the normal exterior differential and $\alpha\in A^*_v$. For $k\ge 1$ and $\gamma \in A_{v+k}$
\begin{equation}\label{redd2} \left\langle d_{A,k}\alpha,\gamma\right\rangle_{v+k} =  \lim_{n\rightarrow \infty}\sum_{\substack{r\ge 0 \\0=i_0<i_1<\ldots<i_r<k}} (-1)^{\star}\mathcal{M}^{v,k}_{i_1,\ldots,i_r} [\alpha, f_{v+i_1}^n,\ldots, f_{v+i_r}^n ,\gamma] 
\end{equation}
where $\star=\sum_{j=0}^r\ddagger(\mathcal{C},\alpha,i_j)$. 
\begin{remark}
\eqref{redd1} can be viewed as the composition of $\rd$ with the projection \eqref{eqn:pi} twisted by a sign. The extra sign could be eliminated if we use a suitable change of coordinates on $A^*_v$ (i.e.\ conjugate by an automorphism of $A^*_v$). Then the sign in \eqref{redd2} would be modified accordingly. The upshot is that there is no canonical orientation and sign convention, but different conventions typically differ by a ``change of coordinates".  
\end{remark}

One important feature of our construction is that the choices we make on the critical manifolds $C_i$, i.e.\ reductions, Thom classes and primitives $f^n_i$, are independent of the structures of the flow categories, flow morphisms or flow homotopies. 

\begin{example}\label{cascades}
    Now we can (heuristically) rephrase the perturbation data for the cascades construction as a reduction. Let $\cC =\{C_i,\cM_{i,j}\}$ be an oriented flow category. We neglect the difference between differential forms and currents as well as orientations and signs for now. For a Morse-Smale pair $(f_i,g_i)$ on a critical manifold $C_i$, let $A_i:=\{[S_x]\}_{x\in \Crit(f_i)}$ and $A^*_i:=\{[U_x]\}_{x\in \Crit(f_i)}$. Then by \cite{harvey2001finite}, we have
	$$[\Delta_i]-\sum_{x\in Cr(f_i)} [S_x][U_x]=d\lim_{t\to \infty}[\bigcup_{t'<t} \graph \phi^i_{t'}].$$
	And $[U_x]$ is the dual of $[S_x]$. Thus $\{A_i,A^*_i\}$ is a reduction\footnote{Note that the ``homotopy operator" $f_i^n$ in our construction might be different from the ``homotopy operator" $[\bigcup_{0<t<\infty} \graph \phi^i_{t}]$ in the cascade construction. However, the homotopy operator in our construction is irrelevant as long as the convergence results in \S \ref{morsebott} hold.}.
	
	One should be able to modify our construction to make the argument above rigorous. In particular, we need an extension of the space of differential forms to include $[S_x],[U_x]$ as well as the homotopy operator. However, such extension will depend on $\cM_{i,j}$, which explains various transversality requirements of the gradient flows of $f_i$ with $\cM_{i,j}$ in the cascades construction. 
\end{example}

In general, a reduction for manifolds of finite type with local systems is defined as follows.
\begin{definition}\label{def:properloc}
	For a manifold $C$ of finite type with a local system $o$, a reduction is a pair $(A,A^*)$, such that the following holds.
	\begin{enumerate}
		\item $A,A^*$ are finite dimensional subspaces of $\Omega_c^*(C,o\otimes \det C)$ and $\Omega^*(C,o^*)$ respectively, such that $\dim A=\dim A^*$.
		\item  There exists a basis $\{\theta_{a}\}$ of $A$ and a basis  $\{\theta_a^*\} $ of $A^*$, such that $\langle\theta_a^*, \theta_b \rangle=\delta_{ab}$.
		\item\label{eqn:thom} $\sum_a\pi_1^* \theta_a\wedge \pi_2^* \theta_a^*$ represents the same map as $\int_{\Delta}$ on $H^*(C,o^*)\otimes H^*_c(C,o\otimes \det C)$.
	\end{enumerate}
\end{definition}

Constructions in \S \ref{morsebott}  combined with results in \S \ref{sub:localsystem} and \S \ref{ss:prop} yield the following results by the identical proofs.
\begin{theorem}\label{thm:ult}
	The following  holds.
	\begin{enumerate}
		\item Let $\cC$ be a proper flow category with local systems and $A$ defining data with reductions. Then \eqref{redcochain}, \eqref{redd1} and \eqref{redd2} define a cochain complex $(\BC(\cC,A), d_A)$, and the homotopy type of $(\BC(\cC,A),d_A)$ is independent of the defining data.
		\item  Let $\cD$ be another proper flow category with local systems,  $B$ defining data with reductions for $\cD$, and $\fH$ a proper flow morphism from $\cC$ to $\cD$ with compatible local systems. Then \eqref{mordef} defines a cochain morphism  $\phi^H_{A,B}:(\BC(\cC,A),d_A)\to (\BC(\cC,B),d_B)$ and the homotopy type of $\phi^H_{A,B}$ is independent of the defining data.
		\item Let $\cE$ be another proper flow category with local systems, $C$ defining data with reductions for $\cE$ and $\fF$ a proper flow morphism from  $\mathcal{D}$ to $\mathcal{E}$ with compatible local systems. Assume $\fH$ and $\fF$ are composable, then $\phi^{F\circ H}_{A,C}$ and $\phi^F_{B,C}\circ \phi^H_{A,B}$ are homotopic.
		\item Let $\fH,\fF$ be two proper flow premorphisms from $\cC$ to $\cD$ with compatible local systems. Assume there exists a proper flow homotopy $\sY$ from $\fH$ to $\fF$ with compatible local systems, then $\phi^H_{A,B}$ is homotopic $\phi^F_{A,B}$.
	\end{enumerate}	
\end{theorem}
\begin{remark}
	When $\cC$ is a single manifold $C$, let $(A,A^*)$ be a reduction. Then the independence result in Theorem \ref{thm:ult} shows that the cohomology of $(A^*, d_{A,0})$ is $H^*(C,o^*)$. In particular, $\dim A = \dim A^* \ge \dim H^*(C,o^*)$. 
\end{remark}

We end this subsection by a general way of constructing a reduction (but not all reductions arise in such way).

\begin{proposition}\label{prop:reduction}
        Let $C$ be a manifold of finite type with a local system $o$, assume $A,A^*$ be two finite dimensional subspaces of $\Omega_c^*(C,o\otimes \det C)$ and $\Omega^*(C,o^*)$ respectively. If $\rd$ is closed on both $A,A^*$, the pairing $A^*\otimes A \to \R, (\alpha,\beta)\mapsto (-1)^{\dim C\cdot |\beta|}\int_{C}\alpha \wedge \beta$ is non-degenerate, and both $A\hookrightarrow \Omega^*(C,o\otimes \det C), A^*\hookrightarrow \Omega^*(C,o^*)$ induce surjection on cohomology, then $(A,A^*)$ is a reduction.
\end{proposition}
\begin{proof}
Let $\{\theta_a\}$ be a basis of $A$, and $\{\theta_a^*\}$ the dual basis under the pairing. It remains to verify \eqref{eqn:thom} of Definition \ref{def:properloc}. We first claim that $T:=\sum_{a}\pi_1^*\theta_a\wedge \pi_2^*\theta_a^*$ is closed. By our assumption that $A,A^*$ are closed under $\rd$, we have $\rd T\in \pi_1^*A \wedge \pi_2^*A^*\subset \Omega^*_{c,\cdot}(C\times C, (o\otimes \det C)\boxtimes o^*)$. Moreover, the pairing on $(\pi_1^*A\wedge \pi_2^*A^*)\otimes (\pi_1^*A^*\wedge \pi_2^*A)$ from integration is non-degenerate by the non-degeneracy of the paring on $A^*\otimes A$. Therefore to show $\rd T=0$, it is sufficient to prove that for any $\theta^*_p\in A^*,\theta_q \in A$, we have
$$\int_{C\times C} \rd T \wedge \pi_1^*\theta^*_p \wedge \pi_2^*\theta_q=0.$$
Hence we compute
\begin{eqnarray*}
\int_{C\times C} \rd T \wedge \pi_1^*\theta^*_p\wedge \pi_2^*\theta_q & = & \int_{C\times C} \left(\sum_{a} \pi_1^* \rd\theta_a \wedge \pi_2^*\theta^*_a + (-1)^{|\theta_a|}\pi_1^*\theta_a \wedge \pi_2^*\rd \theta^*_a\right) \wedge  \pi_1^*\theta^*_p\wedge \pi_2^*\theta_q \\
& = & (-1)^{|\theta_q^*|\cdot |\theta_p^*|+\dim C \cdot |\theta_q|}\int_{C} \rd \theta_q\wedge \theta_p^* + \int_{C} \rd \theta_p^*\wedge \theta_q 
\end{eqnarray*}
Since the only case where the above integration is non-zero is when $|\theta_q|+|\theta^*_p|=\dim C-1$, then the above integration is $\int_{C}\rd (\theta_p^*\wedge \theta_q)=0$. As a consequence $T$ is closed. Since $A\hookrightarrow \Omega^*(C,o\otimes \det C), A^*\hookrightarrow \Omega^*(C,o^*)$ induce surjections on cohomology, we have every class $H^*(C,o^*)\otimes H_c^*(C,o\otimes \det C)$ can be represented by an element in $\pi_1^*A\wedge\pi_2^* A$. Then by the same argument in Proposition \ref{prop:cohomologus}, we have $T$ represents the diagonal. Hence  $(A,A^*)$ is a reduction.
\end{proof}

\subsubsection{Gysin sequences}\label{ss:Gysin} Let $C$ be a manifold and $\pi:E\to C$ an oriented sphere bundle over $C$ with fiber $S^k$. Then we have an exact sequence \cite[\S 14]{bott2013differential}
$$\ldots \to H^*(C) \stackrel{\pi^*}{\to} H^*(E) \stackrel{\pi_*}{\to} H^{*-k}(C) \stackrel{\wedge e}{\to} H^{*+1}(C) \to \ldots,$$
where $e$ is the Euler class of $E$. In this part, we generalize it to the setting of flow categories. Such a construction plays an important role in proving the uniqueness of the cohomology ring of exact symplectic fillings of a flexibly fillable contact manifold in \cite{ring}.

\begin{definition}
	Let $\cC$ be an oriented flow category. A $k$-sphere bundle over $\cC$ is a functor $\pi: \cE\to\cC$, such that $\pi$ maps $E_i$ to $C_i$ and $\cM^E_{i,j}$ to $\cM^C_{i,j}$, both $\pi: E_i\to C_i$ and $\pi: \cM^E_{i,j}\to \cM^C_{i,j}$ are $k$-sphere bundles, $s^E_{i,j},t^E_{i,j}$ are bundle maps covering $s_{i,j},t_{i,j}$. A $k$-sphere bundle $\pi:\cE \to \cC$ is said to be oriented iff $\pi:E_i\to C_i$ are oriented sphere bundles, and there is an orientation on each bundle $\pi:\cM^E_{i,j}\to \cM^C_{i,j}$ with both bundle maps $s^E_{i,j},t^E_{i,j}$ preserving the orientation. 
\end{definition}

\begin{proposition}\label{prop:spherori}
	Let $\pi:\cE\to \cC$ be an oriented $k$-sphere bundle. Then $\cE$ is oriented using the following convention, 
	$$[E_i]=[C_i][S^k], \quad [\cM^E_{i,j}] = (-1)^{k}[\cM^C_{i,j}][S^k].$$
\end{proposition}
\begin{proof}
	It is proven in Proposition \ref{prod}.
\end{proof}

\begin{theorem}\label{thm:gysin}
	Let $\pi:\cE \to \cC$ be an oriented $k$-sphere bundle. There exist flow morphisms $\Pi^*:\cC \Rightarrow \cE, \Pi_*:\cE \Rightarrow \cC$ and defining data $\Theta, \Xi$ for $\cC, \cE$ respectively (where $\Theta$ is minimal but $\Xi$ is defining data with reductions), such that we have a short exact sequence,
	$$0\to	\BC(\cC,\Theta) \stackrel{\phi^{\Pi^*}}{\longrightarrow}  \BC(\cE,\Xi) \stackrel{\phi^{\Pi_*}}{\longrightarrow} \BC(\cC,\Theta) \to 0.$$ 
	Assume $\cC$ has a grading structure (Definition \ref{def:grade}), then we have a long exact sequence
	\begin{equation}\label{Gysin} 
	\ldots \to H^*(\cC) \stackrel{\pi^*}{\to} H^*(\cE) \stackrel{\pi_*}{\to} H^{*-k}(\cC) \to H^{*+1}(\cC)\to \ldots. 
	\end{equation}
	Otherwise, we have an exact triangle (without grading). 
\end{theorem}
Before giving the proof, we first explain the defining data $\Theta,\Xi$.  The defining data for $\cC$ is any minimal defining data $\Theta$. For the defining data of $\cE$, we fix an angular form $\psi_i\in \Omega^k(E_i)$ such that $\rd \psi_i=-\pi^*e_i$, where $e_i$ is the Euler class (viewed in $\Omega^{k+1}(C_i)$) of the sphere bundle $E_i\to C_i$. Since when $k$ is even the cohomology class $[e_i]$ is zero. Hence when $k$ is even, we can choose $\psi$ such that $e_i=0\in \Omega^{k+1}(C_i)$. Assume $\{\theta_{i,a}\}$ is the chosen basis in the minimal defining data $\Theta$, with $\{\theta^*_{i,a}\}$ the dual basis. Then for each $\theta_{i,a}$, there exists a unique $\eta\in \la \theta_{i,a}\ra=\la \theta^*_{i,a} \ra$, such that $[(-1)^{|\theta^*_{i,a}|+1}\theta^*_{i,a}\wedge e_i] =[\eta]$ in cohomology. In other words, we can find $\eta_{i,a}$, such that $(-1)^{|\theta^*_{i,a}|+1}\theta^*_{i,a}\wedge e_i-\rd \eta_{i,a}\in \la \theta_{i,a}\ra$. If we write $m=\dim H^*(C_i)$, then we define
$$A_i=A_i^*:=\la\pi^*\theta_{i,1},\ldots,\pi^*\theta_{i,m}, \pi^*\theta^*_{i,1}\wedge \psi_i-\pi^*\eta_{i,1},\ldots,  \pi^*\theta^*_{i,m}\wedge \psi_i-\pi^*\eta_{i,m}\ra $$
The above construction ensures that $\rd$ is closed on $A_i=A_i^*$. Since $\int_{E_i}\pi^*\theta_{i,a}\wedge (\pi^*\theta^*_{i,b}\wedge \psi_i-\pi^*\eta_{i,b})=\int_{C_i}\theta_{i,a}\wedge \theta^*_{i,b}$, for any non-zero vector $v$ in $A=A^*$, there is a vector $u\in A=A^*$ with $\la v,u\ra\ne 0$. In particular, the pairing is non-degenerate on $A\otimes A^*$. That $A\to \Omega^*(E_i)$ induces a surjection on cohomology follows from the Gysin sequence associated to the sphere bundle $E_i\to C_i$. Therefore by Proposition \ref{prop:reduction}, $(A_i,A_i^*)$ is a reduction for $E_i$. Moreover, we can choose $\eta_{i,a}$ such that the following holds.
\begin{lemma}\label{lemma:reduction}
We write $\pi^*\theta^*_{i,a}\wedge \psi_i-\pi^*\eta_{i,a}$ as  $\xi_{i,a}$. Then there exist $\{\eta_{i,a}\}$ from the construction above such that 
$\la\pi^*\theta_{i,a},\xi_{i,b} \ra_i\ne 0$ iff $a=b$ and $\la \xi_{i,a},\xi_{i,b}\ra_i =0$ for any $a,b$.
\end{lemma}
\begin{proof}
    We have some freedom in the choice of $\eta_{i,a}$, since we can modify it by an element in $\la \theta_{i,a}\ra$. The first claim is obvious by integrating the $S^k$ fiber first. The only non-trivial part is proving  $\la \xi_{i,a},\xi_{i,b}\ra_i =0$ for any $a,b$. We will proceed by induction, assume for $a,b\le N<\dim H^*(C_i)$, we have  $\la \xi_{i,a},\xi_{i,b}\ra_i =0$. Then we can find $\xi_{i,N+1}$ such that $\la \xi_{i,a},\xi_{i,N+1}\ra_i =0$ for any $a\le N+1$. We first take any $\overline{\xi}_{i,N+1}$ in the form of $\pi^*\theta^*_{i,N+1}\wedge \psi_i-\pi^*\overline{\eta}_{i,N+1}\in A$ from the construction above. Then we define
	$$\xi_{i,N+1}:=\overline{\xi}_{i,N+1}-\sum_{a=1}^N \la\xi_{i,a}, \overline{\xi}_{i,N+1}\ra_i \pi^*\theta_{i,a}.$$
	It is straightforward to check that $\la \xi_{i,a},\xi_{i,N+1}\ra_i =0$ for any $a\le N$. Now note that by degree reasons if $\la\xi_{i,N+1},\xi_{i,N+1} \ra_i\ne 0$, we must have $|\xi_{i,N+1}|=\frac{\dim E_i}{2}$. In this case, we have 
	$$\la \xi_{i,N+1},\xi_{i,N+1} \ra_i = ((-1)^{\frac{\dim E_i}{2}+1}-1)\int_{C_i}\eta_{i,N+1}\wedge \theta^*_{i,N+1}.$$
	However, no matter what the parity of $\frac{\dim E_i}{2}$ is, we can add a multiple of $\pi^*\theta_{i,N+1}$ to $\xi_{i,N+1}$ to make sure that $\la\xi_{i,N+1},\xi_{i,N+1} \ra_i= 0$. Note that this modification does not affect the property that $\la \xi_{i,a},\xi_{i,N+1}\ra_i =0$ for any $a\le N$ as $\la \xi_{i,a},\pi^*\theta_{i,N+1}\ra_i=0$ for $a\le N$. Note the above argument also proves the induction foundation when $N=1$. Hence the claim follows.
\end{proof}

In order to obtain the proof of Theorem \ref{thm:gysin}, we need to use the following approximations of Dirac currents of diagonals and primitives $f^n$ on the sphere bundle $E_i\to C_i$.
\begin{proposition}\label{prop:split}
	Let $\pi:E \to C$ be an oriented $k$-sphere bundle over an oriented closed manifold. Let $A=A^*$ be the reduction on $\Omega^*(E)$ built from the previous discussion (in particular, we choose $\psi_i$ such that $\rd \psi_i=0$ if $k$ is even). And $T$ is the closed form in $\pi_1^*A\wedge \pi_2^*A$ representing the diagonal in the definition of reduction. Then there exists approximations $\delta^{E,n}$ of the Dirac current of the diagonal $\Delta_E$ such that the following holds.
	\begin{enumerate}
		\item There exist forms $f^{E,n}$ on $E\times E$, such that $$\rd f^{E,n} = \delta^{E,n} - T$$
		\item Lemma \ref{conv1} and \ref{conv2} hold for $f^{E,n}$, in particular, the construction in \S \ref{ss:reduction} works for $f^{E,n}$.
		\item\label{split} Let $\pi_{1,2}$ denote the projection $E\times E \to C\times C$, then $f^{E,n}$ can be written as sums of differential forms in the form of $(\pi_{1,2}^*\alpha) \wedge \beta$ with $\alpha \in \Omega^*(C\times C)$ and $\deg(\beta) \le k$, i.e.\ the fiber degree of $f^{E,n}$ is at most $k$. In other words, if $v_1,\ldots,v_{k+1}$ are $k+1$ vertical vectors in $T_p(E\times E)$ for $p\in C\times C$, then $f^{E,n}(v_1\wedge \ldots \wedge v_{k+1}\wedge \ldots )=0$.
	\end{enumerate}
\end{proposition}
\begin{proof}
	See Appendix \ref{app:kunneth}.
\end{proof}
\begin{proof}[Proof of Theorem \ref{thm:gysin}]
The defining data $\Theta,\Xi$ are explained above. Next, we explain the flow morphisms $\Pi_*,\Pi^*$. On the space level, $\Pi_*$ is the same as the identity flow morphism $\fI^E$ for $\cE$. The only difference is that the source map on $\Pi^*$ is the projection to $C_i$. Similarly, $\Pi_*$ from $\cE$ to $\cC$ on the space level is the same as the identity flow morphism $\fI^E$,  but the target map for $\Pi_*$ is the projection to $C_i$. We point out here that if the flow category $\cC$ is an actual space, i.e.\ concentrated in one level, then $\Pi^*$ and $\Pi_*$ induce $\pi^*,\pi_*$ on cohomology by definition. 
    
    With the defining data $\Theta,\Xi$, we get maps
    \begin{equation}\label{Gexact}
	0\to \BC(\cC,\Theta) \stackrel{\phi^{\Pi^*}}{\longrightarrow}  \BC(\cE,\Xi) \stackrel{\phi^{\Pi_*}}{\longrightarrow} \BC(\cC,\Theta) \to 0.
	\end{equation}
	We will show \eqref{Gexact} is a short exact sequence. Using the reduction from Lemma \ref{lemma:reduction}, the dual basis of $\{\pi^*\theta_{i,a}\}\cup \{\xi_{i,a}\}$ is  $ \{\xi_{i,a}\} \cup \{\pi^*\theta_{i,a}\}$ up to sign.  Then $\BC(\cE,\Xi)$ can be decomposed into $V_0\oplus V_1$ as a vector space, where $V_0$ is generated by $\langle \pi^*\theta_{i,a} \rangle$ and $V_1$ is generated by $\la \xi_{i,a}\ra$. Next we use approximations of the Dirac currents of the diagonal and primitives $f^n_E$ from Proposition \ref{prop:split}. By \eqref{split} of Proposition \ref{prop:split}, if $\gamma\in \la \pi^*\theta_{i,v+k}\ra $, then ${\Pi^*}^{v,k}_{i_1,\ldots,i_p|j_1,\ldots, j_q} [\alpha, f^{C,n}_{v+i_1},\ldots, f^{C,n}_{v+i_p},f^{E,n}_{v+j_1},\ldots, f^{E,n}_{v+j_q},\gamma]$ in the definition of $\Phi^{\Pi^*}$ is zero, for otherwise, we can not cover the fiber directions to get a nonzero integration, as the total fiber degree contributed by  $f^{E,n}_{v+j_1},\ldots, f^{E,n}_{v+j_q}$ is at most $kq$, but the total fiber dimension in ${\Pi^*}^{v,k}_{i_1,\ldots,i_p|j_1,\ldots, j_q}$ is $k(q+1)$. As a consequence we have $\mathrm{im} \phi^{\Pi^*}\subset V_0$.  Moreover, $\phi^{\Pi^*}$ is an isomorphism onto $V_0$, as it is identity plus a strictly upper triangle matrix similar to the proof of Theorem \ref{thm:main} using the identity flow morphism. Similarly, we have $V_0\subset \ker \phi^{\Pi_*}$ and  $\phi^{\Pi_*}|_{V_1}:V_1 \to \BC(\cC,\Theta)$ is an isomorphism. Therefore \eqref{Gexact} is a short exact sequence and the induced long exact sequence is the Gysin exact sequence \eqref{Gysin}. 
\end{proof}
\begin{remark}
   There are two cases of the Gysin exact sequence that we do not need to appeal to Proposition \ref{prop:split}.
   \begin{enumerate}
       \item When $\cC$ is a single space $C$, then the reduction of the sphere bundle $E$ can be viewed as decomposed into two copies of $H^*(C)$, which corresponds to the classical Gysin exact sequence. This is explained in Proposition \ref{prop:classGysin} below.
       \item When $\dim C_i \le 1$ for all $i$. Then  $\deg f^{E,n_i} = \dim C_i + k -1\le k$, and \eqref{split} of Proposition \ref{prop:split} holds tautologically for any defining data.
   \end{enumerate}
 These two cases are enough for the argument in \cite{ring}.
\end{remark}
By Corollary \ref{cor:dim} and \ref{cor:mordim}, we have the following.
\begin{corollary}
	If $\cC$ is a Morse flow category and $\cE$  an oriented $k$-sphere bundle over $\cC$, then the Gysin exact sequence only depends on $\cM^{E}_{i,j}$ of dimension no greater than $2k$.
\end{corollary}

The following proposition follows from direct computation.
\begin{proposition}\label{prop:classGysin}
	If $\cC$ is a single space $C$, then an oriented $k$-sphere bundle $\cE$ over $\cC$ is an oriented $k$-sphere bundle $\pi:E\to C$. Then the Gysin exact sequence in Theorem \ref{thm:gysin} is the classical Gysin exact sequence
	$$\xymatrix@C=3em{\ldots \ar[r] & H^i(C)\ar[r]^{\pi^*} & H^i(E) \ar[r]^{\pi_*} &  H^{i-k}(C) \ar[rr]^{\wedge (-1)^{\dim C+1} e} & &  H^{i+1}(C) \ar[r] &\ldots,}$$
	where $e\in H^*(C)$ is the Euler class of $\pi:E\to C$ and $\pi_*$ is the integration along the fiber following the convention in \cite[\S 6]{bott2013differential}. 
\end{proposition}
\begin{proof}
	Let $\{\theta_1,\ldots, \theta_k\}$ and $\{\theta^*_1,\ldots, \theta_k^*\}$ be representatives of a basis and the dual basis of $H^*(C)$. Assume $\psi$ is the Thom class of $E$ such that $d\psi = -\pi^*e$, where $e$ is a closed differential form representing the Euler class. $\BC(\cC)$ is $\la \theta^*_1,\ldots, \theta_k^*\ra = \la \theta_1,\ldots, \theta_k\ra$ with zero differential. On the other hand, by the proof of Theorem \ref{thm:gysin}, $\BC(\cE)$ is the reduction $A^*=A$ in the form of 
	$$\la \pi^*\theta_1,\ldots, \pi^*\theta_k, \xi_1:=\pi^*\theta^*_1\wedge \psi-\pi^*\eta_{1},\ldots, \xi_{k}:=\pi^*\theta^*_k\wedge \psi-\pi^*\eta_{k}\ra.$$
	The differential $d_A$ on $\pi^*\theta_i$ is zero. Since \eqref{redd1}, in this case, can be equivalently expressed as for $\gamma \in A$, we have
	$$
	\la d_{A,0}\xi_i,\gamma \ra  =  (-1)^{|\xi_i|(\dim E+1)+\dim E}\int_E \pi^*((-1)^{|\theta_i^*|+1}\theta^*_i\wedge e-\rd \eta_i)\wedge \gamma
	$$
	Since $\int \rd \xi_i \wedge \pi^*\theta_j=0$, it is sufficient to compute the case when $\gamma=\xi_j$ as follows. 
	\begin{eqnarray*}
		\la d_{A,0}\xi_i,\xi_j \ra & = & (-1)^{|\xi_i|(\dim E+1)+\dim E}\int_E \pi^*((-1)^{|\theta_i^*|+1}\theta^*_i\wedge e-\rd \eta_i)\wedge (\pi^*\theta_j^*\wedge -\pi^*\eta_j)\\
		& = &  (-1)^{|\xi_i|(\dim E+1)+\dim E}\int_E \pi^*((-1)^{|\theta_i^*|+1}\theta^*_i\wedge e-\rd \eta_i)\wedge \pi^*\theta_j^*\wedge \psi
	\end{eqnarray*}
	Note that
	$$\int_E \pi^*(\rd \eta_i\wedge \theta^*_j) \wedge \psi = \int_C \rd \eta_i\wedge \theta^*_j =\int \rd (\eta_i\wedge \theta^*_j)=0. $$
	We have
	\begin{eqnarray*}
		\la d_{A,0}\xi_i,\xi_j \ra  & = & (-1)^{|\xi_i|(\dim E+1)+\dim E}\int_E \pi^*((-1)^{|\theta_i^*|+1}\theta^*_i\wedge e\wedge \theta_j^*)\wedge \psi \\
		& = &  (-1)^{|\xi_i|\dim E+\dim C+1}\int_C \theta_i^*\wedge e \wedge \theta^*_j
	\end{eqnarray*}
	On the other hand, we have
	$$\la\pi^*\theta_j,\xi_j \ra = (-1)^{|\theta_j|+|\xi_j|\dim E}.$$
	As a consequence, we have
	$$
		 d_{A,0}\xi_i = \sum_j  (-1)^{|\xi_i|\dim E+\dim C +1+|\theta_j|+|\xi_j|\dim E}(\int_C \theta_i^*\wedge e \wedge \theta^*_j)\pi^*\theta_j.
	$$
	Note that to have a non-zero integration it is necessary to have $|\xi_i|+|\xi_j|+1=\dim E$, hence
	$$|\xi_i|\dim E+\dim C +1+|\theta_j|+|\xi_j|\dim E = \dim C+1+|\theta_j|=\dim C+|\xi_i| \mod 2$$
	and 
	$$ d_{A,0}\xi_i = (-1)^{\dim C+|\xi_i|}\pi^*(\sum _j (\int_C \theta^*_i\wedge e \wedge \theta_j^*) \theta_j).$$
	Since 
	$$\la \theta_j^*,  (\int_C \theta^*_i\wedge e \wedge \theta_j^*) \theta_j\ra =  (-1)^{|\theta_j^*||\theta_j|}\int_C \theta_j^*\wedge \theta^*_i\wedge e = (-1)^{|\theta_j|^2}\la \theta^*_j, \theta_i^*\wedge e \ra, $$
	we know that 
	\begin{equation}\label{eqn:connecting}
	    [(-1)^{\dim C+|\xi_i|}\sum _j (\int_C \theta^*_i\wedge e \wedge \theta_j^*) \theta_j] = [(-1)^{\dim C+1} \theta_i^*\wedge e] \in H^*(C).
	\end{equation}
	
	Next, by Theorem \ref{Mor} and similar computation as above, we have $\phi^{\Pi^*}(\theta_i) =  \pi^*\theta_i$ and $\phi^{\Pi_*}(\xi_i) = \theta^*_i$. Then the connecting map $\delta:H^{*-k}(C) \to H^{*+1}(C)$ is given by $\delta(\theta_i^*) = (-1)^{\dim C+1} \theta^*_i \wedge e$ by \eqref{eqn:connecting}.
\end{proof}
\begin{remark}
To explain the sign twist compared to \cite[\S 14]{bott2013differential}, recall from \eqref{redd1}, $d_A\xi_i$ is, roughly speaking,  $(-1)^{\dim E+|\xi_i|}\rd \xi_i$ (then project to $A$). Note that  $(-1)^{\dim E+|\xi_i|}\rd \xi_i = (-1)^{\dim C+1} \pi^*\theta_i^*\wedge e$. In other words, if we consider the Gysin exact sequence following \cite[\S 14]{bott2013differential} but with the cochain complex $(\Omega^*(E),(-1)^{\dim E+*}\rd)$, then we get the long exact sequence with sign twist in Proposition \ref{prop:classGysin}.
\end{remark}

Next, we consider the functoriality of Gysin exact sequences.
\begin{definition}\label{def:flowbundle}
	Let $\cC, \cD$ be two oriented flow categories and $\pi_E:\cE \to \cC, \pi_F:\cF \to \cD$ be two oriented $k$-sphere bundles. Assume $\fH:\cC \Rightarrow \cD$ is an oriented flow morphism. A compatible $k$-sphere bundle $\fT$ over $\fH$ is a flow morphism (not oriented a priori) from $\cE$ to $\cF$, such that $\cT_{i,j}$ is a $S^k$ bundle over $\cH_{i,j}$ and $s^T,t^T$ are bundle maps covering $s^H,t^H$. It is oriented if the sphere bundles $\cT_{i,j}\to \cH_{i,j}$ are oriented and $s^T,t^T$ preserve the orientation.
\end{definition}

Similar to Proposition \ref{prop:spherori}, we have that if the $k$-sphere bundle $\fT$ over $\fH$ is oriented, then $\cT$ is an oriented flow morphism from $\cE$ to $\cF$ with orientation $[\cT_{i,j}]=[\cH_{i,j}][S^k]$.

\begin{proposition}
		Let $\cC, \cD$ be two oriented flow categories and $\pi_E:\cE \to \cC, \pi_F:\cF \to \cD$ be two oriented $k$-sphere bundles. Assume $\fH:\cC \Rightarrow \cD$ is an oriented flow morphism and $\fT$ is a compatible oriented $k$-sphere bundle over $\fH$. Then we have a morphism between the Gysin exact sequences below assuming they have grading structures, otherwise it is a commutative diagram of exact triangles,
		$$
		\xymatrix{
			\ldots \ar[r] &  H^*(\cC) \ar[r]^{\pi^*}\ar[d]^{\phi^H} & H^*(\cE) \ar[r]^{\pi_*}\ar[d]^{\phi^T} &  H^{*-k}(\cC) \ar[r]\ar[d]^{\phi^H} & H^{*+1}(\cC) \ar[r] \ar[d]^{\phi^H} &\ldots \\
			\ldots \ar[r] &  H^*(\cD) \ar[r]^{\pi^*} & H^*(\cF) \ar[r]^{\pi_*} &  H^{*-k}(\cD) \ar[r] & H^{*+1}(\cD) \ar[r] &\ldots. 
			}
		$$
\end{proposition} 
\begin{proof}
We define $\fP$ to be flow morphism from $\cC$ to $\cF$, which on the space level is same as $\fT$, but the source map is $\pi\circ t^T_{i,j}$, where $\pi$ is the projection $E_i \to C_i$. We claim that  $\phi^{P} = \phi^{T\circ \Pi^*_E} = \phi^{\Pi^*_F\circ H}$. By the argument in Theorem \ref{canonicalthm1}, the contribution from $\fT\circ \Pi^*_E$ containing $(\Pi^*_E)_{i,j}$ for $i<j$ is zero due to the extra interval direction in $(\Pi^*_E)_{i,j}$. Then it is easy to identify $\phi^{P} = \phi^{T\circ \Pi^*_E}$ on the nose. On the other hand, the contribution from $\Pi^*_F\circ \fH$ containing $(\Pi^*_F)_{i,j}$ for $i<j$ is zero and $(\Pi^*_F)_{j,j}\times_{D_j}  \cH_{i,j}\simeq 
\cT_{i,j}$ by Definition \ref{def:flowbundle}, hence  $\phi^{P}$ can also be identified with $\phi^{\Pi^*_F\circ H}$ on the nose. Then by Theorem \ref{CompMor2}, we have $\phi^T \circ \pi^*$ is homotopic to $\pi^*\circ \phi^H$. Similarly, we have $\phi^H\circ \pi_*$ is homotopic to $\pi_*\circ \phi^T$.  By the same argument in Theorem \ref{thm:gysin} using the special defining data in Proposition \ref{prop:split}, the homotopies above and $\phi^T$ satisfy the conditions of Lemma \ref{lemma:commutative}, hence the claim follows.
\end{proof}

%% file: s7.tex
\section{Equivariant Theory}\label{equi}
The aim of this section is to construct an equivariant theory for a flow category with a smooth group action. Our method is based on the approximation of the homotopy quotient. In the context of Floer theory, a construction in this spirit can be found in \cite{bourgeois2013gysin}.  All the results in this section, namely Theorem \ref{s6-1} and \ref{s6-2}, can be generalized to proper flow categories with local systems. However, for simplicity, we only consider oriented flow categories in the following.
\subsection{parameterized cohomology}
Similar to the construction of parameterized symplectic homology in \cite{bourgeois2013gysin}, we need the parameterized cohomology of an oriented flow category, i.e.\ we need to  take the product of a flow category $\cC$ with a closed oriented manifold $B$.   Since taking a product with $B$ automatically falls into the Morse-Bott case, using the theory developed in previous sections, we have a direct, also geometric construction. We will see that all we have to do are some orientation checks.
 
Let $\cC=\{C_i,\cM_{i,j}\}$ be an oriented flow category and $B$ an oriented compact manifold throughout this section. We construct the product flow category $\cC\times B$ first. The parameterized cohomology is defined to be the cohomology of $\cC\times B$.  Each map $f: B_1\to B_2$ induces an oriented flow morphism $\fH(f): \cC\times B_2 \Rightarrow \cC\times B_1$. Similarly, a homotopy induces a flow homotopy. The main result of this subsection is that, after taking the minimal Morse-Bott cochain complex,  we have a contravariant functor by this product construction.
\begin{theorem}\label{s6-1}
	Let $\cC$ be an oriented flow category, then we have  a contravariant functor $\cC\times$
	$$\cC\times:\cK(\cM an)\to \cK(\cC h),$$
	where $\cK(\cM an)$ is the category whose objects are closed oriented manifolds and morphisms are homotopy classes of smooth maps.
\end{theorem} 

\subsubsection{Product flow categories}\label{ss:product}
The first step towards the construction of the functor $\cC\times$ is to construct the functor on objects, i.e.\ the product flow categories.
\begin{definition/proposition}\label{prod}
	If we orient $C_i\times B,\cM_{i,j}\times B$ by $[C_i\times B]=[C_i][B]$ and $[\cM_{i,j}\times B]=(-1)^{\dim B}[\cM_{i,j}][B]$. Then
	$\cC\times B=\{C_i\times B, \cM_{i,j}\times B\}$ is an oriented flow category.
\end{definition/proposition}
\begin{remark}
	The reason we oriented $\cM_{i,j}\times B$ by $(-1)^{\dim B}[\cM_{i,j}][B]$ is that in Definition \ref{os} and Definition \ref{orientation1} we modulo out the $\R$ translation from the right in the construction of coherent orientations in applications which motivate those definitions.
\end{remark}

\begin{definition}
	Let $E_1\to M_1, E_2\to M_2 $ be two vector bundles, $E_1\boxplus E_2$ is defined to $\pi_1^*E_1\oplus \pi_2^*E_2$ over $M_1\times M_2$, where $\pi_1,\pi_2: M_1\times M_2\to M_1, M_2$ are the projections.
\end{definition}
\begin{proof}[Proof of Proposition \ref{prod}]
It is clear that we only need to verify that $\cC\times B$ satisfies the orientation property in Definition \ref{oridef}. Note that
	$$\partial [\cM_{i,k}\times B]=\sum_j (-1)^{\dim B+m_{i,j}}[\cM_{i,j}\times_j\cM_{j,k}][B].$$	
Let $N_B$ be the normal bundle of $\Delta_B$ in $B\times B$, and we orient it by $[\Delta_B][N_B]=[B][B]$. Then the normal bundle of $\Delta_{C_j\times B}$ is $N_j\boxplus N_B$. If we orient $N_j\boxplus N_B$ by the product orientation, then we have $[\Delta_{C_j\times B}][N_j\boxplus N_B]=[C_j\times B][C_j\times B]$, i.e.\ $[N_j\boxplus N_B]$ satisfies our orientation convention \eqref{oricri} for $C_j\times B$.

Then we have $$\begin{array}{c>\displaystyle l} & [N_i\boxplus N_B]\partial [\cM_{i,k}\times B|_{\cM_{i,j}\times_j\cM_{j,k}\times B}]\\
= &(-1)^{\dim B+m_{i,j}}[N_i\boxplus N_B][\cM_{i,j}\times_j \cM_{j,k}][B]\\
= & (-1)^{\dim B+m_{i,j}+\dim B(m_{i,k}-1)+\dim B^2}[N_i][\cM_{i,j}\times_j\cM_{j,k}][\Delta_B][N_B]\\
=& (-1)^{\dim B+m_{i,j}+\dim B(m_{i,k}-1)+c_jm_{i,j}+\dim B^2}[\cM_{i,j}][\cM_{j,k}][B][B]\\
= & (-1)^{\dim B+m_{i,j}+\dim B (m_{i,k}-1)+c_jm_{i,j}+\dim B^2+\dim B m_{j,k}}[\cM_{i,j}\times B][\cM_{j,k}\times B].\end{array}$$
Because $\dim B+m_{i,j}+\dim B (m_{i,k}-1)+c_jm_{i,j}+\dim B^2+\dim B m_{j,k}=\dim B+m_{i,j}+(m_{i,j}+\dim B)(c_j+\dim B) \mod 2$, by Definition \ref{oridef},  $\cC\times B$ is an oriented flow category.
\end{proof}

\begin{remark}\label{natural}
It is very natural to expect a K\"unneth formula for $\cC\times B$. Indeed, we have	$H(\cC\times B)\simeq H(\cC)\otimes H(B)$. Since we will not use it, we omit the proof.
\end{remark}

\subsubsection{Flow morphisms between product flow categories} The second step is to construct the functor on morphisms, i.e. for every smooth map $f:B_1\rightarrow B_2$, we want to associate it with a cochain map $\BC(\cC\times B_2)\to \BC(\cC\times B_1)$. To that end,  we  first construct a flow morphism  $\fH(f)$ from $\cC\times B_2$ to $\cC\times B_1$, which is defined similarly to the identity flow morphism of $\cC \times B_1$. Then the associated cochain map is the cochain map $\phi^{H(f)}$ defined by Theorem \ref{Mor}.

\begin{definition}\label{def:productmor}
	Let $\cC$ be an oriented flow category and $f:B_1\to B_2$ a smooth map between two closed oriented manifolds.  Then we define $\fH(f) = \{\cH^f_{i,j}\}$ as follows.
\begin{enumerate}
	\item We define $\cH^f_{i,j}=\cM_{i,j}\times [0,j-i]\times B_1$ with the product orientation when $i\le j$ and $\cH^f_{i,j}=\emptyset$ when $i>j$.
	\item The source and target maps $s,t$ are defined by
	$$
	s: \begin{array}{rcl}
	\cH^f_{i,j} & \to &C_i\times B_2,\\
	(m,t,b) &\mapsto & (s^C(m), f(b)),
	\end{array} \qquad t:  \begin{array}{rcl}
	\cH^f_{i,j} & \to & C_j\times B_1,\\
	(m,t,b) &\mapsto & (t^C(m),b),
	\end{array}
	$$
	for $m\in \cM_{i,j},t\in [0,j-i],b\in B_1$ and $s^C,t^C$ are source, target maps of $\cC$.  
	\item  For $ m\in \cM_{i,j}, n\in \cM_{j,k}, t\in [0,k-j], b_1\in B_1, b_2\in B_2$ such that $(m,n)\in \cM_{i,j}\times_j\cM_{j,k}$ and $f(b_1)=b_2$, we define
	$$
	m_L: 
	\begin{array}{rcl}
	(\cM_{i,j}\times B_2)\times_j \cH^f_{j,k} & \to & \cH^f_{i,k},\\ (m,b_2,n,t,b_1)  &\mapsto & (m,n,t+j-i,b_1).
	\end{array}
	$$
	\item For $m\in \cM_{i,j}, n\in \cM_{j,k}, t\in [0,j-i], b_1\in B_1$ such that $(m,n)\in \cM_{i,j}\times_k\cM_{j,k}$ and $f(b_1) = b_2$, we define
	 $$
	 m_R:
	 \begin{array}{rcl}
	 \cH^f_{i,j}\times_j (\cM_{j,k}\times B_1) & \to &  \cH^f_{i,k},\\
	  (m,t, b_1, n, b_1) & \mapsto & (m,n,t,b_1).
	 \end{array}
	 $$
\end{enumerate}
\end{definition}

\begin{proposition}\label{prodmor}
$\fH(f)$ defined in Definition \ref{def:productmor} is an oriented flow morphism from $\cC\times B_2\Rightarrow\cC\times B_1$.
\end{proposition}
\begin{proof}
   All we need to do is the orientation check, it is analogous to the proof of Lemma \ref{identitymor}.
\end{proof}

\begin{remark}
In other words, $\fH(f)$ can be viewed as the identity flow morphism on $\cC\times B_1$ with source maps twisted by $f$. In view of the  K\"unneth formula, the morphism induced by $\fH(f)$ is given by $\id \otimes f^*$ twisted by an appropriate sign. We can similarly define another flow morphism from $\cC\times B_1$ to $\cC\times B_2$ as the identity flow morphism on $\cC\times B_1$ with target maps twisted by $f$. Then the induced map on cohomology is $\id \otimes f_*$ twisted by an appropriate sign, where $f_*:H^*(B_1)\to H^{*+\dim B_2-\dim B_1}(B_2)$ is the pushforward.
\end{remark}

\subsubsection{Flow homotopies between product flow categories}
For an oriented flow category $\cC$, we now have enough ingredients to define the functor $\cC\times: \cK(\cM an)\to \cK(\cC h)$,
$$\begin{array}{lrcl}\text{On objects:} & B &\mapsto & \BC(\cC\times B),\\
\text{On morphisms:} & (B_1\stackrel{f}{\to }B_2) & \mapsto& (\BC(\cC\times B_2)\stackrel{\phi^{H(f)}}{\to} \BC(\cC \times B_1)).\end{array}$$
To finish the proof of Theorem \ref{s6-1}, we still need to show that homotopic smooth maps induces homotopic cochain maps and functoriality of  $\cC\times$.

Let $f,g:B_1\to B_2$ be two smooth maps and $D:[0,1]\times B_1\to B_2$ a homotopy between them, such that $D|_{\{0\}\times B_1}=f$ and $D_{\{1\}\times B_1}=g$. We claim there is a flow homotopy $\sY(D)$ between the $\fH(f)$ and $\fH(g)$.  
\begin{definition}\label{def:prodhomo}
 We define $\sY(D) = \{\cY^D_{i,j}\}$ as follows.
\begin{enumerate}
	\item For $i\le j$, we define $\cY^D_{i,j}=[0,1]\times \cM_{i,j}\times [0,j-i]\times B_1$ with the product orientation. For $i<j$, we define $\cY^D_{i,j}=\emptyset$.
	\item The source map $s$ is defined as
	$$s:
	\begin{array}{rcl}[0,1]\times \cM_{i,j}\times [0,j-i]\times B_1 & \to & C_i\times B_2,\\
	 (z,m,t,b) & \mapsto & (s^C(m), D_z(b)).
	\end{array}$$
	\item The target map $t$ is defined as 
	$$t:
	\begin{array}{rcl}[0,1]\times \cM_{i,j}\times [0,j-i]\times B_1 & \to &  C_i\times B_1, \\
	(z,m,t,b) & \mapsto & (t^C(m), b).
	\end{array}$$
	\item we define $\iota_{f}: \cH^f_{i,j}\stackrel{=}{\to} \{0\} \times \cM_{i,j}\times [0,j-i]\times B_1\subset \cY^D_{i,j}$.
	\item we define $\iota_{g}: \cH^g_{i,j}\stackrel{=}{\to}\{1\}\times \cM_{i,j}\times [0,j-i]\times B_1\subset \cY^D_{i,j}$.
	\item we define $$m_L: \begin{array}{rcl} (\cM_{i,j}\times B_2)\times_j ([0,1]\times \cM_{j,k}\times[0,k-j]\times B_1)&\to &[0,1]\times \cM_{i,k}\times [0,k-i]\times B_1= \cK^D_{i,k},\\
	(m,b_2,z,n,t,b_1)&\mapsto&(z,m,n,t+j-i,b_1).
	\end{array}$$
	\item we define $$m_R: \begin{array}{rcl} ([0,1]\times \cM_{i,j} \times [0,j-i]\times B_1)\times_j (\cM_{j,k}\times B_1)&\to &[0,1]\times \cM_{i,k}\times [0,k-i]\times B_1 =  \cY^D_{i,k},\\
	(z,m,t,b_1,n,b_1)&\mapsto&(z,m,n,t,b_1).
	 \end{array}$$
\end{enumerate}
\end{definition}

\begin{proposition}\label{prodhomo1}
	$\sY(D)$ in Definition \ref{def:prodhomo} is an oriented flow homotopy from  $\fH(f)$ to $\fH(g)$.
\end{proposition}
\begin{proof}
We only need to check the orientations, and it is analogous to the proof of Lemma \ref{identitymor}.
\end{proof}	

To complete the proof of Theorem \ref{s6-1}, we still have to prove the functoriality. Let $g:B_1\to B_2, f:B_2\to B_3$ be two smooth maps. It is not hard to see $\fH(f)$ and $\fH(g)$ can be composed. We claim that there is a homotopy $\sY^c$ from $\fH(f)\circ \fH(g)$ to $\fH(f\circ g)\circ \fI$, where $\fI$ is the identity flow morphism on $\cC\times B_3$.
\begin{definition}\label{def:prodcomp}
	$\sY^c = \{\cY^c_{i,j}\}$ is defined as follows. 
\begin{itemize}
	\item We define $\cY^c_{i,j}=[0,2] \times \cM_{i,j}\times [0,j-i]\times B_1$ with product orientation for $i\le j$. We define $\cY^c_{i,j} = \emptyset$ for $i<j$. 
	\item The source map $s$ is defined as 
	 $$s: \begin{array}{rcl}[0,2]\times \cM_{i,j}\times [0,j-i]\times B_1, & \to & C_i\times B_3 \\  (z,m,t,b) & \mapsto & (s^C(m), f\circ g(b)).\end{array}$$
	\item The target map $t$ is defined as  
	$$t: \begin{array}{rcl}[0,2]\times \cM_{i,j}\times [0,j-i]\times B_1 & \to & C_i\times B_1, \\
	(s,m,t,b) & \mapsto & (t^C(m), b).
	\end{array}$$
	\item Since $(\cH^{f\circ g}\circ \cI)_{i,k}=\cup_{i\le j\le k} \cI_{i,j}\times_j \cH^{f\circ g}_{j,k}$, we define $\iota_{\fH(f\circ g)\circ \fI}$ by following two cases.
	\begin{enumerate}
	    \item When $j=i$, 
	    we define $\iota_{\fH(f\circ g)\circ \fI}$ as 
	    \begin{eqnarray*}
	    \cI_{i,i}\times_i \cH^{f\circ g}_{i,k} =  (C_i\times B_3)\times_i (\cM_{i,k}\times [0,k-i]\times B_1) & \to & [0,2] \times \cM_{i,k}\times [0,k-i]\times B_1\\
	    (c,b_3,m,t,b_1) & \mapsto & (0,m,t,b_1).
	    \end{eqnarray*}
	    \item When $j>i$, we 
	    define $\iota_{\fH(f\circ g)\circ \fI}$ on $\cI_{i,j}\times_j \cH^{f\circ g}_{j,k}$ as 
	    \begin{eqnarray*}
	     (\cM_{i,j}\times [0,j-i]\times B_3)\times_j (\cM_{j,k}\times [0,j-i]\times B_1) & \to &   [0,2] \times \cM_{i,k}\times [0,k-i]\times B_1\\
	     (m,t_1,b_3,n,t_2,b_1) & \mapsto & (\frac{t_1}{j-i},m_L(m,n),t_2+j-i,b_1)
	    \end{eqnarray*}
	\end{enumerate}
	\item  For $j<k$, we have $\iota_{\fH(f)\circ \fH(g)}$ on $\cH^f_{i,j}\times_j \cH^g_{j,k}$ is defined as
	 $$\begin{array}{rcl} (\cM_{i,j}\times [0,j-i]\times B_2)\times_j (\cM_{j,k}\times [0,k-j]\times B_1)&\to & [1,2]\times \cM_{i,k}\times [0,k-i]\times B_1,\\ 
	(m,t_1,b_2,n,t_2,b_1) & \mapsto & (\frac{t_2}{k-j}+1, m, n, t_1+k-j, b_1).\end{array}$$
	When $k=j$, $\iota_{\fH(f)\circ \fH(g)}$ is defined as
	\begin{eqnarray*}
	 (\cM_{i,k}\times [0,k-i]\times B_2)\times_j (C_{k}\times B_1)&\to & [1,2]\times \cM_{i,k}\times [0,k-i]\times B_1,\\ 
	(m,t,b_2,c,b_1) & \mapsto & (2, m, t, b_1).
	\end{eqnarray*}
	\item  $m_L: \begin{array}{rcl} (\cM_{i,j}\times B_3)\times_j ([0,2]\times \cM_{j,k}\times[0,k-j]\times B_1)&\to &[1,2]\times \cM_{i,k}\times [0,k-i]\times B_1\subset \cY^c_{i,k},\\
	(m,b_3,z,n,t,b_1)&\mapsto&(\frac{z}{2}+1,(m,n),t+j-i,b_1).
	\end{array}$
	\item  $m_R: \begin{array}{rcl} ([0,2]\times \cM_{i,j} \times [0,j-i]\times B_1)\times_j (\cM_{j,k}\times B_1)&\to &[0,1]\times \cM_{i,k}\times [0,k-i]\times B_1\subset \cY^c_{i,k},\\
	 	(z,m,t,b,n,b)&\to &(\frac{z}{2},(m,n),t,b).
	 	\end{array}$
\end{itemize}
\end{definition}
\begin{proposition}\label{prodhomo2}
	$\sY^c$ in Definition \ref{def:prodcomp} is an oriented flow homotopy from $\fH(f)\circ \fH(g)$ to $\fH(f\circ g)\circ \fI$.
\end{proposition}
\begin{proof}
	The proof is analogous to proof of Lemma \ref{identitymor}.
\end{proof}
\begin{proof}[Proof of Theorem \ref{s6-1}]
	Proposition \ref{prod}, \ref{prodmor},  \ref{prodhomo1} and \ref{prodhomo2} imply Theorem \ref{s6-1}.
\end{proof}

\begin{remark}
There is a generalization of the construction above. Let $B_1\stackrel{f}{\leftarrow} B\stackrel{g}{\to} B_2$ be maps between closed oriented manifolds, then there is a flow morphism $\fH$ from $\cC\times B_2$ to $\cC\times B_1$ with $\cH_{i,j}:=\cM_{i,j}\times [0,j-i]\times B$, where the source and target map are induced by $g,f$. The homotopy type of the induced cochain map is determined by the oriented bordism group $\Omega^*_{SO}(B_1, B_2)$, which is defined as follows: An element in $\Omega^n_{SO}(B_1,B_2)$ is represented by a closed oriented $n$-manifold $M$ and two maps $f,g$ from $M$ to $B_1, B_2$. $(M,f,g)$ and $(N,f',g')$ are equivalent iff there is an oriented bordism $D$ from $M$ to $N$ and two maps $F,G$ from $D$ to $B_1, B_2$ extending $f,g,f',g'$. 
\end{remark}

\subsection{Equivariant cohomology}
The functor $\cC\times$ is not very interesting, because it is quite independent of the flow category $\cC$. However, if $\cC$ has a compact Lie group $G$ acting on it, then the Borel construction, which is just a product modulo the $G$ action, merges some information of $\cC$ into the ``homotopy quotient". Thus nontrivial phenomena may arise from such construction.  The first step towards the Borel construction is to upgrade Theorem \ref{s6-1} to the following form.
\begin{theorem}{\label{s6-2}}
	Let the compact Lie group $G$ acts on $\cC$ in an orientation-preserving way (Definition \ref{Gaction}), then there is a contravariant functor $\cC\times_G$,
	$$ \cC\times_G:\cK(\cP rin_G)\to \cK (\cC h),$$ 
	where $\cK(\cP rin_G)$ is the category whose objects are closed oriented principal $G$ bundles and morphisms are $G$-equivariant homotopy classes of $G$-equivariant maps.
\end{theorem}
Since the classifying space $EG\to BG$ can be approximated by a sequence of closed oriented $G$-bundles $E_n\to B_n$, such that $\ldots \subset E_n\subset E_{n+1}\subset \ldots$. Note that $EG\to BG$ can be understood as the ``$G$-equivariant homotopy colimit" of the diagram $ \ldots \subset E_n\subset E_{n+1}\subset \ldots$. The classical Borel construction of the equivariant cohomology \cite{guillemin2013supersymmetry} suggests that the equivariant cochain complex of a flow category should be the composition of a homotopy limit and the functor $C\times_G$ to the diagram $ \ldots \subset E_n\subset E_{n+1}\subset \ldots$. We will construct this theory in this subsection. In particular, we will show that such a construction is independent of the approximation $\{E_n\to B_n\}$. 

\subsubsection{The functor $\cC\times_G$}\label{ss:quotient}
\begin{definition}\label{Gaction}
	A $G$ action on an oriented flow category $\cC$ is left $G$ actions on $C_i$ and $\cM_{i,j}$, such that the source, target and multiplication maps are $G$-equivariant. We say the $G$-action preserves the orientation, if the $G$-actions on $C_i$ and $\cM_{i,j}$ preserve the orientations.
\end{definition}
Let $E\to B$ be an oriented $G$-bundle. Assume $G$ acts on $\cC$ in a orientation preserving manner, then $G$ acts from right on $C_i\times E$ and $\mathcal{M}_{i,j}\times E$ by $(x,e)\cdot g=(g^{-1}\cdot c,e\cdot g)$. Let $C_i\times_G E$  and $\cM_{i,j}\times_G E$ denote quotients of the $G$ action respectively. If we orient $B$, $C_i\times_G E$ and $\cM_{i,j}\times_G E$ by $[B][G]=[E]$,  $[C_i\times_G E][G]=[C_i][E]$ and $[\cM_{i,j}\times_G E][G]=(-1)^{\dim B}[\cM_{i,j}][E]$, then Proposition \ref{prod} can be generalized to the following statement by an analogous proof. 
\begin{proposition}\label{prop:equicat}
	If $G$ acts on the oriented flow category $\cC$ and preserves orientation, then $\cC \times_GE=\{C_i\times_G E, \cM_{i,j}\times_G E\}$ is an oriented flow category.
\end{proposition}
Moreover, Proposition \ref{prodmor}, \ref{prodhomo1} and \ref{prodhomo2} can be generalized to the equivariant settings.
\begin{proposition}\label{equimor}
	Assume $G$ acts on the oriented flow category $\cC$ and preserves the orientation. Let $E_1\to B_2, E_2\to B_2$ be two oriented $G$-principle bundles.
	\begin{enumerate}
		\item Let $f$ be a smooth $G$-equivariant map $E_1\to E_2$, then there is an oriented flow morphism $\fH_G(f)$ from $\cC\times_G E_2$ to $\cC\times_G E_1$. 
		\item Let $g$ be another $G$-equivariant map $E_1\to E_2$ and $D:[0,1]\times E_1 \to E_2$  an equivariant homotopy between $f$ and $g$, then there is an oriented flow homotopy $\sY_G(D)$ between $\fH_G(f)$ and $\fH_G(g)$.
		\item Let $h:E_2\to E_3$ be another equivariant map between two oriented $G$-principle bundles, then there is an oriented flow homotopy $\sY^c_G$ from $\fH_G(h)\circ \fH_G(f)$ to $\fH_G(h\circ f)\circ \fI$.
	\end{enumerate}
\end{proposition}
Then Theorem \ref{s6-2} follows form Proposition \ref{prop:equicat} and \ref{equimor}.

\subsubsection{Approximations of classifying spaces} 
\begin{definition}\label{appro}
	Let $G$ be a compact Lie group, an approximation of the classifying space $EG\to BG$ is a sequence of oriented principle $G$-bundles $E_n\to B_n$, such that $E_n\subset E_{n+1}$ equivariantly. Moreover,  for each $k\in \N$, there exists $N_k\in \N$, such that for all $n\ge N_k$, $E_n$ is $k$-connected.
\end{definition}	
Given an approximation of the classifying space, we can compute the equivariant cohomology for $G$-actions.
\begin{theorem}[\cite{guillemin2013supersymmetry}]
	Let $M$ be a compact manifold with a smooth $G$ action and $E_n\rightarrow B_n$ an approximation of the classifying space $EG\rightarrow BG$, then 
	$$\varprojlim H^*(M\times_G E_n)=H^*(M\times_G EG)=H^*_G(M).$$	
\end{theorem}
Approximations of the classifying spaces can be constructed as follows. Fix an embedding $G\subset U(m)$. By $H(n,m)$, we mean the set of $m$ orthogonal vectors in $\mathbb{C}^n$, which is a compact orientable smooth manifold.  $U(m)$ acts on it with quotient the Grassmannian $Gr(n,m)$, $\{H(n,m)\rightarrow Gr(n,m)\}$ serves as a finite dimensional approximation of the classifying principle bundle $EU(m)\rightarrow BU(m)$ as $n\rightarrow \infty$. Then $EG\rightarrow BG$ can be approximated by $H(n,m)\rightarrow H(n,m)/G$. It was checked in \cite{guillemin2013supersymmetry}, this construction is an approximation in the sense of Definition \ref{appro}

\subsubsection{Homotopy limit} Since our construction uses an approximation, we need to take limit in the end. Consider a directed system of cochain-complexes,
$$\ldots \to A_3\to A_2\to A_1\to A_0.$$
Then the limit $\varprojlim A_i$ is also a cochain complex. However, this limit is not very nice from the homotopy-theoretic point of view, i.e.\ if we change the maps in the directed system by homotopic maps, then the homotopy type of $\varprojlim A_i$ may change.  In our setting, the cochain map is constructed only up to homotopy (\S \ref{canonical}), thus we need to apply a better limit called the homotopy limit, whose homotopy type is invariant under the replacement of homotopic maps. We recall some of the basic definitions and properties of homotopy limit form  \cite{murfet2006derived}.

Let $\N^{op}$ be the inverse directed set $\{\ldots \to 2\to 1 \to 0\}$ and $\{A_n,\mu_{nm}: A_n\to A_m\}$ an inverse system  of cochain complexes over this directed set, i.e: 
$$\ldots\stackrel{\mu_4}{\longrightarrow}A_3\stackrel{\mu_3}{\longrightarrow } A_2\stackrel{\mu_2}{\longrightarrow} A_1\stackrel{\mu_1}{\longrightarrow }A_0.$$  
Then there  is a map $v:\prod  A_i\to \prod A_i$, such that over the basis $a_n\in A_n, v(a_n)=\mu_n (a_n)$. Then $\holim A_n$ is defined to the homotopy kernel of $1-v$, i.e. $\Sigma^{-1}C(1-v)$, where $C(\cdot)$ denotes the mapping cone and $\Sigma$ is shifting by $1$.\footnote{We assume everything is graded by $\Z/2$ for simplicity. If everything is ungraded, then shifting just means multiplying $-1$ to the differential, this also enters into the definition of mapping cone in the ungraded case.}  Then we have a triangle in $\cK(\cC h)$: 
\begin{equation}\label{holim}
\xymatrix{                                     
 \prod A_n \ar[rr]^{1-v} & &\prod A_n \ar[ld]^{+1}  \\
 & \ar[lu] \holim A_n &
}
\end{equation}
This construction is the infinite telescope construction, thus it  is clear the homotopy limits of any final subsets of $\N^{op}$ are homotopic to each other and changing $\mu_i$ up to homotopy does not affect the homotopy type of the homotopy limit. There is a commutative diagram in $\cK(\cC h)$,
\begin{equation}\label{diagram}\xymatrix{  \holim A_n \ar[r] &\prod A_n\\
	\varprojlim A_n\ar[u]\ar[ur] & }\end{equation}
When $\varprojlim^1 A_n=0$, i.e.\ the Mittag-Leffler condition  holds for $A_n$, then $\varprojlim A_n\to  \holim A_n$ is a quasi-isomorphism \cite[Remark 27]{murfet2006derived}. This is the reason why sometimes we can use limit instead of homotopy limit in applications e.g.\ \cite{bourgeois2013gysin}. The long exact sequence from the triangle \eqref{holim} implies we have the short exact sequence,
$$0\to \textstyle{\varprojlim^1} H^{*-1}(A_n)\to H^*(\holim A_n )\to \varprojlim H^*(A_n)\to 0.$$

\subsubsection{Equivariant cochain complexes}
Now, we are ready to define the equivariant cochain complex of a flow category with a group action. Pick an approximation $E_0\subset \ldots \subset E_i \subset \ldots $ of the classifying space, such that $E_i$ is oriented and $G$ preserves the orientation.  Then applying the functor $\cC\times_G$ to this sequence, we get an inverse system in $\cK(\cC h)$, 
$$\ldots \to \BC(\cC\times_G E_2)\to \BC(\cC\times_G E_1)\to \BC(\cC\times_G E_0).$$
\begin{definition}\label{equico}
	The equivariant cochain complex $\BC_G$ is defined to $\holim \BC(\cC\times_G E_n)$.
\end{definition}
Results in \S \ref{canonical} imply that  the homotopy type of $\BC_G$ is independent of the auxiliary defining data. To get a canonical theory, we still need to check $\BC_G$ does not depend on the choice of the approximation $E_n\to B_n$.

\subsubsection{Independence of approximations}
With another approximation $E'_n \to B'_n$ of the classifying space, we claim that we can form a new sequence of approximations containing both $E'_n\to B'_n$ and $E_n\to B_n$ as final subsets. As preparation, we state following two propositions, the proposition below is a simple application of obstruction theory.
\begin{proposition}
	Let $Y\to X$ be a  smooth fiber bundle, with fiber $F$ is $k$-connected, and $X$ is a $k$ dimensional manifold.  Then there is a cross section for $Y\to X$, and any two cross sections are homotopic. 
\end{proposition}
By this proposition,  \cite[Proposition 1.1.1.]{guillemin2013supersymmetry} can be modified into the following.
\begin{proposition}\label{principle}
	Let $E\to B$ be a $G$-principle bundle, with $E$ is $k$-connected. Then for any closed manifold $M$ with $\dim M\le  k$, the $G$-principle bundles over $M$ are classified by $[M,B]$, i.e.\ the set of homotopy classes of maps from $M$ to $B$.
\end{proposition}
Therefore by the Definition \ref{appro} and Proposition \ref{principle}, there exists $n_1\in \N$, such that there is an equivariant map $E_1\to E'_{n_1}$.  Moreover, there exists $m_1\in \N$,  such that there is an equivariant map $E'_{n_1}\to E_{m_1}$ and the composition $E_1\to E'_{n_1}\to E_{m_1}$ is equivariantly homotopic to $E_1\subset E_{m_1}$. We can keep applying this argument to get a directed system in the equivariant homotopy category of spaces,
$$E_1\to E'_{n_1}\to E_{m_1}\to E'_{n_2}\to E_{m_2}\to \ldots,$$
which is also compatible with the two approximations $\{E_{m_i}\}$ and $\{E'_{n_i}\}$ up to equivariant homotopy. Then Theorem \ref{s6-2} implies that there is a well-defined inverse directed system in the homotopy category of cochain complexes,
\begin{equation}\label{inverse}  \ldots \to \BC(\cC\times_G E_{m_2})\to \BC(\cC\times_G E'_{n_2})\to \BC(\cC\times_G E_{m_1})\to \BC(\cC\times_G E'_{n_1})\to \BC(\cC\times_G E_1). \end{equation}
Let $H$ denote the homotopy limit of \eqref{inverse}. Since both $\BC(\cC\times_G E'_{n_i})$ and $\BC(\cC\times_G E_{m_i})$ are final in the inverse directed systems above, thus we have
$$\holim  \BC(\cC\times_G E'_{n})=\holim \BC(\cC\times_G E'_{n_i})=H=\holim \BC(\cC\times_G E_{m_i})=\holim \BC(\cC\times_G E_{m}).$$
Therefore the homotopy type of $\BC_G$ is independent of the approximation, i.e.\ we have the following theorem.
\begin{theorem}
	Let $\cC$ be an oriented flow category. Assume the compact Lie group $G$ acts on $\cC$ and preserves the orientation. Then the homotopy type of the equivariant cochain complex $\BC_G$ in Definition \ref{equico} is well-defined, i.e.\ independent of all the choices, in particular, the choice of the approximation $\{E_n\to B_n \}$. 
\end{theorem}	

\subsubsection{Spectral sequences}
From \eqref{holim}, the homotopy limit is the shifted mapping cone of $1-v$. Thus the action spectral sequence in Proposition \ref{ss}  on $\BC(\cC\times_G E_n)$ induces spectral sequence on the homotopy limit. In particular, we need to apply the following result.
\begin{proposition}[{\cite[Exercise 5.4.4]{weibel1995introduction}}]\label{sscone}
	Let $f: B\to C$ be a map of filtered cochain complexes. For a fixed integer $r\ge 0$, there is a filtration on the mapping cone $C(f)$ defined by
	$$F_pC(f):=F_{p+r}B_{n+1}\oplus F_pC_n.$$
	Then the $r$th page $E_r(C(f))$ of the induced spectral sequence is the mapping cone of the map on the $r$th page $f^r:E_r(B)\to E_r(C)$.
\end{proposition}
By Proposition \ref{sscone}, let $r=1$, there is a spectral sequence for $\BC_G$ induced from the action filtration on $\Pi BC^{\cC\times_G E_n}$. Since $E^p_1(\Pi \BC(\cC\times_G E_n))=\Pi H^*(C_p\times_G E_n)$ with the differential coming from the $d_1$ term in \eqref{def} for each $\cC\times_G E_n$.  By Proposition \ref{sscone}, $E_1(\BC_G)$ is the (shifted) mapping cone of the cochain morphism $$1-v: \prod_n    \varinjlim_{q\to -\infty} \prod_{p=q}^\infty H^*(C_p\times_G E_n)  \to  \prod_n  \varinjlim_{q\to -\infty} \prod_{p=q}^\infty H^*(C_p\times_G E_n).$$ 
Since $\varprojlim^1 H^*(C_p\times_G E_n)=0$, i.e.\ the Mittag-Leffler condition holds for inverse system $$\ldots H^*(C_p\times_G E_n)\to H^*(C_p\times_G E_{n-1}) \ldots.$$
Thus the natural map \eqref{diagram} 
$$\varinjlim_{q\to-\infty} \prod_{p=q}^\infty H^*_G(C_p)=\varprojlim_n  \varinjlim_{q\to -\infty} \prod_{p=q}^\infty H^*(C_p\times_G E_n)\to E_1(\BC_G)$$ is a quasi-isomorphism.  The induced differential $d^G_1$ on $\varinjlim_{q} \prod_{p=q}^\infty H^*_G(C_p)$ is the limit of $d_1$ for $\cC\times_G E_n$. Since $d_1$ comes from the moduli spaces without boundary, i.e.\ the pullback and pushforward on cohomology, $d^G_1$ is $t_*\circ s^*: H_G^*(C_p)\to H_G^*(C_{p+1})$ up to sign, i.e.\ the pullback and pushforward on equivariant cohomology. The polyfold theoretic version of $d^G_1$ is the analog of the equivariant fundamental class in \cite{efc}.
\begin{corollary}
	There is a spectral sequence for $\BC_G$, such that $$E^p_2(\BC_G)\simeq H^*( \varinjlim_{q\to -\infty}\prod_{p=q}^\infty H_G^*(C_p), d^G_1).$$
\end{corollary}

%% file: s8.tex
\section{Basic Example: Finite Dimensional Morse-Bott Cohomology}\label{s7}
The aim of this section is to construct a flow category for the finite dimensional Morse-Bott theory. The existence of such a flow category is a folklore theorem, stated in various places, e.g.\ \cite{austin1995morse,fukaya1996floer}. The Morse version of the flow category was introduced in \cite{cohen1995floer}, and \cite{wehrheim2012smooth} provided a detailed construction for the flow category of a Morse function for metrics which are standard near critical points. In this section, we prove that there is a flow category for any Morse-Bott function if we choose a suitable metric. The local analysis in our case is just a family version of the analysis in \cite{wehrheim2012smooth}.  

In the Morse case, \cite[\S 3.4]{audin2014morse} provides an argument to reduce constructions of continuation maps and homotopies to counting gradient flow lines on some larger manifolds. Similarly, we can construct the flow morphisms and flow homotopies by looking at flow categories arising from some larger manifolds with suitable Morse-Bott functions. With all of these established, just like the Morse case, we can prove the cohomology of the flow category is independent of the Morse-Bott function. The main theorem of this chapter is the following.
\begin{theorem}\label{s7t1}
Let $f$ be a Morse-Bott function on a closed manifold $M$, then there exists a metric $g$, such that the compactified moduli spaces of (unparameterized) gradient flow lines form a flow category with an orientation structure. The cohomology of the flow category is independent of the Morse-Bott function and is equal to the regular cohomology $H^*(M,\R)$.
\end{theorem}

Let $f$ be a Morse-Bott function on $M$ throughout this section, and the critical manifolds are $C_1,\ldots,C_n$, such that $f(C_i)<f(C_j)$ iff $i<j$. We can fix a real number $\delta>0$, such that $\delta$ is strictly smaller than the absolute values of the nonzero eigenvalues of $Hess(f)$ over all critical manifolds $C_i$. 

\subsection{Fredholm property for the finite dimensional Morse-Bott theory}
Like the Morse case, the moduli spaces of \emph{parameterized} gradient flow line from $C_i$ to $C_j$ is a zero set of a Fredholm operator over some Banach space $\cB_{i,j}$. The construction of $\cB_{i,j}$ was included in the appendix of \cite{frauenfelder2004arnold} as part of the Banach manifolds of the cascades construction, we review the construction briefly.

First we fix an auxiliary metric $g_0$ on $M$. Let $\gamma$ be a smooth curve defined over $\R$, such that \begin{equation}\label{curve1}
\lim_{t\to -\infty}\gamma(t)=x\in C_i \text{ and }   \lim_{t\to +\infty}\gamma(t)=y\in C_j,\end{equation} 
\begin{equation} \label{curve2}|\frac{\rd}{\rd t}\gamma|_{g_0}<Ce^{-\delta|t|} \text{ for } |t|\gg 0 \text{ and some constant }C. \end{equation}
Let $P(C_i,C_j)$ be the space of continuous path defined over $\R$, connecting $C_i$ and $C_j$. The Banach manifold $\cB_{i,j}$ will be a subspace of $P(C_i,C_j)$, we will first describe the neighborhood of $\gamma$ in $\cB_{i,j}$. For this purpose, we fix the following things.
\begin{enumerate}
	\item Fix a smooth function $\chi:\R\to \R$, such that $\chi(t)=|t|$ for $|t|\gg 0$. Then we can define the weighted Sobolev space $H_{\delta}^k(\R,\gamma^*TM)$ with norm  $|f|_{H^k_\delta}:=|e^{\delta \chi(t)} f|_{H^k}$, for $k\ge 1$.
	\item Fix local charts of $M$ near $x,y$, such that $C_i$ near $x$ is a radius $r$ ball in the $x_1,\ldots, x_{c_i}$ coordinates, and $C_j$ near $y$ is a radius $r$ ball in the $y_1,\ldots, y_{c_j}$ coordinates.
	\item $\rho_{\pm}(t)$ are smooth functions which are $1$ near $\pm\infty$ and $0$ near $\mp\infty$, such that \eqref{map} makes sense using the local charts above.
\end{enumerate}
There exists a positive number $K$, such that when $f\in H^k_\delta(\R, \gamma^*TM)$ with $|f|_{H^k_\delta}<K$, then $|f|$ is point-wise smaller than the injective radius of the metric $g_0$. Let $\exp$  denote the exponential map associated to the metric $g_0$. Then there is a map
\begin{equation}\label{map}
\begin{array}{rcl}
  B_K( H^k_\delta(\R, \gamma^*TM)) \times B_r(\R^{c_i})\times B_r(\R^{c_j}) & \to & P(C_i,C_j), \\
(f,x_1,\ldots,x_{c_1},y_1,\ldots, y_{c_j}) & \mapsto &  \exp_\gamma f+\sum_1^{c_i} \rho_-x_i+\sum_1^{c_j}\rho_+ y_i.  
\end{array}
\end{equation}
  $\cB_{i,j}$  consists of images of all such maps in $P(C_i,C_j)$ for all curves $\gamma $ satisfying \eqref{curve1} and \eqref{curve2}. Let $\cE_{i,j}\to \cB_{i,j}$ be the vector bundle, where the fiber over $\gamma\in \cB_{i,j}$ is $H^{k-1}_{\delta}(\R,\gamma^*TM )$, the following was proven in \cite{frauenfelder2004arnold}. 
 \begin{proposition}[\cite{frauenfelder2004arnold}]
 	$\cB_{i,j}$ is a Banach manifold.	$\cE_{i,j}\to \cB_{i,j}$ is a Banach bundle.
 \end{proposition}
Since the evaluation maps $\cB_{i,j}\to C_i\times C_j$ are submersions for all $i<j$,  the fiber products $\cB_{i,j}\times_j \ldots \times_k \cB_{k,l}$ are Banach manifolds. Moreover, $\cE_{i_0,i_1}\times_{i_1} \ldots \times_{i_{k-1}} \cE_{i_{k-1},i_k}\to \cB_{i_0,i_1}\times_{i_1} \ldots \times_{i_{k-1}} \cB_{i_{k-1},i_k}$ are Banach bundles for all $i_0<i_1<\ldots<i_k$. Given a metric $g$, then there is a section $s_{i,j}: \cB_{i,j}\to \cE_{i,j}$ defined by $s(\gamma)=\gamma'-\nabla_g f(\gamma)$. 

\begin{proposition}[{\cite{frauenfelder2004arnold}}]
	$s_{i,j}$ is a Fredholm operator with index $d_j-d_i+c_i+c_j$, where $d_i$ is the dimension of the negative eigenspace of $Hess(f)$ on $C_i$ ($d_i$ is the grading structure for our flow category).
\end{proposition}
\begin{proposition}\label{trans}
	For a generic metric $g$, $s_{i,j}$ is transverse to $0$, and the fiber products $s_{i_0,i_1}^{-1}(0)\times_{i_1} \ldots \times_{i_{k-1}} s_{i_{k-1},i_k}^{-1}(0)$ are cut out transversely for all $i_0<\ldots <i_k$.
\end{proposition}
\begin{proof}
	The proof follows from a standard Sard-Smale argument by considering the universal moduli space of all metrics, and the results for the fiber products follows from applying the Sard-Smale argument to $s_{i_0,i_1}\times_{i_1} \ldots \times_{i_{k-1}} s_{i_{k-1},i_k}:\cB_{i_0,i_1}\times_{i_1} \ldots \times_{i_{k-1}} \cB_{i_{k-1},i_k}\to \cE_{i_0,i_1}\times_{i_1}\ldots\times_{i_{k-1}}\cE_{i_{k-1},i_k}$. 
\end{proof}
We call such a pair $(f,g)$ a \textbf{Morse-Bott-Smale pair} (weaker than the Morse-Bott-Smale condition in Remark \ref{rmk:MBS}). Let $M_{i,j}$ denote $s^{-1}_{i,j}(0)/\R$, then $\cM_{i,j}:=\cup_{i<i_1<\ldots<i_k<j} \allowbreak M_{i,i_1}\times_{i_1} \ldots \times_{i_k} M_{i_k,j}$ can be made into a compact topological space. For the topology one puts on this space, it is completely analogous to the Gromov-Floer topology on the set of broken flow lines in the Morse case, see \cite{audin2014morse, wehrheim2012smooth} for details.

\subsection{Flow categories of Morse-Bott functions}
The main theorem of this section is that we can put smooth structures on $\cM_{i,j}$, such that the following holds.
\begin{theorem}\label{fdflow}
	$\{C_i,\cM_{i,j}\}$ is a flow category with an orientation structure. 
\end{theorem}
To prove this theorem, we need to equip $\mathcal{M}_{i,j}$ with a smooth structure with boundaries and corners. One strategy is using a gluing map  \cite{schwarz1993morse}, which can be generalized to Floer theories. Such method requires certain compatibility between gluing maps to guarantee a smooth structure\footnote{One condition that guarantees compatibility is the so-called ``associative gluing" \cite{wehrheim2012smooth}.}. In the context of Lagrangian Floer theory, such construction was carried out in \cite{barraud2007lagrangian}. Another method is finding a (M-)polyfold description of the moduli spaces, then the manifold structures with boundaries and corners come  from those of the ambient (M-)polyfolds, see \cite{fabert2012polyfolds,hofer2017polyfold}. In this section, we will adopt a more elementary method from \cite{audin2014morse, cohen1995floer,wehrheim2012smooth}, so that the smooth structure on the moduli spaces are inherited from some ambient manifolds. 
\begin{lemma}[\cite{nicolaescu2011invitation}]
	Let $C_i$ be a critical manifold of the Morse-Bott function $f$, then there is a tubular neighborhood of $C_i$ in $M$ diffeomorphic to the normal bundle $N$ of $C_i$. Moreover, $N$ can be decomposed into stable and unstable bundles $N^s, N^u$, and there are metrics $g^s,g^u$ on $N^s,N^u$, such that $f(v)|_N=f(C_i)-|v^s|^2_{g^s}+|v^u|^2_{g^u}$, where $v\in N$, and $v^s,v^u$ are the stable and unstable components of $v$. 
\end{lemma} 
We fix a connection on $N$, then $g^s,g^u$ can be understood as bilinear forms on $TN$. Let $g_{C_i}$ be a metric on $C_i$. If a metric $g$ near $C_i$ has the form $\pi^* g_{C_i}+g^s+g^u$, where $\pi$ is the projection $N\to C_i$, we say the metric $g$ is standard near $C_i$. In fact, we can require the Morse-Bott-Smale pair to have standard metric near all critical manifolds, as we can obtain transversality by perturbing the metric away from critical manifolds. For a standard metric, the gradient vector in $N$ is contained in the fibers of the tubular neighborhood. Therefore the local picture of the gradient flow is just a
family of the Morse flow lines in each fiber. When restricted to a fiber $F$ with coordinate $x_1,\ldots, x_s,y_1,\ldots, y_u$, the pair $(f,g)$ is standard and is in the following form,
\begin{eqnarray}
f|_F & = & -x_1^2-\ldots-x_s^2+y_1^2+\ldots+y_u^2+C, \nonumber \\
g|_F & = & \rd x_1\otimes \rd x_1+\ldots+\rd x_s\otimes \rd x_2+\rd y_1\otimes \rd y_1+\ldots+\rd y_u\otimes \rd y_u. \nonumber
\end{eqnarray}
Inside the fiber $F$, we define \begin{eqnarray*}
S^r_s&:= &\{(x_1,\ldots,x_s)|x_1^2+\ldots+x_s^2=r^2\},\\ S^r_u&:=&\{(y_1,\ldots,y_u)|y_1^2+\ldots+y_u^2=r^2\}, \\
D^r_s& :=&\{(x_1,\ldots,x_s)|x_1^2+\ldots+x_s^2<r^2\},\\  
D^r_u & := &\{(y_1,\ldots,y_u)|y_1^2+\ldots+y_u^2<r^2\}.
\end{eqnarray*}
Let $\cM$ be the moduli space of gradient flow lines and broken gradient flow lines of $(f|_F,g|_F)$ from $S^r_s\times D^r_u$ to $D^r_s\times S^r_u$, let $ev_-,ev_+$ be the two evaluation maps at the two ends defined on $\cM$, the following lemma is essentially contained in \cite{wehrheim2012smooth}.
\begin{lemma}\label{standard}
	$\Ima(ev_-\times ev_+)(\cM)\subset (S^r_s\times D^r_u) \times (D^r_s\times S^r_u)$ is a submanifold with boundary inside the fiber $F$.
\end{lemma}
\begin{proof}
	Since the gradient flow lines are $(e^{-2t}x,e^{2t} y)$, thus the image of unbroken flow lines are $(x,y,\frac{|y|}{r}x, \frac{r}{|y|}y)$, it is a submanifold in $(S^r_s\times D^r_u) \times (D^r_s\times S^r_u)$. The image of broken flow lines are $(x,0,0,y)$, it is also a submanifold in $(S^r_s\times D^r_u) \times (D^r_s\times S^r_u)$. And the boundary chart is given by $(t,x,0,0,y)\rightarrow (x,ty,tx,y)$ for $t\in [0,1)$, thus the lemma is proven. 
\end{proof}
\begin{remark}
	Lemma 4.4 of \cite{wehrheim2012smooth} composes the map $ev_-\times ev_+$ with projection $(x,y',x',y)\to (\frac{|x|'+|y'|}{2r},x,y)$ to get a homeomorphism from $\cM$ to $[0,1)\times S^r_s\times S^r_u$. This homomorphism was used in \cite{wehrheim2012smooth} to construct a smooth structure with boundaries and corners on $\cM$. Since the projection restricted to $\Ima(ev_-\times ev_+)(\mathcal{M})$ is a diffeomorphism,  we can also use the smooth structure on $\Ima(ev_-\times ev_+)(\cM)$ to make $\cM$ into a manifold with boundaries and corners.
\end{remark}

Since $S^r_s\times D^r_u$ and $D^r_s\times S^r_u$ are transverse to the gradient flow, then Lemma \ref{standard} also holds if we replace  $S^r_s\times D^r_u$ and $D^r_s\times S^r_u$ by  open sets in $f|_F^{-1}(C-\epsilon)$ and $f|_F^{-1}(C+\epsilon)$. Now we return to the Morse-Bott case with a standard metric near $C_i$. Let $\phi^t$ be the flow for $\nabla f$, then the stable manifold $S_i$ of $C_i$ is defined to be
$$S_i=\{x\in M| \lim_{t\to \infty }\phi^t(x)\in C_i\}.$$
And the unstable manifold $U_i$ is defined to be 
$$U_i=\{x\in M| \lim_{t\to -\infty}\phi^t(x)\in C_i\}.$$
Both $S_i$ and $U_i$ are equipped with smooth evaluation maps to $C_i$.  Then we have the family version of Lemma \ref{standard} as follows.
\begin{lemma}\label{lemma:family}
	Given a standard metric near $C_i$, let $N_r$ be the radius $r$ open tube of $C_i$. Suppose $\epsilon$ is a small positive real number, and $v^{\pm \epsilon}_i$ denotes $f(C_i)\pm \epsilon$. Let $\cM_{i,\epsilon,r}$ denote the moduli space of flow lines and broken flow lines from  $f^{-1}(v^{-\epsilon}_i)\cap N_r$ to $f^{-1}(v^{+\epsilon}_i)\cap N_r$. Then there exist $\epsilon,r>0$, such that image of  $ev_-\times ev_+|_{\cM_{i,\epsilon,r}}$ is a submanifold with boundary in  $(f^{-1}(v^{-\epsilon}_i)\cap N_r)\times (f^{-1}(v^{+\epsilon}_i))\cap N_r)$, and the boundary is
	$(S_i\cap f^{-1}(v^{-\epsilon}_i))\times_{C_i} (U_i\cap f^{-1}(v^{+\epsilon}_i))$.
\end{lemma}

\begin{proposition}\label{prop:smooth}
   $M_{i,j}\times_jM_{j,k}\cup M_{i,k}$ can be given a structure of a manifold with boundary. 
\end{proposition}
\begin{proof}
	Since we have diffeomorphisms
	$$M_{i,j}\simeq U_i\cap S_{j}\cap f^{-1}(v^{-\epsilon}_j),$$
	$$M_{j,k}\simeq U_{j}\cap S_{k}\cap f^{-1}(v^{+\epsilon}_j).$$
	The Morse-Bott-Smale condition implies that the intersections are transverse. On the other hand, let $M_{i,k}\cap \cM_{j,\epsilon,r}$ be the set of flow lines in $M_{i,k}$ which contains a flow line in $\cM_{j,\epsilon,r}$, then it is an open set of $M_{i,k}$, and we have embedding
	$$ev_-\times ev_+:M_{i,k}\cap \cM_{j,\epsilon,r}\to (f^{-1}(v^{-\epsilon}_j)\cap N_r)\times (f^{-1}(v^{+\epsilon}_j)\cap N_r).$$
	The image is 
	$$\Ima (ev_-\times ev_+)(M_{i,k}\cap \cM_{j,\epsilon,r})= \Ima (ev_-\times ev_+)(\partial_0\cM_{j,\epsilon,r}) \cap\left((U_i\cap  f^{-1}(v^{-\epsilon}_j))\times (S_{k}\cap f^{-1}(v^{+\epsilon}_j))\right),$$
	where $\partial_0\cM_{j,\epsilon,r}$ is the interior (depth-0 boundary, Definition \ref{def:degenercy}) of $\cM_{j,\epsilon,r}$. The Morse-Bott-Smale condition implies that the intersection is transverse. Moreover $\partial \Ima (ev_-\times ev_+)(\cM_{j,\epsilon,r})=(S_j\cap f^{-1}(v^{-\epsilon}_j))\times_{C_j} (U_j\cap f^{-1}(v^{+\epsilon}_j))$ is also transverse to $(U_i\cap  f^{-1}(v^{-\epsilon}_j))\times (S_{k}\cap f^{-1}(v^{+\epsilon}_j))$,  since fiber product $M_{i,j}\times_j M_{j,k}$ is transverse.  Thus $\Ima (ev_-\times ev_+)(M_{i,k}\cap \cM_{j,\epsilon,r})$ can be completed by the boundary structure of $\Ima (ev_-\times ev_+)(\cM_{j,\epsilon,r})$, that is we can add in $(U_i\cap S_{j}\cap f^{-1}(v^{-\epsilon}_j))\times_{C_j} ( S_{k}\cap U_j\cap f^{-1}(v^{+\epsilon}_j))\simeq M_{i,j}\times_jM_{j,k}$ as the boundary of $M_{i,k}\cap \cM_{j,\epsilon,r}$. The topology check is analogous to \cite{wehrheim2012smooth}. 
\end{proof}

Therefore we have a smooth boundary structure on $ M_{i,j}\times_j M_{j,k} \subset \cM_{i,k}$. We still need to construct corner structures near curves with multiple breaking and prove compatibility of smooth structures. The proof is very similar, and the corner structure will be inherited from (fiber) products of the manifolds with boundary in Lemma \ref{lemma:family}.

\begin{proposition}\label{prop:corner}
	$M_{i,j}\times_{j} M_{j,k}\times_{k} M_{k,l} \cup M_{i,k}\times_{k}M_{k,l}\cup M_{i,j}\times_{j}M_{j,l}\cup M_{i,l}$ can be given a structure of manifold with boundaries and corners, which is compatible with structure given in Proposition \ref{prop:smooth}. 
\end{proposition}
\begin{proof}
	Let $N_{*,r}$ denote the radius $r$ open tube around $C_*$. We use $\cM_{j,k,\epsilon,r}$ to denote the moduli space of gradient flow lines from $f^{-1}(v_j^{-\epsilon})\cap N_{j,r}$ to $f^{-1}(v_k^{+\epsilon})\cap N_{k,r}$, passing through $f^{-1}(v_j^{+\epsilon})\cap N_{j,r}$ and $f^{-1}(v_k^{-\epsilon})\cap N_{k,r}$, such that the only breaking allowed is at $C_j$ or $C_k$, or both. Then $ev_{-,+,-,+}:=ev_-\times ev_+ \times ev_-\times ev_+$ defines an embedding
	$$\cM_{j,k,\epsilon,r}\to (f^{-1}(v_j^{-\epsilon})\cap N_{j,r})\times (f^{-1}(v_j^{+\epsilon})\cap N_{j,r})\times (f^{-1}(v_k^{-\epsilon})\cap N_{k,r}) \times (f^{-1}(v_k^{+\epsilon})\cap N_{k,r}).$$
	We define $V\subset f^{-1}(v_j^{+\epsilon})\cap N_{j,r}, U \subset f^{-1}(v_k^{-\epsilon})\cap N_{k,r}$ be the sets such that the flow lines from $V$ will end in $U$ without breaking, then $V,U$ are both open subset and there is a diffeomorphism $\phi:V\to U$ defined using the gradient flow. Then $\Ima(ev_{-,+,-,+})$ is contained inside the fiber product $(f^{-1}(v_j^{-\epsilon})\cap N_{j,r})\times V\times_{\phi} U \times (f^{-1}(v_k^{+\epsilon})\cap N_{k,r})$. By a little abuse of notation, we use $V\cap \cM_{j,\epsilon,r}$ to denote $ev_+^{-1}(V)\subset \cM_{j,\epsilon,r}$ and $U\cap \cM_{k,\epsilon,r}$ to denote $ev_-^{-1}(U)\subset \cM_{k,\epsilon,r}$, which are both open subsets and inherit the structure of a manifold with boundary from Lemma \ref{lemma:family}. Then $\Ima(ev_{-,+,-,+}) = ev_{-,+}(V\cap \cM_{j,\epsilon,r})\times_{\phi} ev_{-,+}(U\cap \cM_{k,\epsilon,r})$. The Morse-Bott-Smale condition implies that the fiber product $ev_{-,+}(V\cap \cM_{j,\epsilon,r})\times_{\phi} ev_{-,+}(U\cap \cM_{k,\epsilon,r})$ is cut out transversely as a manifold with boundaries and corners. Therefore $\cM_{j,k,\epsilon,r}$ inherits the structure of a manifold with corners from its image under $\Ima ev_{-,+,-,+}$, whose depth-1 boundary is  $\left(ev_{-,+}(V\cap \partial_1\cM_{j,\epsilon,r})\times_{\phi} ev_{-,+}(U\cap \partial_0\cM_{k,\epsilon,r})\right) \cup\left( ev_{-,+}(V\cap \partial_0\cM_{j,\epsilon,r})\times_{\phi} ev_{-,+}(U\cap \partial_1\cM_{k,\epsilon,r})\right)$, and the depth-2 boundary (corner) is $ev_{-,+}(V\cap \partial_1\cM_{j,\epsilon,r})\times_{\phi} ev_{-,+}(U\cap \partial_1\cM_{k,\epsilon,r})$.
	
	We define $M_{i,l}\cap \cM_{j,k,\epsilon,r}$ to be the open subset of $M_{i,l}$ consisting of flow lines with a portion in $\cM_{j,k,\epsilon,r}$. Similar to the proof of Proposition \ref{prop:smooth}, we can use the boundary and corner structures on $\cM_{j,k,\epsilon,r}$ to give a corner structure near  $M_{i,l}\cap \cM_{j,k,\epsilon,r}$, by intersecting with the unstable and stable manifolds of $C_i,C_{l}$ with $\Ima(ev_{-,+,-+})$ inside $(f^{-1}(v_j^{-\epsilon})\cap N_{j,r})\times (f^{-1}(v_j^{+\epsilon})\cap N_{j,r})\times (f^{-1}(v_k^{-\epsilon})\cap N_{k,r}) \times (f^{-1}(v_k^{+\epsilon})\cap N_{k,r})$. More explicitly, we get a corner structure near $M_{i,j}\times_jM_{j,k}\times_k M_{k,l}$, which also gives a boundary structure near $M_{i,j}\times_{j} (M_{j,l}\cap (U\cap \partial_0\cM_{k,\epsilon,r}))$ and $(M_{i,k}\cap (V\cap \partial_0\cM_{j,\epsilon,r})) \times_{k} M_{k,l}$. Moreover, the boundary structure is exactly the one constructed in Proposition \ref{prop:smooth}. This finishes the proof.
\end{proof}

\begin{proof}[Proof of Theorem \ref{fdflow}]
	Following the same proof of Proposition \ref{prop:corner}, we can prove that  $\cM_{i,j}$ is endowed with a structure of compact manifold with boundaries and corners. Let $o_i$ be the determinant line bundle of the stable bundle $N^s$ over $C_i$, then $\{C_i,\cM_{i,j}\}$ defines a flow category $\cC_{f,g}$ with an orientation structure following the construction in \S \ref{ss:ori_MB}. 
\end{proof}

\subsection{Morphisms and homotopies}\label{sub:mor}
To derive the flow morphisms between different Morse-Bott functions and flow homotopies between them, we will use the argument from \cite{audin2014morse} to reduce the construction of flow morphisms and flow homotopies back to flow categories.
\subsubsection{Flow morphisms \cite[Theorem 3.4.2. first step]{audin2014morse}}
Let $(f_1,g_1)$ and $(f_2,g_2)$ be two locally standard Morse-Bott-Smale pairs, let $\cC^1=\{C^1_i,\cM^1_{i,j}\}$ and $\cC^2=\{C^2_i,\cM_{i,j}^2\}$ denote the associated flow categories. We can find a smooth function $F: \R\times M\to \R$, such that:
$$F(t,x)=\left\{\begin{array}{cr}
	 f_1(x) & t<\frac{1}{3},\\ & \\
	 f_2(x) & t>\frac{2}{3}.
\end{array}	\right. $$
We consider a Morse function $h$ on $\R$, such that it only has two critical points, one local minima at $0$, and one local maxima at $1$, and 
$$\forall x\in M, t\in (0,1), \frac{\partial F}{\partial t}+\frac{dh}{dt}>0$$
Then $F+h$ defines a Morse-Bott function on $\R\times M$, with critical manifolds $\{C_i^1\times \{0\}\}$ and $\{C_i^2\times \{1\}\}$, we can find a locally standard metric $G$ such that 
$$G(t,x)=\left\{\begin{array}{cr}
g_1+\rd t\otimes \rd t & t<\frac{1}{3},\\ & \\
g_2+\rd t\otimes \rd t & t>\frac{2}{3}.
\end{array}	\right. $$
We can assume $(F,G)$ is a locally standard Morse-Bott-Smale pair. Then by Theorem \ref{fdflow}, we can associate $(F+h, G)$ a flow category with an orientation structure. Let $\cF_{i,j}$ denote the compactified moduli space of flow lines from $C_i^1\times \{0\}$ to $C_j^2\times \{1\}$,  then $\cF_{i,j}$ form a flow morphism $\fF$ from $\cC^1$ to $\cC^2$.  When $F(t,x)=f(x)$, and we can choose metric $g+dt^2$, then $F_{i,i}=C_i$ and $F_{i,j}\simeq\cM_{i,j}\times [0,j-i]\simeq I_{i,j}$, for $i<j$, i.e.\ the construction gives the identity flow morphism \cite[Theorem 3.4.2. second step]{audin2014morse}.

\subsubsection{Flow homotopies \cite[Theorem 3.4.2. third step]{audin2014morse}}
Assume we have continuations $F, G, H$ from $f_1$ to $f_2$, $f_2$ to $f_3$ and $f_1$ to $f_3$ respectively, then we can find $K:\R_s\times \R_t\times M \to \R$, such that: 
$$K(s, t,x)=\left\{\begin{array}{cr}
H(t,x) & s<\frac{1}{3},\\ & \\
F(s,x) & t<\frac{1}{3}, \\ & \\
G(t,x) & s>\frac{2}{3}, \\ & \\
f_3(x) & t>\frac{2}{3}.
\end{array}	\right. $$
We can find $h$ with one local minima at $0$ and local maxima at $1$, such that 
$$\forall (s,t,x)\in (0,1)\times \mathbb{R}\times M, \frac{\partial K}{\partial s}+h'(s)>0,$$
$$\forall (s,t,x)\in \mathbb{R}\times (0,1)\times M, \frac{\partial K}{\partial t}+h'(t)>0.$$
Then $K+h(s)+h(t)$ defines a Morse-Bott function, with critical manifolds $\{C_i^1\times \{(0,0)\}\}$, $\{C_i^2\times \{(1,0)\}\}$, $\{C_i^3\times \{(0,1)\}\}$ and $\{C_i^3\times \{(1,1)\}\}$, and we can find a locally standard Morse-Bott-Smale metric extending the locally standard metrics used in $F,G,H$ and $f_3$. Then the flow lines from $C_i^1\times \{(0,0)\} $ to  $C_j^3\times \{(1,1)\}$ give rise to a flow homotopy between $\fG\circ \fF$ and $\fI\circ \fH$.

\begin{proof}[Proof of theorem \ref{s7t1}]
	By Theorem \ref{fdflow}, we have a flow category $\cC_{f,g}$ with an orientation structure for any locally standard Morse-Bott-Smale pair $(f,g)$. Using flow morphisms and flow homotopies above,  we can see that the cohomology of $\cC_{f,g}$ does not depend on $(f,g)$. Thus we can choose $f\equiv C$, and $g$ be any metric, then $(f,g)$ is a locally standard Morse-Bott-Smale pair. The corresponding flow category has object space and morphism space are both $M$, thus the cohomology of the flow category equals to the cohomology $H^*(M,\R)$. 
\end{proof}
Since a Morse-Smale pair is a special case of Morse-Bott-Smale pair, and our definition of the minimal Morse-Bott cochain complex recovers the Morse cochain complex when the function is Morse. As a corollary, the $\R$ coefficient Morse cohomology equals to the de Rham cohomology of $M$.

\subsection{Noncompact case}
Let $M$ be a noncompact manifold of finite type as introduced in Definition \ref{def:finite} throughout this subsection, i.e.\ $M$ is the set of interior points of a compact manifold with nonempty boundary. Let $\partial_r$ be a nonzero pointing out vector field on the collar neighborhood of the end of $M$. In the following, we will only consider the following two types of Morse-Bott functions.
\begin{enumerate}
	\item\label{c1} A Morse-Bott function $f$, such that $\partial_r f>0$ on the collar.
	\item\label{c2} Constant functions.
\end{enumerate}
In case of type \eqref{c1}, we have a flow category $\cC_f$ by Theorem \ref{fdflow}. In the case of type \eqref{c2}, the flow category is a single space $M$, which is a proper flow category. Next we will show how to associate flow morphism betweens flow categories from different Morse-Bott functions and flow homotopy between them. Once they are set up like the compact case, we have the cohomology of the flow category is independent of the Morse-Bott function. In particular, one can choose a constant, hence the cohomology is the regular cohomology.
\subsubsection{Flow morphisms and homotopies}
Given two admissible Morse-Bott functions $f_1, f_2$ on $M$, the homotopy between them  is a smooth function $F: \R \times M\to \R$, such that
$$F(t,x)=\left\{\begin{array}{cr}
f_1(x) & t<\frac{1}{3},\\ & \\
f_2(x) & t>\frac{2}{3},
\end{array}\right.  $$
and when $t\in (\frac{1}{3},\frac{2}{3})$ we have $\partial_rF(t,x) > 0$ on the collar.  Then $h+F$ defines a Morse-Bott function on $\R \times M$, and we claim the associated flow category defines a proper flow morphism from $\cC_{f_1}$ to $\cC_{f_2}$. We may assume the metric on $\R \times M$ has the property that the gradient for the collar coordinate $r\in (-1,0)$ is $\partial_r$ on the collar. Then $\partial_r F(t,x)\ge 0$ for all $t$ implies that $\partial_r F(t,x) = \partial_r (h+F(t,x))=\la \nabla r, \nabla (h+F(t,x))\ra \ge 0$. Therefore any gradient flow line from a critical point of $f_1$ to a critical point $f_2$ has the property that if it touches the collar then it stays in the collar after the touching point. In addition to the argument in \S \ref{sub:mor}, we need to show the properness of the target maps in order to prove the claim. We divide it into the following cases.
\begin{enumerate}
	\item\label{case1} Both $f_1,f_2$ are of type \eqref{c1}.  Since any gradient flow line touches the collar neighborhood can not return to the interior side. Hence construction in \S \ref{sub:mor} gives compact moduli spaces and a flow morphism from $\cC_{f_1}$ to $\cC_{f_2}$. 
	\item\label{case2} $f_1$ is of type \eqref{c2} and $f_2$ is of type \eqref{c1}. Then the same argument in case \eqref{case1} holds.  
	\item\label{case3} $f_1$ is of type \eqref{c1} and $f_2$ is of type \eqref{c2}. Let $K\subset M = \Crit(f_2)$ be a compact subset. For points outside the collar, we define $r = -1$. Let $R := \max\{r(x)|x\in K\}$. Then $R < 0$ and all gradient flow lines from critical points of $f_1$ to a point in $K$ stays inside the domain $[0,1] \times \{ r\le R\}$, hence the space of such flow lines is compact. This shows that the target maps are proper.
	\item\label{case4} Both $f_1,f_2$ are of type \eqref{c2}. Then the same argument in case \eqref{case3} holds. 
\end{enumerate} 
\begin{remark}
	If we replace the condition on the collar by $\partial_rF(t,x) < 0$, this would force $f_1,f_2$ to have the property that $\partial_r f_1, \partial_rf_2 < 0$ if they are not constant. In this case, the the gradient flow lines in $\R\times M$ will shrink on the collar neighborhood instead of expanding, hence the source map is proper and the target map is not. We can similarly define a cochain complex using the compactly supported cohomology in this case.  The cohomology of the cochain complex is the compactly supported cohomology, which is isomorphic to the homology. 
\end{remark}
The asymmetry of the flow morphism prevents us from constructing a flow morphism from $\cC_f$ to $\cC_f$. Assume $f > 0$ without loss of generality, there exists a flow morphism from $\cC_f$ to $\cC_{2f}$ constructed from $F(t,x) = \phi(t)f(x)$, where $\phi(t)$ is an increasing function with $\phi(t) = 1, t \le 0$ and $\phi(t) = 2, t \ge 1$. The flow morphism is diffeomorphic to the identity flow morphism when we use the metric $g+dt^2$. The flow homotopy follows from the same argument if we require the increasing property on the collar when constructing the homotopy of homotopy. Therefore we have the invariance of the cohomology with respect to the Morse-Bott function, i.e.\ we have the following.
\begin{theorem}
	If $M$ is a noncompact manifold of finite type and $f$ is a Morse-Bott function of type \eqref{c1} or \eqref{c2}. Then the flow category
	 $\cC_f$ is proper and has a local system, such that the cohomology is $H^*(M;\R)$.
\end{theorem}

\subsubsection{The Gysin exact sequence}
Let $M$ be $n$-dimensional manifold of finite type. Assume $f$ is a Morse-Bott function on $M$ and when $M$ is noncompact, $f$ is one of the two admissible types \eqref{c1} or \eqref{c2}. Let $g$ be a metric, such that $(f,g)$ be a locally standard Morse-Bott-Smale pair. Therefore we have a (proper) flow category $\cC_f = \{C_i,\cM_{i,j}\}$. Let $\pi:E\to M$ be a oriented $k$-sphere bundle. Then  $\pi^*f$ is a Morse-Bott function on $E$ with critical manifolds $\{\pi^{-1}(C_i)\}$. We pick a metric $g_F$ on the fibers of $E$, i.e.\ a metric only defined on the subbundle of fiber directions $T^vE$ of $TE$. Fixing a connection of $TE = T^vE\oplus T^hE$, then $g_F$ can be understood as a semi-positive bilinear form on $TE$ vanishing on $T^hE$ and $g_F + \pi^*g$ is a metric on $E$. It can be verified directly a gradient flow line $\widetilde{\gamma}$ of $(\pi^*f, g_F+\pi_*g)$ is a parallel lift of a gradient flow line $\gamma$ of $(f,g)$.  Hence $(\pi^*f, g_F+\pi_*g)$  is again a Morse-Bott-Smale pair and the induced flow category $\cC_{\pi^*f}$ is given by 
$$\Obj(\cC_{\pi^*f}) = \{E_i:=\pi^{-1}(C_i)\}, \quad \Mor(\cC_{\pi^*f}) = \{\cM^E_{i,j} = s_{i,j}^*E_i\}.$$
The source map is the natural map and the target map is given by the parallel transportation along flow lines in $\cM_{i,j}$. As a consequence, we have an oriented $k$-sphere bundle $\cC_{\pi^*f} \to \cC_{f}$. The flow morphisms and flow homotopies defined in the previous discussions can be lifted to the sphere bundle level by the same parallel transportation construction. Therefore the induced Gysin exact sequence is independent of the function $f$. In particular, one may choose $f$ to be constant, hence the Gysin exact sequence will become the usual Gysin exact sequence by Proposition \ref{prop:classGysin}. Therefore we have the following isomorphism of long exact sequences.
\begin{theorem}
	Let $M$ be $n$-dimensional manifold of finite type and $\pi:E \to M$ a $k$-sphere bundle. Suppose $f$ is an admissible Morse-Bott function on $M$, then we have the following isomorphic long exact sequences.
	$$
	\xymatrix@C=3em{
		\ldots \ar[r] & H^i(\cC_f) \ar[r]\ar[d]& H^i(\cC_{\pi^*f}) \ar[r]\ar[d] & H^{i-k}(\cC_f) \ar[rr]\ar[d] & & H^{i+1}(\cC_f) \ar[r]\ar[d] & \ldots\\
		\ldots \ar[r] & H^i(M) \ar[r]^{\pi^*} & H^i(E) \ar[r]^{\pi_*} & H^{i-k}{M} \ar[rr]^{\wedge(-1)^{\dim C+1}e} & & H^{i+1}(M) \ar[r]  & \ldots}
	$$
\end{theorem}

%% file: s9.tex
\section{Transversality by Polyfold Theory}\label{poly}
With the theory on flow categories developed in the previous sections. The remaining problem is to get flow categories in applications, i.e.\ we need to solve the transversality problems. For this purpose, we will adopt the polyfold theory developed by Hofer-Wysocki-Zehnder \cite{HWZ1,HWZ3,hofer2010integration, HWZ2,hofer2017polyfold}. This section outlines some ideas on combining our construction with the polyfold theory, details will appear in a future work. 
\subsection{Polyflow categories}
The main result of \S \ref{morsebott} is that for any oriented flow category, we can construct a well-defined cochain complex up to homotopy. If we want to write down a representative cochain complex of the homotopy class, we need to fix defining data $\Theta$. In applications, take Hamiltonian Floer cohomology as an example, the flow category is the zero sets of some sc-Fredholm sections over a family of polyfolds \cite{wehrheim2012fredholm}. A natural idea is that we replace every manifold $\cM_{i,j}$ in the flow category by strong polyfold bundle $W_{i,j}\to Z_{i,j}$ with an sc-Fredholm section $\kappa_{i,j}$, such that all $W_{i,j}\to Z_{i,j}, \kappa_{i,j}$ are organized just like a flow category.  When all $\kappa_{i,j}$ are transverse to $0$, then $\kappa_{i,j}^{-1}(0)$ defines a flow category. In this case, we expect to assign a well-defined cochain complex to such system of polyfolds up to homotopy. When we need to write down an explicit representative cochain complex for the homotopy class, we need to fix a family of perturbations that are compatible with category structure and defining data (on $C_i$), which does not depend on the perturbation. We first give a preliminary definition of such system.
\begin{definition}\label{def:polyflow}
	A \textbf{polyflow category} is a small category $\cZ$ with following properties.
	\begin{enumerate}
	    \item\label{P1}The object space $\Obj(\cZ)= C:=\sqcup_{i\in \Z} C_i$ is the disjoint union of manifolds $C_i$, such that each connected component of $C_i$ is a manifold of finite type (Definition \ref{def:finite}).
	    \item\label{P2} The morphism space $\Mor(\cZ)=Z$ is a polyfold. The source and target maps $s,t:Z\to C$ are sc-smooth. Let $Z_{i,j}$ denote $(s\times t)^{-1}(C_i\times C_j)$. 
	    \item\label{P3} $Z_{i,i}\simeq C_i$ (i.e.\ the identity morphisms), $Z_{i,j}=\emptyset$ for $j<i$, and $Z_{i,j}$ is a polyfold for $j>i$.
		\item\label{P4} The fiber product $Z_{i_0,i_1}\times_{i_1} Z_{i_1,i_2}\times_{i_2}\ldots \times_{i_{k-1}} Z_{i_{k-1},i_k}$ is cut transversely, for all increasing sequence of $i_0<i_1<\ldots<i_k$.
		\item\label{P5} The composition $m:Z_{i,j}\times_j Z_{j,k}\to Z_{i,k}$ is an $sc$-smooth injective map into the boundary of $Z_{i,k}$. Moreover, $\partial Z_{j,k} = \cup_{i<j<k} m(Z_{i,j}\times_j Z_{j,k})$ and $d(x) + d(y) + 1 = d(m(x,y))$ for $(x,y)\in Z_{i,j}\times_j Z_{j,k}$, where $d$ is the degeneracy index \cite[Definition 2.4.1]{hofer2017polyfold}. When restricted to any stratum of fixed degeneracy index, $m$ is a local sc-diffeomorphism to a stratum with a fixed degeneracy index.
		\item\label{P6} There are strong polyfold bundles $W_{i,j}\to Z_{i,j}$ and sc-Fredholm sections $\kappa_{i,j}$, such that both bundles and sections are compatible with $m$, i.e.\ $m^*W_{i,k}|_{Z_{i,j}\times_j Z_{j,k}}=W_{i,j}\times W_{j,k}$ and $\kappa_{i,k}|_{m(Z_{i,j}\times_j Z_{j,k})}=m(s_{i,j},s_{j,k})$.
		\item\label{P7} For every compact set $K\cap C_j$, $\kappa_{i,j}^{-1}(0)\cap t_{i,j}^{-1}(K)$ is compact. 
	\end{enumerate}
\end{definition}
\begin{remark}
	A few remarks on Definition \ref{def:polyflow} are in order.
	\begin{enumerate}
		\item Condition \eqref{P4} can be replaced by a more convenient condition that $(s\times t)|_{Z_{i,j}}$ are submersions. Then condition \eqref{P4} follows from \cite{filippenko2018polyfold}.
		\item The index $\ind s_{i,j}$ plays the role of $m_{i,j}$. Orientation structures defined in \S \ref{s5} can be generalized to polyflow categories, such that orientation structures are enough to give coherent orientations or local systems on flow categories from perturbations in Claim \ref{claim:trans}.
		\item Condition \eqref{P5}  is stronger than Condition \eqref{F4} of Definition \ref{def:flow}. When we define operators from a flow category, we use integration and Stokes' theorem. Hence an almost identification on the boundary is enough. However, in the polyflow category, we need to perturb $Z_{i,j}$ inductively in a coherent way, which requires a finer identification of all the boundary and corner structures. 
	\end{enumerate}
\end{remark}
When all sections $\kappa_{i,j}$ are transverse to $0$, the zero sets form a proper flow category. Hence our goal is to find a family of $\sc^+$ perturbations $\tau_{i,j}$, such that $s_{i,j}+\tau_{i,j}$ is transverse in general position and consistent with the composition $m$. The consistency depends on the combinatorics of the problem in general. In the case of polyflow category, the combinatorics are relatively simple and we expect to have a perturbation scheme.
\begin{claim}\label{claim:trans}
There exist coherent perturbations $\tau_{i,j}$, such that $\kappa_{i,j}+\tau_{i,j}$ is transverse to $0$ and in general position \cite[Definition 5.3.9]{hofer2017polyfold}.
\end{claim}
\begin{remark}
	The claim does not hold when there are inner symmetries that we want to preserve. To be more precise,  assume we have a strong polyfold bundle $W \to Z$ with two submersive evaluation maps $s,t:Z\to C$. Let $\kappa:Z \to W$ be a Fredholm section. When $\dim C > 0$, given any transverse perturbation $\tau:Z\to W$, it is not necessarily true that $(\tau,\tau)$ is a transverse perturbation to $(\kappa,\kappa)$ on the fiber product $Z_t\times_sZ$. In fact, it is possible that there is no transverse perturbation to $(\kappa,\kappa)$ on $Z_t\times_sZ$ in the form of $(\tau,\tau)$ for a perturbation $\tau:Z\to W$. Such phenomena can appear in a polyflow category, e.g.\ we may have $C_i=C_j=C_k$, $W_{i,j}=W_{j,k}$ and $\kappa_{i,j}=\kappa_{j,k}$. If we require $\tau_{i,j}=\tau_{j,k}$, then we run into this problem. In applications, e.g,\ Hamiltonian Floer cohomology,  we see such phenomenon when the Novikov coefficient has to be used. The requirement of symmetry in perturbations guarantees the cochain complex is a module over the Novikov field. In the $S^1$-Morse theory case, such phenomenon also causes problems (a.k.a.\ self-gluing) in the homotopy argument. The homotopy argument can be viewed as a Morse-Bott problem with critical manifolds copies of $\R$. In these two explicit examples, special methods can be adopted to overcome the challenge. 
	In the most general case, under certain assumptions\footnote{Basically, we require a collar neighborhood near the boundaries and corners of polyfolds, such assumptions are satisfied in all known examples.} of the polyflow category, we can actually perturb the source and target maps consistently to destroy all the inner symmetries. We will discuss this in detail in our future work.
\end{remark}

Although the polyfold perturbation only produces weighted branched suborbifolds as the transverse zero sets,  it causes no problem,  since the convergence results, i.e.\ Lemma \ref{conv1} and  \ref{conv2}, are local in nature. The only thing we need about $\cM_{i,j}$ is Stokes' theorem, which was proven in \cite{hofer2010integration}. Thus all the proofs in \S \ref{morsebott} apply to the weighted branched suborbifold case. Similar to Definition \ref{def:polyflow}, we can define polyflow morphisms and polyflow homotopies by replacing the manifolds by polyfolds with sc-Fredholm sections. Once the perturbation scheme is given for those structures, we can generate flow morphisms and flow homotopies.

\begin{remark}
	To generalize the identity flow category (Definition \ref{identitymor}) to the polyfold case, the naive construction of multiplying by an interval does not work, because product with an interval does not have the right boundary and corner structures to apply an inductive perturbation scheme. However, there is a more natural construction of the identity (poly)flow category, which has the right boundary and corner structures. The construction is closely related to the geometric realization of the category, which will be discussed in a future work.  
\end{remark}

The enrichment to polyflow categories causes more choices, i.e.\ the choice of perturbation. We would like to have the cohomology independent of the perturbation. Such invariance can be proven using the identity polyflow category or a homotopy argument. 
\begin{claim}\label{claim:ind}
Let $\cZ$ be a polyflow category with orientation structures. If there is no inner symmetry\footnote{Or collar neighborhood assumptions on the polyfolds hold, if there are inner symmetries.}, then we can associate it with a Morse-Bott cochain complex $(\BC(\cZ),d_{\BC})$, such that the homotopy type of the cochain complex is independent of defining data and $\sc^+$ perturbations.
\end{claim}

\subsection{Equivariant theory}
In \S \ref{equi}, we discuss the equivariant theory when the flow category is equipped with a group action. However, requiring $G$ symmetry on the flow category is equivalent to requiring $G$-equivariant transversality on the background polyflow category. Since $G$-equivariant transversality is often obstructed, the construction in \S \ref{equi} can not be applied directly. However, the construction in \S \ref{equi} can be generalized to polyflow categories. Hence we can apply the Borel construction on the level of polyfolds.
\begin{definition}
Let $\cZ$ be a polyflow category. A compact Lie group $G$ acts on $\cZ$ iff $G$ acts on $C_i$ and $W_{i,j}\to Z_{i,j}$ in the sense of \cite[Definition 3.66]{quotient}, such that all sc-Fredholm sections $\kappa_{i,j}$  and the structure maps $s,t,m$ are $G$-equivariant. 
\end{definition}
Assume $G$ acts a polyflow category $\cZ$. If we fix an approximation ${E_n}$ of $EG$, then we can form a sequence of polyflow categories $\cZ\times_G E_n$ by the quotient construction in \cite{quotient}. Using the identity polyflow morphism and the construction in \S \ref{equi},  we have a sequence of polyflow morphisms connecting different $\cZ\times_G E_n$. Then we have a directed system in the ``category" of polyflow categories. We can get an inverse system of cochain complexes by applying Claim \ref{claim:ind}, then the equivariant cochain complex will be the homotopy limit of such inverse system.  Details of the construction will appear in a future work.

%% file: app.tex
\appendix
\section{Convergence}\label{conv}
This section proves the convergence results used in \S \ref{morsebott}. We will see that transversality of fiber products is not only natural from the polyfold point of view as explained in \S \ref{poly} but also essential in proving the convergence results, especially Lemma \ref{conv2}. 

\subsection{Thom class}\label{thom} We review the construction of Thom classes in \cite[\S 6]{bott2013differential}. Let $\pi: E\to M$ be an oriented vector bundle with a metric over an oriented manifold. The fiber $F$, the base manifold $M$ and the total space $E$ are oriented in the manner of $[M][F]=[E]$. If $S(E)$ denotes the sphere bundle of $E$, then we can find  a  form $\psi$ (angular form) on $S(E)$, such that the integration over each fiber is $1$, and $\rd \psi=-\pi^* e$, where $e$ is the Euler class of the sphere bundle. Then we pick smooth functions $\rho_n:\R^+\to \R$, such that $\rho_n$ is increasing, supported in $[0,\frac{1}{n}]$ and is $-1$ near $0$.
\begin{figure}[H]
	\begin{center}
		\begin{tikzpicture}
		\draw[->] (-1,0) -- (7,0) node[right] {$r$};
		\draw[->] (0,-2.5) -- (0,1) node[right] {$\rho_n(r)$};
		\draw (0,-2) .. controls (2,-2) and (1,0) .. (3,0);
		\node[left] at (0,-2) {$-1$};
		\node[above] at (3,0) {$\frac{1}{n}$};
		\node[above] at (6,0) {$1$};
		\end{tikzpicture}
		\caption{Graph of $\rho_n$}
	\end{center}
\end{figure}
Then $\rd(\rho_n\psi)$ defines a form on $\R^+\times S(E)$, and it is $\pi^*e$ on an open neighborhood of $\{0\}\times S(E)$, thus $\rd(\rho_n \psi)$ is a lift of some form on $E$, i.e.\ $\rd(\rho_n \psi)=p^* \delta^n$ for $\delta^n \in \Omega^*(E)$, where $p$ is the natural map $\R^+\times S(E)\to E$.  Such $\delta^n$ is a Thom class of $\pi:E\to M$. The next lemma asserts $\delta^n$ actually represent the zero section not only in the cohomological sense, but also in a stronger sense of currents. Let $\delta_M$ denote the Dirac current of the zero section, i.e.\ $\delta_M(\alpha)=\int_{M}i^*\alpha$, for $\alpha \in \Omega^*(E)$, where $i:M\to E$ is the zero section.

\begin{lemma}[Lemma \ref{conv4}]\label{lemma:appcon}
	$\delta^n\to \delta_M$ in the sense of currents, i.e.\ $\forall \alpha \in \Omega^*(E)$, we have
	$$\lim_{n\to \infty} \int_E \alpha \wedge \delta^n \to \delta_M(\alpha).$$
\end{lemma}

\begin{proof}
	Let $F\simeq \R^n$ be a fiber of the bundle, since $\delta^n$ is compactly supported, then the integration over a fiber is
	$$\int_F \delta^n= \int_{F-\{0\}} \delta^n = \int_{(0,\infty)\times S^{n-1}} p^* \delta^n=\int_{[0,\infty)\times S^{n-1}} p^*\delta^n =\int_{[0,\infty)\times S^{n-1}} \rd(\rho_n\psi)=-\int_{\{0\}\times S^{n-1}}\psi=1.$$
	Let $\alpha\in \Omega^*(E)$, since $\int_F \delta^n=1$ for any fiber $F$, then 
	$$\int_E  \pi^*i^*\alpha \wedge \delta^n=\int_M \int_{F} \pi^*i^* \alpha\wedge \delta^n=\int_M i^*\alpha.$$ Therefore, it is enough to show 
	$$\lim_{n\to \infty} \int_E (\alpha-\pi^*i^*\alpha)\wedge \delta^n=0.$$
	We will prove this by partition of unity. Let $\{U_i\}$ be an open cover of $M$ and $\{p_i\}$  a partition of unity subordinated to this open cover. We fix trivializations over each $U_i$. Then over $\pi^{-1}(U_i)$ we have
	$$(\pi^*p_i)\cdot (\alpha-\pi^*i^* \alpha)=\sum f^{I,J}\rd x^I \wedge \rd y^J,$$
	where $x$ are the coordinates in $U_i$ and $y$ are the coordinates in the fiber direction. $I,J$ are sets of indices. Since $\alpha$ and $\pi^*i^*\alpha$ are the same when restricted to the zero section, therefore $\lim_{r\to 0}f^{I,\emptyset}=0$, where $r$ is the radius coordinate in the fiber direction. Hence we have
	\begin{eqnarray}\lim_{n\to \infty}\int_{\pi^{-1}(U_i)}f^{I,\emptyset}\rd x^I \wedge \delta^n
	& = & \lim_{n\to \infty}\int_{\R^+\times S^{n-1} \times U_i}f^{I,\emptyset} \rd x^I \wedge  \rd \rho_n\wedge \psi -f^{I,\emptyset} \rd x^I \wedge \rho_n \pi^* e \nonumber\\
	&=  & \lim_{n\to \infty}\int_{0}^{\frac{1}{n}}\int_{S(E)|_{U_i}} \pm f^{I,\emptyset}\rd \rho_n  \wedge \psi\wedge \rd x^I\pm\rho_n f^{I,\emptyset} \pi^* e\wedge \rd x^I. \nonumber
    \end{eqnarray}
    Since $|\rho_n|$ is supported in $[0,\frac{1}{n}]$ and bounded by $1$, $\int_{0}^{\frac{1}{n}} |\rd \rho_n|=1$, $\lim_{r\to 0}f^{I,\emptyset}=0$, and $\psi$ is bounded on $S(E)$,  we have 
    $$\lim_{n\to\infty}\int_{\pi^{-1}(U)}f^{I,\emptyset}\rd x^I \wedge \delta^n=0.$$
    
	When the cardinality $|J|$ of $J$ is  greater than $0$,  using the spherical coordinate in the fiber direction, we have $\rd y^I=Cr^{|J|} \rd\theta^J+Dr^{|J|-1} \rd r\wedge \rd \theta^{J-1}$, where $\rd \theta^J,\rd \theta^{J-1}$ are forms on the sphere of degree $|J|$ and $|J|-1$ and $C,D$ are bounded functions. Because $\rd \rho_n$ is purely in $\rd r$ direction,  then we have 
	\begin{eqnarray}
	& &\lim_{n\to \infty}\int_{\pi^{-1}(U_i)}f^{I,J}\rd x^I\wedge \rd y^J \wedge \delta^n \nonumber\\
	& = & \lim_{n\to \infty} \int_{0}^{\frac{1}{n}}\int_{S(E)|_{U_i}} f^{I,J}Cr^{|J|} \rd x^I \wedge \rd\theta^J \wedge \rd\rho_n\wedge\psi \nonumber \\
	& & -\lim_{n\to \infty} \int_{0}^{\frac{1}{n}}\int_{S(E)|_{U_i}} f^{I,J}Cr^{|J|}\psi \wedge \rd x^I \wedge \rd\theta^J \wedge \rho_n\pi^* e  \label{eqn:2}\\
	& & -\lim_{n\to \infty} \int_{0}^{\frac{1}{n}}\int_{S(E)|_{U_i}} f^{I,J}Dr^{|J|-1}\wedge\psi \wedge \rd x^I \wedge \rd r \wedge \rd\theta^{J-1}\wedge \rho_n\pi^* e. \label{eqn:3}
	\end{eqnarray}
	Because $f^{I,J}, C$ are bounded, $\rd \theta^J$ is bounded on $S(E)$, $\int_{0}^{\frac{1}{n}} |\rd \rho_n|=1$ and $\lim_{r\to 0} r^{|J|}=0$, thus the first term limits to zero. Since everything in \eqref{eqn:2} and \eqref{eqn:3} are uniformly bounded and $\rho_n$ is supported in $[0,\frac{1}{n}]$, thus \eqref{eqn:2} and \eqref{eqn:3} have limit zero.
    Hence we have
	$$\lim_{n\to \infty} \int_{\pi^{-1}(U_i)} \pi^*p_i (\alpha_i-\pi^*i^*\alpha) \wedge \delta^n=0.$$
	Therefore we have
	$$\lim_{n\to \infty} \int_{E} (\alpha_i-\pi^*i^*\alpha)\wedge \delta^n= \lim_{n\to \infty} \sum_i \int_{\pi^{-1}(U_i)} (\pi^*p_i) \cdot (\alpha_i-\pi^*i^*\alpha)\wedge \delta^n=0.$$
\end{proof}

Next we will show that Lemma \ref{lemma:appcon} is preserved under pullback, when transversality conditions are met.

\begin{lemma}\label{conv3}
	Let $M$ be a compact manifold with boundaries and corners and $E\to B$ a vector bundle over a \emph{closed} manifold $B$. If $f:M \to E$ is transverse to $B$ and we orient $f^{-1}(B)$ by $[f^{-1}(B)]f^*[E]=[TM|_{f^{-1}(B)}]$, then for $\alpha \in \Omega^*(C)$, we have
	$$\lim_{n\to \infty}\int _{M} \alpha \wedge f^*\delta^n=\int_{f^{-1}(B)}\alpha|_{f^{-1}(B)}.$$
\end{lemma}
\begin{proof}
	We fix a tubular neighborhood $\pi: N\to f^{-1}(B)$. For $n$ big enough, $f^*\delta^n$ is the Thom class of $f^{-1}(B)$, i.e.\ $f^*\delta^n$ has integration $1$ along each fiber. This is because the fiber $F$ of $f^{-1}(B)$ is diffeomorphic to a submanifold homotopic to a fiber of $E\to B$ though the map $f$. Since $\delta^n$ is closed and has a small enough support, Stokes' theorem implies  $\int_F f^*\delta_n=\int_{f(F)}\delta^n=\int_{\text{fiber of } E}\delta^n=1$.  Then by the same argument in the proof of Lemma \ref{lemma:appcon}, we only need to prove $$ \lim_{n\to \infty} \int_N (\alpha-\pi^*i^*\alpha)\wedge f^*\delta^n=0.$$
	\begin{figure}[H]
		\centering\includegraphics[width=6in]{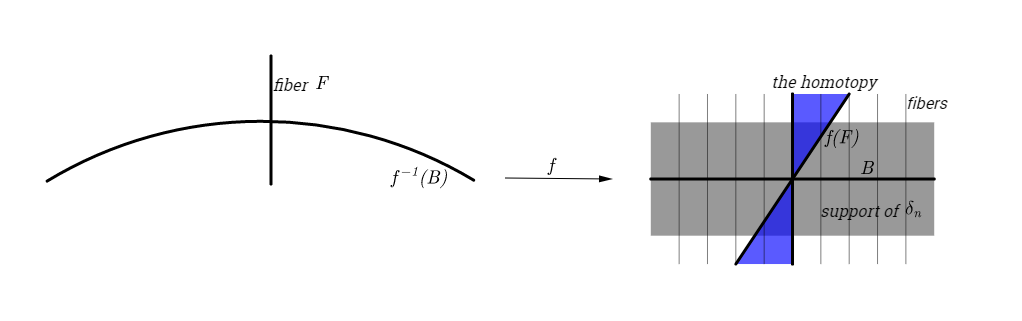}
		\caption{Pullback of Thom classes}
	\end{figure}
	Picking a point $x\in f^{-1}(B)$, then by the implicit function theorem, we can find a local chart of $x$ in $M$, 
	$$\phi: \R_+^k\times \R^n\to M,\quad \phi(0)=x$$
	and local trivialization of $E\to B$ over $f(x)$,
	$$\psi: \R^i\times \R^j\to E, \quad \psi(0,0)=(f(x),0),$$ 
	such that $$\psi^{-1}\circ f\circ \phi(x_1,\ldots, x_k, y_1,\ldots, y_{n-j}, z_{n-j+1},\ldots, z_{n})= (f_1,\ldots, f_i, z_{n-j+1},\ldots, z_{n}),$$
	where $f_1,\ldots, f_i$ are functions of $x_*,y_*,z_*$. We replace the $z$ coordinates by spherical coordinates.  With such coordinates, the pullback of $\rd (\rho_n \psi)$ through $f$ is $\rd(\rho_n \tilde{\psi})$, where $\tilde{\psi}$ is defined on $\R_+^k\times \R^{n-j}\times S^{j-1}\times \R_+$ and uniformly bounded. Then the proof of Lemma \ref{lemma:appcon} can be applied to prove the claim.
\end{proof}

\subsection{Proof of Lemma \ref{conv1} and Lemma \ref{conv2}}
Following the discussion in \S \ref{pertdata}, we pick representatives $\{\theta_{i,a}\}$ of a basis of  $H^*(C_i)$ in  $\Omega^*(C_i)$ to get a quasi-isomorphic embedding, 
$$H^*(C_i)\to \Omega^*(C_i)$$
and the dual basis is denoted by $\{\theta_{i,a}^*\}$, such that $\{\theta_{i,a}^*\}$ are in the image of the chosen embedding $H^*(C_i)\to \Omega^*(C_i)$ and $(-1)^{\dim C_i |\theta_i^b|}\int_{C_i} \theta_{i,a}^*\wedge \theta_{i,b}=\delta_{ab}$. Then by Proposition \ref{prop:cohomologus}, the Thom class $\delta^n_i=\rd (\rho_n \psi_i)$ of $\Delta_i\subset C_i\times C_i$ and $\sum_a \pi_1^*\theta_{i,a}\wedge \pi_2^*\theta_{i,a}^*$ both represent the Poincar\'e dual of the diagonal $\Delta_i$, thus they are cohomologous in $\Omega^*(C_i\times C_i)$. Therefore we can find $f^n_i$, such that $\rd f^n_i=\delta^n_i-\sum_a \pi_1^*\theta_{i,a}\wedge \pi_2^*\theta_{i,a}^*$ and 
\begin{equation}\label{prim}
f^n_i-f^m_i=(\rho_n-\rho_m)\psi_i.
\end{equation}
Thus the support of  $f^n_i-f^m_i$ converges to a measure zero set. To show the convergence results, i.e.\ Lemma \ref{conv1} and \ref{conv2}, we need to show $f^n_i$ is uniformly bounded. The uniform boundedness is not necessarily true in $C_i\times C_i$, but it holds if we use spherical coordinates near the diagonal $\Delta_i$. To apply spherical coordinates in an intrinsic way, we  recall blow-ups of real submanifolds from  \cite[Chapter 5]{analysis}.
\begin{definition}[{\cite[Chapter 5]{analysis}}]
	Let $p: E\to M$ be vector bundle over a  manifold, then the blow-up of $E$ along $M$ is the following manifold $Bl_ME$,
	$$Bl_ME=\{(v,e)\in E\times S(E)| p(v)=p(e), \exists a\ge 0, \text{ such that } ae=v\},$$
	where $S(E)$ is the sphere bundle $(E\backslash \{0_M\})/\mathbb{R}^+$, and $0_M$ is the zero section of $E\rightarrow M$.
\end{definition}

Then one can define a blow-up of a submanifold $N\subset M$ in the sense of Definition \ref{def:submanifold} by blowing up $N$ in the tubular neighborhood, which is identified with the normal bundle. Moreover, the blow-up of the submanifold $N$ can be described intrinsically as follows,
$$Bl_NM:=(M\backslash N) \cup S (TM/TN|_N),$$
where $S (TM/TN|_N)$ is the sphere bundle of the quotient bundle (normal bundle) $TM/TN|_N$ over $N$. The smooth structure on $Bl_NM$ can be given using an auxiliary tubular neighborhood and it is independent of the tubular neighborhood  \cite[Chapter 5]{analysis}. The natural map $Bl_NM\rightarrow M$ is smooth and is a diffeomorphism up to measure zero sets. Thom classes $\delta_i^n=\rd (\rho_n \psi_i)$ can be pulled back to $Bl_{\Delta_i}C_i\times C_i$, and the primitives $\rho_n\psi_i$ is uniformly bounded on $Bl_{\Delta_i} C_i\times C_i$. 

\begin{figure}[H]
	\centering\includegraphics[width=2in]{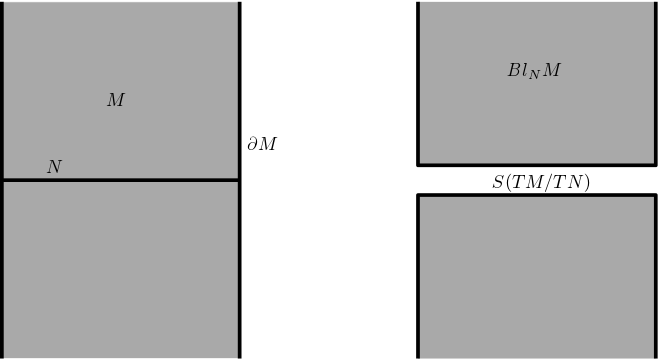}
	\caption{Blow up one submanifold}
\end{figure}
Using this intrinsic description, when a smooth map $f:M\times N\to C\times C$ is transverse to the diagonal $\Delta$, there is a natural map $Bl_\Delta f:Bl_{M\times_C N} M\times N\to Bl_{\Delta_C} C\times C$ induced by $f:M\times N\to C\times C$. Moreover, we have the following commutative diagram of smooth maps,
$$\xymatrix{Bl_{M\times_C N} M\times N \ar[d] \ar[r]^{\quad Bl_\Delta f} & Bl_{\Delta_C} C\times C \ar[d] 	
	\\ M\times N \ar[r]^f & C\times C.}$$   

If we have two submanifolds $N_1,N_2$ of $M$, such that $N_1$ is transverse to $N_2$ in the sense of Definition \ref{def:trans}, then we can blow up $N_1,N_2$. It was shown in \cite[Chapter 5]{analysis} that the order of blowing up does not influence the diffeomorphism type. The resulted blow-up is denoted by $Bl_{N_1,N_2}M$. Similarly, if we have a sequence of submanifolds $N_1,N_2,\ldots, N_k$, such that $(\cap_{\alpha\in A} N_\alpha)$ is transverse to $N_{\beta}$ for $\beta \notin A$, then we can blow up all $N_1,\ldots, N_k$. The diffeomorphism type does not depend on the order and let $Bl_{N_1,\ldots,N_k}M$ denote the blow-up.
\begin{figure}[H]
	\centering\includegraphics[width=4in]{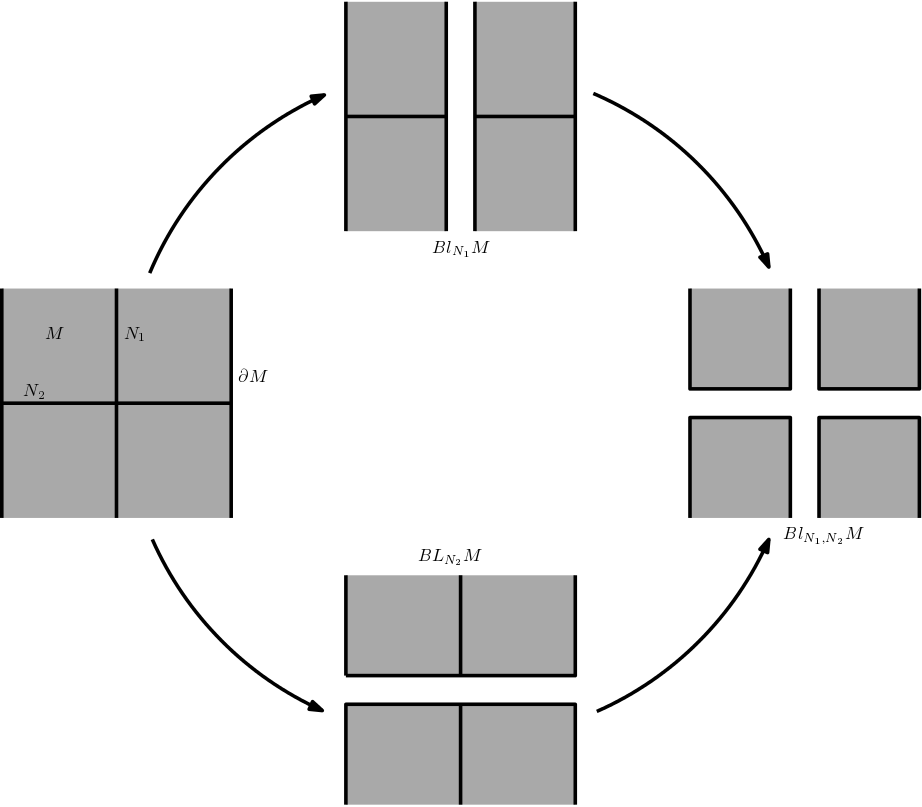}
	\caption{Blow up two submanifolds}
\end{figure}
In the setting of a flow category (Definition \ref{def:flow}), any fiber product
$\cM_{i_0,i_1}\times_{i_1}\cM_{i_1,i_2}\times_{i_2}\ldots \times_{i_{n}}\cM_{i_n,i_{n+1}}$ is cut out transversely in $\cM_{i_0,i_1}\times\cM_{i_1,\i_2}\times\ldots \times\cM_{i_n,i_{n+1}}$. Therefore 
$N_j:=\cM_{i_0,i_1}\times\cM_{i_1,i_2}\times\ldots \times \cM_{i_{j-1},i_j}\times_{i_j}\cM_{i_j,i_{j+1}}\times \ldots \times\cM_{i_n,i_{n+1}}$ are submanifolds in the product $\cM_{i_0,i_1}\times\cM_{i_1,i_2}\times\ldots \times\cM_{i_n,i_{n+1}}$, such that $(\cap_{\alpha\in A} N_\alpha)$ is transverse to $N_{\beta}$ for $\beta \notin A$. Then we have a blow-up $Bl_n:=Bl_{N_1,\ldots,N_n}\cM_{i_0,i_1}\times\cM_{i_1,i_2}\times\ldots \times\cM_{i_n,i_{n+1}}$ and we have similar commutative diagrams of smooth maps,
\begin{equation}\label{cd}\xymatrix{Bl_n \ar[d] \ar[r]^{Bl_{\Delta_i}(t\times s)} & Bl_{\Delta_{i_j}} C_{i_j}\times C_{i_j} \ar[d] 	
	\\ \cM_{i_0,i_1}\times\cM_{i_1,i_2}\times\ldots \times\cM_{i_n,i_{n+1}} \ar[r]^{\qquad \qquad \quad t\times s} & C_{i_j}\times C_{i_j}.}
\end{equation}

Now, we start to prove Lemma \ref{conv1} and \ref{conv2}. The definition of $\cM^{s,k}_{i_1,\ldots,i_r}[\alpha, f^{n_1}_{s+i_1},\ldots,f^{n_r}_{s+i_r}, \gamma]$ is \eqref{pair1}.

\begin{lemma}[Lemma \ref{conv1}]\label{lemma:appconv1}
	For every $\alpha\in \Omega^*(C_v),\gamma \in \Omega^*(C_{v+k})$ and any defining data $\Theta$, we have $\displaystyle \lim_{n\to \infty} \cM^{v,k}_{i_1,\ldots,i_r}[\alpha, f^{n}_{v+i_1},\ldots,f^{n}_{v+i_r}, \gamma]$ exists.
\end{lemma}
\begin{proof}
	Since $\cM^{v,k}_{i_1,\ldots,i_r}[\alpha, f^{n_1}_{v+i_1},\ldots,f^{n_r}_{v+i_r}, \gamma]$ is an integration over $\cM^{v,k}_{i_1,\ldots,i_r}$, and $\bigcup_{j} \cM^{v,k}_{i_1,\ldots, \overline{i_j},\ldots, i_r}$ is a measure zero set in $\cM^{v,k}_{i_1,\ldots,i_r}$, thus we can restrict the integral to $\cM^{v,k}_{i_1,\ldots,i_r}-\bigcup_{j} \cM^{v,k}_{i_1,\ldots, \overline{i_j},\ldots, i_r}$ to get the same value. 
	
	We have a blow-up $Bl_r\cM^{v,k}_{i_1,\ldots,i_r}$, by blowing up all $\cM^{v,k}_{i_1,\ldots, \overline{i_j},\ldots, i_r}, 1\le j \le r$.   The primitives $f_i^n$ can be lifted to $Bl_{\Delta_i}C_i\times C_i$ and $t\times s$ can be lifted to the blow-ups to $Bl_{\Delta_i}(t\times s)$. We define $Bl_r\cM^{v,k}_{i_1,\ldots,i_r}[\alpha, f^{n_1}_{v+i_1},\ldots, f^{n_r}_{v+i_r},\gamma]$ to be the integration of the wedge product of pullbacks of $\alpha, f^{n_1}_{v+i_1},\ldots, f^{n_r}_{v+i_r},\gamma$ to $Bl_r\cM^{v,k}_{i_1,\ldots,i_r}$. Because $Bl_r\cM^{v,k}_{i_1,\ldots,i_r}$ and $\cM^{v,k}_{i_1,\ldots,i_r}-\bigcup_{j} \cM^{v,k}_{i_1,\ldots, \overline{i_j},\ldots, i_r}$ also differ by a measure zero set, by commutative diagram \eqref{cd}, we have
	$$Bl_r\cM^{v,k}_{i_1,\ldots,i_r}[\alpha, f^{n}_{v+i_1},\ldots, f^{n}_{v+i_r},\gamma]=\cM^{v,k}_{i_1,\ldots,i_r}[\alpha, f^{n}_{v+i_1},\ldots, f^{n}_{v+i_r},\gamma].$$ 
	Since
	\begin{equation}\label{differ}
	\begin{aligned}Bl_r\cM^{v,k}_{i_1,\ldots,i_r}[\alpha, f^{n}_{v+i_1},\ldots, f^{n}_{v+i_r},\gamma]-Bl_r\cM^{v,k}_{i_1,\ldots,i_r}[\alpha, f^{m}_{v+i_1},\ldots, f^{m}_{v+i_r},\gamma] \\ =\sum_{p=1}^r Bl_r\cM^{v,k}_{i_1,\ldots,i_r}[\alpha, f^{m}_{v+i_1},\ldots,f^{m}_{v+i_{p-1}}, f^{n}_{v+i_p}-f^{m}_{v+i_p}, f^{n}_{v+i_{p+1}},\ldots f^{n}_{v+i_r},\gamma].
	\end{aligned}
	\end{equation}
	Note that $f_{v+i_j}^{n}$ are uniformly bounded over $Bl_{\Delta_{v+i_j}}C_{v+i_j}\times C_{v+i_j}$ for every $n\in \N$ and the support of $f_{v+i_j}^{n}-f_{v+i_j}^{m}$ converges to a measure zero set in $Bl_{\Delta_{v+i_j}}C_{v+i_j}\times C_{v+i_j}$ when $n,m\to \infty$. By commutative diagram \eqref{cd}, the pullbacks of $f_{v+i_j}^{n}$ to $Bl_r \cM^{v,k}_{i_1,\ldots,i_r}$ have the same properties. Thus \eqref{differ} implies the  convergence.
\end{proof}

\begin{lemma}[Lemma \ref{conv2}] 
	For an oriented flow category $\cC$ and any defining data, we have 
	\begin{equation*}
	\lim_{n\rightarrow \infty}  \mathcal{M}^{v,k}_{i_1,\ldots,i_r}[\alpha, f_{v+i_1}^{n},\ldots ,\delta^{n}_{v+i_p},\ldots, f_{v+i_r}^{n}, \gamma]
	= (-1)^* \lim_{n\rightarrow \infty}\mathcal{M}^{v,k}_{i_1,\ldots,i_{p-1},\overline{i_p}, i_{p+1},\ldots, i_r}[\alpha, f_{v+i_1}^{n},\ldots, f_{v+i_r}^{n},\gamma],\end{equation*}
	where $*=(|\alpha|+m_{v,v+i_p})c_{v+i_p}$. 
\end{lemma}
\begin{proof}
    The limit $\displaystyle\lim_{n\to\infty}\cM^{v,k}_{i_1,\ldots,i_{p-1},\overline{i_p}, i_{p+1},\ldots, i_r}[\alpha, f_{v+i_1}^{n},\ldots, f_{v+i_r}^{n},\gamma]$ exists by the same argument used in the proof of Lemma \ref{lemma:appconv1}. To prove the limit on the left-hand-side exists, we can blow up everything except for $\cM^{v,k}_{i_1,\ldots \overline{i_p},\ldots, i_r}$ to get $Bl_{r-1}$. Assume the pullback of $\delta^n_{v+i_p}$ is supported in tubular neighborhood $U$ of $\cM^{v,k}_{i_1,\ldots \overline{i_p},\ldots, i_r}$ in $\cM^{v,k}_{i_1,\ldots,i_r}$, then $U$ can be lifted to the blow-up $Bl_{r-1}$ to get $Bl_{r-1}U$ (c.f. figure 7). For simplicity, we suppress the wedge and pullback notation, then we have  
	$$ \lim_{n\to \infty}\int_{\cM^{v,k}_{i_1,\ldots,i_r}} \alpha f_{v+i_1}^{n}\ldots \delta^{n}_{v+i_p}\ldots f_{v+i_r}^{n} \gamma  
	=\lim_{ n\to \infty}\int_{Bl_{r-1}U} \alpha f_{v+i_1}^{n}\ldots \delta^{n}_{v+i_p}\ldots f_{v+i_r}^{n} \gamma.$$
	\begin{figure}[H]
		\centering\includegraphics[width=4in]{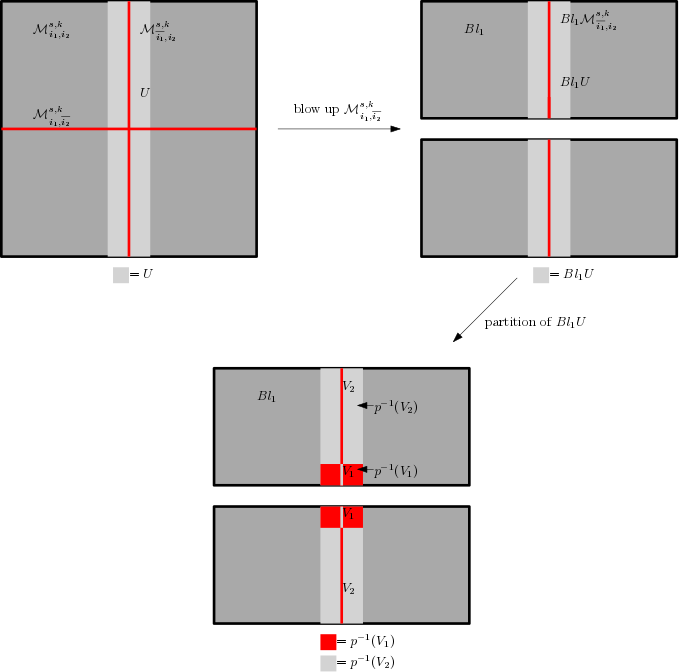}
		\caption{$r=2,p=1$ case}\label{fig:partition}
	\end{figure}
	Let $Bl_{r-1}\cM^{v,k}_{i_1,\ldots \overline{i_p},\ldots, i_r} $ denote the lift of $\cM^{v,k}_{i_1,\ldots \overline{i_p},\ldots, i_r}$ in $Bl_{r-1}$, then $Bl_{r-1}U$ is still a tubular neighborhood of  $Bl_{r-1}\cM^{v,k}_{i_1,\ldots \overline{i_p},\ldots, i_r}$. Let $p: Bl_{r-1}U\to Bl_{r-1}\cM^{v,k}_{i_1,\ldots \overline{i_p},\ldots, i_r}$ denote the projection of the tubular neighborhood. Then we can divide $Bl_{r-1}\cM^{v,k}_{i_1,\ldots \overline{i_p},\ldots, i_r}$ into two parts $V_1,V_2$, such that $V_1$ is a small open set containing the blow-up domain, and $V_2$ is the complement. Then $p^{-1}(V_1)$ and $p^{-1}(V_2)$ are partitions of $Bl_{r-1}U$ (see figure \ref{fig:partition}). Using the same local coordinate in Lemma \ref{conv3}, if we integrate the fiber direction of the tubular neighborhood, because $f^n_{v+i_1},\ldots, f^n_{v+i_{p-1}}, f^n_{v+i_{p+1}},\ldots, f^n_{v+i_r}$ are uniformly bounded over $Bl_{r-1}$, we have
	\begin{equation}\label{small} |\int_{p^{-1}(V_1)} \alpha f_{v+i_1}^{n}\ldots \delta^{n}_{v+i_p}\ldots f_{v+i_r}^{n} \gamma|\le K \vol(V_1),
	\end{equation}
	where $K$ is a constant.	
	Over $p^{-1}(V_2)$, the pullback of $f_{v+i_1}^{n},\ldots, f_{v+i_{p-1}}^{n},f_{v+i_{p+1}}^{n},\ldots, f_{v+i_r}^{n}$ do not change for $n$ large enough, because $p^{-1}(V_2)$ stays away from the blown-up area. Thus the only thing that varies over $p^{-1}(V_2)$ is $\delta^{n}_{v+i_p}$. Note that
	$$
		\lim_{n\to \infty} \int_{p^{-1}(V_2)} \alpha f^n_{v+i_1}\ldots \delta^n_{v+i_p}\ldots f^n_{v+i_r}\gamma  =  	(-1)^{(|\alpha|+\sum_{j<p}(c_{v+i_j}-1)) c_{v+i_p}}\lim_{n\to \infty} \int_{p^{-1}(V_2)} \delta^n_{v+i_p} \alpha f^n_{v+i_1}\ldots f^n_{v+i_r}
	$$
	By Definition \ref{oridef}, we have the orientation relation on $\cM^{v,k}_{i_1,\ldots \overline{i_p},\ldots, i_r}\supset V_2$ satisfies
	$$[N_{v+i_p}][\cM^{v,k}_{i_1,\ldots,\overline{i_p},\ldots i_r}]=(-1)^{(\sum_{j\le p} m_{v+i_{j-1},v+i_j})c_{v+i_p}}[\cM^{v,k}_{i_1,\ldots,i_r}]$$
	Combining with Lemma \ref{conv3} and  that
	$$(|\alpha|+\sum_{j<p}(c_{v+i_j}-1)) c_{v+i_p}+(\sum_{j\le p} m_{v+i_{j-1},v+i_j})c_{v+i_p}=(|\alpha|+m_{v,v+i_p})c_{v+i_p} \mod 2,$$
	we can conclude that 
	\begin{equation}\label{dirac} 
	\lim_{n\to \infty} \int_{p^{-1}(V_2)} \alpha f^n_{v+i_1}\ldots \delta^n_{v+i_p}\ldots f^n_{v+i_r}\gamma=\lim_{n\to\infty}(-1)^{(|\alpha|+m_{v,v+i_p})c_{v+i_p}}\int_{V_2} \alpha f^n_{v+i_1}\ldots f^n_{v+i_{p-1}}f^n_{v+i_{p+1}}\ldots f^n_{v+i_r}\gamma.
\end{equation}
	By \eqref{small} and \eqref{dirac}, because $V_1$ can be arbitrarily small, we have $\displaystyle\lim_{n\to \infty}  \cM^{s,k}_{i_1,\ldots,i_r}[\alpha, f_{v+i_1}^{n},\ldots ,\delta^{n}_{v+i_p},\ldots, f_{v+i_r}^{n}, \gamma]$ exists. Since $f^n_{v+i_1}, \ldots, f^n_{v+i_{p-1}}, f^n_{v+i_{p+1}},\ldots, f^n_{v+i_r}$ are uniformly bounded over $Bl_{r-1} \cM^{v,k}_{i_1,\ldots, \overline{i_p},\ldots, i_r}$, we have 
	\begin{equation}\label{small1}
	|\int_{V_1}\alpha f^n_{v+i_1}\ldots f^n_{v+i_{i_{p-1}}}f^n_{v+i_{p+1}}\ldots f^n_{v+i_r}\gamma|<K'\vol(V_1).
	\end{equation}
	Since $Bl_{r-1}\cM^{v,k}_{i_1,\ldots, \overline{i_p},\ldots, i_r}$ and $\cM^{v,k}_{i_1,\ldots,\overline{i_p}, \ldots, i_r}$ differ by a measure zero set, we have 
	\begin{eqnarray}
	\int_{\cM^{v,k}_{i_1,\ldots,\overline{i_p}, \ldots, i_r}} \alpha f^n_{v+i_1}\ldots f^n_{v+i_{p-1}} f^n_{v+i_{p+1}} \ldots f^n_{v+i_r}\gamma &= & 	\int_{Bl_{r-1}\cM^{v,k}_{i_1,\ldots, \overline{i_p},\ldots, i_r}} \alpha f^n_{v+i_1}\ldots f^n_{v+i_{p-1}} f^n_{v+i_{p+1}} \ldots f^n_{v+i_r}\gamma \nonumber \\
	&=& 	\int_{V_1\cup V_2} \alpha f^n_{v+i_1}\ldots f^n_{v+i_{p-1}} f^n_{v+i_{p+1}} \ldots f^n_{v+i_r}\gamma. \label{sum}
	\end{eqnarray}
	Therefore by \eqref{small}, \eqref{dirac}, \eqref{small1} and \eqref{sum}, we have 
	$$\begin{aligned}\left|\lim_{n\to\infty}\left(\int_{\cM^{v,k}_{i_1,\ldots,i_r}} \alpha f^n_{v+i_1}\ldots \delta^n_{v+i_p}\ldots f^n_{v+i_r}\gamma-(-1)^{(|\alpha|+m_{v,v+i_p})c_{s+i_p}}\int_{\cM^{v,k}_{i_1,\ldots,\overline{i_p},\ldots, i_r}} \alpha f^n_{v+i_1}\ldots f^n_{v+i_r}\gamma \right) \right| \\
	\le (K+K')\vol(V_1).\end{aligned}$$
	Since $V_1$ can be arbitrarily small, thus we have
	$$\lim_{n\to\infty}\int_{\cM^{v,k}_{i_1,\ldots,i_r}} \alpha f^n_{v+i_1}\ldots \delta^n_{v+i_p}\ldots f^n_{v+i_r}\gamma=\lim_{n\to\infty}(-1)^{(|\alpha|+m_{v,v+i_p})c_{v+i_p}}\int_{\cM^{v,k}_{i_1,\ldots,\overline{i_p},\ldots, i_r}} \alpha f^n_{v+i_1}\ldots f^n_{v+i_r}\gamma.$$
\end{proof}

\section{Proof of Proposition \ref{prop:split}}\label{app:kunneth}
\begin{proposition}[Proposition \ref{prop:split}]
	Let $\pi:E \to C$ be an oriented $k$-sphere bundle over an oriented closed manifold. Let $A=A^*$ be the reduction on $\Omega^*(E)$ built from the  discussion after the statement of Theorem \ref{thm:gysin} (in particular, we choose $\psi_i$ such that $\rd \psi_i=0$ if $k$ is even). And $T$ is the closed form in $\pi_1^*A\wedge \pi_2^*A$ representing the diagonal by the definition of reduction. Then there exists approximations $\delta^{E,n}$ of the Dirac current of the diagonal $\Delta_E$ such that the following holds.
	\begin{enumerate}
		\item There exist forms $f^{E,n}$ on $E\times E$, such that $$\rd f^{E,n} = \delta^{E,n} - T$$
		\item Lemma \ref{conv1} and \ref{conv2} hold for $f^{E,n}$, in particular, the construction in \S \ref{ss:reduction} works for $f^{E,n}$.
		\item\label{split} Let $\pi_{1,2}$ denote the projection $E\times E \to C\times C$, then $f^{E,n}$ can be written as sums of differential forms in the form of $(\pi_{1,2}^*\alpha) \wedge \beta$ with $\alpha \in \Omega^*(C\times C)$ and $\deg(\beta) \le k$, i.e.\ the fiber degree of $f^{E,n}$ is at most $k$. In other words, if $v_1,\ldots,v_{k+1}$ are $k+1$ vertical vectors in $T_p(E\times E)$ for $p\in C\times C$, then $f^{E,n}(v_1\wedge \ldots \wedge v_{k+1}\wedge \ldots )=0$.
	\end{enumerate}
\end{proposition}
\begin{proof}	
	Let $\delta^{C,n}$ be the Thom classes of $\Delta_C \subset C\times C$ constructed using \eqref{eqn:delta} with the angular form $\Psi_C$ of the normal bundle. Let $\delta^{S^k,n}$ be the Thom classes of $\Delta_{E} \subset E\times_C E$ constructed using \eqref{eqn:delta}. We define $p:U\to E\times_CE$ to be a projection in a tubular neighborhood $U$ of $E\times_CE$ in $E\times E$. Then $\pi_{1,2}(U)$ is a tubular neighborhood of $\Delta_{C}\subset C\times C$. By the same argument in Lemma \ref{conv4}, $\lim_{n\to \infty}\pi_{1,2}^*\delta^{C,n}\wedge p^*\delta^{S^k,n}$ is the Dirac current of the diagonal $\Delta_E \subset E\times E$. Since for $n\gg 0$, the support of $\pi_{1,2}^*\delta^{C,n}$ is contained in $U$, $\pi_{1,2}^*\delta^{C,n}\wedge p^*\delta^{S^k,n}$ are cohomologous to each other and represent Thom classes of $\Delta_E$ for $n\gg 0$.
	
	Next, we show that we can find the desired primitives $f^{E,n}$. Let $p_1,p_2: E\times_CE \to E$ be the projections to the first and second component respectively. Then $(-1)^k p_1^*\psi + p_2^*\psi$ is a closed form on $E\times_CE$ because $\rd((-1)^kp_1^*\psi + p_2^*\psi) = (-1)^{k+1}q^*e - q^*e = 0$ for any $k$ (when $k$ is even, $e$ is zero by assumption), where $q:E\times_CE \to C$ is the projection. We claim $(-1)^kp_1^*\psi + p_2^*\psi$ is cohomologous $\delta^{S^k,n}$, i.e.\ there are $f^{S^k,n}\in \Omega^{k-1}(E\times_CE)$, such that
	\begin{equation}\label{eqn:exact_rel}
	    \delta^{S^k,n} -(-1)^kp_1^*\psi - p_2^*\psi = \rd f^{S^k,n}.
	\end{equation}

    We first proceed assuming \eqref{eqn:exact_rel}. Let $\Pi_1,\Pi_2$ be the two projections $E\times E \to E$. Note that $(-1)^k \Pi_1^*\psi+\Pi_2^*\psi$ is not closed on $U$. We have $\rd ((-1)^k \Pi_1^*\psi+\Pi_2^*\Psi)=\pi_{1,2}^*((-1)^{k+1}e\otimes 1-1\otimes e)$, and the closed form $(-1)^{k+1}e\otimes 1-1\otimes e$ is zero on $\Delta_C$, hence   $(-1)^{k+1}e\otimes 1-1\otimes e$  is exact on $\pi_{1,2}(U)$. Therefore we can find $h\in \Omega^k(\pi_{1,2}(U))$ with $h|_{\Delta_C}=0$ and $(-1)^k \Pi_1^*\psi+\Pi_2^*\Psi+\pi_{1,2}^*h$ is closed on $U$. Since $((-1)^k \Pi_1^*\psi+\Pi_2^*\psi+\pi_{1,2}^*h)|_{E\times_CE}=(-1)^k p_1^*\psi + p_2^*\psi$, we know that there exists $g\in \Omega^{k-1}(U)$, such that 
	$$p^*\left((-1)^kp_1^*\psi+p_2^*\psi\right)- (-1)^k\Pi_1^*\psi - \Pi_2^*\psi = \rd g+\pi_{1,2}^*h,$$
	Now we make any extension of $h$ to $C\times C$, the extended form is still denoted by $h$. We have
	\begin{eqnarray*}
	\pi_{1,2}^*\delta^{C,n}\wedge p^*\delta^{S^k,n} & = & \pi_{1,2}^*\delta^{C,n}\wedge p^*((-1)^kp_1^*\psi + p_2^*\psi) + \pi_{1,2}^*\delta^{C,n}\wedge p^*\rd f^{S^k,n} \\
	& = &\pi_{1,2}^*\delta^{C,n}\wedge \left( \left((-1)^k\Pi_1^*\psi +\Pi_2^*\psi+\pi_{1,2}^*h\right)+ \pi_{1,2}^*\delta^{C,n}\wedge (\rd g+p^*\rd f^{S^k,n})\right)
	\end{eqnarray*}
	If we write $\rd f^{C,n} = \delta^{C,n} -\sum_a \pi_1^*\theta_a\wedge \pi_2^*\theta^*_a$, then we have the following
	\begin{eqnarray*}
     \pi_{1,2}^*\delta^{C,n}\wedge ((-1)^k\Pi_1^*\psi +\Pi_2^*\psi+\pi_{1,2}^*h) & = & 
     \pi_{1,2}^*(\rd f^{C,n} +\sum_a \pi_1^*\theta_a\wedge \pi_2\theta_a^*)\wedge \left((-1)^k\Pi_1^*\psi +\Pi_2^*\psi+\pi_{1,2}^*h\right) \\
     & = & \rd (\pi_{1,2}^*f^{C,n}\wedge \left((-1)^k\Pi_1^*\psi +\Pi_2^*\psi+\pi_{1,2}^*h)\right)\\
     &  & + \pi_{1,2}^*(\sum_a \pi_1^*\theta_a\wedge \pi_2\theta_a^*)\wedge \left((-1)^k\Pi_1^*\psi +\Pi_2^*\psi+\pi_{1,2}^*h\right)\\
     & & +(-1)^{\dim C} \pi_{1,2}^*f^{C,n}\wedge \rd\left((-1)^k\Pi_1^*\psi +\Pi_2^*\psi+\pi_{1,2}^*h)\right)
	\end{eqnarray*}
	Now let $S_n$ denote the last two terms. Then $S_n-S_m=0$ for $n,m\gg 0$ as $\supp (f^{C,n}-f^{C,m})\subset (\pi_{1,2})(U)$ and $\rd\left((-1)^k\Pi_1^*\psi +\Pi_2^*\psi+\pi_{1,2}^*h)\right)$ is zero on $U$. 
	
	Next, recall from Lemma \ref{lemma:reduction}, $A=A^*$ has basis in the form of 
	$$\la \pi^*\theta_1,\ldots, \pi^*\theta_k, \xi_{1}:=\pi^*\theta^*_{1}\wedge \psi-\pi^*\eta_1,\ldots,\xi_k:=\pi^*\theta^*_{k}\wedge \psi-\pi^*\eta_k \ra$$
	such that the dual basis is $\la \xi_1,\ldots, \xi_k,\pi^*\theta_1,\ldots,\pi^*\theta_k\ra$ up to sign. It is easy to check that $S_n-T$ is in the form of $\pi_{1,2}^*\alpha$ for $\alpha \in \Omega^*(C\times C)$. Since $T$ and $\pi_{1,2}^*\delta^{C,n}\wedge p^*\delta^{S^k,n}$  both represent $\Delta_E$, we have $S_n-T$ is exact. Therefore $\alpha$ is closed class in $\Omega^*(C\times C)$, such that $[\pi_{1,2}^*\alpha]=0$. As a consequence $[\alpha]=\sum_i ([\alpha_i]\wedge [e])\otimes [\beta_i] + \sum_j [\alpha'_j]\otimes ([\beta'_j]\wedge [e])$ on cohomology. Therefore there exist $\alpha_0,\alpha_1,\alpha_2\in \Omega^*(C\times C)$, such that 
	$$S_n-T=\pi_{1,2}^*\alpha = \rd (\pi_{1,2}^*\alpha_0\wedge \Pi_1^*\Psi+\pi_{1,2}^*\alpha_1\wedge \Pi_2^*\Psi+\pi_{1,2}^*\alpha_2)=\rd w.$$
	As a consequence we can take $\delta^{E,n}:=\pi_{1,2}^*\delta^{C,n}\wedge p^*\delta^{S^k,n}$ and 
	\begin{equation}\label{eqn:fe}
	f^{E,n}:=w + f^{C,n}\wedge ((-1)^k\Pi_1^*\psi +\Pi_2^*\psi+\pi_{1,2}^*h)+ (-1)^{\dim C}(\pi_1\times \pi_2)^*\delta^{C,n}\wedge (g+p^*f^{S^k,n}),
	\end{equation}
	Since $f^n_C$ and $f^n_{S^k}$ can be chosen such that \eqref{eqn:diff} holds, the Lemma \ref{conv1} and \ref{conv2} hold for $f^{E,n}$ using the same argument in Appendix \ref{conv}. By \eqref{eqn:fe}, the third property of the proposition holds, since each component has the property.
\end{proof}

\begin{proof}[Proof of \eqref{eqn:exact_rel}]
Note that $p_1:E\times_CE\to E$ is also a sphere bundle (it is the pullback of the bundle $\pi:E\to C$ through $\pi$ itself). Then $p_2^*\psi$ is the angular form of $p_1$. After fixing representatives  $\{\alpha_1,\ldots,\alpha_m\}$ of a basis of $H^*(E)$, we get a reduction of $\Omega^*(E\times_CE)$ by the same argument after the statement of Theorem \ref{thm:gysin} as follows,
$$B=B^*=\la p_1^*\alpha_1,\ldots,p_1^*\alpha_m,  \chi_1:=p_1^*\alpha_1\wedge p_2^*\psi-p_1^*f_1,\ldots,\chi_m:=p_1^*\alpha_m\wedge p_2^*\psi-p_1^*f_m\ra.$$
Since $\rd$ is closed on $B$, and the cohomology is the cohomology of $E\times_C E$ (since it is a reduction). It suffices to prove that  for any $\beta \in B$
$$\int_{E\times_CE} \beta \wedge ((-1)^kp_1^*\psi+p_2^*\psi) = \int_{\Delta_E} \beta.$$
If $\beta = p_1^*\alpha_i$, then 
$$\int_{E\times_CE} p_1^*\alpha_i \wedge ((-1)^kp_1^*\psi+p_2^*\psi) = \int_{E\times_CE} (-1)^{k}p_1^*(\alpha_i\wedge \psi) + \int_{E\times_CE} p_1^*\alpha_i\wedge p_2^*\psi.$$
The first term is clearly zero, the second term is $\int_E \alpha_i=\int_{\Delta_E}(p_1^*\alpha_i)|_{\Delta_E}$ by integration along the fiber of $p_1$.
If $\beta = \chi_i=p_1^*\alpha_i\wedge p_2^*\psi-p_1^*f_i$. Then by the same argument as above, we have
$$\int_{E\times_CE} \chi_i\wedge ((-1)^kp_1^*\psi+p_2^*\psi) = \int_{E\times_C E} (p_1^*\alpha_i\wedge p_2^*\psi)\wedge ((-1)^kp_1^*\psi+p_2^*\psi)+\int_{\Delta E} (p_1^*f_i)|_{\Delta_E}.$$
The first term is 
$\int_{E\times_C E} p_1^*\alpha_i\wedge p_1^*\psi\wedge p_2^*\psi = \int_{E}\alpha_i\wedge \psi=\int_{\Delta_E}(p_1^*\alpha_i \wedge p_2^*\psi)|_{\Delta_E}$. Hence the claim follows.
\end{proof}